\documentclass[reqno]{amsart}
\usepackage{amsthm,amsmath,amsfonts,amssymb,amscd,mathrsfs,diagrams,graphics,cases,amsbsy}
\usepackage{txfonts}
\usepackage[pdftex,breaklinks=true]{hyperref}

\newcommand{\vtr}{\vartheta^{rig}} 

\newcommand{\BG}{\B G} 

\newcommand{\coker}{\operatorname{coker}}
\newcommand{\im}{\operatorname{im}}
\newcommand{\rk}{\operatorname{rank}}
\newcommand{\aut}{\operatorname{Aut}}

\newcommand{\wt}{\operatorname{wt}}
\newcommand{\diag}{\operatorname{diag}}

\newcommand{\mop}{J^{-1}}

\newcommand{\irightarrow}{\stackrel{\sim}{\longrightarrow}}

\newcommand{\B}{{\mathscr B}}
\newcommand{\D}{{\mathscr D}}

\newcommand{\ch}{{\mathscr H}}

\newcommand{\M}{{\mathscr M}}
\newcommand{\N}{{\mathscr N}}
\renewcommand{\O}{{\mathscr{O}}}

\newcommand{\U}{{\mathscr U}}
\newcommand{\V}{{\mathscr V}}
\newcommand{\W}{{\mathscr W}}

\newcommand{\C}{{\mathbb{C}}}
\newcommand{\Q}{{\mathbb{Q}}}
\newcommand{\R}{{\mathbb{R}}}
\newcommand{\Z}{{\mathbb{Z}}}

\newcommand{\eul}{{e}} 

\newcommand{\w}{{\mathbf w}}
\newcommand{\wit}{{\D}} 
\newcommand{\z}{{\mathbf z}}
\newcommand{\bu}{{\mathbf u}}
\newcommand{\bphi}{\boldsymbol{\phi}}

\newcommand{\comb}{\operatorname{Comb}}
\newcommand{\ind}{\operatorname{index}}

\newcommand{\bpat}{\bar{\partial}}
\newcommand{\pat}{{\partial}}
\newcommand{\Vd}{{V_{deform,\sigma}}}
\newcommand{\Vm}{{V_{map,\sigma}}}

\newcommand{\Vmt}{{V^t_{map,{\sigma}}}}
\newcommand{\Vr}{{V_{resol,\sigma}}}

\newcommand{\papp}{{{\mathbf \phi}_{app,y,\zeta}}}
\newcommand{\uapp}{{{\mathbf u}_{app,y,\zeta}}}

\newcommand{\ts}{\tilde{s}}
\newcommand{\bb}{{\mathbf b}}
\newcommand{\cC}{\mathscr C} 
\newcommand{\LL}{{\mathscr L}} 
\newcommand{\MM}{\overline{\M}} 
\newcommand{\WW}{\overline{\W}}
\newcommand{\bgamma}{\protect{\boldsymbol{\gamma}}} 
\newcommand{\grp}{G} 
\newcommand{\frkc}{\mathfrak{C}} 

\newcommand{\hGamma}{\hat{\Gamma}}
\newcommand{\barGamma}{\overline{\Gamma}}
\newcommand{\bone}{\mathbf{e}}

\newcommand{\bvkappa}{{\boldsymbol{\varkappa}}}

\newcommand{\MMr}{\WW^{\mathrm{rig}}}
\newcommand{\MMrs}{\WW^{\mathrm{s}}}
\newcommand{\st}{st} 
\newcommand{\so}{so} 
\newcommand{\str}{st^{\mathrm{rig}}} 
\newcommand{\chat}{\hat{c}}

\newtheorem{thm}{Theorem}[subsection]
\newtheorem{lm}[thm]{Lemma}
\newtheorem{prop}[thm]{Proposition}
\newtheorem{crl}[thm]{Corollary}

\theoremstyle{definition}
\newtheorem{rem}[thm]{Remark}

\newtheorem{df}[thm]{Definition}
\newtheorem{ex}[thm]{Example}
\theoremstyle{remark}

\begin{document}

\date{\today}

\title[Witten equation]{The Witten equation and its virtual fundamental cycle}
\author{Huijun Fan}
\thanks{Partially Supported by NSFC 10401001, NSFC 10321001, and NSFC 10631050}
\address{School of Mathematical Sciences, Peking University, Beijing 100871, China}
\email{fanhj@math.pku.edu.cn}
\author{Tyler Jarvis}
\thanks{Partially supported by the National Science Foundation and the Institut Mittag-Leffler (Djursholm, Sweden)}
\address{Department of Mathematics, Brigham Young University, Provo, UT 84602, USA}
\email{jarvis@math.byu.edu}
\author{Yongbin Ruan}
\thanks{Partially supported by the National Science Foundation and
the Yangtz Center of Mathematics at Sichuan University}
\address{Department of Mathematics, University of Michigan Ann Arbor, MI 48105 U.S.A
and the Yangtz Center of Mathematics at Sichuan University,
Chengdu, China} \email{ruan@umich.edu}

\maketitle

\begin{abstract}
We study a system of nonlinear elliptic PDEs associated with a
quasi-homogeneous polynomial. These equations were proposed by Witten
as the replacement for the Cauchy-Riemann equation in
the singularity (Landau-Ginzburg) setting. We introduce a perturbation to
the equation and construct a virtual cycle for the moduli space of
its solutions. Then, we study the wall-crossing of the deformation
of the virtual cycle under perturbation and match it to classical
Picard-Lefschetz theory. An extended virtual cycle is obtained for
the original equation. Finally, we prove that the extended virtual
cycle satisfies a set of axioms similar to those of Gromov-Witten
theory and $r$-spin theory.
\end{abstract}

\tableofcontents

\section{Introduction}

This is the second installment in a series of papers devoted to
the mathematical theory of the Witten equation and its applications.
The Witten equation is a system of nonlinear elliptic PDEs
associated to a quasi-homogeneous polynomial $W$. It has the simple
form
$$\bar{\partial}u_i+\overline{\frac{\partial W}{\partial
u_i}}=0,$$
where $W$ is a quasi-homogeneous polynomial (called the \emph{super-potential}), and $u_i$ is
    interpreted as a section of an appropriate orbifold line bundle
    on a Riemann surface $\cC$.
    Some simple examples are
    \begin{description}
    \item[($A_r$-case)] $\bar{\partial}u+\bar{u}^r=0.$
    \item[($D_n$-case)]
    $\bar{\partial}u_1+n\bar{u}^{n-1}_1+\bar{u}_2^2=0,
    \bar{\partial}u_2+2\bar{u}_1\bar{u}_2=0.$
    \item[($E_7$-case)]
    $\bar{\partial}u_1+3\bar{u}^2_1+\bar{u}_2^3=0,
    \bar{\partial}u_2+3\bar{u}_1\bar{u}^2_2=0$.
    \end{description}
    To understand the importance of the Witten equation, we have to go back to the famous
    Landau-Ginzburg/Calabi-Yau correspondence \cite{Mar}, \cite{VW}, \cite{Wi4}. Most  known examples
    of Calabi-Yau three-folds are constructed as a hypersurface or a complete intersection of
    toric varieties. The defining equations of these three-folds have a natural interpretation in terms of
    Landau-Ginzburg/singularity theory. To compute Gromov-Witten invariants is basically to
    solve the Cauchy-Riemann equation for maps from Riemann surfaces. It is generally difficult to solve the Cauchy-Riemann
    equation on a nonlinear space such as a Calabi-Yau manifold. If we compare the Cauchy-Riemann
    equation on a Calabi-Yau manifold to the anti-self-dual equation in Donaldson theory,
    the Witten equation plays the role of the Seiberg-Witten equation. It takes a while to set up, but we
    expect that its invariants will be much easier to compute.  
    
A program was launched
by the authors in 2001 to establish a mathematical theory of these
equations and
 their various consequences in geometry. We have achieved some initial success, although a lot
 more remains to be done. In \cite{FJR2}, we formulated the
algebraic-geometric foundations of the theory, its main properties
in terms of axioms, as well as the applications to mirror symmetry.
In the same paper and in
\cite{FJR3} we solve the famous Witten conjecture for integrable hierarchies associated to the simple singularities $D_n$ and $E_{6,7,8}$ (the conjecture was proved in \cite{FSZ} for $A_n$ singularities). The paper \cite{FJR4} shows that our theory for the $A_{r-1}$ singularity agrees with the $r$-spin
curve theory of \cite{JKV1}.

In this article, we establish the analytic foundations of the theory.

\subsection{Subtleties of the Witten equation}

A casual investigation of the Witten equation reveals that the
Witten equation is much more subtle than its simple appearance
would suggest. Suppose $u_i\in \Omega^0(\LL_i)$. A simple
computation shows
$$\bar{\partial}u_i\in \Omega^{0,1}(\LL_i), \quad \overline{\frac{\partial W}{\partial
u_i}}\in \Omega^{0,1}_{log}(\overline{\LL}^{-1}_i),$$ where log means a
$(0,1)$ form with possible singularities of order $\leq 1$. Namely,
the Witten equation has singular coefficients! 
This is a fundamental
phenomenon for the application of  the Witten equation. A Calabi-Yau
manifold has cohomological information which we expect to discover
from the Witten equation. At  first sight, it seems to be difficult
to reproduce it from the Witten equation, but it has become clear that the
singularity of the Witten equation is the key to producing such
cohomological data in our theory. Unfortunately, the appearance of
singularities makes the Witten equation very difficult to study. This is
the main reason it took us such a long time to construct the theory.

Another subtle issue is the fact that we need an isomorphism
$\overline{\LL}^{-1}_i\cong \LL_i$ for the two terms of the Witten
equation living in the same space. The required isomorphism can be
obtained by a choice of metric. A nontrivial fact is that such a
metric can be constructed uniformly from a metric of the underlying
Riemann surface. Then the question is, which metric should we choose
on the Riemann surface? We should mention that a different choice of
metric often leads to a completely different looking theory,
including a different dimension for its moduli space. Apparently,
there is no physical guidance for the correct metric to 
choose. We have experimented with both smooth and cylindrical metric
near marked points, and now we understand that \emph{both} choices are
important for the theory. In \cite{FJR1} we studied the theory for
the smooth metric. In this article, our main choice is a cylindrical
metric.

\subsection{Outline of the construction and results}

Throughout this paper we will use the notation and definitions of \cite{FJR2}, with the notable exception that what was called \emph{Ramond} in \cite{FJR2} will be called \emph{broad} here, and what was called \emph{Neveu-Schwarz} will be called \emph{narrow} here.

 Let $\WW_{g,k}(\gamma_1,
\dots, \gamma_k)$ be the moduli space of $W$-curves with orbifold structure $\gamma_i$ at the marked point $z_i$ (see
Section 2). Technically, it is more convenient for us to work on the
moduli space $\MMr_{g,k}(\gamma_1,
\dots, \gamma_k)$ of rigidified $W$-structures. These are the background data for the Witten
equation. One can show by the Witten Lemma that in the case that all the coefficients are non-singular (we call this case \emph{narrow}), the Witten equation has only
the trivial solution. Unfortunately, it is impossible to
obtain a meaningful theory using only the zero solution. 
Hence, as in the study of singularity theory, we need to consider the perturbed Witten
equation given by the perturbed superpotential $W+W_0$, where $W_0$ is a
linear perturbation term such that $W_{\gamma}+W_{0\gamma}$ is a
holomorphic Morse function for every $\gamma$. Here  $W_{\gamma}$ and $W_{0\gamma}$ are the restrictions of $W$ and $W_0$ to the fixed point set
$(\C^N)^{\gamma}$ (also denoted $\C^{N_{\gamma}}$). 
The crucial part of the analysis is to
show that a solution of the perturbed Witten equation converges to a
critical point of $W_{\gamma_i}+W_{0\gamma_i}$. This enables us to
construct a moduli space
$$
\MMr_{g,k}(\kappa_{j_1}, \dots, \kappa_{j_k}),$$
where $\kappa_{j_i}$ is a critical point of
$W_{\gamma_i}+W_0|_{\C^{N}_{\gamma_i}}$. We call $W_0$ {\em
strongly regular} if (i) $W_{\gamma_i}+W_{0\gamma_i}$ is
holomorphic Morse, and (ii) the critical values of
$W_{\gamma_i}+W_{0\gamma_i}$ have distinct imaginary parts. Our
first important result is Theorem \ref{thm-stro-kura}, which can
be written in the following simple form:
 \begin{thm}
 If $W_0$ is strongly regular, then $\MMr_{g,k}(\kappa_{j_1}, \dots, \kappa_{j_k})$
 is compact and has a virtual fundamental cycle $[\MMr_{g,k}(\kappa_{j_1}, \dots, \kappa_{j_k})]^{vir}$
 of  degree
$$
2\left((\chat_W-3)(g-1)+k -\sum_i \iota_{\gamma_i}\right)-\sum_i N_{\gamma_i}.
$$
Here $\chat_W$ is the central charge, $\iota_{\gamma_i}$ is the degree shifting  number
defined in \cite[Def 3.2.3]{FJR2}, and $N_{\gamma_i}$ is the complex dimension of the fixed point locus $(\C^N)^{\gamma_i}\subseteq \C^N$  of $\gamma_i$.
\end{thm}

It turns out to be notionally convenient to map the above virtual
cycle from $H_*(\MMr_{g,k}(\kappa_{j_1}, \dots, \kappa_{j_k}),\Q)$
into 
$H_*(\MMr_{g,k}(\gamma_{j_1}\dots, \gamma_{j_k}),\Q)$, even though
       the former is not a subspace of the latter in any way. This is the
       first step of our construction. We have not yet seen the
       cohomology data we hope for.
        Then, a crucial
       new phenomenon comes into play when we study how the above
       virtual cycle changes when we vary the perturbation. It
       turns out that the above virtual cycle {\it does} depend on the
       perturbation. It will change when $W_0$ fails to be
       strongly regular. The ``wall crossing formula" proved in Theorem \ref{crl-cobo-sing} shows the
       following \emph{quantum Picard-Lefschetz theorem}:

       \begin{thm}
       When $W_0$ varies, $[\MMr_{g,k}(\kappa_{j_1}, \dots, \kappa_{j_k})]^{vir}$
       transforms in the same way as the so called Lefschetz thimble
       $S_{j_i}$ attached to the critical point $k_{j_i}$.
       \end{thm}

The ``wall crossing formula'' of the above virtual cycle can
       be neatly packaged into the following formula. Let
       $\{\tilde{S_i}\}$ be a basis for the space of dual Lefschetz thimbles. To
       simplify the notation, we assume that there is only one
       marked point with the orbifold decoration $\gamma$. Then, the
       wall crossing formula of $[\MMr_{g,1}(\kappa_i)]^{vir}$
       shows precisely that
       $$\sum_j [\MMr_{g,1}(\kappa_j)]^{vir}\otimes
       \tilde{S_j},$$
       viewed as a class in $H_*(\MMr_{g,1}(\gamma), \Q)\otimes
       H_{N_{\gamma}}(\C^{N_{\gamma}}, W^{\infty}_{\gamma},
       \Q)$, is independent of the perturbation. Now, we define
       $$[\MMr_{g,1}(\gamma)]^{vir}=\sum_j [\MMr_{g,1}(\kappa_j)]^{vir}\otimes
       \tilde{S_j}.$$
       The above definition can be generalized with multiple
       marked points in an obvious way. It is clear that
       $$[\MMr_{g,k}(\gamma_1, \dots, \gamma_k)]^{vir}\in
       H_*(\MMr_{g,k}(\gamma_1, \dots, \gamma_k), \Q)\otimes \prod_i H_{N_{\gamma_i}}(\C^{N_{\gamma_i}},
       W^{\infty}_{\gamma_i}, \Q)$$
       has  degree
       $$2\left((\chat_w-3)(g-1)+k -\sum_i \iota_{\gamma_i}\right).$$

 \begin{crl}
    $[\MMr_{g,k}(\gamma_1, \dots, \gamma_k)]^{vir}$ is
    independent of the perturbation $W_0$.
    \end{crl}

Eventually, we want to work on
       $\WW_{g,k}$. In Section~\ref{sec:rigmoduli} we show that the \emph{softening} map $so: \MMr_{g,k}\rTo \WW_{g,k}$ that forgets all the rigidifications is a finite, representable morphsim. We can define
$$
[\WW_{g,k}(\gamma_1, \dots,
\gamma_k)]^{vir}:=\frac{1}{\deg(so)}(so)_*[\MMr_{g,k}(\gamma_1,
\dots, \gamma_k)]^{vir}.
$$
This new virtual cycle is independent of rigidification, which
implies that
    $$[\WW_{g,k}(\gamma_1, \dots, \gamma_k)]^{vir}\in
       H_*(\WW_{g,k}(\gamma_1, \dots, \gamma_k), \Q)\otimes \prod_i H_{N_{\gamma_i}}(\C^{N_{\gamma_i}},
       W^{\infty}_{\gamma_i}, \Q)^G.$$

Similarly, we can define the virtual cycle $[\WW(\Gamma)]^{vir}$
corresponding to a decorated dual graph $\Gamma$ (see \cite[\S2.2.2]{FJR2}).

Besides the quantum Picard-Lefschetz theory, another main aim of
this paper is to prove Theorem 1.0.2 in our paper \cite{FJR2}. This
theorem shows that the virtual cycles constructed in this paper
satisfy a set of axioms similar to the standard Gromov-Witten
theory. We state it below, after the following simple definition.

\begin{df}
For any (decorated) graph $\Gamma$, let $E(\Gamma)$ denote the set of edges of $\Gamma$, let $T(\Gamma)$ denote the set of tails, and let $V(\Gamma)$ denote the set of vertices.

Furthermore, let $D_{\Gamma}$ be defined as follows:
\begin{equation}
D_{\Gamma}:=\chat_W(g-1)+\sum_{j=1}^k \iota_{\gamma_j}.
\end{equation}
\end{df}
Note that when $D_{\Gamma}$ is an integer, then $-D_{\Gamma}$ is the sum of the indices of the $W$-structure bundles $$-D_{\Gamma} = \sum_{i=1}^N \ind(\LL_i) .$$   

\begin{thm}\label{thm:main}
The following axioms are satisfied for $[\WW(\Gamma)]^{vir}$:
\begin{enumerate}
\item \textbf{Dimension:}\label{ax:dimension} If $D_\Gamma$ is not
an integer, 
then $\left[\WW(\Gamma)\right]^{vir}=0$. Otherwise, the cycle
$\left[\WW(\Gamma)\right]^{vir}$ has degree
\begin{equation}\label{eq:dimension}
6g-6+2k-2D_\Gamma -2\#E(\Gamma)=2\left((\chat-3)(1-g) + k -\#E(\Gamma) - \sum_{\tau\in
T(\Gamma)} \iota_{\tau}\right).
\end{equation}
So the cycle lies in $H_r(\WW(\Gamma),\Q)\otimes \prod_{\tau \in
T(\Gamma)}
H_{N_{\gamma_{\tau}}}(\C^N_{\gamma_{\tau}},W^{\infty}_{\gamma_\tau},
\Q),$ where
$$
r:=6g-6+2k-2\#E(\Gamma) -2D  -\sum_{\tau\in T(\Gamma)}N_{\gamma_\tau} =
2\left((\hat{c}-3)(1-g)+k -\#E(\Gamma) -\sum_{\tau\in
T(\Gamma)}\iota(\gamma_{\tau})- \sum_{\tau\in
T(\Gamma)}\frac{N_{\gamma_\tau}}{2}\right).
$$

\item \label{ax:symm}\textbf{Symmetric group invariance}: There is
a natural $S_k$-action on $\WW_{g,k}$ obtained by permuting the
tails.  This action induces an action on homology.  That is, for any
$\sigma \in S_k$ we have
$$\sigma_*: H_*(\WW_{g,k},\Q)\otimes \prod_i
H_{N_{\gamma_{i}}}(\C^N_{\gamma_i}, W^{\infty}_{\gamma_i}, \Q)^G \to
H_*(\WW_{g,k},\Q)\otimes \prod_i
H_{N_{\gamma_{\sigma(i)}}}(\C^N_{\gamma_{\sigma(i)}},
W^{\infty}_{\gamma_{\sigma(i)}}, \Q)^G.$$ For any decorated graph
$\Gamma$, let $\sigma\Gamma$ denote the graph obtained by applying
$\sigma$ to the  tails of $\Gamma$.

We have
\begin{equation}\sigma_*\left[\WW(\Gamma)\right]^{vir} = \left[\WW(\sigma\Gamma)\right]^{vir}.\end{equation}

\item \textbf{Degenerating connected graphs:} Let $\Gamma$ be a
connected, genus-$g$, stable, decorated $W$-graph.

The cycles $\left[\WW(\Gamma)\right]^{vir}$ and
$\left[\WW_{g,k}(\bgamma)\right]^{vir}$ are related by
\begin{equation}
\label{eq:cvgamma} \left[\WW(\Gamma)\right]^{vir}= 
\tilde{i}^*\left[\WW_{g,k}(\bgamma)\right]^{vir},
\end{equation}
where $\tilde{i} : \WW(\Gamma) \to \WW_{g,k}(\bgamma)$ is the
canonical inclusion map. \item \textbf{Disconnected graphs:} Let
$\Gamma =\coprod_{i} \Gamma_i$ be a stable, decorated $W$-graph
which is the disjoint union of connected $W$-graphs $\Gamma_i$. The
classes $\left[\WW(\Gamma)\right]^{vir}$ and
$\left[\WW(\Gamma_i)\right]^{vir}$ are related by
\begin{equation}
\left[\WW(\Gamma)\right]^{vir}= \left[\WW(\Gamma_1)\right]^{vir}
\times \dots \times \left[\WW(\Gamma_d)\right]^{vir}.
\end{equation}
\smallskip

\item \textbf{The Topological Euler class for 
narrow sectors:}  Suppose that all the decorations on tails of
$\Gamma$ are
  \emph{narrow}, meaning that $\C^N_{\gamma_i}=0$, and so we can
  omit $H_{N_{\gamma_{i}}}(\C^N_{\gamma_i}, W^{\infty}_{\gamma_i}, \Q)=\Q$ from our notation.

Consider the universal $W$-structure $(\LL_1,\dots,\LL_N)$ on the
universal curve $\pi:\cC \rTo{}{} \W(\Gamma)$ and the two-term
complex of sheaves
$$\pi_*(|\LL_i|){\rTo} R^1\pi_*(|\LL_i|).$$
There is a family of maps
$$W_i=\frac{\partial W}{\partial x_i}: \pi_*(\bigoplus_j |\LL_j|)
\rTo \pi_*(K\otimes |\LL_i|^*)\cong R^1\pi_*(|\LL_i|)^*.$$ The above
two-term complex is quasi-isomorphic to a complex of vector bundles
\cite{PV}
$$E^0_i\rTo^{d_i} E^1_i$$
such that
$$\ker( d_i)\rTo \coker(d_i)$$
is isomorphic to the original two-term complex.
    $W_i$ is naturally extended (denoted by the same notation) to
    $$\bigoplus_i E^0_i\rTo (E^1_i)^*.$$
    Choosing an Hermitian metric on $E^1_i$ defines an

isomorphism $\bar{E}^{1*}_i\cong E^1_i.$ Define the \emph{Witten map} to be the following
    $$\wit=\bigoplus (d_i+\bar{W}_i):
    \bigoplus_i E^0_i\rTo\bigoplus_i \bar{E}^{1*}_i\cong \bigoplus_i E^1_i.$$
 Let $\pi_j: \bigoplus_i E^j_i\rTo \MM$ be
    the projection map. The Witten map defines a proper section
    (denoted by the same notation)
    of the bundle $\wit: \bigoplus_i E^0_i\rTo \pi^*_0 \bigoplus_i
    E^1_i.$ The above data  defines a topological Euler
    class $\eul(\wit: \bigoplus_i E^0_i\rTo \bigoplus_i
    E^1_i)$. Then,
    $$[\WW(\Gamma)]^{vir}=(-1)^{D}\eul(\wit: \pi^*_0 \bigoplus_i E^1_i\rTo \bigoplus_i
    E^0_i)\cap [\WW(\Gamma)].$$

    The above axiom implies two subcases.
\begin{enumerate}
\item{\bf Concavity}:\label{ax:convex}\footnote{This axiom was
  called \emph{convexity} in \cite{JKV1} because the original form of
  the construction outlined by Witten in the $A_{r-1}$ case involved
  the Serre dual of $\LL$, which is convex precisely when our $\LL$ is
  concave.}

    Suppose that all tails of $\Gamma$ are narrow.  If
$\pi_*\left(\bigoplus_{i=1}^t\LL_i\right)=0$, then the virtual cycle
is given by capping the top Chern class of the dual $\left(R^1 \pi_*
\left(\bigoplus_{i=1}^t\LL_i\right)\right)^*$ of the pushforward
with the usual fundamental cycle of the moduli space:
\begin{equation}\begin{split}
\left[\WW(\Gamma)\right]^{vir}& = c_{top}\left(\left(R^1\pi_*\bigoplus_{i=1}^t\LL_i \right)^*\right) \cap \left[\WW(\Gamma)\right]\\
& = (-1)^D c_{D}\left(R^1\pi_*\bigoplus_{i=1}^t\LL_i \right) \cap
\left[\WW(\Gamma)\right].
\end{split}
\end{equation}
\item   {\bf  Index zero:} \label{ax:wittenmap} Suppose that  $\dim (\W(\Gamma))=0$
and all the decorations on tails  are narrow.

If the pushforwards $\pi_* \left(\bigoplus\LL_i\right)$ and
$R^1\pi_* \left(\bigoplus \LL_i\right)$ are both vector bundles of
the same rank, then the virtual cycle is just the degree
$\deg(\wit)$ of the Witten map times the fundamental cycle:
$$\left[\WW(\Gamma)\right]^{vir} = \deg(\wit)\left[\WW(\Gamma)\right].$$
\end{enumerate}

\item\textbf{Composition law:}\label{ax:cutting} Given any
genus-$g$ decorated stable $W$-graph $\Gamma$ with $k$ tails, and
given any edge $e$ of $\Gamma$, let $\Gamma_{cut}$ denote the
graph obtained by ``cutting'' the edge $e$ and replacing it with
two unjoined tails $\tau_+$ and $\tau_-$ decorated with $\gamma_+$
and $\gamma_-$, respectively.

In view of the gluing/cutting commutative diagram:
\begin{equation}\label{eq:CuttingDiagr}
\begin{diagram}
        &           &F          &\rTo^{pr_2}    &\WW(\Gamma)\\
        &\ldTo^q    \\
\WW(\Gamma_{cut}) &&\dTo_{pr_1}&              &\dTo_{\st_\Gamma}          \\
    &\rdTo^{\st_{\Gamma_{cut}}}     \\
        &           &\MM({\Gamma_{cut}}) & \rTo^{\rho} & \MM({\Gamma})\\
\end{diagram},
\end{equation}
the fiber product
$$F:=\MM(\Gamma_{cut})\times_{\MM(\Gamma)} \WW(\Gamma)$$
has morphisms
 $$ \WW({\Gamma_{cut}})\lTo^{q} F
\rTo{pr_2}\WW(\Gamma).$$

We have
\begin{equation}\label{eq:cutting}
\left\langle
[\WW(\Gamma_{cut})]^{vir}\right\rangle_{\pm}=\frac{1}{\deg(q)}q_*pr_2^*\left(\left[\WW(\Gamma)\right]^{vir}\right),
\end{equation}
where
\begin{multline}\langle\,\,\,\,  \rangle_{\pm}:H_*(\WW(\Gamma_{cut})\otimes\prod_{\tau \in T(\Gamma)} H_{N_{\gamma_{\tau}}}(\C^N_{\gamma_{\tau}},W^{\infty}_{\gamma_\tau}, \Q) \otimes  H_{N_{\gamma_{+}}}(\C^N_{\gamma_{+}},W^{\infty}_{\gamma_+}, \Q)\otimes H_{N_{\gamma_{-}}}(\C^N_{\gamma_{-}},W^{\infty}_{\gamma_-}, \Q) \rTo \\
 H_*(\MM_W(\Gamma_{cut})\otimes\prod_{\tau \in T(\Gamma)} H_{N_{\gamma_{\tau}}}(\C^N_{\gamma_{\tau}},W^{\infty}_{\gamma_\tau}, \Q)
\end{multline}
is the contraction of the last two factors via the pairing
$$
\langle\, , \rangle:
H_{N_{\gamma_{+}}}(\C^N_{\gamma_{+}},W^{\infty}_{\gamma_+},
\Q)\otimes
H_{N_{\gamma_{-}}}(\C^N_{\gamma_{-}},W^{\infty}_{\gamma_-}, \Q)
\rTo \Q.
$$

 \item \textbf{Forgetting tails:}\label{ax:tails}
\begin{enumerate}
\item
 Let $\Gamma$
have its $i$th tail decorated with $J^{-1}$, where $J$ is the
exponential grading element of $G$. Further, let $\Gamma'$ be the
decorated $W$-graph obtained from $\Gamma$ by forgetting the $i$th
tail and its decorations.  Assume that $\Gamma'$ is stable, and
denote the forgetting tails morphism by $$\vartheta: \WW(\Gamma) \to
\WW(\Gamma').$$  We have
\begin{equation}
\left[\WW(\Gamma)\right]^{vir}
=\vartheta^*\left[\WW(\Gamma')\right]^{vir}.
\end{equation}
    \item In the case of $g=0$ and $k=3$, then the space $\WW(\gamma_1, \gamma_2, J^{-1})$ is empty if
    $\gamma_1\gamma_2\neq 1$ and $\WW_{0,3}(\gamma, \gamma^{-1}, J^{-1})=\BG$.  We omit
     $H_{N_{J^{-1}}}(\C^N_{J^{-1}},W^{\infty}_{J^{-1}}, \Q)^G = \Q$ from the notation.  In this case, the cycle
            $$\left[\WW_{0,3}(\gamma, \gamma^{-1}, J^{-1})\right]^{vir}\in
           H_*(\BG,\Q)\otimes  H_{N_{\gamma}}(\C^N_{\gamma},W^{\infty}_{\gamma}, \Q)^G \otimes
        H_{N_{\gamma^{-1}}}(\C^N_{\gamma^{-1}},W^{\infty}_{\gamma^{-1}}, \Q)^G$$ is
        the fundamental cycle of $\BG$
        times the Casimir element. Here the Casimir element
        is defined as follows. Choose a basis $\{\alpha_i\}$
        of $H_{N_{\gamma}}(\C^N_{\gamma},W^{\infty}_{\gamma},
        \Q)^G,$ and a basis $\{ \beta_j\}$ of $H_{N_{\gamma^{-1}}}(\C^N_{\gamma^{-1}},W^{\infty}_{\gamma^{-1}},
        \Q)^G$. Let $\eta_{ij}=\langle \alpha_1, \beta_j\rangle $ and $(\eta^{ij})$ be the inverse
        matrix of $(\eta_{ij})$. The Casimir element is
        defined as $\sum_{ij}\alpha_i\eta^{ij}\otimes \beta_j.$
\end{enumerate}

\item \textbf{Sums of Singularities:}

     If $W_1 \in \C[z_1,\dots,z_t]$ and $W_2 \in
C[z_{t+1},\dots,z_{t+t'}]$ are two quasi-homogeneous polynomials
with diagonal automorphism groups $G_1$ and $G_2$, and if we write
$W=W_1+W_2$, then the diagonal automorphism group of $W$ is $G =
G_1 \times G_2$. Further, the state space $\ch_W$ is naturally
isomorphic to the tensor product
\begin{equation}
\ch_W = \ch_{W_1} \otimes \ch_{W_2},
\end{equation}
and the space $\WW_{g,k}$ is naturally isomorphic to the fiber
product $$\WW_{g,k} = {(\WW_1)}_{g,k} \times_{\MM_{g,k}}
{(\WW_2)}_{g,k}.$$ Indeed, since any $G$-decorated stable graph
$\Gamma$ is equivalent to the choice of a $G_1$-decorated graph
$\Gamma_1$ and $G_2$-decorated  graph $\Gamma_2$ with the the same
underlying graph $\barGamma$, we have
\begin{equation}
\WW(\Gamma) = {(\WW_1)}(\Gamma_1) \times_{\MM(\barGamma)}
(\WW_{2})(\Gamma_2).
\end{equation}
The natural inclusion
$$\WW_{g,k} = {(\WW_1)}_{g,k} \times_{\MM_{g,k}}
{(\WW_2)}_{g,k}\rInto^{\Delta} {(\WW_1)}_{g,k} \times
{(\WW_2)}_{g,k}$$ together with the isomorphism of middle homology
induces a homomorphism
\begin{eqnarray*}
\Delta^*&:&\left(H_*({(\WW_1)}_{g,k},\Q)\otimes \prod_{i=1}^k
H_{N_{\gamma_{i,1}}}(\C^N_{\gamma_{i,1}},(W_1)^{\infty}_{\gamma_{i,1}},
\Q)^G\right) \\
& &\otimes \left(H_*({(\WW_2)}_{g,k},\Q)\otimes \prod_{i=1}^k
H_{N_{\gamma_{i,2}}}(\C^N_{\gamma_{i,2}},(W_2)^{\infty}_{\gamma_{i,2}},
\Q)^G\right) \\
& &\rTo H_*(\overline{(\W_1+\W_2)}_{g,k},\Q)\otimes \prod_{i=1}^k
H_{N_{(\gamma_{i,1},\gamma_{i,2})}}(\C^N_{(\gamma_{i,1},\gamma_{i,2})},(W_1+W_2)^{\infty}_{(\gamma_{i,1},\gamma_{i,2})},
\Q)^{G_1\times G_2}.
\end{eqnarray*}

The virtual cycle satisfies
\begin{equation}
\Delta^*\left(\left[{(\WW_1)}_{g,k}\right]^{vir}\otimes
\left[{(\WW_2)}_{g,k}\right]^{vir}\right) =
\left[\overline{(\W_1+\W_2)}_{g,k}\right]^{vir}.
\end{equation}

\item \textbf{Deformation Invariance:} Let $W_\tau\in
\C[z_1,\dots,z_N]$ be a family of non-degenerate quasi-homogeneous
polynomials depending smoothly on a parameter $\tau\in S$, where
$S$ is a path connected domain in $\C^m$. Suppose that $G$ is the
common automorphism group of $W_\tau$, then the virtual cycle
$[\WW_\tau(\Gamma)]^{vir}$ associated to $(W_\tau,G)$ is independent
of $\tau$.

\item \textbf{$G_{max}$-Invariance:} The virtual cycle $[\WW(\Gamma)]^{vir}$ associated to $(W,G)$ is
$G_{max}$-invariant (refer to the proof of this axiom for the
explanation of the $G_{max}$ action).
\end{enumerate}
\end{thm}

    The paper is organized as follows. In Section \ref{sec:rigmoduli}, we will
    summarize the algebro-geometric set up of \cite{FJR2} and
    some of the easy consequences. The perturbed Witten equation and its nonlinear analysis on a
    smooth curve or orbicurve will be defined in Section \ref{sec:witten}. In Section \ref{sec:compact},
     the perturbed Witten map and equation will be defined on the moduli space of rigidified $W$-curves.
     We will prove the Gromov compactness theorem for the space of $W$-sections. The virtual cycle will be
    constructed in Section \ref{sec:construct}. The main theorems will be
    proved in the last section.

\section{Rigidified $W$-curves and their
moduli}\label{sec:rigmoduli}

      Recall that the variable $u_i$ of the Witten equation is supposed to
      be a section of certain orbifold line bundles $\LL_i$ 
over the
      Riemann surface. The fact that $u_i$ satisfies the  Witten
      equation forces a certain condition on $\LL_i$. $\LL_i$ satisfying the required condition
      is called {\em a $W$-structure}. The $W$-structure is the
      background data of the Witten equation. A detailed construction of
      the moduli of $W$-structures has been worked out in our
      previous paper \cite{FJR2}. It is technically convenient to work over
      the moduli space of rigidified $W$-structures. In this
      section, we summarize the construction of \cite{FJR2}.

Let us recall the definition of the non-degenerate quasi-homogeneous
polynomials in \cite{FJR2}.

\begin{df}
A \emph{quasi-homogeneous} (or \emph{weighted homogeneous})
\emph{polynomial} $W \in \mathbb{C} [x_1, \dots,
x_N]$\glossary{N@$N$
  & The number of variables in the quasi-homogeneous polynomial $W$
  and the number of line bundles in the $W$-structure}\glossary{W@$W$
  & A quasi-homogeneous polynomial} is a polynomial for which there
exist positive rational degrees $q_1, \dots, q_N \in \Q^{>0}$, such
that for any $\lambda \in \mathbb{C}^*,$
\begin{equation}\label{eq:qhomocharge}
W(\lambda^{q_1}x_1, \dots, \lambda^{q_N}x_N) =\lambda W(x_1, \dots,
x_N).
\end{equation}
 We will call $q_j$ the \emph{weight}\glossary{qj@$q_j$ & The weight (charge)
   of the variable $x_j$}\glossary{qy@$q_y$ & The weight of the
   variable $y$} of $x_j$.
   We define
$d$ and $n_i$ for $i\in \{1,\dots,N\}$ to be the unique positive
integers such that $(q_1,\dots,q_N) = (n_1/d, \dots, n_N/d)$ with
$\gcd(d,n_1,\dots,n_N) =1$.

\end{df}

\begin{df}\label{df:nondegenerate} We call $W$ \emph{nondegenerate}
if \begin{enumerate}
\item\label{it:nondegen-unique}
All weights $q_i$ satisfy $q_i\le 1/2$ and are uniquely determined
by $W$ for all $i \in \{1,\dots,N\}$, and
\item The hypersurface defined by $W$ in weighted
projective space is non-singular, (or equivalently, the affine
hypersurface defined by $W$ has an isolated singularity at the
origin).
\item If $W$ has a pair of variables $x_{i_1},x_{i_2}$ with weights $q_{i_1},q_{i_2}=1/2$, then $W$ is
required to have no cross term $x_{i_1}x_{i_2}$.
  \end{enumerate}
    \end{df}

\begin{rem} The only reason we ruled out the case that $W$ has 
cross
terms $x_{i_1}x_{i_2}$ is to simplify our analysis of the corresponding P.D.E. On the
other hand, this case has a normal form with only square terms, after
a holomorphically equivalent transformation, and will not affect the
topological properties of the singularities.
\end{rem}

Assume that the variables $x_{i_1},\dots,x_{i_l}$ in the
non-degenerate quasi-homogeneous polynomial $W(x_1,\dots,x_N)$ have
weights $q_{i_j}=1/2,j=1,\dots,l$. This case is called 
the
borderline
case. Since we do not allow the existence of the monomial
$x_{i_j}x_{i_k}, j\neq k$ in $W$. $W$ can be written as the form:
\begin{equation}\label{eq:form-half-iden}
W(x_1,\dots,x_N)=\sum_{j=1}^l
(x_{i_j}\hat{W}_{i_j}+c_{i_j}x_{i_j}^2),
\end{equation}
where $\hat{W}_{i_j}$ do not contain the variables
$x_{i_j},j=1,\dots,l$ because of degree reasons.

We will show that in the borderline case the condition (2) in the
definition of nondegeneracy of $W$ gives a restriction to the
admissible form of formula (\ref{eq:form-half-iden}). In detail, we
have

\begin{lm}\label{lm:poly-half-form} Suppose that $W(x_1,\dots,x_N)$ is a non-degenerate quasi-homogeneous
polynomial such that the variables $\{x_{i_j},j=1,\dots,l\}$ have
weights $1/2$, then $W$ has only the following two cases:
\begin{enumerate}

\item $l=N-1$. $W$ has the form
$$
W(\dots,u,\dots)=(\sum_{j=1}^{N-1}c_{i_j}x_{i_j})u^n+(\sum_{j=1}^{N-1}b_{i_j}x^2_{i_j}),n\ge
2.
$$
In particular, if $l=1$, $W$ is just the $D^T_{n+1}$ singularity.

\item $l=N$, then
$$
W(x_1,\dots,x_N)=\sum_{j=1}^N c_j x_j^2.
$$
\end{enumerate}
\end{lm}

\begin{proof} Consider the case that $l\le N-2$. We want to show that the identities
$$
\frac{\pat W}{\pat x_i}=0, \quad \text{ for $i=1,\dots,N$}
$$
have a nontrivial zero locus.

The above equations have the form
\begin{numcases}{}
\frac{\pat W}{\pat x_{i_j}}=2 c_{i_j} x_{i_j}+ \hat{W}_{i_j}=0, \;\;
j=1,\dots,l \\
\frac{\pat W}{\pat x_k}=\sum_{j=1}^l x_{i_j}\frac{\pat
\hat{W}_{i_j}}{\pat x_k}=0,  \;\; k \not\in \{i_1,\dots,i_l\}.
\end{numcases}

Now we can set
\begin{numcases}{}
2 c_{i_j} x_{i_j}+ \hat{W}_{i_j}=0, \;\;
j=1,\dots,l \\
\sum_{j=1}^l x_{i_j} \hat{W}_{i_j}=0.
\end{numcases}

Thus the solution space is at least a dimension $1$ space in $\C^N$,
which contradicts with the non-degeneracy of $W$.

The proofs for the other cases are straightforward.
\end{proof}

\begin{df}\label{df:W-structure}
For any non-degenerate, quasi-homogeneous polynomial $W \in
\C[x_1, \dots$,$ x_N]$, we define a \emph{$W$-structure} on an
orbicurve $\cC$ to be  the data of an $N$-tuple $(\LL_1, \dots,
\LL_N)$ of orbifold line bundles on $\cC$ and
 isomorphisms $\varphi_j:W_j(\LL_1, \dots, \LL_N) \irightarrow K_{\cC,log}$ for every  $j\in \{1,\dots,s\}$, where
by $W_j(\LL_1, \dots, \LL_N)$ we mean the $j$th monomial
  of $W$ in $\LL_i$:
   $$W_j(\LL_1, \dots, \LL_N)= \LL^{\otimes b_{1,j}}_1 \otimes \dots
\otimes \LL^{\otimes b_{N,j}}_N.$$

Note that for each point $p \in \cC$,  an orbifold line bundle
$\LL$ on $\cC$ induces a representation $G_p \rTo \aut(\LL) $.
Moreover, a $W$-structure on $\cC$ will induce a representation
$\rho_{p}: G_p \rTo \grp \subseteq U(1)^N$. For all our
$W$-structures we require that this representation $\rho_p$ be
faithful at every point.
\end{df}

The moduli space $\WW_{g,k}$ of $W$-structures was constructed in
\cite{FJR2} along the lines of the Delign-Mumford moduli space. It
is a smooth non-effective orbifold. There is an obvious finite map
$$\st: \WW_{g,k}\rTo \MM_{g,k}.$$
It was shown in \cite{FJR2} that
$$\WW_{g,k}=\coprod_{(\gamma_1, \dots,
\gamma_k)}\WW_{g,k}(\gamma_1, \dots, \gamma_k),$$
    where $\WW_{g,k}(\gamma_1, \dots, \gamma_k)$ is the moduli
    space of $W$-structures with the orbifold structure at the $i$-th
    marked point decorated by $\gamma_i$.

It is convenient for us to work on the rigidified $W$-structure.
Let $p$ be a
    marked point. A {\em rigidification} (at $p$) is a homomorphism
    $$\psi: j^*_{p_{\pm}}(\LL_1\oplus \dots
    \oplus\LL_N){\rTo}  [\C^N/G_{p_{\pm}}].$$

A more geometric way to understand the rigidification is follows.
Suppose the fiber of the $W$-structure at the marked point is
$L_1\oplus L_2\dots\oplus L_N/G_p$. The rigidification can be
thought of as a $G_p$-equivariant map $\psi: \oplus_i L_i\rTo {\mathbb
C}^N$ commuting with the $W$-structure. For any element $g\in G_p$,
$g\psi$ is considered to be an equivalent rigidification.
Alternatively, $\psi$ is equivalent to a choice of basis $e_i\in
L_i$ such that $W_j(e_1, \dots, e_N)=dz/z$ and $g(e_1), \dots
g(e_N)$ is considered to be an equivalent choice. In particular, if
$L_i$ is fixed by $G_p$, 
there is a unique choice of basis $e_i$ with $W_j(e_i, \dots,
e_N)=dz/z$. On the other hand, the choice of basis on the
un-fixed variables is only unique up to the action of $G_p$. It
is clear that the group $G/G_p$ acts transitively on the set of
rigidifications within a single orbit. Let $\WW^{rig_p}(\Gamma)$ be
the moduli space of equivalence classes of $W$-structures with a
rigidification at $p$. $G/G_p$ acts on $\WW^{rig_p}(\Gamma)$ by
changing the rigidification. It is clear that
$$\WW^{rig_p}(\Gamma)/(G/G_p)=\WW(\Gamma).$$
We use $\MMr_{g,k}$ to denote the moduli space of rigidified
$W$-structures at all the marked points.

Forgetting the rigidification gives a morphism
$$\so:\MMr_{g,k} \rTo \WW_{g,k},$$
which we call \emph{softening}. The morphism $\so$ is quasi-finite
since one can always construct a rigidification of any
unrigidified $W$-structure,  and $G^k$ acts transitively, but
usually not effectively, on the set of rigidifications.  It easy
to see that $\so$ is proper and of Deligne-Mumford type.
Furthermore, $\so$ is representable, since the automorphisms of
any rigidified $W$-curve are a subgroup of the automorphisms of
the corresponding unrigidified $W$-curve.

 Now we describe the gluing. To simplify notation, we ignore the orbifold
    structures at other marked points and denote the type of the marked point $p_+, p_-$ being glued by $\gamma_+,
    \gamma_-$. Because the resulting orbicurve must be balanced,  we require that
    $\gamma_-=\gamma^{-1}_+$. Let
    $$\psi_{\pm}: j^*_{p_{\pm}}(\LL_1\oplus \dots
    \oplus\LL_N){\rTo}  [\C^N/G_{p_{\pm}}]$$
    be the rigidifications. Moreover, the residues at $p_+, p_-$
    are opposite to each other. The obvious  identification will
    not preserve the rigidifications. Here, we fix once and for all an
    isomorphism
    $$I: \C^N\rTo \C^N$$
    such that $W(I(x))=-W(x)$. $I$ can be explicitly constructed
    as follows. Suppose that $q_i=n_i/d$ for common denominator
    $d$. Choose $\xi^d=-1$. Then,
 \begin{equation}
 I(x_1, \dots,  x_N)=(\xi^{k_1}x_1, \dots, \xi^{k_N}x_N).
 \end{equation}
If $I'$ is another choice, then $I^{-1}I'\in G_W$.
    Furthermore, $I^2\in G_W$ as well. The identification by $I$
    induces a $W$-structure on the nodal orbifold Riemann surface with a
    rigidification at the nodal point. Forgetting the rigidification at the node yields the lifted gluing morphisms
\begin{equation}\label{eq:glueTreeW}
\widetilde{\rho}_{tree,\gamma}:\MMr_{g_1, k_1+1}(\gamma)\times
\MMr_{g_2,
    k_2+1}(\gamma^{-1})\rTo \WW_{g_1+g_2, k_1+k_2},
\end{equation}\glossary{rhotildetree@$\widetilde{\rho}_{tree}$ & The gluing-trees morphism for rigidified $W$-curves}

\begin{equation}\label{eq:glueLoopW}
\widetilde{\rho}_{loop,\gamma}: \MMr_{g, k+2}(\gamma,
\gamma^{-1})\rTo \WW_{g, k},
\end{equation}\glossary{rhotildeloop@$\widetilde{\rho}_{loop}$ & The gluing-loops morphism for rigidified $W$-curves}
and
$$\tilde{\rho}: \MMr_W(\hGamma)\rTo \WW(\Gamma),$$
where $\tilde{\rho}$ is defined by gluing the rigifications at the
extra tails and forgetting the rigidification at the node.

Next, we summarize some of the basic properties of $\MMr_{g,k}$.
They are  easy consequences of the existence of a universal family.
We leave the proof to the interested reader. As a warm-up, we start
with the Deligne-Mumford moduli space $\MM_{g,k}$.

\subsection{The structure of the Deligne-Mumford space $\MM_{g,k}$}\label{sec1}

Let $\M_{g,k}$ be the moduli space of Riemann surfaces of genus
$g$ and having $k$ marked points (assuming $k+2g\ge 3$). The
Deligne-Mumford compactification $\MM_{g,k}$ of $\M_{g,k}$ is the
set of all isomorphism classes of stable nodal Riemann surfaces in
$\MM_{g,k}$. The following is well-known and an easy consequence
of the existence of a universal family.

 The moduli space $\MM_{g,k}$ is a
stratified space, indexed by a set $\comb(g,k)$. Each element in
$\comb(g,k)$ is represented by a triple $(\Gamma,g_\nu, o)$
satisfying the following requirement:

\begin{itemize}

\item[(1)] $\Gamma$ is the dual graph of some nodal curve $\cC$
with each vertex assigned the genus $g_\nu$ of the component of
$\cC$ corresponding to $\nu$.

\item[(2)] $o$ is a map from $\{1,\dots, k\}$ to the set
$V(\Gamma)$.

\item[(3)] Let $k_\nu$ be the number of elements of the set
$T(\nu)$, then
$$
k_\nu+2g_\nu\ge 3\;\;\text{and}\;\;\sum_\nu g_\nu+\rk
H_1(\Gamma,\Q)=g.
$$

\end{itemize}

From the definition, it is easy to see that the number of
combinatorial types $(\Gamma,g_\nu, o)$ in $\comb(g,k)$ is finite.

Each stable curve $\cC\in \MM_{g,k}$ has a dual graph $\Gamma$ and
corresponds to a combinatorial type $(\Gamma, g_\nu,o)$.

There is a partial order $\succ$ on $\comb(g,k)$ defined as
follows. Let $(\Gamma, g_\nu, o)\in \comb(g,k)$. We consider
$(\Gamma_\nu, g_{\nu w}, o_\nu)\in \comb(g_\nu, k_\nu)$ for some
vertices $\nu=\nu_1,\dots, \nu_a$ of $\Gamma$. Now we replace the
vertex $\nu$ by the graph $\Gamma_\nu$, and join the edges
containing $\nu$ to $o_\nu(j)$; then we obtain a new graph
$\tilde{\Gamma}$. $\tilde{g}_\nu$ is determined by $g_\nu$ and
$g_{\nu w}$ in a natural way. If $o(j)\neq \nu_i, i=1,\dots,
a$,let $\tilde{o}(j)=o(j)$; if $o(j)=\nu_i$, then the $j$th marked
point corresponds to some $j'\in \{1,\dots,k_{\nu_i}\}$, and we
let $\tilde{o}(j)=o_{\nu_i}(j')$. So we get an element
$(\tilde{\Gamma},\tilde{g}_\nu,\tilde{o})\in \comb(g,k)$. If
$(\tilde{\Gamma},\tilde{g}_\nu,\tilde{o})$ is obtained in this way
from $(\Gamma, g_\nu, o)$, we define $(\Gamma, g_\nu,
o)\succ(\tilde{\Gamma},\tilde{g}_\nu,\tilde{o})$. There is a
minimal element in $\comb(g,k)$ which satisfies, for any vertex
$\nu$, $g_\nu=0$ and $k_\nu=3$.

\begin{df}[Cylindrical metric] We can view a disc neighborhood of a
nodal or marked point as a half-infinite  cylinder $S^1\times
[T_0, \infty)$ for some $T_0>0$ plus the infinity point $z$. We
say that a metric is cylindrical near a neighborhood of a nodal or
marked point if we take the flat metric $dw\otimes d\bar{w}$ on
the cylinder $S^1\times [T_0, \infty)$, where $w=s+i\theta$ is the
cylindric coordinate.
\end{df}

\begin{rem} Other uniform metrics near the marked points can also be chosen; for instance, one can
choose the smooth metric or hyperbolic metric. See the discussion
in \cite{CL}.
\end{rem}


 The moduli stack $\MM_{g,k}$ has the following
structure:

\begin{prop}\label{deli-mumf}

Let $\M_{g,k}(\Gamma)$ be the set of all stable curves having
combinatorial type $(\Gamma, g_\nu, o)$, then

\begin{itemize}

\item For each marked point $z_i$, there is an orbifold disc
bundle $D_i\rTo \MM_{g,k}$ such that the fiber at $\Sigma$
is a disc neighborhood of $z_i$. Furthermore, there is a smooth
family of metrics on $D_i$ parameterized by $\Sigma\in \MM_{g,k}$
such that in a uniformly small neighborhood of the zero point the
metric is a cylindrical metric.

\item $\MM_{g,k}$ is a compact complex orbifold of (complex)
dimension $3g-3+k$ which admits a stratification with finitely
many strata, and each stratum is of the form $\M_{g,k}(\Gamma)$.
There is a minimal stratum containing only one point $(\cC,\z)$.

\item There is a fiber bundle $\C_{g,k}(\Gamma)\rTo
\M_{g,k}(\Gamma)$ which has the following property. For each
$x=(\cC_x, \z_x)\in \M_{g,k}(\Gamma)$, there is a neighborhood of
$x$ in $\M_{g,k}(\Gamma)$ of the form $U_x=V_x/G_x$, where
$G_x=\aut(\cC_x,\z_x)$ such that the inverse image of $U_x$ in
$\U_{g,k}(\Gamma)$ is diffeomorphic to $V_x\times \cC_x/G_x$.
There is a complex structure on each fiber such that the fiber of
$y=(\cC_y,\z_y)$ is identified with $(\cC_y,\z_y)$ itself,
together with a K\"ahler metric $\mu_y$ which is cylindrical in a
neighborhood of the nodal points and varies smoothly in $y$.

\item $\M_{g,k}(\Gamma')$ is contained in the compactification of
$\M_{g,k}(\Gamma)$ in $\MM_{g,k}$ only if $(\Gamma, g_\nu,o)\succ
(\Gamma',g'_\nu,o')$.

\item Different strata are patched together in a way which is
described in the following local model of a neighborhood of a
stable curve in $\MM_{g,k}$. A neighborhood of $x=(\cC,\z)$ in
$\MM_{g,k}$ is parametrized by
$$
\frac{V_x\times (\oplus_z ([T_0,\infty]\times
S^1)_z)}{\aut(\cC,\z) },
$$
where $z=\pi_\nu(z_\nu)=\pi_w(z_w)$ (here it may happen that
$\nu=w$) runs over all nodal points of $\cC$, and $T_0$ is a
positive number. Each $y\in V_x$ represents a stable curve
$(\cC_y, \z_y)$ homeomorphic to $(\cC, \z)$ with a K\"ahler metric
$\mu_y$ which is cylindrical in a neighborhood of the nodal
points. Given $y\in V_x$, a stable curve
$(\cC_{y,\zeta},\z_{y,\zeta})$ is obtained as follows. Each
component $\cC_\nu$ of $\cC_y$ is given a K\"ahler metric $\mu_y$
which is cylindrical in a neighborhood of the nodal points. For
each $\zeta_z=(s_z,\theta_z)\in [T_0,\infty)\times S^1$, there is
a biholomorphic map $\Psi_{\zeta_z}: ([T_0,\infty)\times
S^1)_{z_\nu}\rTo ([T_0,\infty)\times S^1)_{z_w}$ defined by
$(s_w,\theta_w)=(s_z,\theta_z)-(s_\nu,\theta_\nu)$, where
$([T_0,\infty)\times S^1)_{z_\nu}$ is the cylindrical neighborhood
of $z_\nu$ in the component $\cC_\nu$. Let $s_z=2T$; then for
sufficiently large $T$, the map $\Psi_{\xi_z}$ is a biholomorphism
between $[\frac{1}{2}T, \frac{3}{2}T]\times S^1$ and
$[\frac{1}{2}T, \frac{3}{2}T]\times S^1$. We glue $\cC_\nu$ and
$\cC_w$ by this biholomorphism. If $\zeta_z=\infty$, we do not
make any changes. 

Thus we obtain
 $(\cC_{y,\zeta},\z_{y,\zeta})$. Moreover, there is a K\"ahler
 metric $\mu_{y,\zeta}$ on $\cC_{y,\zeta}$ which coincides with
 the K\"ahler metric $\mu_y$ on $\cC_y$ outside a neighborhood
 of the nodal points, and varies smoothly in $\zeta$. Each
 $\gamma\in \aut(\cC,\z)$ takes $(\cC_y,\z_y)$ to
 $(\cC_{\gamma y}, \z_{\gamma y})$ isometrically, so it acts on
 $\oplus_z([T_0,\infty]\times S^1)_{z}$. $\gamma$
 induces an isomorphism between $(\cC_{y,\zeta},\z_{y,\zeta})$
 and $(\cC_{\gamma(y,\zeta)},\z_{\gamma(y,\zeta)})
$, which is also an isometry.

\end{itemize}
\end{prop}

\subsection{The structure of the moduli space $\MMr_{g,k}
$}

The structure of $\MMr_{g,k, W}$ can be described in the same way as
that of $\WW_{g,k}$. However, there are some new ingredients from
the rigidified $W$-structure.

\begin{rem}[Decorated dual graph] As shown in \cite{FJR2},
the dual graph describing a stratum of $\MMr_{g,k}$ has an
additional decoration $\gamma\in G$ at each tail and a pair of
decorations $(\gamma, \gamma^{-1})$ at each internal edge. We use
the same $\Gamma$ to denote the original dual graph together with
the additional decoration, and $\MMr(\Gamma)$ to denote the space of
rigidified $W$-structures  whose combinatorial type is described by
$\Gamma$. Whether $\Gamma$ contains the additional decoration should
be clear from the context.
\end{rem}

\begin{rem}[Canonical trivialization over the cylindrical coordinate]
  Given a marked point $z_j$, a rigidification defines a basis $e_i$ of $\LL_i$ corresponding to
  the fixed variables. Hence, it defines a canonical trivialization
  of $\LL_i|_{z_j}$. We can extend this canonical trivialization
  over the cylindrical coordinate at $z_j$. For $\LL_i$ corresponding to a variable moved by the local group,
  $\LL_i$ has nontrivial orbifold structure at $z_j$. A
  rigidification  defines a canonical trivialization over the cylindrical coordinate in the orbifold sense.
  \end{rem}

  \begin{rem}[Gluing rigidifed $W$-structures]
  Suppose that $\Sigma=\Sigma_1\bigvee\Sigma_2$ is a nodal curve
  obtained by gluing the marked point $p\in \Sigma_1, q\in \Sigma_2$.
  Given rigidified $W$-structures $(\LL^1_i, \psi^1_j)$,$ (\LL^2_i,
  \psi^2_j)$ over $\Sigma_1, \Sigma_2$, there is a canonical way
  to glue them to get a rigidified $W$-structure on $\Sigma$ with an
  additional rigidification at the node. The rigidified $W$-structure
  on $\Sigma$ is obtained by forgetting the rigidification at the
  node. Alternatively, a rigidified $W$-structure of $\Sigma$ is
  obtained by an  isomorphism $I: \LL^1_i|_p\cong \LL^2_i|_q$ preserving the $W$-structure while
  forgetting the rigidifications at $p,q$.
  \end{rem}

\begin{prop}\label{deli-mumf}

\begin{itemize}

\item $\MMr_{g,k}$ is a compact complex orbifold which is a
finite cover of $\WW_{g,k}$ and admits a stratification with
finitely many strata, and each stratum is of the form
$\MMr_{g,k}(\Gamma)$.

\item There is a trivialization of the restriction of $\LL_i$ to
each fiber of $D_i$. Furthermore, the trivialization can be chosen
smoothly depending on $\Sigma$.

\item There are  two  bundles (universal families). The first one
is $\C(\Gamma)\rTo \MMr(\Gamma),$ which is the pull-back of
the  universal family of $\MM(\Gamma)$. Then there is a collection
of (orbifold) universal line bundles
$\overline{\U}_i(\Gamma)\rTo \C(\Gamma),$ whose fiber is
$\LL_i$.

\item $\MMr(\Gamma')$ is contained in the compactification of
$\MMr(\Gamma)$ in $\MMr_{g,k}$ only if $\Gamma\succ \Gamma'$.

\item Different strata are patched together in a way which is
described in the following local model of a neighborhood of a
nodal rigidified $W$-structure in $\MMr_{g,k}$. Let $x=(\cC, \z)$
be the underlying nodal stable Riemann surface. Recall that a
neighborhood of $x=(\cC,\z)$ in $\MM_{g,k}$ is parameterized by
$$
\frac{V_x\times (\oplus_z ([T_0,\infty]\times
S^1)_z)}{\aut(\cC,\z) },
$$
where $z=\pi_\nu(z_\nu)=\pi_w(z_w)$. An element $(\cC_y,\z_y)$ in
a neighborhood of $(\cC_x, \z_x)$ is obtained by gluing
cylindricals neighborhoods of corresponding marked points on each
component. Now, the gluing of $\LL_i$ is obtained as follows.
Recall that a rigidified $W$-structure on $\cC_x$ is obtained by
identifying the fiber of $\LL_i$ at the corresponding marked point
and forgetting the corresponding rigidifications. Then, the
canonical trivialization of $\LL_i$ near the marked point extends
the identification at the marked point to the cylindrical
neighborhoods which glue to a rigidified $W$-structure on $(\cC_y,
\z_y)$.

\end{itemize}
\end{prop}

\section{The Witten equation}\label{sec:witten}

We have laid down the algebraic-geometric foundations in the last
     several sections. Now we turn our attention to the analytic
     aspects of our theory. Even though we formulate our axioms in
     algebraic-geometric terms, the analytic aspect of the theory is at
     the heart of our construction.
     It remains a challenging problem to have a completely
     algebraic treatment of our theory.

     Roughly speaking, the moduli space of $W$-structures
     $\WW_{g,k}$ plays the role of Deligne-Mumford
     moduli space in ordinary Gromov-Witten theory. To construct
     the moduli space of stable maps, we use the solution of the
     Cauchy-Riemann equation $\bar{\partial}_J f=0,$ where  $f$ is
     a map from a Riemann surface. In our theory, we replace $f$
     by a section $s_i$ of $\LL_i$ or equivalently $|\LL_i|$. The
     replacement of the Cauchy-Riemann equation is the Witten equation,
     which we will describe in this section.

\subsection{The perturbed Witten equation}
\subsubsection{Cylindrical metric and Witten-equation}
In this section, we fix a rigidified $W$-structure \\
$(\LL_1, \dots, \LL_N,\varphi_1, \dots,
\varphi_s,\dots,\psi_1,\dots,\psi_k)$. Let $\gamma_l\in G$ be
the group element generating the orbifold structure at the marked
point $z_l$. Recall that

\begin{align*}
W_i(|\LL_1|),\dots,|\LL_N|)\cong K_{log}\otimes
\O(-\sum^k_{l=1}\sum^N_{j=1}b_{i j}\Theta_j^{\gamma_l}z_l).
\end{align*}

Let $D=-\sum^k_{l=1}\sum^N_{j=1}b_{i
j}(\Theta^{\gamma_l}_j-q_j)z_l$ be a divisor; then there is a
canonical meromorphic section $s_0$ with divisor $D$. This section
provides the identification
$$
K_\Sigma\otimes\O(D)\stackrel{s_0^{-1}}{\cong} K_\Sigma(D),
$$
where $K_\Sigma(D)$ is the sheaf of local possibly meromorphic
sections of $K_{\Sigma}$ with zero or pole at $D$.

A rigidification $\psi_l$ at $z_l$ actually gives a local
trivialization of the orbifold line bundle $\oplus \LL_i$, or
equivalently determines a   basis set $\eta_1,\dots, \eta_N$.
Taking a good coordinate system $\{z\}$ near $z_l$ , we can assume
that the holomorphic basis $e_i$ of $|\LL_i|$ corresponding to
$\eta_i$ satisfies the relation providing by $\varphi_i$,
$$
W_i(e_1,\dots, e_N)=\frac{dz}{z}z^{\sum^N_{j=1}b_{i
j}\Theta^{\gamma_l}_j}.
$$

The relation between $e_i$ and $\eta_i$ will be fully discussed
when we define the Witten equation on orbicurves.

Choose the cylindrical metric on the line bundle $K_\Sigma(D)$
such that $|\frac{dz}{z}|=1$. The above identity together with the
nondegeneracy of $W$ imply that the modulus
$|e_i|=|z|^{\Theta^{\gamma_l}_j}$ near marked points. In
particular, if $x_i$ is a fixed variable, then $|e_i|=1$. Let
$u_j=\tilde{u}_j e_j$; then it is easy to see that
\begin{equation}
\overline{\frac{\partial W}{\partial u_j}}\in
\overline{K}_{log}\otimes |\overline{\LL_j}|^{-1}.
\end{equation}
The bundle $|\overline{\LL_j}|^{-1}$ is isomorphic to $ |\LL_j|$
topologically. But there is no canonical isomorphism. However, we
can choose an isomorphism compatible with the metric. It induces
an isomorphism $I_1:\Omega(\Sigma, |\overline{\LL_j}|^{-1}\otimes
\Lambda^{0,1})\rTo \Omega(\Sigma, |\LL_j|\otimes
\Lambda^{0,1})$ such that for a section $v=\tilde{v}e'_j$, we have
$$
I_1(\tilde{v}\bar{e}'_j\otimes
d\bar{z})=\tilde{v}|e'_j|^2e_j\otimes d\bar{z},
$$
where $e'_j$ is the holomorphic basis of $|\LL_j|^{-1}$ such that
$e'_j\cdot e_j=1$.

It is obvious that $I_1$ is  a metric-preserving isomorphism
between the corresponding two spaces and that it is independent of
the choice of the local charts.

Since $I_1(\overline{\frac{\partial W}{\partial u_j}})\in
\overline{K}_{log}\otimes |\LL_j|$, the so-called Witten equation
is defined below as the first order system for the sections
$u_1,\dots,u_N$:
$$
\bar{\partial}u_j+I_1\left(\overline{\frac{\partial W}{\partial
u_j}}\right)=0, \quad\mbox{ for all } j=1,\dots, N.
$$

\begin{rem} In this paper, we choose the cylindrical metric near
the marked points; in [FJR] we used the smooth metric near marked
points, i.e., let $|dz|=1$ near marked points. Different choices
give different equations near marked points. This will produce
different theories. We will discuss the relations between the
theories in future work.

\end{rem}

Let $u_i=\tilde{u}_i e_i, \forall i$. Then near a marked point,
the Witten equation can be written locally as
\begin{equation}\label{lgw-equ}
\frac{\bpat \tilde{u}_i}{\pat\bar{z}}+\overline{\sum_j \frac{\pat
W_j(\tilde{u}_1,\dots, \tilde{u}_N)}{\pat
\tilde{u}_i}z^{\sum_{s=1}^{N}b_{js}(\Theta^{\gamma}_s-q_s)}}|z|^{-2\Theta_i^{\gamma}}=0,
\end{equation}
for $i=1,\dots,N$.

\

{\bf Note:} For simplicity, we often drop "$\;\tilde{ }\;$" from
$\tilde{u}_i$ when discussing the local equation near the marked
point. The reader can easily distinguish from the context when we
are talking about the sections or the coordinate functions.

\

Near $z_l$,  $W$ has the decomposition $W=W_{\gamma_l}+W_N$.
Without loss of generality, we can assume the fixed point set
$\C^N_{\gamma_l}=\{(x_1, \dots, x_{N_l}, 0, \dots, 0)\}$. Then,
the line bundles $\LL_i$ for $i\leq N_l$ have no orbifold
structure at $z_l$ and $\Theta^{\gamma_l}_i=0$. Now we drop the
subscript $l$  without any confusion. (\ref{lgw-equ}) can be
rewritten as
\begin{equation}\label{lgw-r-equ}
\bar{z}\frac{\bpat u_i}{\pat \bar{z}}+\overline{\frac{\pat
W_{\gamma}}{\pat u_i}}+\overline{\sum_{W_j: W_j \text{in} W_N }
\frac{\pat W_j(u_1,\dots, u_N)}{\pat
u_i}z^{\sum_{s=1}^{N}b_{js}(\Theta^{\gamma}_s)}}=0,
\end{equation}
for any $0\le i \le N_l$, and
\begin{equation}\label{lgw-n-equ}
\bar{z}\frac{\bpat u_i}{\pat \bar{z}}+(\overline{\sum_{W_j: W_j
\text{in} W_N } \frac{\pat W_j(u_1,\dots, u_N)}{\pat
u_i}z^{\sum_{s=1}^{N}b_{js}(\Theta^{\gamma}_s)}})|z|^{-2\Theta^{\gamma}_i}=0,
\end{equation}
for any $i\ge N_l$.

\begin{ex}[$D_n$-equation]\

In this case, the quasi-homogeneous polynomial is
$W(x,y)=x^n+xy^2$, so $W_1=x^n, W_2=xy^2$, and $b_{11}=n,
b_{12}=0; b_{21}=1, b_{22}=2$. The fractional degree is
$(q_1,q_2)=(\frac{1}{n},\frac{n-1}{2n})$. $G$ is generated by
$J^{-1}=(e^{\frac{2\pi i}{n}}, e^{\frac{2\pi (n-1)i}{2n}})$. There
are some cases near a marked point:

\begin{enumerate}
\item[{Case 1.}] $\gamma=(1,1)$. We have
$\Theta_1^{\gamma}=\Theta_2^{\gamma}=0$ and the equation becomes
\begin{equation}
\left\{
\begin{array}{l}
\bar{z}\frac{\bpat u_1}{\pat
\bar{z}}+\overline{nu_1^{n-1}}+\overline{u_2^2}=0\\
 \\
\bar{z}\frac{\bpat u_2}{\pat \bar{z}}+\overline{2u_1u_2} =0.
\end{array}
\right.
\end{equation}
\item[{Case 2.}] $\gamma=J^{-n}$ and $n$ is even. Then
$\Theta_1^{\gamma}=0$ and $\Theta_2^{\gamma}=\frac{1}{2}>0$.  In
this case the Witten equation is
\begin{equation}
\left\{
\begin{array}{l}
\bar{z}\frac{\bpat u_1}{\pat
\bar{z}}+\overline{nu_1^{n-1}}+\overline{u_2^2z}=0.\\
 \\
\bar{z}\frac{\bpat u_2}{\pat \bar{z}}+\overline{2u_1u_2
z}|z|^{-1}=0.
\end{array}
\right.
\end{equation}

\item[{Case 3.}] $\gamma\neq (1,1), J^{-n}$, where $n$ is odd.
This is the  narrow-case and $\Theta_1(\gamma)>0,
\Theta_2^{\gamma}>0$, the equation is
\begin{equation}
\left\{
\begin{array}{l}
\bar{z}\frac{\bpat u_1}{\pat
\bar{z}}+(\overline{nu_1^{n-1}z^{n\Theta_1^{\gamma}}}+\overline{u_2^2z^{\Theta_1^{\gamma}
+2\Theta_2^{\gamma}}})|z|^{-2\Theta_1^{\gamma}}=0\\
 \\
\bar{z}\frac{\bpat u_2}{\pat \bar{z}}+\overline{2u_1u_2
z^{\Theta_1^{\gamma}
+2\Theta_2^{\gamma}}}|z|^{-2\Theta_2^{\gamma}}=0.
\end{array}
\right.
\end{equation}
\end{enumerate}
\end{ex}

\subsubsection{The perturbed Witten equation}

When a variable is fixed by the local group at a point, we call the variable a \emph{broad variable} and we call the point a \emph{broad point}.  

When there are broad points, the Witten Lemma fails and
there are nontrivial solutions to  the Witten equation. Then, we
have to study the moduli space of solutions of the  Witten
equation. The first step is to study the asymptotic behavior of
the solution at the marked points. The moduli problem makes sense
only if the solution has nice asymptotic behavior. From the
analytic point of view, the Witten equation is highly degenerate
in the same way that zero is a highly degenerate critical point of
$W$. There is a sophisticated $L^2$-moduli space theory by
Morgan-Mrowka-Ruberman \cite{MMR} and Taubes \cite{Ta} in the
literature
 to deal with the situation being considered. However, we choose to perturb the
 equation,
 which simplifies the analysis immensely. An unexpected bonus of our
 approach is the appearance of vanishing cycles. Of course, we have to pay a price by
 studying the dependence of our invariants  on the perturbation.

  Our strategy is to modify the
Witten equation for the broad variable without changing the
equation for the narrow variable. We only need to choose linear
perturbations.

\begin{df} Let $W$ be a quasi-homogeneous polynomial. A linear polynomial $W_0=\sum_i b_i x_i$ is
called $W$-regular if $W+W_0$ has only nondegenerate critical
points. Namely, $W+W_0$ is a holomorphic Morse function. $W_0$ is
called strongly $W$-regular if $W_0$ is regular and for any two
different critical points $\kappa^i\neq \kappa^j$,
$\mbox{Im}(W+W_0)(\kappa^i)\neq \mbox{Im}(W+W_0)(\kappa^j)$.
\end{df}

\begin{thm}\label{thm-para-chamber} Given a quasi-homogeneous polynomial $W$. The set of $W$-regular $W_0$ forms a non-empty path connected open
   submanifold of all $W_0$. The subset of $W$-regular $W_0$ where two or more critical points have
   the same $\mbox{Im}(W+W_0)$ is a union of real hypersurfaces separating the set of $W$-regular $W_0$
   into a system of chambers.
\end{thm}

\begin{proof}

Let
$$X=\{(x_1, \dots, x_N; b_1, \dots, b_N); \frac{\partial
W}{\partial x_i}+b_i=0, \det(\frac{\partial^2 W}{\partial
x_i\partial x_j})=0\}\subset \C^N\times \C^N.$$
    We claim that $X$ is an affine variety of complex dimension $N-1$.
    It is clear that the
    set of regular $W_0$ is $\C^N-\pi_2(X)$, where $\pi_2:
    \C^N\times \C^N\rTo \C^N$ is defined by
    $(x_1, \dots, x_N, b_1, \dots, b_N)\rTo (b_1, \dots,
    b_N)$. Since $\pi_2$ is an algebraic map, $\pi_2(X)$ is an
    algebraic subset of dimension $\leq N-1$. Therefore,
    $\C^N-\pi_2(X)$ is an open connected submanifold.

    Let
    $$F=(\frac{\partial W}{\partial x_1}+b_1, \dots, \frac{\partial W}{\partial x_N}+b_N ):
    \C^N\times \C^N\rTo \C^{N}.$$
    Let $\tilde{x}_i, \tilde{b}_j$ be the corresponding tangent
    vectors. Then,
    $$DF_{x, b}(\tilde{x}_1, \dots, \tilde{x}_N, \tilde{b}_1,
    \dots, \tilde{b}_N)=(\sum_{j}\frac{\partial^2 W}{\partial
x_1\partial x_j}\tilde{x}_j+\tilde{b}_1, \dots,
\sum_{j}\frac{\partial^2 W}{\partial x_N\partial
x_j}\tilde{x}_j+\tilde{b}_N)$$
    is obviously surjective. Hence, $F^{-1}(0)$ is a smooth affine
    variety. $X=F^{-1}(0)\cap \det(\frac{\partial^2 W}{\partial
x_i\partial x_j})^{-1}(0)$ is a hypersurface of $F^{-1}(0)$ of
    dimension $N-1$ unless
    $\det(\frac{\partial^2 W}{\partial
x_i\partial x_j})$ vanishes on some component of $F^{-1}(0)$.
    On the other hand, $(x_1, \dots, x_N)$ is a nondegenerate
    critical points iff $(x_1, \dots, x_N)$ is a regular value of
    the projection map $\pi_1: F^{-1}(0)\rTo \C^t$. Hence,
    the set of nondegenerate critical points is dense and $\det(\frac{\partial^2 W}{\partial
x_i\partial x_j})$ can not vanish on any open subset of
$F^{-1}(0)$.

    It is known that the cardinality of the set of nondegenerate critical points
    is
    precisely the dimension of the middle homology of the Milnor fiber and
    hence a constant number. The manifold of nondegenerate critical points
    $Y=F^{-1}(0)-\det(\frac{\partial^2 W}{\partial
x_i\partial x_j})^{-1}(0)$ is a disjoint union of components
$Y=\bigcup_{1\leq i\leq \mu_W} Y_i$. Furthermore, each component
$Y_i$ is isomorphic to $\C^N-\pi_2(X)$. Let
    $$f_i: \C^N-\pi_2(X)\rTo Y_i\stackrel{\mbox{Im}(W+W_0)}{\rTo}
    \R.$$
     Note that
     $D \mbox{Im}(W+W_0)(\tilde{x}_1, \dots, \tilde{x}_N)=0$ at critical points.
     $$D \mbox{Im} (W+W_0)(\tilde{b}_1, \dots, \tilde{b}_N)=\mbox{Im} (x_1\tilde{b}_, \dots, x_N\tilde{b}_N).$$
     Let $(x^i_1, \dots, x^i_N, b_1, \dots, b_N)\in Y_i$.
     $$D(f_i-f_j)(\tilde{b}_1, \dots,
     \tilde{b}_N)=\mbox{Im}((x^i_1-x^j_1)\tilde{b}_1, \dots,
     (x^i_t-x^j_N)\tilde{b}_N)$$
     is surjective for $i\neq j$. Therefore,  zero is a regular
     value of $f_i-f_j$ and $(f^i-f^j)^{-1}(0)$ is a smooth real
     hypersurface. We have finished the proof.

     \end{proof}

Suppose that $\gamma$ parameterizes the orbifold structure at the
marked point. The quasi-homogeneous polynomial can be decomposed
into the sum of the polynomials $W_{\gamma}$ and $W_N$. Take a
fixed $W_\gamma$-regular polynomial
$$
W_0=W_0(x_{1},\dots, x_{N_l})=\sum_{j=1}^{N_l}b_{j} x_{j}.
$$
 There exist line bundles $O_j(l), j=1,\dots,N_l$ over the orbicurve $\tilde{\Sigma}$ such that
\begin{enumerate}
\item $O_{j}(l)$ has no orbifold structure near the marked point
$z_l$;

\item the following isomorphisms hold:
$$
O_{j}(l)\otimes \LL_{j}\cong K_{\log}.
$$
\end{enumerate}
Hence, near the marked point $z_l$ we have the local isomorphisms
$$
|O_{j}(l)|\otimes |\LL_{j}|\cong K_{\log}.
$$
If $j\leq N_l$, we define the section $\beta_{j}=\beta_T(z)
dz/z\otimes e'_j\in \Omega^0(|O_{j}|),$ and $\beta_T(z)$ is given
by
\begin{equation}
\beta_T(z)=\left\{
\begin{array}{ll}
1 & \;\text{if}\; |z|\le e^{-T-1}\\
0 & \; \text{if}\; |z|\ge e^{-T}.
\end{array}
\right.
\end{equation}
Let
$$
W_{0,\beta}(u_{1},\dots, u_{N_l})=\sum_{j=1}^{N_l} b_{j}
\beta_{j} u_{j}.
$$
Hence the global perturbed Witten equation is defined as
\begin{equation}\label{plgw-equ}
\bar{\partial}u_i+I_1\left(\overline{\frac{\partial W}{\partial
u_i}+\frac{\pat W_{0,\beta}}{\pat u_i}}\right)=0, \quad\mbox{ for
all } i=1,\dots, N.
\end{equation}

Locally near a marked point $z_l$, the perturbed Witten-equation
has the following form:
\begin{equation}\label{plgw-r-equ}
\bar{z}\frac{\bpat u_i}{\pat \bar{z}}+\overline{\frac{\pat
W_{\gamma}}{\pat u_i}}+\beta_T\overline{\frac{\pat W_0}{\pat
u_i}}+\overline{\sum_{W_j: W_j \text{in} W_N } \frac{\pat
W_j(u_1,\dots, u_N)}{\pat
u_i}z^{\sum_{s=1}^{N}b_{js}(\Theta^{\gamma}_s)}}=0,
\end{equation}
for any $i\leq N_l$, and
\begin{equation}\label{plgw-n-equ}
\bar{z}\frac{\bpat u_i}{\pat \bar{z}}+(\overline{\sum_{W_j: W_j
\text{in} W_N } \frac{\pat W_j(u_1,\dots, u_N)}{\pat
u_i}z^{\sum_{s=1}^{N}b_{js}(\Theta^{\gamma}_s)}})|z|^{-2\Theta^{\gamma}_i}=0,
\end{equation}
for any $i>N_l$. Here $u_i, \beta_T$ in the equations are
understood not as sections, but as the coordinate functions with
respect to each basis of  the line bundles.

If we set $v_i=u_i$, for all $i\leq N_l$ and $v_i=u_i
z^{\Theta^{\gamma}_i}$, then the above two equations can be
changed to a more symmetric form:
\begin{equation}\label{plgw-r-equ1}
\bar{z}\frac{\bpat v_i}{\pat \bar{z}}+\overline{\frac{\pat
(W_{\gamma}+W_{0,\beta})}{\pat v_i}}+\overline{\frac{\pat
W_N(v_1,\dots, v_N)}{\pat v_i}}=0,\;\forall i\leq N_l
\end{equation}
and
\begin{equation}\label{plgw-n-equ1}
\bar{z}\frac{\bpat v_i}{\pat \bar{z}}+\overline{\frac{\pat
W_N(v_1,\dots, v_N)}{\pat v_i}}=0, \;\forall i>N_l,
\end{equation}

or simply written as one equation

\begin{equation}\label{plgw-equ1}
\bar{z}\frac{\bpat v_i}{\pat \bar{z}}+\overline{\frac{\pat
(W+W_{0,\beta})}{\pat v_i}}=0, \; i=1,\dots,N.
\end{equation}

Those sections $v_i$ for $i>N_l$ are not well defined in any
neighborhood containing $z_l$. Hence the equations
(\ref{plgw-r-equ1}), (\ref{plgw-n-equ1}) and (\ref{plgw-equ1})
only hold in any fan-shaped domain of $z_l$. However, the absolute
value $|v_i|$ of each $v_i$ is well-defined.

Take a cylindrical coordinate near a marked point $z_l$, i.e., let
$z=e^{-(s+i\theta)}$; then the equations
(\ref{plgw-r-equ1})+(\ref{plgw-n-equ1}) can be rewritten as the
simpler form:
\begin{align}
&\frac{\pat v_i}{\pat s}+\sqrt{-1}\frac{\pat v_i}{\pat
\theta}-2\overline{\frac{\pat (W+W_{0,\beta})}{\pat v_i}}=0,\;
\forall i\leq N_l\label{plgw-r-equ2}\\
&\frac{\pat v_i}{\pat s}+\sqrt{-1}\frac{\pat v_i}{\pat
\theta}-2\overline{\frac{\pat W_{N}}{\pat v_i}}=0,\; \forall
i>N_l\label{plgw-n-equ2}.
\end{align}
Set $v_i=x_i+\sqrt{-1}y_i$. Let
$$
H_{R0}:=2 Re(W_{\gamma}+W_{0,\beta}),\; H_N:=2 \mbox{Re}(W_N),
$$
then $\overline{\frac{\pat (W+W_{0,\beta})}{\pat
v_i}}=\overline{\frac{\pat (W+W_{0,\beta})}{\pat v_i}+\frac{\pat
\overline{(W+W_{0,\beta})}}{\pat v_i}}=\frac{\bpat }{\pat
\bar{v}_i}(H_{R0}+H_N)$; the equation (\ref{plgw-r-equ2}) becomes
$$
\frac{\pat x_i}{\pat s}+\sqrt{-1}\frac{\pat y_i}{\pat
s}+\sqrt{-1}(\frac{\pat x_i}{\pat \theta}+\sqrt{-1}\frac{\pat
y_i}{\pat \theta})-(\frac{\pat}{\pat x_i}+\sqrt{-1}\frac{\pat
}{\pat y_i})(H_{R0}+H_{N})=0,
$$
i.e.,
\begin{equation}
\left\{
\begin{array}{l}
\frac{\pat x_i}{\pat s}-\frac{\pat y_i}{\pat
\theta}-\frac{\pat}{\pat x_i}(H_{R0}+H_{N})=0\\
 \\
\frac{\pat y_i}{\pat s}+\frac{\pat x_i}{\pat
\theta}-\frac{\pat}{\pat y_i}(H_{R0}+H_{N})=0, \;\forall i\leq
N_l.
\end{array}
\right.
\end{equation}

If we define $v_R=(x_{1},\dots,x_{N_l}, y_{1},\dots, y_{N_l})$,
and let
$$
J_R=\left(
\begin{matrix}
0& -I_{N_l}\\
I_{N_l}&0
\end{matrix}\right),
$$
then the above equation is just
\begin{equation}
\frac{\pat v_R}{\pat s}+J_R \frac{\pat v_R}{\pat \theta}-\nabla_R
H_{R0}=\nabla_R H_{N},
\end{equation}
where $\nabla_R$ is the gradient operator with respect to the
variables $\{x_{1},\dots,x_{N_l}, y_{1},\dots, y_{N_l}\}$.

Similarly, if we set $v_N=(x_{N_l+1},\dots, x_{N},
y_{N_l+1},\dots, y_{N})$  and
$$
J_N=\left(
\begin{matrix}
0& -I_{N-N_l}\\
I_{N-N_l}&0
\end{matrix}\right),
$$
then the equation (\ref{plgw-n-equ2}) becomes
\begin{equation}
\frac{\pat v_N}{\pat s}+J_N \frac{\pat v_N}{\pat \theta}=\nabla_N
H_N,
\end{equation}
where $\nabla_N$ is the gradient operator with respect to the
variables $\{x_{N_l+1},\dots, x_{N}, y_{N_l+1},\dots, y_{N} \}$.
Thus the equations (\ref{plgw-r-equ2}) and (\ref{plgw-n-equ2}) are
changed to
\begin{equation}\label{holo-equ}
\left\{
\begin{array}{l}
\frac{\pat v_R}{\pat s}+J_R \frac{\pat v_R}{\pat \theta}-\nabla_R
H_{R0}(v_R)=\nabla H_{N}(v_R, v_N)\\
 \\
\frac{\pat v_N}{\pat s}+J_N \frac{\pat v_N}{\pat \theta}=\nabla_N
H_N(v_R, v_N).
\end{array}
\right.
\end{equation}

Later we will show that the first equation of (\ref{holo-equ}) is
just the perturbation equation of the trajectory equation of a
Hamiltonian system which has frequently appeared in symplectic
geometry, and the second one is the perturbation equation of the
Cauchy-Riemann equation of pseudo-holomorphic curves.

\begin{df} Sections $(u_1,\dots, u_N)$ are said to be the
solutions of the perturbed Witten equation (\ref{plgw-equ}); if
they satisfy the following conditions:
\begin{enumerate}

\item for each $j, u_j\in
L^2_{1,loc}(\Sigma\setminus\{z_1,\dots,z_k\},
|\LL_j|),I_1\left(\overline{\frac{\partial W}{\partial
u_j}+\frac{\partial W_{0,\beta}}{\partial u_j}}\right)\in
[L^2_{loc}(\Sigma\setminus\{z_1,\dots,z_k\}, |\LL_j|\otimes
\Lambda^{0,1})]$;

\item $(u_1,\dots,u_N)$ satisfy the perturbed  Witten equation
(\ref{plgw-equ}) almost everywhere;

\item near each marked point, the integral
$$
\sum_j\int^\infty_0\int_{S^1} |\frac{\pat u_j}{\pat s}|^2 d\theta
ds<\infty.
$$
\end{enumerate}
\end{df}

\subsubsection{ Witten equations over orbicurves}\label{subsub-witten-orbi}

In the last section, we defined the  Witten equation on a
smooth Riemann surface. However, since our moduli theory will be
constructed on orbicurves, actually we should define our
 Witten equation on orbifolds. In this part, we will define the
 Witten equation over orbifolds and will show that the resolution
of these solutions of  the Witten equation just satisfy the
previously defined  Witten equation on a Riemann surface and vice
versa.

Because $W_j(\LL_1,\dots, \LL_N)\cong K_{\log}$, we can choose a
holomorphic basis $\eta_i$ of the orbifold line bundle $\LL_i$
such that $W_j(\eta_1,\dots, \eta_N)=dz$ when away from the
marked points and take the smooth metric in this part. If $z_l$ is
a marked point, we can assume that the orbifold structure near
$z_l$ is given by $(\tilde\Delta, \pi, G_{z_l}=\Z/m)$, where
$\tilde\Delta$ is a disc with coordinate $z$, 
$\tilde\Delta\rTo \Delta$ is given by $\pi(z)=z^m$, and
$G_{z_l}$ acts by $z\rTo \exp(2\pi i/m)z$. The orbifold
structure of $\LL_i$ is given by the uniformizing system
$\tilde\Delta\times \C\longrightarrow [(\tilde\Delta\times
\C)/G_{z_l}]$, with the action of $G_{z_l}$ given by
$(z,w)\rTo (\exp(2\pi i/m)z, \exp(2\pi v_i i/m)w)$ with
$m>v_i\ge 0$. Now the sheaf of sections of $\LL_i$ is generated as
an $\O_{\tilde\Delta}$-module by an element $\eta_i$ such that for
each $j$ there is
\begin{equation}\label{isorm-base-W}
W_j(\eta_1,\dots,\eta_N)=mdz/z.
\end{equation}
$G_{z_l}$ acts on $\eta_i$ by $\eta_i\rTo \exp(-2\pi v_i
i/m)\eta_i$. The invariant sections are of the form
$u_i(z)=z^{v_i}\tilde{g}(z^m)\eta_i$. Let $\tilde{u}_i$ be the
coordinate function of $u_i(z)$ with respect to $\eta_i$; then
$\tilde{u}_i(z)=z^{v_i}\tilde{g}_i(z^m)$, which is equivariant. Let
$w=z^m$ be the coordinate in $\Delta$. The sheaf of sections of
$|\LL_i|$ is generated by the section $e_i(w)=z^{v_i}\eta_i(z)$.
So $u_i=\tilde{g}_i(z^m)e_i(z^m)$ and the corresponding coordinate
function is $\tilde{g}_i(z^m)$ and actually we have the relation
between two coordinate functions:
\begin{equation}\label{chang-two-coord}
\tilde{u}_i(z)=z^{v_i}\tilde{g}_i(z^m)=w^{\Theta_i}\tilde{g}_i(w).
\end{equation}

Now we choose the cylindrical metric in the small charts of all of
the marked points on the resolved surface, i.e., let $|dw/w|=1$.
This metric induces the metric on $\tilde{\Sigma}$ such that $|m
dz/z|=1$. By (\ref{isorm-base-W}) this metric induces the metric
on each orbifold line bundle such that $|\eta_i|=1, \forall i,$ in
small charts of all marked points. Let $\eta'_i$ be the dual basis
of $\eta_i$ on the dual line bundle $\LL_i^{-1}$. We define a
metric preserving isomorphism:
$$
\tilde{I}_1: \Omega(\tilde{\Sigma}, \overline{\LL}_j^{-1}\otimes
\Lambda^{0,1})\rTo \Omega(\tilde{\Sigma}, \LL_j\otimes
\Lambda^{0,1})
$$
such that for $v_i=\tilde{v}_i \overline{\eta'_i}$, there holds
$$
\tilde{I}_1(v_i\otimes d\bar{z})=\tilde{v}_i |\eta'_i|^2
\eta_i\otimes d\bar{z}.
$$

It is easy to see if $u_i\in \Omega(\tilde{\Sigma}, \LL_i)$, then
$$
\tilde{I}_1\left(\overline{\frac{\pat W}{\pat u_i}}\right)\in
\Omega(\tilde{\Sigma}, \LL_i\otimes \Lambda^{0,1}).
$$
Now the  Witten equation on orbifolds is defined as
\begin{equation}\label{lgw-orbi}
\bpat u_i+\tilde{I}_1\left(\overline{\frac{\pat W}{\pat
u_i}}\right)=0, \forall i=1,\dots,N.
\end{equation}

We consider the local expression of the Witten equation in a
neighborhood of the marked point $z_l$. Assume that
$u_i=\tilde{u}_i \eta_i, i=1,\dots, N$, are solutions of equation
(\ref{lgw-orbi}) in this uniformizing system of $z_l$; then the
coordinate function $\tilde{u}_i$ satisfies
\begin{equation}\label{lgw-orbi3}
\frac{\bpat \tilde{u}_i}{\pat\bar{z}}+\overline{\frac{\pat
W(\tilde{u}_1,\dots,\tilde{u}_N)}{\pat
\tilde{u}_i}\frac{m}{z}}=0.
\end{equation}

Replacing $\tilde{u}_i$ in (\ref{lgw-orbi3}) by the equality
(\ref{chang-two-coord}), we find out that
\begin{equation}\label{lgw-orbi2}
\frac{\bpat \tilde{g}_i}{\pat\bar{w}}+\overline{\sum_j \frac{\pat
W_j(\tilde{g}_1,\dots, \tilde{g}_N)}{\pat
\tilde{g}_i}w^{\sum_{s=1}^{N}b_{js}(\Theta^{\gamma}_s-q_s)}}|w|^{-2\Theta^{\gamma}_i}=0,
\end{equation}
where $w=z^m$. This is just  equation (\ref{lgw-equ}) when
evaluated at $z^m$.

Conversely, if one knows that the function $\tilde{g}_i$ of
$|\LL_i|$ satisfies  equation (\ref{lgw-orbi2}), then
$\hat{u}_i(w)=w^{\Theta_i}\tilde{g}_i(w)$ satisfies
\begin{equation}\label{lgw-orbi4}
\frac{\bpat \hat{u}_i}{\pat\bar{w}}+\overline{\frac{\pat
W(\hat{u}_1,\dots,\hat{u}_N)}{\pat \hat{u}_i}\frac{1}{w}}=0.
\end{equation}

Notice that  equation (\ref{lgw-orbi4}) holds locally at a broad
point. 
On the other hand, the solutions of this Witten equation in
the neighborhood of any orbifold point can be viewed as solutions
satisfying some twisted periodic boundary condition. Equation
(\ref{lgw-orbi4}) is changed to  equation (\ref{lgw-orbi3}) when
the potential $W$ is replaced by $mW$. This observation is
beneficial in estimating of the solutions of (\ref{lgw-orbi3}).

Similarly, we can define the perturbed Witten equation if there
exist any broad marked points. The definition is completely the
same as before, since there is no orbifold structure for these
broad line bundles.

\
\subsection{Interior estimate of solutions}\

\

When away from the marked points, the perturbed  Witten equation
can be written in the form
\begin{equation}\label{bdd-lgw}
\bar{\partial}u_j+\varphi_j\left(\overline{\frac{\partial
W}{\partial u_j}+\beta_j\frac{\partial W_0}{\partial
u_j}}\right)=0, \quad\mbox{ for all } j=1,\dots, N,
\end{equation}
where $\varphi_j$ is a smooth positive function. Here we
understand that for $i>N_l, \beta_i\equiv 0$ and for $i\leq N_l,
\beta_i$ is defined as before. The polynomial $W_0$ can be chosen
differently corresponding to different marked points, but they
should satisfy the following uniform control:

\begin{equation}\label{w0-bdd}
\left|\frac{\pat W_0}{\pat u_i}\right|\le b.
\end{equation}

To study the solutions, we also need the following lemma.

\begin{lm}[\cite{FJR1}, Theorem 5.7]\label{thm-new}
Let $W \in \mathbb{C}[x_1, \dots, x_N]$ be a non-degenerate,
quasi-homogeneous polynomial with weights $q_i:=\wt(x_i)<1$ for
each variable $x_i,i=1, \dots,N$. Then for any $t$-tuple $(u_1,
\dots, u_N) \in \mathbb{C}^N$ we have

\[ |u_i| \leq C \left(\sum^N_{i=1}\left|\frac{\partial W}{\partial x_i}(u_1, \dots,
u_N)\right|+1 \right)^{\delta_i},\] where
$\delta_i=\frac{q_i}{\min_j(1-q_j)}$ and the constant $C$ depends
only on $W$. If $q_i\leq 1/2$ for all $ i \in \{1, \dots, N\}$,
then $\delta_i\le 1$ for all $ i \in \{1, \dots, N\}$. If
$q_i<1/2$ for all $ i \in \{1, \dots, N\}$, then $\delta_i<1$ for
all $ i \in \{1, \dots, N\}$.
\end{lm}

So by this lemma we have
\begin{equation}\label{W-bdd}
|u_i| \leq C \left(\sum^N_{i=1}\left|\frac{\partial W}{\partial
x_i}(u_1, \dots, u_N)\right|+1 \right)^{\delta_0},
\;\delta_0=\max\{\delta_1,\dots,\delta_N\}.
\end{equation}

In the rest of this paper, we always assume that $q_i<1/2$ for any
$i$; therefore we have $\delta_0<1$.

\begin{thm}\label{bdd-inn-thm} Let $(u_1,\dots, u_N)$ be the solutions of the
perturbed  Witten equation (\ref{plgw-equ}); then for any $m\in
\Z$, there holds
$$
||u_j||_{C^{m}(B_R)}\le C,
$$
where $B_R\subset \Sigma\setminus \{z_1,\dots,z_k\}$, and $C$ is
a constant  depending only on the metric in $B_{2R}$, the
fractional degree $q_i$, $ b$ and $R$.
\end{thm}

\begin{proof}
Multiplying (\ref{bdd-lgw}) by $\frac{\pat W}{\pat u_j}$ and
taking the sum, one has
\begin{equation}\label{bdd-pro1}
\frac{\bpat W}{\pat \bar{z}}+\left( \sum_i \varphi_i\left|
\frac{\pat W}{\pat u_i}\right|^2+\sum_i \varphi_i\beta_i
\frac{\pat W}{\pat u_i}\overline{\frac{\pat W_0}{\pat
u_i}}\right)=0.
\end{equation}

Take a small ball $B_{2R}(z_0)$ and a test function $\psi$ such
that $\psi\in C^\infty_0(B_{2R}(z)), \psi\equiv 1$ on $ B_R(z)$,
and $\psi \equiv 0$ outside $B_{2R}(z)$, and $|\nabla \psi|\le
\frac{C}{R}$.

Multiplying (\ref{bdd-pro1}) by $\psi^\beta
\frac{\sqrt{-1}}{2}dz\wedge d\bar{z}$ and integrating over $B$,
one has
\begin{align}\label{bdd-pro2}
&\int_B \frac{\bpat W}{\pat \bar{z}}\psi^\beta
\frac{\sqrt{-1}}{2}dz\wedge d\bar{z}+\int_B  \sum_j \varphi_j
\left|\frac{\pat W}{\pat u_j}\right|^2|dz d\bar{z}|\psi^\beta \nonumber\\
&+\int_B  \sum_j \varphi_j\beta_j \frac{\pat W}{\pat
u_j}\overline{\frac{\pat W_0}{\pat u_j}}\psi^\beta=0.
\end{align}
Now there holds
\begin{align*}
&\int \frac{\bpat W}{\pat \bar{z}}\psi^\beta
\frac{\sqrt{-1}}{2}dz\wedge d\bar{z}=-\int
\frac{\sqrt{-1}}{2}\psi^\beta d(Wdz)\\
&=\int \frac{\sqrt{-1}}{2}\beta \psi^{\beta-1} d\psi\wedge
Wdz=\int \pat_{\bar{z}}\psi \beta W \psi^{\beta-1}.
\end{align*}
By (\ref{bdd-pro2}) and (\ref{w0-bdd}), one obtains for large
$\beta$ that
\begin{align*}
&\int_B \psi^\beta \sum_j \varphi_j\left| \frac{\pat W}{\pat
u_j}\right|^2 \le \int_B |\pat_{\bar{z}}\psi| \beta|W|
\psi^{\beta-1}+\int_B  \sum_j \varphi_j\left|\frac{\pat W}{\pat
u_j}\right| b\psi^\beta\\
&\le \int_B \beta \psi^{\beta-1}|\pat_{\bar{z}}\psi| \sum_i q_i
|u_i| |\frac{\pat W}{\pat u_i}|+b\max_j \varphi_j\sum_j
\left|\frac{\pat W}{\pat
u_j}\right| \psi^\beta\\
&\le C\int_B \psi^{\beta-1}|\bpat \psi|\cdot \left(\sum_i \left|
\frac{\pat W}{\pat u_i} \right|^{1+\delta_0}\right)+C_2\int_B
|\bpat
\psi|\\
&+C_1 \max_j \varphi_j\int_B \sum_j \left|\frac{\pat W}{\pat u_j}\right|^{1+\delta_0}\psi^\beta+C_1 R^2\\
&\le \varepsilon \int \sum_j \left|\frac{\pat W}{\pat
u_j}\right|^2 \psi^\beta + C'_\varepsilon \int
\psi^{(\beta-\frac{1+\delta_0}{2}\beta-1)\frac{2}{1-\delta_0}}|\pat_{\bar{z}}\psi|^{\frac{2}{1-\delta_0}}+C_2 R\\
&+\varepsilon  \int \sum_j \left|\frac{\pat W}{\pat
u_j}\right|^2\psi^\beta+C''_\varepsilon R^2+C_1 R^2\\
&\le 2\varepsilon \int_B \sum_j \left|\frac{\pat W}{\pat
u_j}\right|^{2}\psi^\beta+C'_\varepsilon
R^{-\frac{2\delta_0}{1-\delta_0}}+C''_\varepsilon R^2+C_1 R^2+C_2
R.
\end{align*}
Then take $\varepsilon$ small enough such that
\begin{equation}\label{bdd-pro3}
\int_{B_R} \sum_j \left|\frac{\pat W}{\pat u_j}\right|^2\le C(q_i,
 b, \varphi_j, R).
\end{equation}

On the other hand, for the non-homogeneous $\bpat$-equation we
have the interior estimate
\begin{equation}
||u_j||_{L^2_1(B_{\frac{R}{2}})}\le C\left(\sum_i ||\frac{\pat
W}{\pat u_i}||_{L^2(B_R)}+\sum_i ||\frac{\pat W_0}{\pat
u_i}||_{L^2(B_R)}+||u_j||_{L^2(B_R)}\right).
\end{equation}

By (\ref{W-bdd}), there holds
$$
||u_j||^2_{L^2}\le C\int \left(\sum_j \left|\frac{\pat W}{\pat
u_j}\right|+1\right)^{2\delta_0}\le C_1 \sum_j ||\frac{\pat
W}{\pat u_j}||^2_{L^2}+C_2,
$$
and hence by (\ref{bdd-pro3}) we obtain
$$
||u_j||_{L^2_1(B_{\frac{R}{2}})}\le C_1 \sum_j ||\frac{\pat
W}{\pat u_j}||_{L^2}+C_2\le C.
$$
Applying the embedding theorem, there holds
\begin{equation}\label{bdd-pro4}
||u_j||_{L^p}\le C||u_j||_{L^2_1}\le C,\;\forall 1<p<\infty.
\end{equation}
Hence, for any $p>1$, we have
\begin{equation}\label{bdd-pro5}
||\frac{\pat W}{\pat u_j}||^p_{L^p}\le C,
\end{equation}
where $C$ depends on $||u_i||_{L^2}$ and $||\frac{\pat W}{\pat
u_j}||_{L^2}$.

Now using the $L^p$-estimate of the $\bpat$-equation,
(\ref{bdd-pro4}) and (\ref{bdd-pro5}), eventually we have
$$
||u_j||_{L^p_1(B_{\frac{R}{2}})}\le C\left(\sum_j ||\frac{\pat
W}{\pat u_j}||_{L^p(B_R)}+||\frac{\pat W_0}{\pat
u_j}||_{L^p(B_R)}+||u_j||_{L^p(B_R)}\right)\le C.
$$
Using the embedding theorem again, we have
$$
||u_j||_{C^\alpha(B_R)}\le C.
$$
Furthermore, by the bootstrap argument, we have the following
estimate for any $m\in \Z$:
$$
||u_j||_{C^m}\le C.
$$
\end{proof}

\

\subsection{Asymptotic behavior}\label{subsec:asym-beha}

\

Set the complex variable $\xi=s+i\theta$; then the equations
(\ref{plgw-n-equ2}) and (\ref{plgw-r-equ2}) can be written as
\begin{equation}\label{plgw-equ2}
\bpat_\xi v_i-2\overline{\frac{\pat(W+W_{0,\beta})}{\pat
v_i}}=0,\;\forall i=1,\dots,N.
\end{equation}
Note that the function $v_i$ is only locally defined; hence the
above relation only holds in a bounded domain $D\subset S^1\times
[0,\infty)$ such that $D$ is contractible.

\begin{thm}\label{bdd-bdry-thm} Let $B_R\subset S^1\times [0,\infty)$ be a ball with
radius $R$ such that $B_{2R}$ is contractible. Suppose that
$v_1,\dots, v_N$ satisfy (\ref{plgw-equ2}) in $B_{2R}$; then for
any $m>0$, there is a constant $C$ depending only on $q_i, b, R,m$
such that
$$
||v_j||_{C^m(B_R)}\le C,
$$
where the derivatives taken in the $C^m$ modulus  correspond to
the variables $s, \theta$.
\end{thm}

\begin{proof} Multiplying the two sides of (\ref{plgw-equ2}) by
$\frac{\pat W}{\pat v_i}$ and taking the sum over $i$ from $1$ to
$N$, we have
$$
\pat_{\bar{\xi}} W=2\sum_i \left|\frac{\pat W}{\pat v_i}
\right|^2+2\sum_i \overline{\frac{\pat W_{0,\beta}}{\pat
v_i}}\frac{\pat W}{\pat v_i}.
$$
Multiplying the above equality by $\psi^\beta \frac{\sqrt{-1}}{2}
d\xi\wedge d\bar{\xi}$ and integrating over $B_{2R}$, one has
$$
\int_{B_{2R}} \pat_{\bar{\xi}}W (\psi^\beta \frac{\sqrt{-1}}{2}
d\xi\wedge d\bar{\xi})=2\int_{B_{2R}} \sum_i\left|\frac{\pat
W}{\pat v_i} \right|^2 \psi^\beta |d\xi d\bar{\xi}|+2\int_{B_{2R}}
\sum_i \overline{\frac{\pat W_{0,\beta}}{\pat v_i}}\frac{\pat
W}{\pat v_i} \psi^\beta |d\xi d\bar{\xi}|.
$$
Using Stokes' theorem for the $\bpat$ operator, we have
$$
2\int \sum_i\left|\frac{\pat W}{\pat v_i} \right|^2 \psi^\beta
|d\xi d\bar{\xi}|=\int \pat_{\bar{\xi}}\psi \beta W
\psi^{\beta-1}-2\int_{B_{2R}} \sum_i \overline{\frac{\pat
W_{0,\beta}}{\pat v_i}}\frac{\pat W}{\pat v_i} \psi^\beta |d\xi
d\bar{\xi}|.
$$
Now we can proceed in the same way as in the proof of Theorem
\ref{bdd-inn-thm} to obtain the conclusion.
\end{proof}

Define $\hat{v}_i=u_i e^{-\Theta^{\gamma_l}_is}$ for any $i$; then
the  $\hat{v}_i$ are well-defined functions in the cylinder
$S^1\times [0,\infty)$. We have the relations:
\begin{align}
&\hat{v}_i=v_i=u_i,\;\forall i\leq N_l;\\
&|\hat{v}_i|=|v_i|=|u_i|e^{-\Theta_i^{\gamma_l}s},\;\forall i>N_l.
\end{align}

The equation (\ref{plgw-n-equ}) of $u_i$ for any $i>N_l$ can be
rewritten as
$$
\pat_{\bar{\xi}}u_i+\overline{\frac{\pat W_N(v_1,\dots,v_t)}{\pat
v_i}}e^{\Theta_i(s+i\theta)}=0.
$$
Hence $\hat{v_i}$ for $i>N_l$ satisfies
\begin{equation}
(\pat_s+\sqrt{-1}\pat_\theta+\Theta_i)\hat{v}_i=-\overline{\frac{\pat
W_N(v_1,\dots,v_N)}{\pat v_i}}e^{i \Theta_i\theta}.
\end{equation}

\begin{lm}\label{inte-1} Let $u_1,\dots, u_N$ be the solutions of the equations
(\ref{plgw-n-equ}) and (\ref{plgw-r-equ}) defined in the cylinder
$S^1\times [0, \infty)$. Then there holds
$$
\int^\infty_0 \int_{S^1} \sum_{i\leq N_l} \left| \frac{\pat
(W_{\gamma}+W_{0,\beta})}{\pat u_i} \right|^2+\sum_{i>N_l}
\left|\sum_j \frac{\pat W_j(u_1,\dots,u_N)}{\pat u_i} e^{-\sum_l
b_{jl}\Theta_l (s+i\theta)+\Theta_i s} \right|^2 ds d\theta<\infty.
$$
\end{lm}

\begin{proof} Note that $\sum_l b_{jl} \Theta_l=0 \mod \Z$; hence
$$
\sum_j \frac{\pat W_j(u_1,\dots,u_N)}{\pat u_i} e^{-\sum_l
b_{jl}\Theta_l (s+i\theta)+\Theta_i s}
$$
is well defined on the whole cylinder and we can represent it as
the combination of the locally defined functions $v_i, i=1,\dots,
N$, as follows:
\begin{align}
&\sum_j \frac{\pat W_j(u_1,\dots,u_N)}{\pat u_i} e^{-\sum_l
b_{jl}\Theta_l (s+i\theta)+\Theta_i s}=\sum_j \frac{\pat
W_j(v_1,\dots,v_N)}{\pat v_i} e^{-i \Theta_i\theta}\nonumber\\
&=\frac{\pat W_N(v_1,\dots,v_N)}{\pat v_i} e^{-i \Theta_i\theta}.
\end{align}
 So what we need to prove is actually the following estimate:
\begin{equation}
\int^\infty_0 \int_{S^1} \sum_{i\leq N_l} \left| \frac{\pat
(W_{\gamma}+W_{0,\beta})}{\pat v_i} \right|^2+\sum_{i>N_l} \left|
\frac{\pat W_N(v_1,\dots,v_N)}{\pat v_i} \right|^2 ds
d\theta<\infty,
\end{equation}
or in equivalent form:
\begin{equation}\label{asym-n-proof1}
\int^\infty_0 \int_{S^1} \sum_{i=1}^N \left| \frac{\pat
(W+W_{0,\beta})}{\pat v_i} \right|^2<\infty.
\end{equation}
To prove (\ref{asym-n-proof1}), we multiply the two sides of the
equation (\ref{plgw-equ2}) by $\frac{\pat (W+W_{0,\beta})}{\pat
v_i}$ and take the sum of $i$ from $1$ to $N$, then there holds
$$
\pat_{\bar{\xi}} (W+W_{0,\beta})=2\sum_i \left| \frac{\pat
(W+W_{0,\beta})}{\pat v_i} \right|^2+(\sum_{j=1}^{N_l} b_j
v_j)\pat_{\bar{\xi}}\beta_T.
$$
Integrating the above equality over $S^1\times [0,T_0]$ for
$T_0>T,$ where $T$ appears in the definition of $\beta_T$, and
noting that $(W+W_{0,\beta})(v_1,\dots, v_N)$ is a well-defined
function on the cylinder, one has
$$
\int^{T_0}_0 \pat_s \int_{S^1} W+\int^{T_0}_0
\int_{S^1}\sum_{j=1}^{N_l} b_j \pat_s v_j
\beta_T=2\int^{T_0}_0\int_{S^1} \sum_i \left| \frac{\pat
(W+W_{0,\beta})}{\pat v_i} \right|^2.
$$
Hence
\begin{align*}
&2\int^{T_0}_0\int_{S^1} \sum_i \left| \frac{\pat
(W+W_{0,\beta})}{\pat v_i} \right|^2 \\
&=\int_{S^1} W(T_0, \theta)d\theta-\int_{S^1} W(0,
\theta)d\theta+\int^{T_0}_T \int_{S^1}\sum_j b_j \pat_s
v_j+\int^T_{T-1}\int_{S^1}
\sum_j b_j \pat_s v_j \beta_T.\\
&=\int_{S^1}(W+W_0)(T_0,\theta) d\theta-\int_{S^1}
W_0(T,\theta)-\int_{S^1} W(0,\theta)+\int^T_{T-1}\int_{S^1} \sum_j
b_j \pat_s v_j \beta_T.
\end{align*}
By Theorem \ref{bdd-inn-thm} and Theorem \ref{bdd-bdry-thm}, one
knows that $|v_i|$ and $|\pat_s v_i|$ are uniformly bounded for
any $i=1,\dots, N$. Thus the inequality (\ref{asym-n-proof1})
holds; in particular, the upper bound is independent of $T$.
\end{proof}

\begin{lm}\label{homo-int-esti} Let $0<\Theta<1$. Suppose that $v$ is a smooth bounded
solution of the equation
\begin{equation}\label{homo-cyln-equa}
(\pat_s+\sqrt{-1}\pat_\theta+\Theta)v=0
\end{equation}
satisfying the integrability condition
$$
\int^\infty_0 \int_{S^1} |\pat_s v|^2<\infty.
$$
Then for any $T\in (0,\infty)$, there holds
\begin{align}
\int^\infty_T \int_{S^1} |v|^2\le \frac{1}{\Theta^2} \int^\infty_T
\int_{S^1} |\pat_s v|^2
\end{align}
and
\begin{equation}
|v(s,\theta)|\le \sqrt{\frac{2}{\Theta}}e^{-\Theta
s}(\int^\infty_0 \int_{S^1} |\pat_s
v|^2)^{\frac{1}{2}}(1-e^{-2T})^{-\frac{1}{2}},\;\forall s\in
[T,\infty).
\end{equation}

\end{lm}

\begin{proof} Since $v(s,\theta)$ is a smooth function defined in
$S^1\times [0,\infty)$, by Fourier analysis, any smooth solution
of (\ref{homo-cyln-equa}) has the following form:
$$
v(s,\theta)=\sum^\infty_{n=-\infty} C_n e^{-(n+\Theta)s}\cdot
e^{-in\theta}.
$$
Since $v$ is assumed to be bounded, hence $C_n=0,\forall n<0$.
Thus the bounded smooth solution has the form
$$
v(s,\theta)=\sum_{n=0}^\infty C_n e^{-(n+\Theta)s}\cdot
e^{-in\theta}.
$$
Now we have the integral estimate
\begin{align}
&\int^\infty_T \int_{S^1} |v|^2=\sum_{n=0}^\infty |C_n|^2
\int^\infty_T e^{-2(n+\Theta)s}ds=\frac{1}{2}\sum_{n=0}^\infty
|C_n|^2(\Theta+n)^{-1}e^{-2(n+\Theta)T}\nonumber\\
&\le \frac{1}{2}\sum_{n=0}^\infty
|C_n|^2(\Theta+n)e^{-2(n+\Theta)T}(\Theta+n)^{-2}\le
\frac{1}{\Theta^2}\int^\infty_T\int_{S^1} |\pat_s v|^2
\end{align}
and the pointwise estimate
\begin{align}
&|v(s,\theta)|\le \sum_{n=0}^\infty |C_n|e^{-(n+\Theta)s}\le
e^{-\Theta
s}(\sum^\infty_{n=0}|C_n|^2)^{\frac{1}{2}}(\sum^\infty_{n=0}
e^{-2nT})^{\frac{1}{2}}\nonumber\\
&\le \sqrt{\frac{2}{\Theta}}e^{-\Theta
s}(\int^\infty_0\int_{S^1}|\pat_s
v|^2)^{\frac{1}{2}}(1-e^{-2T})^{-\frac{1}{2}}.
\end{align}
\end{proof}

\

Extend the function $-\overline{\frac{\pat
W_N(v_1,\dots,v_N)}{\pat v_i}}e^{i \Theta_i\theta}$ symmetrically
to $(-\infty,\infty)$, and  view it as the free term of the
following equation for $v$ defined in $(-\infty,\infty)$:
\begin{equation}
(\pat_s+\sqrt{-1}\pat_\theta+\Theta_i)v=-\overline{\frac{\pat
W_N(v_1,\dots,v_N)}{\pat v_i}}e^{i \Theta_i\theta}.
\end{equation}
This equation has a unique bounded solution in the whole interval
$(-\infty,\infty)$ (see [D]). Since the ``free term" is $L^2$
integrable by Lemma \ref{inte-1}, we have the Fourier expansion
$$
-\overline{\frac{\pat W_N(v_1,\dots,v_N)}{\pat v_i}}e^{i
\Theta_i\theta}=\sum_n \rho_n(s) e^{-in\theta}.
$$
The unique bounded solution
$$ \hat{v}_{i,0}(s,\theta)=-\sum_{n=0}^\infty
e^{-(n+\Theta_i)s}\int^\infty_s e^{(n+\Theta_i)\tau}\rho_n
(\tau)d\tau e^{-in\theta}+\sum_{n=-\infty}^{-1}
e^{(n+\Theta_i)s}\int^s_{-\infty}
e^{-(n+\Theta_i)\tau}\rho_n(\tau)d\tau e^{in\theta}.
$$

\begin{lm}\label{inho-int-esti} Let $T>1$. For any $i>N_l$, $\hat{v}_{i,0}$ satisfies the
integral estimate
\begin{equation}
||\hat{v}_{i,0}||_{L^2_1(S^1\times [T,\infty])}\le C
 \left|\left| \frac{\pat
W_N}{\pat v_i}\right| \right|_{L^2(S^1\times [T-1,\infty))},
\end{equation}
where $C$ is a constant.
\end{lm}

\begin{proof}
By Lemma 3.22 of [D], we have
$$
||\hat{v}_{i,0}||_{L^2(S^1\times [T,\infty))}\le C\int^\infty_T
\int_{S^1}\left| \frac{\pat W_N}{\pat v_i} \right|^2,
$$
where $C$ depends only on $ \Theta_i^{\gamma_l}$. Hence by $L^2$
interior estimates of the Cauchy-Riemann operator, for any
$1<p<\infty$, there holds
\begin{align*}
&||\hat{v}_{i,0}||_{L^2_1(S^1\times [T,\infty))}\\
&\le C \left(||\hat{v}_{i,0}||_{L^2(S^1\times
[T-1,\infty))}+\left|\left| \frac{\pat W_N}{\pat v_i}
\right|\right|_{L^2(S^1\times [T-1,\infty))}\right)\\
&\le C \left|\left| \frac{\pat W_N}{\pat v_i}
\right|\right|_{L^2(S^1\times [T-1,\infty))}.
\end{align*}
\end{proof}

\begin{thm}\label{conv-n-solu} For any $i>N_l$, $\hat{v}_i(s,\theta)$ and its
derivatives of any order tend to zero uniformly as $s\to +\infty.$
\end{thm}

\begin{proof} The bounded function
$\hat{v}_i-\hat{v}_{i,0}$ satisfies the homogeneous equation
$$
(\pat_s+\sqrt{-1}\pat_\theta+\Theta_i)v=0.
$$
On the other hand, by Lemma \ref{inte-1}, Lemma \ref{inho-int-esti}
and the definition of the solution $u_i$ we have
\begin{align}
&||\pat_s(\hat{v}_i-\hat{v}_{i,0})||_{L^2(S^1\times
[T,\infty))}\le ||\pat_s\hat{v}_i||_{L^2(S^1\times
[T,\infty))}+||\pat_s\hat{v}_{i,0}||_{L^2(S^1\times [T,\infty))}\\
&\le ||\pat_s\hat{v}_i||_{L^2(S^1\times
[T-1,\infty))}+C\left|\left| \frac{\pat W_N}{\pat v_i}\right|
\right|_{L^2(S^1\times[T-1,\infty))}\le \infty.
\end{align}
Hence by Lemma \ref{homo-int-esti}, we obtain
\begin{align}
&||\hat{v}_i||_{L^2(S^1\times [T,\infty))} \le ||
\hat{v}_{i,0}||_{L^2(S^1\times
[T,\infty))}+||\hat{v}_i-\hat{v}_{i,0}||_{L^2(S^1\times
[T,\infty))}\\
&\le \frac{1}{\Theta_i}
||\pat_s(\hat{v}_i-\hat{v}_{i,0})||_{L^2(S^1\times [T,\infty))}+||
\hat{v}_{i,0}||_{L^2(S^1\times
[T,\infty))}\nonumber\\
&\le C ||\pat_s\hat{v}_i||_{L^2(S^1\times
[T-1,\infty))}+C\left|\left| \frac{\pat W_N}{\pat v_i}\right|
\right|_{L^2(S^1\times[T-1,\infty))}.
\end{align}

By $L^2$ estimates and the above inequality, there holds
\begin{align}
&||\hat{v}_i||_{L^2_1(S^1\times [T,\infty))}\nonumber\\
&\le C\left(||\hat{v}_i||_{L^2(S^1\times [T-1,\infty))}+
\left|\left| \frac{\pat W_N}{\pat v_i}\right|
\right|_{L^2(S^1\times[T-1,\infty))}\right)\nonumber\\
&\le C\left( ||\pat_s\hat{v}_i||_{L^2(S^1\times [T-1,\infty))}+
\left|\left| \frac{\pat W_N}{\pat v_i}\right|
\right|_{L^2(S^1\times[T-1,\infty))}\right).
\end{align}

Notice that the term on the right hand side tends to zero when
$T\to\infty$.

By the embedding theorem, any $L^p$ norm of $\hat{v}_i$ can be
controlled by the $L^2_1$ norm of $\hat{v}_i$; therefore, the $L^p$
norm of $|\frac{\pat W_N}{\pat v_i}|$ is also been controlled.
Using $L^p$ estimates of the Cauchy-Riemann operator and
furthermore by the embedding theorem and Schauder estimates, any
$C^m$ norm of $\hat{v}_i$ can be controlled by the $L^2_1$ norm of
$\hat{v}_i$. Therefore, any $C^m$ norm of $\hat{v}_i$ in
$[T,\infty)$ tends to zero as $T\to \infty.$

\end{proof}

Now we turn to the study of the solution $u_i=v_i=\hat{v}_i$ for
$i\leq N_l$.

\begin{lm}\label{ramo-conv}
Suppose that $u_i, i\leq N_l$ is the solution of
(\ref{plgw-r-equ2}) and (\ref{plgw-n-equ2}); then $\pat_s
u_i(s,\theta)$, $\pat_\theta u_i(s, \theta) \rTo 0$
uniformly for $\theta\in S^1$ as $s\rTo \infty$.
\end{lm}

\begin{proof}
Take the derivative $\pat_s$ of the equation (\ref{plgw-r-equ2});
then we get
\begin{equation}\label{proof-expo-deca1}
\pat_s (\pat_s u_i)+\sqrt{-1}\pat_\theta (\pat_s u_i)-2\sum_{j\leq
N_l}\overline{\frac{\pat^2 (W_{\gamma}+W_{0,\beta})}{\pat u_i \pat
u_j}} \overline{\pat_s u_j}=\pat_s (2\overline{\frac{\pat
W_{N}(v_1,\dots, v_N)}{\pat v_i}}).
\end{equation}
For simplicity, we set $X_{ij}=2\frac{\pat^2
(W_{\gamma}+W_{0,\beta})}{\pat u_i \pat u_j}$, which is a complex
symmetric matrix, and set $L_i(s,\theta)=\pat_s
(2\overline{\frac{\pat W_{N}}{\pat v_i}})$. So
(\ref{proof-expo-deca1}) becomes
\begin{equation}\label{proof-expo-deca2}
\pat_s (\pat_s u_i)+\sqrt{-1}\pat_\theta (\pat_s
u_i)-\overline{X_{ij}(\pat_s u_j)}=L_i.
\end{equation}
Here we also omit the summation sign if no confusion arises.

We have
\begin{align}
&\Delta (\pat_s u_i)=\overline{
(\pat_s+\sqrt{-1}\pat_\theta)(X_{ij}(\pat_s u_j))}+(\pat_s-\sqrt{-1}\pat_\theta)L_i\nonumber\\
&=\overline{X_{ijk}(\pat_s u_k+\sqrt{-1}\pat_\theta
u_k)(\pat_s u_j)}+\overline{X_{ij}}(X_{jk}(\pat_s u_k)+\bar{L}_j)+(\pat_s-\sqrt{-1}\pat_\theta)L_i\nonumber\\
&=\overline{X_{ijk}(\pat_s u_k+\sqrt{-1}\pat_\theta u_k)}\cdot
\overline{\pat_s u_j}+(\overline{X_{ij}}X_{jk})\cdot (\pat_s
u_k)+\overline{X_{ij}}\cdot
\bar{L}_j+(\pat_s-\sqrt{-1}\pat_\theta)L_i.
\end{align}
Denote by $F_i$ the term on the right hand side. We know that
$F_i$ is a bounded term.

We have
\begin{align}
&\Delta |(\pat_s u_i)|^2=4\Delta(\pat_s u_i)\cdot\overline{(\pat_s
u_i)}+4\Delta\overline{(\pat_s u_i)}\cdot (\pat_s u_i)+4|\bpat
(\pat_s u_i)|^2+4|\pat (\pat_s u_i)|^2\nonumber\\
&\ge 4(\bar{F}_i\cdot (\pat_s u_i)+F_i\cdot\overline{\pat_s u_i}).
\end{align}

By the maximum principle, we have
\begin{align}
&|(\pat_s u_i)|^2(s,\theta)\le
C(\int^{s+1}_{s-1}\int_{S^1}|(\pat_s
u_i)|^2+\int^{s+1}_{s-1}\int_{S^1}|F_i|^2
|(\pat_s u_i)|^2)\\
&\le C_1 \int^{s+1}_{s-1}\int_{S^1}|(\pat_s u_i)|^2.
\end{align}
This shows that $\pat_s u_i$ converges to zero uniformly for
$\theta\in S^1$. In the same way, one can prove that the
conclusion is also true for $\pat_\theta u_i$.
\end{proof}

\begin{thm}\label{conv-r-solu} Let $u_j, j\leq N_l$ be the solution of (\ref{plgw-r-equ2})
and (\ref{plgw-n-equ2}), then there holds
$||u_{j}(s,\theta)-\kappa_{j}||_{C^0(S^1)}\rTo 0$ as $s\to
0$, where $\kappa=(\kappa_{1}, \dots, \kappa_{N_l})$ is one of
the critical points of $W_{\gamma}+W_{0}$.
\end{thm}

\begin{proof}  Let $s_n$ be any sequence tending to
infinity as $n\to\infty$. By Theorem \ref{bdd-bdry-thm};
$||u_i(s_n,\cdot)||_{C^m(S^1)}$ is uniformly bounded for any $m\in
\Z$. Hence by the Arzel\`a-Ascoli theorem there is a subsequence
$s_{n_k}$ such that $u_i(s_{n_k},\theta)$ converges uniformly in
$C^{m-1}$ norm to a limit function $u_i^\infty(\theta)$.\

$u_i(s_{n_k},\theta)$ satisfies the following relation:
$$
\pat_s u_i(s_{n_k},\theta)+\sqrt{-1}\pat_\theta
u_i(s_{n_k},\theta)-2\overline{\frac{\pat (W_{\gamma}+W_{0})}{\pat
u_i}}=2\overline{\frac{\pat W_N(v_1,\dots, v_N)}{\pat v_i}}.
$$
By Theorem \ref{conv-n-solu} and Lemma \ref{ramo-conv}, we know
that
$$
\pat_s u_i(s_{n_k})\rTo 0,\;\frac{\pat W_N(v_1,\dots,
v_N)}{\pat v_i}\rTo 0,
$$
as $n_k\to \infty$. So let $k\to \infty$ in the above equation; we
can obtain the limit equation
\begin{align}
&\sqrt{-1}\pat_\theta u_i^\infty-2\overline{\frac{\pat
(W_{\gamma}+W_{0})(u_1^\infty, \dots, u_{N_l}^\infty)}{\pat
u_i}}=0,\; \forall i\leq N_l\label{limi-r-equ2}.
\end{align}

Multiply $\frac{\pat (W_{\gamma}+W_0)}{\pat u_i}$ to the two sides
of (\ref{limi-r-equ2}) and take the sum; then there holds
\begin{align}
\sum_{i\leq N_l} \sqrt{-1} \frac{\pat u_i^\infty}{\pat \theta}
\frac{\pat (W_{\gamma}+W_0)}{\pat u_i}=\sum_{i\leq N_l}
2\left|\frac{\pat (W_{\gamma}+W_0)}{\pat u_i}\right|^2.
\end{align}
Integrating, we obtain
$$
\int_{S^1} \sum_{i\leq N_l} 2\left|\frac{\pat
(W_{\gamma}+W_0)}{\pat u_i}\right|^2d\theta=0,
$$
and so
$$
\sqrt{-1}\pat_\theta u_i^\infty=\frac{\pat (W_{\gamma}+W_0)}{\pat
u_i}=0, \;\forall i\leq N_l.
$$

Thus the solutions $u_i^\infty=\kappa_i, i\leq N_l$ are constants,
and $(\kappa_1, \dots, \kappa_{N_l})$ is just one of the critical
points of the polynomial $W_{\gamma}+W_0$. Notice that
$W_{\gamma}+W_0$ is a holomorphic Morse function and has finitely
many critical points, and so the values that $u_i^\infty$ can
attain are also  finitely many. Assume that $\kappa^1, \dots,
\kappa^m$ are different critical points. Define
$$
r=\min_{l\neq k}|\kappa^l-\kappa^k|.
$$

Now if we assume that the conclusion of this theorem does not
hold, then there is a sequence $s_n\to\infty$ such that
$u_i(s_n,\theta)$ does not converge. Hence there are at least two
subsequences $s_{n_k}$ and $s_{n_k'}$ tending to infinity such
that
$$
||u(s_{n_k},\theta)-\kappa^l||_{C^0}\to
0,\;||u(s_{n_k'},\theta)-\kappa^{l'}||_{C^0}\to
0,\;\text{as}\;k,k'\to \infty,
$$
where $\kappa^l\neq \kappa^{l'}$, and we set
$u=(u_1,\dots,u_{N_l})$. In particular, for sufficiently large
$k,k'$, we have
$$
|\fint_{S^1} u(s_{n_k},\theta)d\theta-\kappa^l|\le
\frac{r}{4},\;|\fint_{S^1}
u(s_{n_k'},\theta)d\theta-\kappa^{l'}|\le \frac{r}{4}.
$$
Now by continuity for sufficiently large $k,k'$, there exists
$s_{k''}\in [s_{n_k},s_{n_k'} ]$ such that
$$
\frac{r}{4}\le |\fint_{S^1}
u(s_{k''},\theta)d\theta-\kappa^{l'}|\le \frac{r}{3}.
$$
Again by the Arzel\`a-Ascoli theorem, there exists a subsequence
of $s_{k''} $, still denoted by $s_{k''}$ such that $u_i(s_{k''},
\theta )$ converges uniformly to some $\kappa^{l''}_i$ which
should satisfy
$$
\frac{r}{4}\le |\kappa^{l''}-\kappa^{l'}|\le \frac{r}{3}.
$$
However, we know that there is no $\kappa^{l''}$ whose distance
from $\kappa^{l'}$ is between $\frac{r}{4}$ and $\frac{r}{3}$. This
is a contradiction.
\end{proof}

As one application of our analysis of asymptotic behavior, we have
the following conclusion:

\begin{thm}[\textbf{Witten Lemma}]\label{lm-witt-noda} Suppose that $z_k$ is the unique
broad marked point on $\Sigma$, and $(u_1,\dots,u_{N_l})$ are all
the possible broad solutions near $p$ such that their local
coordinate functions $(\tilde{u}_1,\dots, \tilde{u}_{N_l})$
converge to a critical point $\kappa$ of $W_\gamma+W_{0}$. Then we
have
\begin{equation}
\sum_i ||\bpat u_i||^2_{L^2(\Sigma)}=\sum_i ||\frac{\pat
(W+W_{0,\beta})}{\pat u_i}||^2_{L^2(\Sigma)}=\pi
(W_\gamma+W_0)(\kappa)-\frac{1}{2}\int^{T+1}_T \sum_i
(b_i\tilde{u}_i\pat_s \beta_T) ds d\theta.
\end{equation}
If there is no broad marked point, then the equation has only the
zero solution.
\end{thm}

\begin{proof} Let $\{z_1,\dots,z_k\}$ be $k$ marked points and
$z_k$ be the only broad marked point. Our perturbed Witten
equation is
$$
\bpat u_i+I_1(\overline{\frac{\pat (W+W_{o,\beta})}{\pat u_i}})=0.
$$
Here the $u_i$ are global sections on $\Sigma$. Let
$u_i=\tilde{u}_i e_i$ in the disc $D_r(z_l)$ (with radius $r$ and
centered at $z_l$) of the marked point $z_l$. Then
\begin{align*}
&\sum_i (\bpat u_i, I_1(\overline{\frac{\pat (W+W_{o,\beta})}{\pat
u_i}}))_{L^2(\Sigma)}=\sum_i \int_{\Sigma\setminus \cup_{l=1}^{k}
D_r(z_l)} (\bpat u_i,
I_1(\overline{\frac{\pat W}{\pat u_i}}))+\\
&\sum_{l=1}^{k-1}  \sum_i \int_{D_r(z_l)} (\frac{\bpat
\tilde{u}_i}{\pat \bar{z}}d\bar{z}\otimes e_i, \sum_j
\overline{\frac{\pat W_j(\tilde{u}_1,\dots, \tilde{u}_N)}{\pat
\tilde{u}_i}z^{\sum^N_{s=1}b_{js}
(\Theta^{\gamma_l}_s-q_s)}}|e'_i|^2
e_i\otimes d\bar{z})\\
&+\sum_i \int_{D_r(z_k)} (\frac{\bpat \tilde{u}_i}{\pat
\bar{z}}d\bar{z}\otimes e_i, \sum_j \overline{(\frac{\pat
W_j(\tilde{u}_1,\dots, \tilde{u}_N)}{\pat
\tilde{u}_i}z^{\sum^N_{s=1}b_{js}
(\Theta^{\gamma_l}_s)}+b_i\beta_T)\frac{1}{z}}|e'_i|^2 e_i\otimes
d\bar{z})\\
&=-\lim_{r=|z|\to 0}\pi \sum_{l=1}^{k-1} \sum
W_j(\tilde{u}_1(z),\dots, \tilde{u}_N(z))z^{\sum^N_{s=1}b_{js}
(\Theta^{\gamma_l}_s)}\\
&-\lim_{r=|z|\to 0} \pi (\sum W_j(\tilde{u}_1(z),\dots,
\tilde{u}_N(z))z^{\sum^N_{s=1}b_{js} (\Theta^{\gamma_l}_s)}+\sum_i
b_i \beta_T \tilde{u}_i)+\sum_i \frac{\sqrt{-1}}{2}\int
\frac{dz\wedge d\bar{z}}{z}b_i \bpat \beta_T \tilde{u}_i \\
&=-\pi (W_\gamma+W_0)(\kappa)+\frac{1}{2}\int^{T+1}_T \sum_i
(b_i\tilde{u}_i\pat_s \beta_T) ds d\theta.
\end{align*}
We have used Stokes' theorem in the second equality and Lemma
\ref{conv-n-solu}, \ref{conv-r-solu} in the last equality. Using
the equation, we obtain
$$
\sum_i ||\bpat u_i||^2_{L^2(\Sigma)}=\sum_i ||\frac{\pat W}{\pat
u_i}||^2_{L^2(\Sigma)}=\pi
(W_\gamma+W_0)(\kappa)-\frac{1}{2}\int^{T+1}_T \sum_i
(b_i\tilde{u}_i\pat_s \beta_T) ds d\theta.
$$
If there is no broad marked point, we have the equality
$$
\sum_i ||\bpat u_i||_{L^2(\Sigma)}=\sum_i ||\frac{\pat W}{\pat
u_i}||_{L^2(\Sigma)}=0.
$$
Since near each marked point the norm $|u_i|$ of the section $u_i$
equals  $|\hat{v}_i|$ , which is an $L^2$-integrable function in
view of the proof of Theorem \ref{conv-n-solu}, $u_i$ is a
constant section near each marked point. Furthermore, $\frac{\pat
W}{\pat u_i}=0$; then non-degeneracy of $W$ forces $u_i\equiv 0$
for each $i=1,\dots, N$.
\end{proof}

Now it is easy to prove the following two corollaries.

\begin{crl}\label{crl-wittenlemma1} Suppose that $\Sigma=\Sigma^+\cup\Sigma^-$ is a nodal
curve having two components $\Sigma^\pm$ and the nodal point $p$ is
a broad nodal point. If all the marked points are narrow
points, then we have the estimates
\begin{equation}
\sum_i ||\bpat u_i||^2_{L^2(\Sigma)}=\sum_i ||\frac{\pat
(W+W_{0,\beta})}{\pat u_i}||^2_{L^2(\Sigma)}\le C\max_i|b_i|.
\end{equation}
\end{crl}

\begin{crl}\label{crl-wittenlemma2} The non-perturbed Witten equation has only the trivial
solution.
\end{crl}

\subsection{Exponential decay}\label{subsec:expo-deca}

\

Let $u_i,i=1,\dots,N,$ be the solutions of the perturbed Witten
equation near a marked point $z_l$. Setting as before $v_i=u_i
z^{\Theta^{\gamma_l}_i}=u_i e^{-\Theta^{\gamma_l}_i(s+i\theta)}$
and $\hat{v}_i=u_i e^{-\Theta^{\gamma_l}_i s}$, we have
\begin{equation}
\hat{v}_i=v_i e^{i\Theta^{\gamma_l}_i \theta}.
\end{equation}

Note that $\hat{v}_i$ is a well-defined function on $S^1\times
[0,\infty),$ but $v_i$ is only a locally defined function. The
locally defined section $v_i$ satisfies the equation
\begin{equation}
\pat_s v_i+\sqrt{-1}\pat_\theta v_i-2\overline{\frac{\pat
(W+W_0)}{\pat v_i}}=0.
\end{equation}

Suppose that the quasi-homogeneous polynomial $W$ has the form
$$
W(u_1,\dots, u_N)=\sum_j W_j=\sum_j c_j\prod_{k=1}^{N}
u_k^{b_{jk}}.
$$
Set $\Theta^\gamma(j):=\sum_k b_{jk} \Theta^\gamma_k\in \Z$. It is
known that $W_j$ is broad at $z_l$ iff $\Theta^\gamma(j)=0$, and
$W$ is broad at $z_l$ iff there exsits one monomial $W_{j_0}$
such that $\Theta^\gamma(j_0)=0$.

Now an easy computation shows that $\hat{v}_i$ satisfies the
equation
\begin{equation}\label{expo-decay-equ1}
\pat_s \hat{v}_i+\sqrt{-1}\pat_\theta \hat{v}_i+\Theta^\gamma_i
\hat{v}_i= \overline{\frac{\pat
\hat{W}(\hat{v}_1,\dots,\hat{v}_N,\theta)}{\pat \hat{v}_i}},
\end{equation}
where
\begin{equation}
\hat{W}(\hat{v}_1,\dots,\hat{v}_N,\theta):=2\left(\sum_j
W_j(\hat{v}_1,\dots,\hat{v}_N)e^{-i
\Theta^\gamma(j)\theta}+W_0(\hat{v}_1,\dots,\hat{v}_{N_l})\right).
\end{equation}

Define $\hat{H}=2\mbox{Re} \hat{W}$, and let
$\hat{v}_i=\hat{x}_i+\sqrt{-1}\hat{y}_i$. We want to change the
complex system $(\ref{expo-decay-equ1})$ into a real system. By
$(\ref{expo-decay-equ1})$, we have
$$
\pat_s(\hat{x}_i+\sqrt{-1}\hat{y}_i)+\sqrt{-1}\pat_\theta(\hat{x}_i+\sqrt{-1}\hat{y}_i)
+\Theta^\gamma_i(\hat{x}_i+\sqrt{-1}\hat{y}_i)=\frac{1}{2}(\pat_{\hat{x}_i}+\sqrt{-1}\pat_{\hat{y}_i})(2\hat{H}),
$$
for $i=1,\dots,N$. So we obtain
\begin{equation}\label{expo-decay-equ2}
\left(\begin{array}{l} \pat_s \hat{x}_i-\pat_\theta
\hat{y}_i+\Theta^\gamma_i \hat{x}_i=\pat_{\hat{x}_i}\hat{H}\\
 \\
\pat_s \hat{y}_i+\pat_\theta \hat{x}_i+\Theta^\gamma_i
\hat{y}_i=\pat_{\hat{y}_i}\hat{H},
\end{array}
\right.
\end{equation}
for $i=1,\dots, N$.

Define the following quantities:
\begin{align*}
&\hat{v}_R:=(\hat{x}_1,\dots,
\hat{x}_{N_l},\hat{y}_1,\dots,\hat{y}_{N_l})^T,\;\;\nabla_R=(\pat_{\hat{x}_1},\dots,
\pat_{\hat{x}_{N_l}},\pat_{\hat{y}_1},\dots,\pat_{\hat{y}_{N_l}})^T\\
&\hat{v}_N:=(\hat{x}_{N_l+1},\dots,
\hat{x}_{N},\hat{y}_{N_l+1},\dots,\hat{y}_{N})^T,\;\;\nabla_N=(\pat_{\hat{x}_{N_l+1}},\dots,
\pat_{\hat{x}_{N}},\pat_{\hat{y}_{N_l+1}},\dots,\pat_{\hat{y}_{N}})^T\\
&J_R:=\left(
\begin{matrix}
 0& -I_R\\
I_R&0
\end{matrix}\right),\;\;I_R\; is\; the\; N_l\times N_l\;
\text{identity matrix}\\
&J_N:=\left(
\begin{matrix} 0& -I_N\\
I_N&0
\end{matrix}\right),\;\;I_N \;is\; the\; (N-N_l)\times (N-N_l)\;
\text{identity matrix} \\
&\Theta_N:=\diag(\Theta^\gamma_{N_l+1},\dots,\Theta^\gamma_N)\\
&A_N:=\left(\begin{matrix}
 \Theta_N& 0\\
0&\Theta_N.
\end{matrix}\right)
\end{align*}

Then the system (\ref{expo-decay-equ2}) can be written as follows:
\begin{equation}\label{expo-decay-equ3}
\left(
\begin{array}{l}
\pat_s \hat{v}_R+J_R\cdot \pat_\theta \hat{v}_R=\nabla_R \hat{H}\\
 \\
\pat_s \hat{v}_N+J_N\cdot \pat_\theta \hat{v}_N+A_N\cdot \hat{v}_N=\nabla_N \hat{H}.\\
\end{array}
\right.
\end{equation}

Set
\begin{align*}
&\hat{\mathbf{v}}:=\left(\begin{matrix}\hat{v}_R\\
\hat{v}_N\end{matrix}\right),\;\;\mathbf{\nabla}:=\left(\begin{matrix}\nabla_R\\
\nabla_N\end{matrix}\right)\\
&J:=\left(\begin{matrix}J_R&0\\
0&J_N\end{matrix}\right),\;\; A:=\left(\begin{matrix}0&0\\
0&A_N\end{matrix}\right).
\end{align*}

Now we can write the equation (\ref{expo-decay-equ3}) in the
simple form
\begin{equation}\label{expo-decayequ}
\pat_s \mathbf{v}+J\cdot \pat_\theta \mathbf{v}+A\cdot \mathbf{v}=\mathbf{\nabla}\hat{H}.\\
\end{equation}
Here $J$ is an almost complex structure in $\R^{2N}$, since
$J^2=-I$.

Take the derivative $\pat_s$ of the two sides of
(\ref{expo-decayequ}) and let $\w=\pat_s \hat{\mathbf{v}}$; we
obtain the system for $\w$:
\begin{equation}\label{expo-decayequ-deri}
\pat_s \w+J\cdot \pat_\theta \w+S\cdot \w=0,\\
\end{equation}
where $S:=(A-\mathbf{\nabla}^2\hat{H})$.

The following lemma gives a formulation of the asymptotic behavior
of the coefficient matrix $S=A-\mathbf{\nabla}^2\hat{H}$.

\begin{lm}\label{expo-decay-cond} The following conclusions hold when $s\to
\infty$:
\begin{itemize}
\item[(1)] $ S(s,\theta)\rTo S^\infty:=
\begin{pmatrix} -\nabla^2_R 2\mbox{Re}(W_\gamma+W_0)|_{(\kappa_1,\dots,\kappa_{N_l})}&0\\
0&\Theta_N-\nabla^2_N
2\mbox{Re}(W_N)|_{(\kappa_1,\dots,\kappa_{N_l})}
\end{pmatrix},
$ where $(\kappa_1,\dots,\kappa_{N_l})$ is some critical point of
the polynomial $W_\gamma+W_0$.

\item[(2)]$\sup_{\theta\in S^1} ||\pat_s S(s,\theta)||\to
0$, as $s\to \infty$.

\item[(3)]$\sup_{S^1\times [0,\infty)}||\pat_\theta
S(s,\theta)||<\infty$.
\end{itemize}
\end{lm}

\begin{proof}
When $s\to \infty$, we have $(\hat{v}_1,\dots,
\hat{v}_{N_l},\hat{v}_{N_l+1},\dots, \hat{v}_N)\to
(\kappa_1,\dots,\kappa_{N_l},0,\dots,0)$. Note that the
polynomial $W_\gamma+W_0$ only contains the variables
$\hat{v}_1,\dots, \hat{v}_{N_l}$; so we have two cases:
\begin{itemize}
\item for $1\le i, j\le 2N_l,$ there holds $(\nabla^2
\hat{H})_{ij}=(\nabla^2_R 2\mbox{Re}(W_\gamma+W_0+W_N))_{ij}$;

\item if $i,j$ do not satisfy case 1, then $(\nabla^2
\hat{H})_{ij}=(\nabla\nabla_N 2\mbox{Re}(W_N))_{ij}$.
\end{itemize}

There are at least two narrow sections in each
monomial of $W_N$, so if $1\le i\le 2N_l, N_l<j$ or $1\le j\le
2N_l, N_l<i$, then $(\nabla^2 \hat{H})_{ij}\to 0$. Hence
(1) is proved.

The proofs of (2) and (3) are easily obtained by Theorem
\ref{bdd-bdry-thm}, Theorem \ref{conv-n-solu} and Lemma
\ref{ramo-conv}.
\end{proof}

System (\ref{expo-decayequ-deri}) for $\w$ becomes the standard
system that appears frequently in symplectic geometry, whose
decaying behavior has been studied in detail. For example, we can
cite Lemma 2.11 of D. Salamon's lecture in [ET] as below (after
minor changes in notation):

\begin{lm}\label{expo-decay-salamon} Let $J$ be a $2N\times 2N$ real matrix such that
$J^2=-I$, and $S(s,\theta)$ be a $2N\times 2N$ matrix function
defined on $S^1\times [0,\infty)$ satisfying
$$
\lim_{s\to \infty}\sup_{\theta\in S^1}||\pat_s
S(s,\theta)||=0,\;\;\sup_{s,\theta}||\pat_\theta
S(s,\theta)||<\infty.
$$
Define the operator $D=\pat_s+J\pat_\theta+S$. If the unbounded
self-adjoint operator $J\pat_\theta+S: L^2(S^1,
\R^{2N})\rTo L^2(S^1, \R^{2N})$ is invertible, then there
exists a constant $\delta>0$ such that the following holds. For
every $C^2$-function $\xi: S^1\times [0,\infty)\rTo
\R^{2N}$ which satisfies $D\xi=0$ and does not diverge to $\infty$
as $s\to\infty$ there exists a constant $c>0$ such that for
sufficiently large $s$ and any $\theta\in S^1$,
$$
|\xi(s,\theta)|<ce^{-\delta s}.
$$
\end{lm}

\begin{lm} Let $W_0(u_1,\dots, u_{N_l})=\sum_{k=1}^{N_l}  b_k u_k$ be any
$W_{\gamma_l}$-regular polynomial at a broad marked point $z_l$.
Then if $\sum_{k=1}^{N_l} |b_k|$ is sufficiently small, there
exist constants $C, T,\delta>0$ such that for any $(s,\theta)\in
S^1\times [T,\infty)$,
$$
|\pat_s \hat{v}_i|< C e^{-\delta s}, i=1,\dots, N.
$$
\end{lm}

\begin{proof} By Theorem \ref{bdd-bdry-thm} and (2), (3) of Lemma
\ref{expo-decay-cond}, we know that all the hypotheses in Lemma
\ref{expo-decay-salamon} except invertibility are satisfied. So we
only need to show that the operator $J\pat_\theta+S: L^2(S^1,
\R^{2N})\rTo L^2(S^1, \R^{2N})$ is invertible when
$\sum_{k=1}^{N_l} |b_k|$ is sufficiently small. However, it is
easy to see that if $\sum_{k=1}^{N_l} |b_k|$ is sufficiently
small, the absolute value of any critical point of
$W_{\gamma_l}+W_0$ also becomes  arbitrarily small. Hence we can
take sufficiently small $\kappa_1,\dots, \kappa_{N_l}$ such that
\begin{equation}
||S^\infty||<1,\;\; and \; S^\infty\;\text{is a non-degenerate
matrix}.
\end{equation}
Assume $\w\in L^2_1(S^1,\R^{2N})$ is  a solution of
$J\pat_\theta\w+S^\infty\w=0$. By Poincare inequality, there is
$$
||\w-\fint_{S^1}\w||_{L^2}\le ||\pat_\theta
(\w-\fint_{S^1}\w)||_{L^2}.
$$
Since the mean value $\fint_{S^1}\w$ vanishes by the equation and
the matrix $S^\infty$ is non-degenerate, we obtain
$$
||\w||_{L^2}\le ||\pat_\theta \w||_{L^2}=||S^\infty \w||_{L^2}.
$$
Since $||S^\infty||<1$, we conclude that $\w=0$. Therefore the
invertiblity is proved. Applying Lemma \ref{expo-decay-salamon} to
our case shows that we are done.
\end{proof}

By the above lemma, it is easy to obtain the following theorem.

\begin{thm}\label{thm:asymp}
 Let $W_0(u_1,\dots, u_{N_l})=\sum_{k=1}^{N_l}  b_k u_k$ be any
$W_{\gamma_l}$-regular polynomial at the broad marked point
$z_l$. Then if $\sum_{k=1}^{N_l} |b_k|$ is sufficiently small,
then there exist constants $C, T,\delta>0$ such that for any
$(s,\theta)\in S^1\times [T,\infty)$,
\begin{equation}
\begin{aligned}
&|\hat{v}_i-\kappa_i|< C e^{-\delta s},
\text{for}\;i=1,\dots,N_l,\\
&|\hat{v}_i|< C e^{-\delta s}, \text{for}\;i=N_l+1,\dots,N,
\end{aligned}
\end{equation}
where $(\kappa_1,\dots,\kappa_{N_l})$ is some critical point of
$W_{\gamma_l}+W_0$.
\end{thm}

\subsection{A Liouville type theorem for the $A_1$-equation}\

\

If $W=u^2$, the Witten equation becomes
\begin{equation}
\bpat u+I_1(\overline{2u})=0.
\end{equation}

This is a linear equation and the existence of solutions is
related to the spectrum problem. The previous interior estimate
does not hold here since it holds for $W$ with all weights
$q_i<1/2$. Since $W=x^2$ is already a holomorphic Morse function,
it is not useful to do linear perturbation again. To avoid the
spectrum, we choose the following global perturbed equation
\begin{equation}\label{equa-A1-case}
\bpat u+I_1(\overline{(2+\epsilon)u})=0.
\end{equation}

We have the following Liouville type theorem

\begin{thm} For generic perturbation parameter $\epsilon\in \R$, the
perturbed $A_1$ equation (\ref{equa-A1-case}) has only the zero
solution.
\end{thm}
\begin{proof} Let $[1,\infty]$ be the cylindrical neighborhood of the marked
point $p=\infty$. Choose the cylindrical coordinate
$\zeta=s+i\theta$, then (\ref{equa-A1-case}) becomes
\begin{equation}
\bpat_\zeta u-2(\overline{(2+\epsilon)u})=0.
\end{equation}
Since the integral $\int^\infty_0\int_{S^1} |\frac{\pat u}{\pat
s}|d\theta ds<\infty$, we can use Lemma \ref{expo-decay-salamon}
to show that if $\epsilon\in \R$ is generic, then the solution
converges to zero exponentially. Applying the Witten lemma to
Equation (\ref{equa-A1-case}), we are done.
\end{proof}

\subsection{An estimate for the borderline case}

If $W(x_1,\dots,x_N)$ is a quasi-homogeneous polynomial with some
variable $x_i$ having weight $q_i=1/2$, then the condition
$\delta_i<1$ in Theorem \ref{thm-new} does not hold for any variable
$x_i$. Therefore we can't obtain the $C^0$ norm estimate for the
solutions $\bu$ of the perturbed Witten equations. In fact, the
$C^0$ norm of the solutions of the (perturbed) Witten equations may
not be controlled a priori. A Typical example is the $D_{n+1}^T$
singularity: $x^ny+y^2$. Hence in this case, the $C^0$ estimates in
the previous sections do not hold anymore. We will modify the
definition of solutions and only consider the bounded solutions.
Later we will show those bounded solutions a priori have a uniform
bound and so the solution set is compact in related topology.

\begin{df}\label{df: solu-half-case} Sections $(u_1,\dots, u_N)$ are said to be the
solutions of the perturbed Witten equation (\ref{plgw-equ}) if they
satisfy the following conditions:
\begin{enumerate}

\item for each $j, u_j\in
L^2_{1,loc}(\Sigma\setminus\{z_1,\dots,z_k\}, |\LL_j|)\cap
L^\infty(\Sigma,|\LL_j|).$

\item $(u_1,\dots,u_N)$ satisfy the perturbed  Witten equation
(\ref{plgw-equ}) almost everywhere;

\item near each marked point, the integral
$$
\sum_j\int^\infty_0\int_{S^1} |\frac{\pat u_j}{\pat s}|^2 d\theta
ds<\infty.
$$
\end{enumerate}
\end{df}

\textbf{An Estimate of the $D_n^T$ equation.} Consider the $D_n^T$
equation:\
\begin{numcases}{}
\bpat u_1+I_1(n\overline{u_1^{n-1}u_2+\bar{b}_1\beta_1})=0\\
\bpat u_2+I_1(\overline{u_1^n+2u_2+\bar{b}_2\beta_2})=0,
\end{numcases}
where $b_1,b_2$ are perturbation parameters and $\beta_1,\beta_2$
are cut-off functions near the marked points.

In a holomorphic coordinates system $\{e_1,e_2\}$, this equation can
be written as the following form (we still use $u_i$ to represent
the coordinate functions)

\begin{numcases}{}
\bpat u_1+\phi_1^2 n\overline{u_1^{n-1}u_2}+b_1\beta_1\phi_1^2=0\label{eq-dtn-1}\\
\bpat
u_2+\phi_2^2\overline{u_1^n+2u_2}+b_2\beta_2\phi_2^2=0\label{eq-dtn-2},
\end{numcases}
where $\phi_1=|e_1|^{-1}=|dz|^{-1/2n}$ and
$\phi_2=|e_2|^{-1}=|dz|^{-1/2}.$

Let $\lambda>0$, and set
\begin{numcases}{}
v_1=\lambda^{\frac{1}{n}} u_1(z)\nonumber\\
v_2=\lambda u_2(z)\nonumber.
\end{numcases}
Then $(v_1,v_2)$ satisfy

\begin{numcases}{}
\bpat v_1+\phi_1^2 s_0 n\overline{v_1^{n-1}v_2}+\lambda^{\frac{1}{n}}b_1\beta_1\phi_1^2=0\label{eq-dtn-3}\\
\bpat v_2+\phi_2^2\overline{v_1^n+2v_2}+\lambda
b_2\beta_2\phi_2^2=0\label{eq-dtn-4},
\end{numcases}

where $s_0=\lambda^{\frac{2}{n}-2}$.

Define
\begin{numcases}{}
w_1= v_1^{n}\nonumber\\
w_2=v_2\nonumber;
\end{numcases}
we have (setting $\beta=n-1/n$)

\begin{numcases}{}
\bpat w_1+\phi_1^2 s_0 n^2|w_1|^{2\beta}\bar{w}_2+\lambda^{\frac{1}{n}}nb_1\beta_1\phi_1^2w_1^\beta=0.\label{eq-dtn-5}\\
\bpat w_2+\phi_2^2\overline{w_1+2w_2}+\lambda
b_2\beta_2\phi_2^2=0\label{eq-dtn-6}.
\end{numcases}

\begin{lm}\label{lm:dtn-1} Suppose that $s_0\ge 1$ and $||w_2||_{B_R(z)}<\infty$ for some $z\in \Sigma$,
 then for any $1<p<\infty$ there exists $C$ depending only on
$p,||w_2||_{L^2(B_R(z))}$ such that the solutions of
(\ref{eq-dtn-5})-(\ref{eq-dtn-6}) satisfies
\begin{align}
&\int_{B_{\frac{R}{8}}(z)} s_0
n^2|w_1|^{2\beta}|w_2|^2\varphi^2+|w_1|^2\varphi^2\le
C(||w_2||_{L^2(B_R(z))})\\
&||w_2||_{L^2_1(B_{\frac{R}{8}})},\;||w_2||_{L^p(B_{\frac{R}{8}})}\le
C(||w_2||_{L^2(B_R(z))}).
\end{align}
\end{lm}

\begin{proof} Multiplying (\ref{eq-dtn-5}) and (\ref{eq-dtn-6}) by $w_2,w_1$
respectively,  we obtain
\begin{numcases}{}
w_2\bpat w_1+\phi_1^2 s_0 n^2|w_1|^{2\beta}|w_2|^2+\lambda^{\frac{1}{n}}nb_1\beta_1\phi_1^2w_2w_1^\beta=0\\
w_1\bpat w_2+\phi_2^2 |w_1|^2+2\phi_2^2 \bar{w}_2
w_1+\lambda\beta_2\phi_2^2 w_1=0.
\end{numcases}
Sum the two identities together, and we have
\[
\bpat (w_1 w_2)+\phi_1^2 s_0 n^2|w_1|^{2\beta}|w_2|^2+\phi_2^2
|w_1|^2+\phi_2^2 2\bar{w}_2
w_1+\lambda^{\frac{1}{n}}nb_1\beta_1\phi_1^2w_2w_1^\beta+\lambda\beta_2\phi_2^2
w_1=0.
\]

Let $\varphi$ be a cut-off function  supported in $B_R(z)$.
Multiplying the above identity by $\varphi^2$ and integrating on
$B_R(z)$, we have
\begin{align}
\int \phi_1^2 s_0
n^2|w_1|^{2\beta}|w_2|^2\varphi^2&+\phi_2^2|w_1|^2\varphi^2=-\int
(w_1w_2)2\varphi \bpat
\varphi+\phi_2^2 2w_1\bar{w}_2\varphi^2\label{ineq-dtn-6}\\
&-\int
(\lambda^{\frac{1}{n}}nb_1\beta_1\phi_1^2w_2w_1^\beta+\lambda\beta_2\phi_2^2
w_1)\varphi^2\\
&\le C\int |w_1 w_2 \varphi| (|\bpat\varphi|+1)+\varepsilon\int\phi_2^2|w_1|^2+C+C||w_2||_{L^2(B_R(z))}\\
&\le \varepsilon \int \phi_2^2
|w_1|^{2}\varphi^2+C_\varepsilon(\phi_2^{-1}) \int |\varphi
w_2|^2(|\bpat \varphi|+1)^2\nonumber.
\end{align}
Therefore we obtain
\begin{equation}
\int
\phi_1^2s_0n^2|w_1|^{2\beta}|w_2|^2\varphi^2+\phi_2^2|w_1|^2\varphi^2\le
C(\varphi,||w_2||_{L^2(B_R(z))}),
\end{equation}
where $C(\varphi,||w_2||_{L^2(B_R(z))})$ is a constant depending
only on $\varphi$ and the $L^2$-norm of $w_2$.

Applying the $L^2$ inner estimate to (\ref{eq-dtn-6}), we have
\begin{align}
||w_2||_{L^2_1(B_{\frac{R}{4}})}\le&
C(||w_2||_{L^2(B_{\frac{R}{2}})}+||(w_1+w_2)||_{L^2}+C)\nonumber\\
\le &C(||w_2||_{L^2}+|| w_1||_{L^2}+C)\le
C(\varphi,||w_2||_{L^2(B_R(z))}).
\end{align}

By the Sobolev embedding theorem, we have
\begin{equation}
||w_2||_{L^p(B_{\frac{R}{8}})}\le
C(\varphi,||w_2||_{L^2(B_R(z))}),\;\forall 1<p<\infty.
\end{equation}
\end{proof}

Now we let $\lambda=1$, and return to the discussion of equations
(\ref{eq-dtn-1})-(\ref{eq-dtn-2}).

\begin{crl} Let $(u_1,u_2)$ be the solutions of
(\ref{eq-dtn-1})-(\ref{eq-dtn-2}), if
$||u_2||_{L^2(B_R(z))}<\infty$, then we have
\begin{align}
||u_1||_{L^{2n}(B_{\frac{R}{8}})}, ||u_2||_{L^2_1(B_{\frac{R}{8}})},
||u_2||_{L^p(B_{\frac{R}{8}})}\le
C(||u_2||_{L^2(B_R)})\label{ineq-dtn-7}.
\end{align}
\end{crl}

\begin{lm} Let $(u_1,u_2)$ be the solutions of
(\ref{eq-dtn-1})-(\ref{eq-dtn-2}), if $||u_2||_{B_R(z)}<\infty$,
then for any $k$, there exists a constant $C$ depending only on the
norm $||u_2||_{L^2(B_R(z))}$ such that
\begin{align*}
&||u_1||_{C^k(B_{\frac{R}{8}})}, ||u_2||_{C^k(B_{\frac{R}{8}})}<C.
\end{align*}
\end{lm}

\begin{proof}
Applying the $L^2$ interior estimate to the equation
(\ref{eq-dtn-1}), we have
$$
||u_1||_{L^2_1(B_{\frac{R}{4}})}\le
C(||u_1||_{L^2(B_{\frac{R}{2}})}+||nu_1^{n-1}u_2||_{L^2(B_R(\frac{R}{2}))}+C)\le
C(||u_2||_{L^2(B_R)}),
$$
where we used the inequalities (\ref{ineq-dtn-7}). By
the  Sobolev
embedding theorem, we know that
\begin{equation}
||u_1||_{L^p(B_{\frac{R}{8}})}\le C(||u_2||_{L^2(B_R)}).
\end{equation}

By the $L^p$-estimate, we have
\[
||u_1||_{L^p_1(B_{\frac{R}{4}})}\le
C(||u_1||_{L^p(B_{\frac{R}{2}})}+||nu_1^{n-1}u_2||_{L^p(B_R(\frac{R}{2}))}+C)\le
C(||u_1||_{L^p(B_R)},||u_2||_{L^p(B_R)})\le C(||u_2||_{L^2(B_R)}).
\]
By the Schauder estimate and the bootstrap argument, we can reach the
conclusion.
\end{proof}

\begin{crl}[regularity] Any solutions $(u_1,u_2) $of the equations
(\ref{eq-dtn-1})-(\ref{eq-dtn-2}) are smooth. Furthermore, if for
fixed $r>0$ and any point $z\in \cC$, the $L^2$-norm
$||u_2||_{L^2(B_r(z))}$ of the solutions $(u_1,u_2)$ are uniformly
bounded, then any $C^k$-norm of $(u_1,u_2)$ are uniformly bounded.
\end{crl}

Once we know the $C^m$ norm of a solution $\bu$ is bounded by its
$C^0$ norm, the proof about the asymptotic behavior of the solutions
near the marked points in Subsection \ref{subsec:asym-beha} and
\ref{subsec:expo-deca} can go through without change. Hence we have
the following result.

\begin{thm}[Asymptotic estimate] Suppose that $\bu$ is a solution of
the perturbed Witten equation, then Theorem \ref{thm:asymp} holds.
\end{thm}

\begin{thm}[compactness] For any $k\in \N$, there exists a constant
$C$ depending only on $k$ such that for any solutions $(u_1,u_2)$ of
(\ref{eq-dtn-1})-(\ref{eq-dtn-2}), there is the uniform bound
\begin{equation}
||u_1||_{C^k},||u_2||_{C^k}\le C.
\end{equation}
\end{thm}

\begin{proof} By the asymptotic behavior of the solutions near the
marked points, the $C^0$ norm of the solution $u_2$ in each pair
$(u_1,u_2)$ will be bounded by a uniform constant near the infinity
part. Hence the value of $\sup_{z\in \Sigma} ||u_2||_{L^2(B_r(z))}$
can only be achieved in the interior of the Riemann surface
$\Sigma$.

Now we will prove the theorem by contradiction. Assume that the
compactness does not hold, then there is a sequence of solutions
$(u_{m1},u_{m2})$ satisfying the equations
(\ref{eq-dtn-1})-(\ref{eq-dtn-2}), and
\begin{equation}
\sup_{z\in \Sigma}
||u_{m2}||_{L^2(B_r(z))}=||u_{m2}||_{L^2(B_r(z_m))}=\lambda_m^{-1}\to
\infty, z_m\in \Sigma.
\end{equation}

Assume that $z_m\to p\in\Sigma$ and we have for large $m$:
\begin{equation}
\lambda_m^{2}\int_{B_{2r}(p)}|u_{m2}|^2\ge 1.
\end{equation}

Do the scaling transformation

\begin{numcases}{}
v_{m1}=\lambda_m^{\frac{1}{n}} u_{m1}(z)\nonumber\\
v_{m2}=\lambda_m u_{m2}(z)\nonumber,
\end{numcases}

then $v_{m1},v_{m2}$ satisfy the following equation (where
$s_{0m}=\lambda_m^{\frac{2}{n}-2}$):
\begin{numcases}{}
\bpat v_{m1}+\phi_1^2 s_{0m}n\overline{v_{m1}^{n-1}v_{m2}}
+\lambda^{\frac{1}{n}}b_1\beta_1\phi_1^2=0\label{eq-dtn-7}\\
\bpat
v_{m2}+\phi_2^2\overline{(v_{m1}^n+2v_{m2})}+\lambda\beta_2\phi_2^2=0\label{eq-dtn-8}.
\end{numcases}

By Lemma \ref{lm:dtn-1}, we have for any fixed $0<R\le r$ and point
$z$,
\begin{align}
&\int_{B_{\frac{R}{8}}(z)} \phi_1^2 s_{0m}
n^2|v_{m1}|^{2(n-1)}|v_{m2}|^2\varphi^2+\phi_2^2|v_{m1}|^{2n}\varphi^2\le
C(||v_{m2}||_{L^2(B_R(z))})\le C_R\label{ineq-dtn-8}\\
&||v_{m2}||_{L^2_1(B_{\frac{R}{8}})},\;||v_{m2}||_{L^p(B_{\frac{R}{8}})}\le
C(||v_{m2}||_{L^2(B_R(z))})\le C_R\label{ineq-dtn-9}.
\end{align}

Hence there are functions $v_2^\infty \in L^2_{1,loc}(\Sigma),
v_1^\infty \in L^{2n}_{loc}(\Sigma)$ such that the following
sequences converge:
\begin{align*}
&\bpat v_{m2}\rightharpoonup \bpat v_2^\infty \;\text{weakly
in}\;L^2_{loc},\;\;v_{m2}\to v_2^\infty\;\text{strongly in}\; L^p_{loc},1<p<\infty.\\
&v_{m1}\rightharpoonup v_1^\infty\;\text{weakly in}\;L^{2n}_{loc}.\\
&v_{m1}^{n-1}v_{m2}\to 0,\;\text{strongly in}\;L^2_{loc}.
\end{align*}

Hence we get the limit equation for $v_2^\infty$:
\begin{equation}
\bpat
v_2^\infty+\phi_2^2\overline{(v_1^\infty+2v_2^\infty)}=0,\;\text{a.e}.
\end{equation}

Multiplying $2v_2^\infty$ and using the fact that $v_1^\infty
v_2^\infty=0$ almost everywhere, we get
\begin{equation}
\bpat (v_2^\infty)^2+\phi_2^2 4|v_2^\infty|^2=0.
\end{equation}
Since the $L^p$-norm of $v_2^\infty$ is bounded, by regularity
theory and approximation analysis, $v_2^\infty$ is a bounded smooth
solution tending to zero in the ends. By Witten Lemma, we know that
$v_2^\infty\equiv 0$, but this contradicts the fact that
$||v_2^\infty||_{L^2(B_{2r}(p))}\ge 1$.

\end{proof}

\section{Compactness}\label{sec:compact}

\subsection{Sequence convergence of the moduli space $\MMr_{g,k}
$}\ 

\

Traditionally, the compactness theorem of various moduli spaces is
formulated in terms of possible geometric limits of its sequence. In
Section two,  we gave a description of the topology of
$\MMr_{g,k}(\bgamma)$ in terms of its local neighborhood. Here, we
take a different point of view it in terms of converging sequences.

Take a sequence of rigidified $W$-curves $\frkc^n=(\cC^n,p_1,
\dots,p_k, \LL^n_1,\dots, \LL^n_t, \varphi^n_1, \dots, \varphi^n_s,
\psi^n_1,\dots,\psi^n_k)$ in $\MMr_{g,k}(\bgamma)$. Here we assume
that the $W$-structure $(\LL^n_1,\dots, \LL^n_t,\varphi^n_1, \dots,
\varphi^n_s)$ induces the data $\bgamma=(\gamma_1,\dots,
\gamma_k)$, i.e., we assume the sequence of $W$-curves has type
$\bgamma$. The underlying curves $(\cC^n,p_1, \dots,p_k)$ have dual
graph $\Gamma^n$. Since the combinatorial types of the dual graphs
are finite in number, we can assume that $\Gamma^n=\Gamma$ for all
$n$. Similarly, we can let $\psi^n_i=\psi_i$ at each marked point
$p_i$, because the set of rigidifications at one marked point is
finite and characterized by a group element in $G$. Hence it is
sufficient to consider the sequence of $W$-curves $\frkc^n$ in
$\MMr_W(\Gamma)$, where $\Gamma$ is the $G$-decorated stable graph
with each tail labelled by a group element in $G$ and having the
same rigidification. We can further assume that $\frkc^n\in
\MMr_W(\Gamma)$, where $\Gamma$ is also a fully $G$-decorated stable
graph. In fact, each half-edge $\tau$ of $\Gamma$ corresponds to an
orbifold point $p_\tau$ of the normalization of the underlying
curve, thus has a corresponding choice of $\gamma_\tau\in G$. Since
$G$ is a finite group, we can assume that all $W$-curves attain the
same group element $\gamma_\tau\in G$. Notice that we don't fix the
rigidification at the nodal points. We start from the following
obvious lemma:

\begin{lm} Suppose that the sequence of $W$-curves $\frkc^n\in
\MMr_W(\Gamma)$ converges to $\frkc$ in $\MMr_W(\Gamma)$. Then
$\str(\frkc^n)\to \str(\frkc)$ in $\MM_{g,k}(\Gamma)$.
\end{lm}

In general when considering the convergence in $\MMr_{g,k}(\bgamma)$
the limit curve of the sequence of underlying curves of $W$-curves
$\frkc^n$ may change the combinatorial type. So to define the
convergence in $\MMr_{g,k}(\bgamma)$, we have to study the
degeneracy behavior of the line bundles and the $W$-structures.

Consider a sequence of the rigidified $W$-curves $\frkc^n$ such
that $st^{rig}(\frkc^n)=\cC^n\rTo \cC$ degenerating only
along a circle to a nodal point $z\in \cC$. Hence by the
description of Proposition \ref{deli-mumf}, there is a sequence of
gluing parameters $\zeta^n=(s^n,\theta^n)\in [T_0,\infty)\times
S^1$( $s^n\to \infty$) such that $\cC^n=\cC_{0,\zeta^n}$. $\cC^n$
is obtained through gluing the corresponding cylinders
$U_n=[\frac{1}{2}T^n,\frac{3}{2}T^n]$ ($s^n=2T^n$) in two
different components of $\cC$ by a biholomorphic map.

The canonical bundle $K_{\cC^n}$ when restricted to  the cylinder
$U_n$ can be trivialized. Fixing a trivialization, the
$W$-structure gives the isomorphism $\varphi^n_j:
W_j(\LL^n_1,\dots,\LL^n_t)\rTo \C$.

\begin{lm} For each $n$, the line bundles $\LL_i^n$ are flat and
the structure group of the $N$-tuple $(\LL^n_1,\dots,\LL^n_N)$ on
$U_n$ is $G$.
\end{lm}

\begin{proof} Let $e^n_i(\alpha), i=1,\dots, N$ be the basis of $\LL^n_i$
in a small chart $U_\alpha$ of $U_n$ such that
$W_j(e^n_1(\alpha),\dots,e^n_N(\alpha))=1$ for all $j$. Take an
atlas $\{U_\alpha\}$ of $U_n$ such that all the line bundles can
be trivialized in each chart and the chosen basis satisfies the
above relation. Now it is easy to see that the transition function
on each $U_\alpha\cap U_\beta$ must lie in the group $G$. Hence
each $\LL^n_i$ is a flat line bundle.

\end{proof}

On $U_n$, each $\LL^n_i$ is completely classified by its holonomy
$$
\rho^n_i: \Z=\pi_1(U_n)\rTo U(1),
$$
and $\rho^n_i$ maps the generator to $e^{2\pi c^n_i/d}$ for some
$0\le c^n_i<d$ such that $\gamma^n=(e^{2\pi
c^n_1/d},\dots,e^{2\pi c^n_t/d})\in G$. Such a flat line bundle
can be explicitly constructed as follows. Let
$(\widetilde{U_n})^d$ be a degree-$d$ unramified cover of $U_n$.
The line bundle is of the form
$\LL^n_i=(\widetilde{U_n})^d\times_{e^{2\pi c^n_i/d}}\C$. After
choosing a subsequence, we can assume that for all $n$, $\LL^n_i$
corresponds to a fixed value $e^{2\pi c_i/d}$. When $n$ goes to
infinity, we can associate a natural "limit." We change the
coordinate of $U_n$ to $U_{n0}=[0, T^n]\times S^1$ and change
$\tilde{U}_n$ to $\tilde{U}_{n0}$ accordingly.  When $n$ goes to
infinity, $U_{n0}$ converges to a punctural disc $U_0-\{0\}$.
$\LL^n_i$ extends uniquely to a flat orbifold line bundle
$\LL_{i0}=(\widetilde{U}_0)^d\times_{e^{2\pi c_i/d}}\C$ over
$(\widetilde{U}_0,\Z_d)$. Now we change coordinates of $U_n$ to
$U_{n1}=[-T^n, 0]\times S^1$. The same argument obtains a limiting
orbifold line bundle $\LL_{i1}=(\widetilde{U}_1)^d\times_{e^{-2\pi
c_i/d}}\C$ over $(\widetilde{U}_1,\Z_d)$.  Finally, since
$\LL_{i0}, \LL_{i1}$ are really the limit of the same flat line
bundle, there is a canonical identification $\LL_{i0}|_0\cong
\LL_{i1}|_0$. Therefore, we obtain a $W$-structure $\LL_{i}$ on
the nodal curve $\cC$ locally represented by $U_0\wedge U_1$.
Furthermore, $\varphi^n_j$ converges to an isomorphism
$\varphi_{j}: W(\LL_{i},\dots, \LL_{N})\rTo \C$.

When $\mbox{gcd}(c_1,\dots, c_t,d)=a>1$, the orbifold structure
$(\widetilde{U}_0,\Z_d)$ is redundant. The group action is not
faithful; hence we can modify the uniformizing system of $\LL_0$
by redefining the local group to be $\Z_{\frac{d}{a}}$.

Hence  degeneracy of the underlying curves $\frkc^n$ along a
circle induces a geometric limit $(\LL_{1},\dots,
\LL_{N},\varphi_{1},\dots, \varphi_{N})$ on $\cC$.

\begin{thm}
$(\LL_{1}^i,\dots, \LL_{N}^i,\varphi_{i1},\dots, \varphi_{iN})$
converges to $(\LL_{1},\dots, \LL_{N},\varphi_{1},\dots,
\varphi_{N})$ in $\MMr_{g,k}$.
    \end{thm}

    \begin{proof} For this purpose, we need to
show that the above sequence is in any neighborhood of
$(\LL_{1},\dots, \LL_{N},\varphi_{1},\dots, \varphi_{N})$ for
$i>>0$.

But this is automatic from our construction. Recall the
construction of a neighborhood of $(\LL_{1},\dots,
\LL_{N},\varphi_{01},\dots, \varphi_{0N})$ in Proposition 2.2.4.
Let $\frkc=(\cC, p_1,\dots, p_k, \LL_{1},\dots,
\LL_{N},\varphi_{1},\dots, \varphi_{N}, \linebreak
\psi_1,\dots,\psi_k)$ be a rigidified $W$-curve having a orbifold
nodal point $z$ with local group $G_z\cong\Z/d$, then we can
construct the nearby rigidified $W$-curves. Let $\zeta=(s_z,
\theta_z), s_z=2T$ be the gluing parameter; then by Proposition
\ref{deli-mumf} there is a nearby curve $\cC_{0, \zeta}$. Take one
component $([T_0,\infty)\times S^1)_k, k=1,2$ with coordinates
$(s_k, \theta_k)$. Let $U_k=([\frac{1}{2}T,\frac{3}{2}T]\times
S^1)_k$ and $\widetilde{U_k}^d$ be the degree-$d$ covering of
$U_k$. A local trivialization of $\LL_{j}$ is given by
$\widetilde{U_k}^d\times \C$ with the group action $e^{2\pi
 i/d}: (z_k, w_j)\mapsto (e^{2\pi i /d}z_k, e^{(-1)^k 2\pi i
 c_j/d}w_j)$. Thus the orbifold line bundle $\LL_{j}$, when
 restricted to $U_k$, is a flat line bundle with holonomy
$$
\rho_{k,j}: \pi_1(U_k)\rTo U(1), k=1,2, j=1,\dots, N,
$$
where $\rho_{k,j}$ sends the generator of $\pi_1(U_k)\cong\Z/d$ to
$e^{(-1)^k 2\pi i c_j/d}w_j$.  Part of the data of a rigidified $W$-structure on a
nodal curve is an identification of $\LL_{0j}$
at nodal point to a compatible $W$-structure.
 Now we use this
identification and the trivialization on $U_k$ to glue the
$W$-structures on $U_1$ and $U_2$ to get a $W$-structure on
$\cC_{0,\zeta}$. In our construction of $\LL_{0j}$, it naturally
 carries a trivialization at the nodal point. The above gluing process gives precisely
 the $\LL_{i}^n=(\widetilde{U_n})^d\times_{e^{2\pi c^n_i/d}}\C$ on the cylinder.
\end{proof}

\subsection{Topology of the moduli spaces $\MMr_{g,k}
$ and $\MMrs_{g,k}
$
}

\
\subsubsection{The perturbed Witten map over $\MMr_{g,k}(\bgamma)$}\label{sect-7.2.1}

\

\subsubsection*{Fr\'echet stratified orbibundles over $\MMr_{g,k}(\bgamma)$} By
the  algebraic geometric construction in \cite{FJR2}, there exists
the universal rigidified $W$-curve $\cC^{rig}_{g,k}\rTo
\MMr_{g,k}(\bgamma)$. Also, by Proposition \ref{deli-mumf}, we can
obtain the required family of smooth metrics on $\cC^{rig}_{g,k}$.
We first construct the stratified Fr\'echet orbibundles $B^0$ and
$B^{0,1}$ over $\MMr_{g,k}(\bgamma)$.

Let $\frkc=(\cC, p_1,\dots, p_k, \LL_{1},\dots,
\LL_{N},\varphi_{1},\dots, \varphi_{N}, \psi_1,\dots,
\psi_k)=:(\cC, \LL, \Psi)$ be a rigidified $W$-curve representing an
element $[\frkc]$ in $\MMr_{g,k}(\bgamma)$. $\cC$ can be decomposed
into irreducible components: $\cC=\cup_\nu \cC_\nu$. Denote by
$\pi_\nu: \cC_\nu\rTo \cC$ the projection map from the
renormalized component $\cC_\nu$ to the nodal curve $\cC$. The
automorphism group of $\frkc$ is $\aut(\frkc),$ which is the
extension of $\aut(\cC)$ along $\aut_\cC(\LL, \Psi)$. We have the
$\aut(\cC)$-invariant metric of $\cC$ and the induced metrics on the
line bundles $\LL_i$. Define
\begin{align*}
&C^\infty(\cC, \LL_j):=\{(u_{j,\nu})\in \oplus_\nu
C^\infty(\cC_\nu, \LL_j)| u_{j,\nu}(p_\nu)=u_{j,\mu}(p_\mu),
\;\text{if}\;\pi_\nu(p_\nu)=\pi_\mu(p_\mu)\}\\
&C^\infty(\cC, \LL_j\otimes \Lambda^{0,1}):=\{(u_{j,\nu})\in
\oplus_\nu C^\infty(\cC_\nu, \LL_j\otimes \Lambda^{0,1})\},
\end{align*}
and let $L^p_1(\cC,\LL_j)$ and $L^p(\cC,\LL_j\otimes \Lambda^{0,1})$
be the $p$-integrable Sobolev spaces with respect to $C^\infty(\cC,
\LL_j)$ and $C^\infty(\cC, \LL_j\otimes \Lambda^{0,1})$
respectively. We always take $p>2$ to ensure the continuity of the
functions in $L^p_1(\cC,\LL_j)$. Let $B_\frkc^0$ represent the
Fr\'echet spaces $L^p_1(\cC,\LL_1)\times \dots L^p_1(\cC,\LL_N)$ or
$C^\infty(\cC, \LL_1)\times \dots C^\infty(\cC, \LL_N)$ and let
$B_\frkc^{0,1}$ represent $L^p(\cC,\LL_1\otimes \Lambda^{0,1})\times
\dots L^p(\cC,\LL_N\otimes \Lambda^{0,1})$ or $C^\infty(\cC,
\LL_1\otimes \Lambda^{0,1})\times \dots C^\infty(\cC, \LL_N\otimes
\Lambda^{0,1})$. The automorphism group acts naturally on
$B_\frkc^0$ and $B^{0,1}_\frkc$. For example, if $(\tau, g)\in
\aut(\cC)\times \aut_\cC(\LL,\Psi)$ is an element in $\aut(\frkc)$
and $u\in B^0$, then $(\tau,g)\cdot u(z)=gu(\tau\cdot z)$. Now
$B^0_\frkc$ (resp. $B^{0,1}_\frkc$) become the fiber at $[\frkc]$ of
the Fr\'echet orbibundle $B^0$ (resp. $B^{0,1}$). In fact, we can
give a description of the uniformizing system of $B^0$ around
$[\frkc]$. Let $(U, \aut(\cC))$ be a uniformizing system of
$[\cC]\in \MM_{g,k}$ such that $[\cC]\in U/\aut(\cC)$; then the
uniformizing system of $[\frkc]\in \MMr_{g,k}(\bgamma)$ is given by
$(U, \aut(\frkc))$. Hence the uniformizing system of $B^0$ around
$[\frkc]$ is $(U\times B^0_\frkc, \aut(\frkc))$. The action of $g\in
\aut(\frkc)$ on $B^0|_U$ is an isometry from the Fr\'echet space
$B^0_{\frkc'}$ to $B^0_{g\cdot\frkc'}$. Actually a local
trivialization is given by the gluing map $Glue$ in Lemma
\ref{lm-appro-1}. Similarly, one can get the description of the
uniformizing system of $B^{0,1}$.

\subsubsection*{Witten map} Since the projection map $\pi: B^0\rTo
\MMr_{g,k}(\bgamma)$ is an orbifold morphism, we have the pull-back
Fr\'echet orbibundle $\pi^*B^{0,1}$ over $B^0$. The Witten map is
defined as the section from $B^0$ to $\pi^* B^{0,1}$:
\begin{align*}
&\widetilde{WM}(\frkc,\bu)=\widetilde{WM}_\frkc(u_1,\dots, u_N)\\
&=\left(\bpat_{\cC} u_1+\tilde{I}_1\left(\overline{\frac{\pat
W}{\pat u_1}}\right),\dots, \bpat_{\cC}
u_N+\tilde{I}_1\left(\overline{\frac{\pat W}{\pat
u_N}}\right)\right).
\end{align*}
Note that the Witten map is actually defined between uniformizing
systems; hence it is required to be $\aut(\frkc)$-equivariant.
This equivariance can be easily seen from the equation.

\begin{rem} It is not hard to see that Witten map descends to the moduli space of
$W$-structures $\WW_{g,k}(\bgamma)$. However, the perturbed Witten
map is only defined over $\MMr_{g,k}(\bgamma)$. This is the main
reason that we use $\MMr_{g,k}(\bgamma)$. Alternatively, one can
think that the perturbed Witten map is multi-valued over
$\WW_{g,k}(\bgamma)$.
\end{rem}

\subsubsection*{The perturbed Witten map}  Since the nonlinear
term appearing in the Witten map is degenerate, it is hard to get
the asymptotic behavior of the solutions lying in the zero locus.
We need to perturb the Witten map near the broad marked or broad
nodal points. We will treat the two cases in different ways.

Let $\frkc=(\cC,\LL, \Psi)$ be a rigidified $W$-curve. Choosing a
rigidification $\psi$ at a marked point $p$ means fixing a basis
$e=(e_1,\dots,e_N)$ of the line bundles $\LL_1\times \dots
\LL_N$ near $p$ such that
$$
W(e_1,\dots,e_N)=dz/z.
$$
The group $G$ acts on $e$. We can fix a basis $e^0$ and call the
corresponding rigidification $\psi^0$ the \emph{ standard
rigidification}. Then the set of rigidifications at $p$ is $g\cdot
\psi^0$. Let $(\psi^0, \bb)$ be a pair of data, where
$\bb=(b_1,\dots,b_N)$ is the perturbation parameter. We define
the group action of $G$ by $g\cdot(\psi^0,\bb)=(g\cdot\psi^0,
g\cdot \bb)$.

For any rigidification $\psi=g\cdot\psi^0$, we take the pair
$(\psi, g\cdot \bb)$. As the first step we will construct a
perturbed map over $\frkc$ based on the data $(\psi, g\cdot\bb)$.
If $p$ is a narrow marked point, then we leave the Witten
map unchanged. We only perturb the Witten map near a broad marked
point $p$. Let $[0,\infty]\times S^1$ be the cylindrical
neighborhood of $p$. Since $p$ is broad, the line bundles $\LL_j$
are classified into broad line bundles and narrow line
bundles. As shown in Section \ref{sec:witten}, the group action of
$\gamma$ at $p$ determines $W_{\gamma}$, the sum of partial
monomials of $W$, and thus determines the choice of perturbed
function $W_{0,\gamma}$. As done in Section \ref{sec:witten}, we
choose a number $\bar{T}_0>T_0$, where $T_0$ is the number related
to the gluing parameter occurring in Proposition \ref{deli-mumf},
and a section $\beta_j$ whose derivative has compact support in
$[\bar{T}_0,\bar{T}_0+1]\times S^1$ of some line bundle such that
$\beta_j \in \Omega(\LL_j^{-1}\otimes K_{log})$. After choosing
$\bar{T}_0$, we fix $\bar{T}_0$ in this paper. Hence the perturbed
map on $\frkc$ is defined as
\begin{equation}\label{pert-witt-oper}
\widetilde{WM}^{(\psi,g\cdot\bb)}_\frkc(u_1, \dots,
u_N):=\left(\bpat_{\cC} u_1+\tilde{I}_1\left(\overline{\frac{\pat
(W+\beta_1 W_{0,\gamma})}{\pat u_1}}\right),\dots, \bpat_{\cC}
u_N+\tilde{I}_1\left(\overline{\frac{\pat (W+\beta_N
W_{0,\gamma})}{\pat u_N}}\right)\right).
\end{equation}

One can easily verify that
\begin{equation}
\widetilde{WM}^{(g\cdot\psi_0,g\cdot\bb)}_\frkc=\widetilde{WM}^{(\psi_0,\bb)}_\frkc.
\end{equation}

We can perturb the Witten map near each broad marked point in
this way and take $\widetilde{WM}^{(\Psi,\bb)}_\frkc$ as the
global perturbed map whose restriction near each marked $p$ is
$\widetilde{WM}^{(\psi_p,g_p\cdot\bb)}_\frkc$.

\begin{lm} $\widetilde{WM}^{(\Psi,\bb)}$ is a section from $B^0$ to $\pi^* B^{1,0}$.
\end{lm}

\begin{proof} Assume that $\cC$ has several components $\cC_i$; then the group $\aut_\cC(\LL,
\Psi)\subset \prod_i \aut_{\cC_i}(\LL_i,\Psi_i)$. Each
$\aut_{\cC_i}(\LL_i,\Psi_i)$ acts trivially on the broad line
bundles at each marked point on $\cC_i$. Hence
$\widetilde{WM}^{(\Psi,\bb)}$ is $\aut(\frkc)$-equivariant.
\end{proof}

Although we obtained the perturbed map
$\widetilde{WM}^{(\Psi,\bb)}$, we still need to perturb this map at
the broad nodal points. There is no natural rigidification at
nodes. Here, we run into the same problem as when we defined the 
perturbed
Witten map over $\WW_{g,k}(\bgamma)$. In this case, we simply treat
it as a multi-valued perturbation or multi-section. We notice that
Fukaya-Ono \cite{FO} has used the same idea already in their
construction of virtual fundamental cycles.  The reader can find the
definition and properties of  multisection in subsection
\ref{subsec:MK}.

Now assume that $p$ is the only broad nodal point connecting two
components $\cC_\nu$ and $\cC_\mu$ of $\cC$. We want to define the
new perturbed map over $\frkc$. First we give the standard
rigidification $\psi^0$ to the two half edges represented by $p$
in the $\mu$ and $\nu$ components. Denote them by $\psi^0_+$ and
$\psi^0_-$ respectively. Take the standard pair $(\psi^0_+,\bb)$.
If we take the map $\widetilde{WM}^{(\Psi,\bb)}$ as the Witten map
$\widetilde{WM}$, then we can define the new perturbed map over
the component $\mu$:
$WM^{(\psi^0_+,\bb)}_{\frkc_\mu}:=(\widetilde{WM}^{(\bb)})^{(\psi^0_+,\bb)}_{\frkc_\mu}$,
where the latter operator has the same form as
(\ref{pert-witt-oper}). The new perturbed map on the
$\nu$-component is not independent and it is defined as
$WM^{(\psi^0_-,-I(\bb))}_{\frkc_\nu}:=(\widetilde{WM}^{(\bb)})^{(\psi^0_-,-I(\bb))}_{\frkc_\nu}$.
Here $I=\diag(\xi^{k_1},\dots,\xi^{k_N})$ and $\xi^d=-1$. This is
the requirement of the compatibility from the gluing operation.
The details will be discussed later in Remark
\ref{rem-wequ-equal}. Thus we have the new perturbed map over the
two components:
$$
WM^{(\psi^0,\bb)}_{\frkc}:=WM^{(\psi^0_+,\bb)}_{\frkc_\mu}\#WM^{(\psi^0_-,-I(\bb))}_{\frkc_\nu}.
$$
Here we denote $(\psi^0,\bb):=(\psi^0_+,\psi^0_-, \bb, -I(\bb))$.

Now we consider two cases:
\begin{enumerate}

\item[(1)] \emph{Tree case}\quad Let $g=(g_1,g_2)\in
\aut_\cC(\Psi, \LL)=\aut_{\cC_\mu}(\Psi,
\LL)\times_{G/<\gamma_p>}\aut_{\cC_\nu}(\Psi, \LL)$, then there
exists a $\delta \in <\gamma_p>$ such that $g_1=g_2\delta$. We
define $g\cdot(\psi^0,\bb)=$
$(g\cdot\psi^0,g\cdot\bb):=(g_1\cdot\psi^0_+,g_2\cdot\psi^0_-,g_1\cdot\bb,
-g_2\cdot I(\bb))$. So we can define
$WM^{(g_1\cdot\psi^0_+,g_1\cdot\bb)}_{\frkc_\mu}$ and
$WM^{(g_2\cdot\psi^0_-,-g_2\cdot I(\bb))}_{\frkc_\nu}$. The latter
map equals  $WM^{(g_1\cdot\psi^0_-,-g_1\cdot I(\bb))}_{\frkc_\nu}$
since $\delta$ fixes the rigidification $\psi^0$ and it acts
trivially on the perturbed term. Hence the compatibility condition
for the two maps is satisfied and we obtain
$WM^{(g\cdot\psi^0,g\cdot\bb)}_{\frkc}$.

\item[(2)]\emph{Loop case}\quad In this case, we have
$\aut_\cC(\Psi, \LL)=<\gamma_p>\cap(\cap_{i=1}^k <\gamma_i>)$. Let
$g\in \aut_\cC(\Psi, \LL)$, then the following map is well defined
$$
WM^{(g\cdot\psi^0,g\cdot\bb)}_{\frkc}=WM^{(g\cdot\psi^0_+,g\cdot\bb)}_{\frkc_\mu}\#WM^{(g\cdot\psi^0_-,-g\cdot
I(\bb))} _{\frkc_\mu}.
$$
\end{enumerate}
Thus, in either case, for any $g\in \aut_\cC(\LL,\Psi)$ the map
$WM^{(g\cdot\psi^0,g\cdot\bb)}_{\frkc}$ is well defined and equal
to $WM^{(\psi^0,\bb)}_{\frkc}$.

\begin{lm}\label{pert-witt-oper1} If $g\in \aut_\cC(\LL, \Psi)$, then for any $\bu\in B^0_\frkc$, there is
$WM^{(\psi^0,\bb)}_\frkc(g\cdot \bu)=g\cdot
WM^{(g^{-1}\psi^0,\bb)}_\frkc(\bu)=g\cdot
WM^{(\psi^0,g\cdot\bb)}_\frkc(\bu)$.
\end{lm}

\begin{proof} We only treat the tree case; the proof of the loop case is left to the reader.
Assume that $g=(g_1,g_2)$; then
$(g^{-1}\psi^0,\bb)=(g_1^{-1}\psi^0_+,g_2^{-1}\psi^0_-,\bb,-I(\bb))$.

Let $e=(e_1,\dots,e_N)$ be the basis corresponding to the
rigidification $\psi^0_+$; then the $i$-th component of
$WM^{(\psi^0_+,\bb)}_{\frkc_\mu}(\bu)$ is
\begin{equation}
\bpat_{\cC} u_i+\tilde{I}_1\left(\overline{\frac{\pat W }{\pat
u_i}}\right)+\bar{b}_i\beta_T e_i\otimes \frac{d\bar{z}}{\bar{z}}.
\end{equation}
Since $g_1\in G$, there is a number $\lambda,\lambda^d=1$ such
that $g_1=(\lambda^{k_1},\dots, \lambda^{k_N})$. Hence it is easy
to check that the $i$-th component of
$WM^{(\psi^0_+,\bb)}_{\frkc_\mu}(\lambda^{k_1}u_1,\dots,\lambda^{k_N}u_N)$
is $\lambda^{k_i}WM^{(g_1^{-1}\psi^0_+,\bb)}_{\frkc_\mu}(\bu)$. We
have a conclusion  similar to the $\nu$ component. Hence
\begin{align*}
&WM^{(\psi^0,\bb)}_\frkc(g\cdot
\bu)=\left(WM^{(\psi^0_+,\bb)}_{\frkc_\mu}(g_1\cdot
\bu_1),WM^{(\psi^0_-,-I(\bb))}_{\frkc_\nu}(g_2\cdot
\bu_2)\right)\\
&=\left(g_1\cdot WM^{(g_1^{-1}\psi^0_+,\bb)}_{\frkc_\mu}(\bu_1),
g_2\cdot
WM^{(g_2^{-1}\cdot\psi^0_-,-I(\bb))}_{\frkc_\nu}(\bu_2)\right)=g\cdot
WM^{(g^{-1}\psi^0,\bb)}_\frkc(\bu).
\end{align*}
\end{proof}

\begin{df} \emph{The perturbed Witten map $WM_\frkc(u)$ over $\frkc$} is
the multisection $[WM^{(g\cdot\psi^0,\bb)}_\frkc: g\in
\aut_\cC(\LL,\Psi)]$ from $B_\frkc^0$ to $B^{0,1}_\frkc$ which is
$\aut(\frkc)$-equivariant by Lemma \ref{pert-witt-oper1}. In
general, if $\frkc$ has $n$ broad nodal points, then we have the
multiple index $(\psi^0, \bb):=(\psi^0_+(p_i),\psi^0_-(p_i),
\bb_i, -I(\bb_i),i=1,\dots, n),$ and the group
$\aut_\cC(\LL,\Psi)$ acts naturally on it. The perturbed Witten
map can be defined as
$$
WM_\frkc(u):=[WM^{(g\cdot\psi^0,\bb)}_\frkc: g\in
\aut_cC(\LL,\Psi)],
$$
which is also $\aut(\frkc)$-equivariant. Here
$WM^{(g\cdot\psi^0,\bb)}_\frkc$ is called a branch map of
$WM_\frkc$. The zero locus of $WM_\frkc$ is defined as the union
of the zero locus of all of its branch maps, which is also an
$\aut(\frkc)$-invariant set.
\end{df}

Since the deformation domain in $\frkc$ has empty intersection with
the perturbation domain, the perturbed Witten map is naturally
defined over $\M^{rig}_{g,k,W}(\Gamma;\bgamma)$ for any
combinatorial type $\Gamma$. However, to construct a global
perturbed Witten map over $\MMr_{g,k}(\bgamma)$, we need to modify
the existing perturbed Witten maps over the strata such that the
compatibility condition holds for these multisections. We construct
such a map by induction with respect to the order $\succ$ of the
dual graph.

As the first step, we consider the minimal stratum in
$\MMr_{g,k}(\bgamma)$. The dual graph $(\Gamma, (g_\nu), o)$ is
minimal if
$$
g_\nu =0, \;\; k_\nu=3,\;\; \forall \nu.
$$
In this case, the stratum $\MM_{g,k}(\Gamma)$ consists of only one
point, and $\MMr_{g,k}(\Gamma)$ as the covering space consists of
only finitely isolated points. Then the perturbed Witten map
$WM_\frkc$ is well-defined over $\MMr_{g,k}(\Gamma)$. The second
step is to redefine the perturbed Witten map on the nearby curves.
Define a smooth and monotone increasing function $\varpi: \R^+\times
S^1\rTo \R$ such that
\begin{align*}
\varpi(s,\theta):=\varpi(s)=\left\{
\begin{array}{ll}
0,\;& s\le 8\bar{T}_0\\
1, \;& s\ge 9\bar{T}_0.
\end{array}
\right.
\end{align*}

Without loss of generality, we show how to redefine the perturbed
Witten map on nearby curves of $\frkc$ which have only one broad
nodal point $p$ connecting two components $\frkc_\mu$ and
$\frkc_\nu$. Then a neighborhood of $\frkc$ in $\MMr_{g,k}(\gamma)$
is given by
$$
\frac{V_{\cC}\times ([T_0,\infty]\times S^1)_p}{\aut(\frkc)},
$$
where $(V_{\cC}\times ([T_0,\infty]\times S^1)_p, \aut(\cC))$ is a
uniformizing system of $[\cC]\in \MM_{g,k}(\Gamma)$.

The rigidified $W$-curve $\frkc_{y,\zeta}, \zeta=({s_p,
\theta_p})$ is constructed by gluing the $W$-structures on the
corresponding domains $[\frac{1}{4}s_p, \frac{3}{4}s_p]$ of two
components of $\frkc$. The redefined perturbed Witten map on
$\frkc_{y,\zeta}$ is defined as
$$
WM_{\frkc_{y,\zeta}}:=[WM^{(g\cdot\psi^0,\bb)}_{\frkc_{y,\zeta}}:g\in
G],
$$
where
\begin{equation}
WM^{(\psi^0,\bb)}_{\frkc_{y,\zeta}}(u_1, \dots,
u_N):=\left(\bpat_{\cC_{y,\zeta}}
u_1+\tilde{I}_1\left(\overline{\frac{\pat (W+\varpi(\zeta)\beta_1
W_{0,\gamma})}{\pat u_1}}\right),\dots, \bpat_{\cC_{y,\zeta}}
u_N+\tilde{I}_1\left(\overline{\frac{\pat (W+\varpi(\zeta)\beta_N
W_{0,\gamma})}{\pat u_N}}\right)\right).
\end{equation}
It is easy to see that the redefined map $WM$ is
$\aut(\frkc)$-equivariant on $V_{\cC}\times ([T_0,\infty]\times
S^1)_p$.

So in this way we can redefine the perturbed Witten map on the
nearby curves of $[\frkc]\in \MMr_{g,k}(\Gamma)$. Note that the
perturbation term $\varpi\beta_j W_{0,\gamma}$ will disappear if the
gluing domain approaches $2\bar{T}_0\times S^1$; thus the redefined
perturbed Witten map will be compatible with the original perturbed
Witten map defined on $\M^{rig}_{g,k,W}(\Gamma')$ satisfying
$\Gamma\prec \Gamma'$.

Now we do the induction assumption. Take the space
$\MMr_{g,k}(\Gamma')$ and assume that for any dual graph $\Gamma$
such that $\Gamma\prec\Gamma'$ we have already redefined the
perturbed Witten map on the nearby curves of $\MMr_{g,k}(\Gamma)$.
We want to define the perturbed Witten map on the nearby curves of
$\MMr_{g,k}(\Gamma')$. Take a finite open covering $\U$ of
$\MMr_{g,k}(\Gamma)$ in $\MMr_{g,k}(\Gamma')$ such that any point in
the open set of this covering has the redefined perturbed Witten
map. The perturbed Witten map was already constructed over any point
in the compact complement $\MMr_{g,k}(\Gamma')-\U$. We redefine the
perturbed Witten map on nearby curves around those points as done in
the first step. Now it is possible that for a nearby curve
$[\hat{\frkc}]$ there are two definitions of the perturbed Witten
maps; one comes from the original definition in $\MMr_{g,k}(\Gamma)$
and the other one from the new definition. But these two definitions
are identical since the deformation domain and the resolution domain
are separated. One can either deform the complex structure first and
then redefine the perturbed Witten map or redefine the perturbed
Witten map first and then deform the complex structure. Thus we
construct the redefined perturbed Witten map on the nearby curves of
$\MMr_{g,k}(\Gamma')$. By induction, we can define the global
perturbed Witten map at any point of $\MMr_{g,k}(\bgamma)$.

This global multisection $WM: B^0\rTo \pi^* B^{0,1}$ is
called \emph{ the perturbed Witten map over $\MMr_{g,k}(\bgamma)$}.

\begin{df} Let $WI_{\frkc}$ represent the branch $WM^{(\psi^0,\bb)}_\frkc$ of the perturbed Witten map
$WM_\frkc$; then \emph{the perturbed Witten equation over $\frkc$}
is defined as
\begin{equation}\label{witt-equa-noda}
WI_{\frkc}(u_1,\dots, u_N)=0.
\end{equation}
In the following, we also call the branch map $WI_{\frkc}$
the \emph{perturbed Witten map.}
\end{df}

This means that on each component $\frkc_\nu$, we have
\begin{equation}\label{witt-equa-noda1}
\bpat_{\cC_\nu} u_{i,\nu}+\tilde{I}_1\left(\overline{\frac{\pat(
W+W_{0,\beta})}{\pat u_{i,\nu}}}\right)=0, \forall i=1,\dots,N,
\end{equation}
where $u_{i,\nu}$ is the $\nu$-component of the section $u_i$ and
$W_{0,\beta}$ represents the perturbed term.

In view of the definitions, we have
\begin{equation}
WM_\frkc^{-1}(0)=\aut_{\frkc}\cdot WI_\frkc^{-1}(0).
\end{equation}

Because of the group action, we have $WM^{-1}(0)=\cup_{\frkc\in
\MMr_{g,k}(\bgamma)}WI_\frkc^{-1}(0)/\aut(\frkc)\subset B^0$. This
set is rather complicated because  transversality does not hold. Its
topology is weak and can't be characterized by the strong topology
from the Banach bundle. But we can give it the Gromov-Hausdorff
topology, and prove the Gromov compactness theorem. Finally, we can
show that it carries an orientable Kuranishi structure if the
perturbation is strongly regular. This method to construct the
virtual cycle has already been used in the proof of the Arnold
conjecture and in the construction of Gromov-Witten invariants in
general symplectic manifolds (see \cite{FO,LiT,LT,R,Sb} etc.).

\begin{rem}\label{rem-wequ-equal} Now we discuss the gluing of the Witten equations
on two components connected by a nodal point. Let $\cC$ be a nodal
curve with one broad nodal point $p$ connecting two components
$\cC_\nu$ and $\cC_\mu$. Let $(e_{i,\nu}), z_\nu)$ and
$(e_{i,\mu}, z_\mu)$ be the standard basis and coordinates of line
bundles $\LL_i$ on $\cC_\nu$ and $\cC_\mu$, respectively, such
that
$$
W(e_{1,\nu}, \dots, e_{N,\nu})=\frac{d z_\nu}{z_\nu}=-\frac{d
z_\mu}{z_\mu}=-W(e_{1,\mu}, \dots, e_{N,\mu}).
$$
Then the perturbations we choose in the  $\mu$-component and
$\nu$-component are not independent. Now the relation between the
two families of perturbation parameters is shown as the conclusion
of the following facts.

As before, we take the cylindrical coordinates
$z_\nu=e^{\zeta_\nu}$ and $z_\mu=e^{\zeta_\mu}$. Then we have the
relation
\begin{align*}
&W(e_{1,\nu}, \dots, e_{N,\nu})=-d\zeta_\nu,\;\zeta_\nu\in
[0,\infty)\\
&W(e_{1,\mu}, \dots, e_{N,\mu})=-d\zeta_\mu,\;\zeta_\mu\in
[0,\infty),
\end{align*}
and $d\zeta_\nu=-d\zeta_\mu, \zeta_\nu+\zeta_\mu=\zeta_p$, where
$\zeta_p$ is the gluing parameter if we want to do the gluing
operation.

We have the following facts:

\begin{enumerate}
\item In the coordinate system $(z_\nu, e_{i,\nu})$, the perturbed
polynomial
$$
W_{0,\nu}=\sum_j b_{j,\nu}\beta_{j,\nu}u_{j,\nu}\in
\Omega(\frac{dz_\nu}{z_\nu}).
$$
A similar expression holds on the $\mu$ component.

\item Assume that $u_{i,\nu}=\tilde{u}_{i,\nu}e_{i,\nu}$; then the
perturbed Witten equation has the form
\begin{equation}\label{equ-nu}
\bpat_{\zeta_\nu} \tilde{u}_{i,\nu}-\overline{\frac{2\pat W}{\pat
\tilde{u}_{i,\nu}}}-\overline{2 b_{i,\nu}}=0.
\end{equation}
In $(z_\mu,e_{i,\mu})$ coordinates, we have
$u_{i,\mu}=\tilde{u}_{i,\mu}e_{i,\mu}$ and the corresponding
equation
\begin{equation}\label{equ-mu}
\bpat_{\zeta_\mu} \tilde{u}_{i,\mu}-\overline{\frac{2\pat W}{\pat
\tilde{u}_{i,\mu}}}-\overline{2 b_{i,\mu}}=0.
\end{equation}

\item Let $\xi^d=-1$ and $\hat{e}_{i,\mu}=\xi^{k_i}e_{i,\mu}$, we
have
$$
W(\hat{e}_{1,\mu}, \dots, \hat{e}_{N,\mu})=d\zeta_\mu.
$$
When we do gluing,  first we need to identify the coordinate
$\zeta_\mu=\zeta_p-\zeta_\nu$ and so $d\zeta_\nu=-d\zeta_\mu$.
Secondly, we should identify the line bundles on two components by
identifying the basis $e_{i,\nu}$ with $\hat{e}_{i,\mu}$.

\item The local expression of the perturbed Witten equation in the
coordinate system $(\zeta_\mu, \hat{e}_{i,\mu})$ is given as
below. We have
$$
u_{i,\mu}=\hat{u}_{i,\mu}\hat{e}_{i,\mu}=\tilde{u}_{i,\mu}e_{i,\mu},
$$
and so
$$
\hat{u}_{i,\mu}=\xi^{-k_i}\tilde{u}_{i,\mu}.
$$
Substituting the above equality into Equation (\ref{equ-mu}), one
has
$$
\bpat_{\zeta_\mu}(\xi^{k_i} \hat{u}_{i,\mu})-\overline{\frac{2\pat
W}{\pat \hat{u}_{i,\mu}}(\xi^{k_1}\hat{u}_{1,\mu},\dots,
\xi^{k_N}\hat{u}_{N,\mu})}-\overline{2 b_{i,\mu}}=0.
$$
This is equivalent to
\begin{equation}
\bpat_{\zeta_\mu} \hat{u}_{i,\mu}+\overline{\frac{2\pat W}{\pat
\hat{u}_{i,\mu}}}-\overline{2 \xi^{k_i}b_{i,\mu}}=0.
\end{equation}
So in $(\zeta_\nu, \hat{e}_{i,\mu})$ coordinates, we have
\begin{equation}\label{equ-mu-change}
\bpat_{\zeta_\nu} \hat{u}_{i,\mu}-\overline{\frac{2\pat W}{\pat
\hat{u}_{i,\mu}}}+\overline{2 \xi^{k_i}b_{i,\mu}}=0.
\end{equation}

\item After transformation, Equation (\ref{equ-mu-change}) should
be the same as Equation (\ref{equ-nu}). So we obtain the relation
between two parameter groups:
\begin{equation}
b_{i,\nu}=-\xi^{k_i}b_{i,\mu}.
\end{equation}
And at the nodal point $p$ one has
$$
e_{i,\nu}=\hat{e}_{i,\mu},
\tilde{u}_{i,\nu}(+\infty)=\hat{u}_{i,\mu}(+\infty).
$$

\item Let $I:\C^N\rTo \C^N$ be the map defined by the
multiplication of the diagonal matrix $\diag(\xi^{k_1},
\dots,\xi^{k_N})$. If $W+\xi^{d+k_i}\sum_j
b_{i,\mu}\hat{u}_{i,\mu}$ has critical point $\hat{\kappa}^j$ and
corresponding critical value $\hat{\alpha}^j$, it is easy to show
that $(I(\hat{\kappa}^j), -\hat{\alpha}^j)$ is the critical point
and critical value of $W+\sum_j b_{i,\mu}\hat{u}_{i,\mu}$.

\item If the rigidifications over the two components are not
standard but arbitrary, we can still glue the perturbed Witten
equations. However the identification map $I$ is replaced by the
composition of $I$ and the inverse of the corresponding
rigidifications.
\end{enumerate}
\end{rem}

\begin{df} Sections $(u_1,\dots, u_t)$ of $\LL_1\times \LL_2\times\dots\times\LL_t$ on $\cC$
are said to be  solutions of the perturbed Witten equation
(\ref{witt-equa-noda}) if they satisfy the following conditions:
\begin{enumerate}

\item for each $j, u_{j,\nu}\in
L^2_{1,loc}(\cC_\nu\setminus\{z_1,\dots,z_{k_\nu}\},
\LL_j|_{\cC_\nu}),I_1\left(\overline{\frac{\partial W}{\partial
u_{j,\nu}}+\frac{\partial W_{0,\beta}}{\partial
u_{j,\nu}}}\right)\in
L^2_{loc}(\cC\setminus\{z_1,\dots,z_{k_\nu}\}, \LL_j\otimes
\Lambda^{0,1}|_{\cC_{\nu}})$, where $u_{j,\nu}$ is the component
of $u_j$ and $k_\nu$ is the number of marked points and nodal
points on $\cC_\nu$;

\item $(u_{1,\nu},\dots,u_{t,\nu})$ satisfy the perturbed
Witten-equation (\ref{witt-equa-noda1}) on $\cC_\nu$ almost
everywhere;

\item near each marked or nodal point, the integral
$$
\sum_j\int^\infty_0\int_{S^1} |\frac{\pat u_{j,\nu}}{\pat s}|^2
d\theta ds<\infty.
$$

\item If $p$ is a broad nodal point of $\cC$ between two
components $\cC_\nu$ and $\cC_\mu$, then
$$
\lim_{s_\nu\to +\infty} (u_{1,\nu}(s_\nu, \theta_\nu),\dots,
u_{t,\nu}(s_\nu, \theta_\nu))=\lim_{s_\mu\to +\infty}
(u_{1,\mu}(s_\mu, \theta_\mu),\dots, u_{t,\mu}(s_\mu,
\theta_\mu)).
$$
\end{enumerate}
\end{df}

\begin{rem} Using the notations from Remark \ref{rem-wequ-equal}, we
discuss the condition (4) in the above definition. The condition
(4) is given in section form. Assume that on the $\nu$-component,
$\tilde{u}_{i,\nu}$ satisfies Equation (\ref{equ-nu}) with
$\tilde{u}_{i,\nu}(+\infty)=\kappa^j_i$, where $\kappa^j,
\alpha^j$ are the $j$-th critical point and critical value of
$W+\sum b_{i,\nu} \tilde{u}_{i,\nu}$. Now the condition (4)
implies that the corresponding value in the $\mu$-coordinate
should satisfy
$$
\tilde{u}_{i,\mu}(+\infty)=\xi^{k_i} \kappa^j_i,
$$
for a local solution $\tilde{u}_{i,\mu}$ of Equation
(\ref{equ-mu}).
\end{rem}

On the component $\frkc_\nu$, the equation (\ref{witt-equa-noda1})
has a corresponding equation defined on the resolved curve
$\tilde{\cC}$:
\begin{equation}\label{witt-equa-reso}
\bpat_{\cC_\nu} g_{i,\nu}+I_1\left(\overline{\frac{\pat(
W+W_{0,\beta})}{\pat g_{i,\nu}}}\right)=0, \forall i=1,\dots,N,
\end{equation}
where $g_{i,\nu}\in \Omega(\tilde{\cC_\nu}, |\LL_i|_{\cC_\nu})$.
From  subsection \ref{subsub-witten-orbi}, we know that there is a
one-to-one correspondence between the solutions of these two
equations. In fact, let $p$ be a marked or nodal point on the
component $\frkc_\nu$, and take a uniformizing system
$(\Delta\times \C, G_p)$ of the orbifold line bundle $\LL_i$ near
$p$. If $(z, u_{i,\nu}(z))$ be the local solution of
(\ref{witt-equa-noda1}), then
$u_{i,\nu}(z)=g_{i,\nu}(z^m)(z^m)^{\Theta_i^{\gamma_p}}$. The
properties of the solutions of Equation (\ref{witt-equa-reso})
have been fully studied in Section 3. Hence we can easily get the
following theorems:

\begin{thm}\label{thm-solu-noda}

 The solutions of the perturbed Witten equation
(\ref{witt-equa-noda}) satisfy the following conclusions:
\begin{enumerate}
\item Interior estimate: for any ball $B_{2R}$ lying outside the
cylindrical neighborhood of any marked or nodal points, there
exists a constant $C$ depending only on $R, m\in {\mathbb Z},$ and
the metric in $B_{2R}$ such that
$$
||u_i||_{C^m(B_{2R})}\le C.
$$

\item Boundedness estimate on the cylinder: Let $p$ be a marked or
nodal point with cylindrical neighborhood $[0,\infty)\times S^1$.
Then for any compact set $K\subset [0, \infty)\times S^1$, there
is a constant depending only on $m\in \Z, K$, and the perturbation
parameters $b$ in $W_{0,\gamma_p}$ such that
$$
||u_i||_{C^m(K)}\le C.
$$
Here the derivatives are taken with respect to the cylindrical
coordinates $(s,\theta)$.

\item Asymptotic behaviors: Without loss of generality, we assume
that for $1\le i\le N_p$, the solutions $u_i$'s are sections of
broad line bundles, and for $N_p+1\le i\le N$, the solutions
$u_i$'s are sections of narrow line bundles. Then we have

\begin{itemize}

\item Let $W_{0,\gamma_p}(u_1,\dots, u_{N_l})=\sum_{k=1}^{N_l}
b_k u_k$ be any $W_{\gamma_p}$-regular polynomial at $p$. Then if
$\sum_{k=1}^{N_l} |b_k|$ is sufficiently small, there exist
constants $C, T,\delta>0$ such that for any $(s,\theta)\in
S^1\times [T,\infty)$,

\begin{align*}
&|u_i-\kappa_i|< C e^{-\delta s},
\text{for}\;i=1,\dots,N_p,\\
\end{align*}
where $(\kappa_1,\dots,\kappa_{N_p})$ is some critical point of
$W_{\gamma_p}+W_{0,\gamma_p}$.

\item If $N_p+1\le i\le t$, then there exist constants $C,
T,\delta>0$ such that for any $(s,\theta)\in S^1\times
[T,\infty)$,
\begin{equation}
||u_i||_{C^1([T,\infty))}\le C e^{-\delta s}.
\end{equation}

\end{itemize}

\item Continuity: Assume that $p$ is a nodal point connecting two
components $\frkc_\nu$ and $\frkc_\mu$; then
$u_{i,\nu}(p)=0=u_{i,\mu}(p)$ naturally hold for narrow
sections $u_i, N_p+1\le i\le N$. For broad sections $u_i,1\le
i\le N_p$, we have $u_{i,\nu}(p)=\lim_{s_\nu\to \infty}
u_{i,\nu}(s_\nu)=\kappa_i=\lim_{s_\mu\to \infty}
u_{i,\mu}(s_\mu)=u_{i,\mu}(p)$, which is required by the
definition of solutions.  Here
$\varkappa=(\kappa_1,\dots,\kappa_{N_p})$ is some critical point
of $W+W_{\gamma_p}$.

\end{enumerate}
\end{thm}

\begin{rem} The above conclusions hold for any solutions of the
perturbed Witten equations at any point $\frkc\in
\MMr_{g,k}(\bgamma)$. One may worry about the equations on
$\frkc_{y, {\boldsymbol \zeta}}$, which is the resolution curve of
$\frkc$. In this case, the perturbed Witten map has a perturbation
term appear in the interior of the glued curve. Now the interior
estimate on $\frkc_{y, {\boldsymbol \zeta}}$ is the combination of
the interior estimate and the bounded estimate of $\frkc$.
\end{rem}

From the description of Theorem \ref{thm-solu-noda}, we know that
any solutions of the perturbed Witten equation actually lie in the
space $C^\infty(\cC, \LL_1)\times \dots\times C^\infty(\cC,
\LL_N)$.

\subsubsection{Soliton space}

Let $\frkc=(\cC, p_1,\dots, p_k, \LL_{1},\dots,
\LL_{N},\varphi_{1},\dots, \varphi_{s},\psi_1,\dots,\psi_k)$ be
a rigidified semistable $W$-curve with $k$ marked points. The
$W$-structure induces a group element $\gamma_{p_i}\in G$ at
$p_i$. The action of this local group determines whether the line
bundles are broad or narrow at $p_i$, and hence determines
the choice of the polynomial $W_{\gamma_{p_i}}$ and the perturbed
polynomial $W_{0,\gamma_{p_i}}$. Choose the coefficients such that
each $W_{0,\gamma_{p_i}}$ is $W_{\gamma_{p_i}}$-regular and then
fix $W_{0,\gamma_{p_i}}$. Under the rigidification, (i.e.,
choosing the local bases around the marked point), we obtain
$\deg(W_{\gamma_{p_i}})-1$ critical points of the polynomial
$W_{0,\gamma_{p_i}}+W_{\gamma_{p_i}}$ in $\C^N$, which are all
non-degenerate. We use $\varkappa_i$ to denote the critical points
at the marked point $p_i$. Remember $\varkappa_i$ can take
$deg(W_{\gamma_{p_i}})-1$ values. Denote
$\bvkappa:=(\varkappa_1,\dots,\varkappa_k)$.

Now we study a special solution space related to a marked point
$p$. The data attached to $p$ is $\gamma\in G$, the image of the
generator of local group $G_p$ in $G$, and $\varkappa_p$, some
nondegenerate critical points of
$W_{0,\gamma_{p}}+W_{\gamma_{p}}$. Let $N_p$ be the number of
broad line bundles at $p$.  We give an order to the set of these
critical points $\kappa^1, \dots, \kappa^{\deg(W_{\gamma_p})-1}$
in $\C^{N_p}$ space such that
$\mbox{Re}((W_\gamma+W_{0,\gamma})(\kappa^i))\ge
\mbox{Re}((W_\gamma+W_{0,\gamma})(\kappa^j))$ if $i>j$.

Let $\LL_1,\dots,\LL_N$ be  flat orbifold line bundles defined on
the infinitly long cylinder $\R\times S^1$, where the monodromy
representation is given by
$$
\rho_p: 1\in \Z\cong \pi_1(\R\times S^1)\mapsto \gamma\in
G\subset U(1)^N.
$$
At the  point $-\infty$, the orbifold action on $\LL_1,\dots,
\LL_N$ is given by $\gamma^{-1}$ and at $\infty$ it is given by
$\gamma$. We have the perturbed Witten equation (in cylindrical
coordinates) defined on ${\mathbb R}\times S^1$:
\begin{equation}\label{soli-witt-equa}
\frac{\bpat u_i}{\pat \bar{\xi}}-2\overline{\frac{\pat
(W+W_{0,\gamma})}{\pat u_i}}=0.
\end{equation}

By the asymptotic analysis in Section \ref{sec:witten}, we know
that $\bu$ takes values $(\kappa^{\pm}, 0)$ at $\mp\infty$ (notice
the sign!), where $\kappa^{+}$ or $\kappa^-$ are two critical
points of $W_\gamma+W_{0,\gamma}$. The solution of this equation
is called {\bf the soliton solution} of type $\gamma_p$ connecting
$\kappa^{+}$ and $\kappa^-$ and is denoted by
$(\bu_{\kappa^+,\kappa^-},\gamma_p)$.

We have the following Witten lemma.

\begin{lm} If $p$ is a narrow point, then the related
soliton equation (\ref{soli-witt-equa}) at $p$ has only the zero
solution.
\end{lm}

\begin{proof}
This is a special case of the Witten lemma we proved in section
\ref{sec:witten}.
\end{proof}

By the above lemma, we only need to consider the soliton equation
related to broad marked points. Set $p$ as a broad marked point.
Due to the proof of Lemma \ref{inte-1}, we have the following lemma.

\begin{lm} Let $(\bu_{\kappa^+,\kappa^-},\gamma)$ be a soliton; then
we have
$$
(W_\gamma+W_{0,\gamma})(\kappa^-)-(W_\gamma+W_{0,\gamma})(\kappa^+)=2\int^{+\infty}_{-\infty}\int_{S^1}
\sum_i \left| \frac{\pat(W+W_{0,\gamma})}{\pat u_i} \right|^2.
$$
\end{lm}

\begin{crl} There is no soliton solution connecting
$\kappa^-$ and $\kappa^+$, if
$\mbox{Im}(W_\gamma+W_{0,\gamma})(\kappa^-)\neq
\mbox{Im}(W_\gamma+W_{0,\gamma})(\kappa^+)$, and  only the trivial
solution if  $\kappa^-=\kappa^+$.
\end{crl}

According to this corollary, the soliton solution connecting two
different critical points $\kappa_i$ and $\kappa_j$ exist, only if
$\mbox{Im}(W_\gamma+W_{0,\gamma})(\kappa^i)
=\mbox{Im}(W_\gamma+W_{0,\gamma})(\kappa^j)$.

\begin{crl} Let $p$ be a broad marked point on a $W$-curve. If
the perturbed polynomial $W_{0,\gamma_p}$ is strongly
$W_\gamma$-regular, then the related soliton equation has no
nontrivial solution.
\end{crl}

Now we consider a kind of special soliton solution which is
independent of the angle $\theta$. Assume as before that the first
$N_p$ components of $\bu$ are broad sections and the last
components are narrow sections. Since the narrow
line bundles are nontrivial bundles,  there is no nontrivial
section which is independent of $\theta$. Hence the equation
(\ref{soli-witt-equa}) becomes
\begin{equation}
\frac{\pat u_i}{\pat s}-2\overline{\frac{\pat
(W_\gamma+W_{0,\gamma})}{\pat u_i}}=0, \;s=1,\dots, N_p.
\end{equation}

This special solution is a called BPS soliton with respect to the
superpotential $W_\gamma+W_{0,\gamma}$ in  Landau-Ginzburg theory
in physics. By the above equation, we can easily get
$$
\pat_s (W_\gamma+W_{0,\gamma})(s)=\left|\frac{\pat
(W_\gamma+W_{0,\gamma})}{\pat u_i}\right|^2.
$$
This shows that the imaginary part of $(W_\gamma+W_{0,\gamma})(s)$
is invariant under the evolution of the flow, and the real part
increases monotonically.

\subsubsection*{Stable manifolds and vanishing cycles} So to count the
number of solitons connecting two critical points, e.g.,
$\kappa^1$ and $\kappa^2,$ we need to study the intersection
behavior of the stable manifold at $\kappa^2$ and the unstable
manifold at $\kappa^1$. Here we define the unstable manifold at
$\kappa^1$ to be the following set in $\C^{N_p}$:
$$
\C^u(\kappa^1):=\{(u_1,\dots, u_{N_p})\in \C^{N_p}|((u_1,\dots,
u_{N_p}))\cdot s\to \kappa^1, \;\text{as}\;s\to -\infty
\},
$$
where $(u_1,\dots, u_{N_p})\cdot s$ represents the flow line
going through $(u_1,\dots, u_{N_p})$ at time $s=0$. Similarly, we
can define the stable manifold of $\kappa^2$:
$$
\C^s(\kappa^2):=\{(u_1,\dots, u_{N_p})\in \C^{N_p}|((u_1,\dots,
u_{N_p}))\cdot s\to \kappa^2, \;\text{as}\;s\to +\infty
\}.
$$
It is wellknown that $\C^s(\kappa^2)$ and $\C^u(\kappa^1)$ are all
real   dimension $n$. They are  submanifolds of the
$(2n-1)$-dimensional subspace
$$
\{(u_1,\dots, u_{N_p})\in \C^{N_p}|
Im(W_\gamma+W_{0,\gamma})(u_1,\dots,u_{N_p})=Im(W_\gamma+W_{0,\gamma})(\kappa^1))\}.
$$
So for generic parameters $b_1,\dots, b_{N_p}$ such that
$\mbox{Im}(W_\gamma+W_{0,\gamma})(\kappa^1))=\mbox{Im}(W_\gamma+W_{0,\gamma})(\kappa^2)$,
the geometric orbits connecting $\kappa^1$ and $\kappa^2$ are
finitely many. Actually this can be given by the intersection
numbers of two vanishing cycles which represent the two critical
points respectively. In fact, take a point
$w=(W_\gamma+W_{0,\gamma})((u_1,\dots,u_{N_p})\cdot (s_0))$ on
the segment connecting
$\mbox{Im}(W_\gamma+W_{0,\gamma})(\kappa^1))$ and
$\mbox{Im}(W_\gamma+W_{0,\gamma})(\kappa^2))$; then the two
$(n-1)$-dimensional intersection submanifolds
$\Delta^1:=(W_\gamma+W_{0,\gamma})^{-1}(w)\cap \C^u(\kappa^1)$ and
$\Delta^2:=(W_\gamma+W_{0,\gamma})^{-1}(w)\cap \C^s(\kappa^2)$ are
just vanishing cycles representing $\kappa^1$ and $\kappa^2$. For
example, when $s_0\to -\infty$,
$(W_\gamma+W_{0,\gamma})^{-1}(w)\cap \C^u(\kappa^1)$ will shrink
to the critical point $\kappa^1$.

We have the well-known result from Picard-Lefschetz theory:

\begin{thm} The number of BPS solitons connecting $\kappa^1$ and
$\kappa^2$ is given by the intersetion number $\Delta^1\circ
\Delta^2$.
\end{thm}

The computation of the number of BPS soliton is an important part
of classical singularity theory.

\begin{df} Let $\{\kappa^1,\dots, \kappa^{\deg(W_\gamma)-1}\}$ be the set of
nondegenerate critical points of $W_\gamma+W_{0,\gamma}$. Let
$i_0<j_1<\dots< j_p<i_1$, define
$S_\gamma(\kappa^{i_0},\kappa^{j_1},\dots, \kappa^{j_p},
\kappa^{i_1})$ to be the space of broken soliton, $\{\bu_{i_0j_1},
\bu_{j_1j_2},\dots, \bu_{j_p i_1}\}$ , where $\bu_{i_0j_1}$ is
the soliton of type $\gamma$ connecting $\kappa^{i_0}$ and
$\kappa^{j_1}$, and so on. We denote by $S_\gamma(\kappa^{i_0},
\kappa^{i_1})$ the space of all kinds of solitons of type $\gamma$
(including broken solitons) from $\kappa^{i_0}$ to
$\kappa^{i_1}$.
\end{df}

Obviously $S_\gamma(\kappa^{i_0},\kappa^{j_1}, \kappa^{j_p},
\kappa^{i_1})$ is not empty if and only if
$\mbox{Im}(W_\gamma+W_{0,\gamma})(\kappa^{i_0}))=\mbox{Im}(W_\gamma+W_{0,\gamma})(\kappa^{i_1})$.

The group $\C$ acts on $S_\gamma(\kappa^i, \kappa^{i+1})$ by
$(s_0, \theta_0)\cdot\bu(\cdot,
\cdot)=\bu(\cdot+s_0,\cdot+\theta_0)$ and the group $\C^{p+1}$
also acts on $S_\gamma(\kappa^{i_0},\kappa^{j_1},\dots,
\kappa^{j_p}, \kappa^{i_1})$ in an obvious way.

If the space $S_\gamma(\kappa^{i_0},\kappa^{j_1},\dots,
\kappa^{j_p}, \kappa^{i_1})$ is not empty, the soliton
$\{\bu_{i_0j_1}, \bu_{j_1j_2},\dots, \bu_{j_p i_1}\}$ can be
viewed as the broken trajectory connecting the periodic solutions
$\kappa^{i_0}$ and $\kappa^{i_1}$ in some sense. So imitating the
construction of a trajectory space in symplectic geometry, one can
consider the following equivalence relation:
$$
\{\bu_{i_0j_1}, \bu_{j_1j_2},\dots, \bu_{j_p i_1}\}\sim
\{\bu'_{i_0j_1}, \bu'_{j_1j_2},\dots, \bu'_{j_p i_1}\}
$$
iff there exist real constants $s_1,\dots, s_{p+1}$ such that
$$
(s_1,\dots, s_{p+1})\cdot \{\bu_{i_0j_1}, \bu_{j_1j_2},\dots,
\bu_{j_p i_1}\}=\{\bu'_{i_0j_1}, \bu'_{j_1j_2},\dots, \bu'_{j_p
i_1}\}.
$$
Then we define the moduli space of "geometrical" solitons as
$\hat{S}_\gamma(\kappa^{i_0},\kappa^{j_1},\dots, \kappa^{j_p},
\kappa^{i_1})=S_\gamma(\kappa^{i_0},\kappa^{j_1},\dots,
\kappa^{j_p},\kappa^{i_1})/\sim$. There is an $S^1$ action on the
space $\hat{S}_\gamma(\kappa_i,\kappa_{i+1})$. The fixed points of
this space are just the BPS solitons. However, in our setting, we
will not care about the space
$\hat{S}_\gamma(\kappa_i,\kappa_{i+1})$, but the quotient space
$\bar{S}_\gamma(\kappa_i,\kappa_{i+1}):=\hat{S}_\gamma(\kappa_i,\kappa_{i+1})/S^1$.
 BPS solitons become the singularities of the $S^1$ action.

\subsubsection{Moduli space $\MMr_{g,k}(\bgamma,\varkappa)$}

\

Define the section $\bu:=(u_1,\dots, u_N)\in C^\infty(\cC,
\LL_1\times\dots \LL_N):=
 C^\infty(\cC, \LL_1)\times \dots\times C^\infty(\cC, \LL_N)$.

\begin{df} Let $\frkc=(\cC, p_1,\dots, p_k, \LL_{1},\dots,
\LL_{N},\varphi_{1},\dots, \varphi_{s},\psi_1,\dots,\psi_k)$ be
a rigidified stable $W$-curve, and $\bu$ be a solution of the
perturbed Witten equation on $\frkc$. Then the tuple $(\frkc,
\bu)$ is said to be a stable $W$-section.
\end{df}

\begin{df} The automorphism group of $(\frkc, \bu)$ is defined as
\begin{align*}
\aut(\frkc, \bu):=&\{\tau\in \aut(\frkc)|\tau(\bu)=\bu \}.
\end{align*}
\end{df}

Since $\aut(\frkc,\bu)$ is a subgroup of $\aut(\frkc)$, we have the
following lemma.
\begin{lm} The automorphism group $\aut(\frkc, \bu)$ is a finite
group if $\frkc$ is a $W$-stable curve.
\end{lm}


\begin{df} We say that a stable section $(\frkc, \bu)$ is of type
$(\bgamma, \bvkappa)$ if at each marked point $p_i$ of $\frkc$ the
generator of the local group $G_{p_i}$ is $\bgamma_{p_i}$ and the
section $\bu$ which is viewed as the section in $\C^N$ by the
standard rigidification at $p_i$ takes the given value
$\varkappa_i$ for $1\le i\le k$.
\end{df}

\begin{df} The moduli space of $\MMr_{g,k}(\bgamma, \bvkappa)$ is
the space consisting of  isomorphism classes of all  sections
$(\frkc, \bu)$ of type $(\bgamma, \bvkappa)$ under the action of the
automorphism group $\aut((\frkc, \bu))$. The subspace
$\MMr_{g,k}(\Gamma;\bgamma,\bvkappa)$ of $\MMr_{g,k}(\bgamma,
\bvkappa)$ is the space consisting of the elements in
$\MMr_{g,k}(\bgamma, \bvkappa)$ having dual graph $\Gamma$.
\end{df}

If the perturbed polynomial $W_{0,\gamma_{p_i}}$ at {\bf any broad
marked or nodal point} $p_i$ is strongly $W_{\gamma_i}$-regular,
{\bf {we call the moduli space $\MMr_{g,k}(\bgamma, \bvkappa)$
strongly regular}}. However, when $\MMr_{g,k}(\bgamma, \bvkappa)$ is
not strongly regular, then it is not compact with respect to the
Gromov convergence which will be defined later. The loss of
compactness is due to the existence of solitions. So we need to add
the corresponding limits to our moduli space.

\begin{df} Let $(\frkc, \bu)$ be a stable $W$-section of type $(\bgamma,
\bvkappa)$, where $\bgamma=(\gamma_1,\dots, \gamma_k)$ and $
\bvkappa=(\varkappa_1,\dots,\varkappa_k)$.  Take a marked point
$p_{i_0}$ of $\frkc$ such that
$\psi_{i_0}(\bu)(p_{i_0})=\varkappa_{i_0}^{j_0}$, where
$\psi_{i_0}$ is the rigidification at $p_{i_0}$ and
$\varkappa_{i_0}^{j_0}$ is the $j_0$th critical point of the
perturbed polynomial $W_{\gamma_{i_0}}+W_{\gamma_{i_0}}$. We
simply write it as $\bu(p_{i_0})=\varkappa_{i_0}^{j_0}$ if no
ambiguity can occur. Let $(\bu_{j_0,j_1},\gamma_{i_0})$ be a
soliton in $S_{\gamma_{i_0}}(\varkappa^{j_0}_{i_0},
\varkappa^{j_1}_{i_0})$. We define $\frkc\#_{p_{i_0}}(\R\times
S^1)$ to be the connected sum of the $W$-curve $\frkc$ and the
$W$-curve $\R \times S^1$ by identifying the marked point
$p_{i_0}$ in $\frkc$ and the $-\infty$ point of $\R\times S^1$.
Similarly, on the new $W$-curve, we can define the connected sum
$\bu\#_{p_{i_0}} \bu_{j_0,j_1}$ in a natural way. We call the pair
$(\frkc\#_{p_{i_0}}(\R\times S^1),\bu\#_{p_{i_0}} \bu_{j_0,j_1} )$
a {\bf soliton $W$-section}. If $(\frkc',\bu')$ is another stable
$W$-curve having a marked point $p_{i_1}$ labelled by
$\gamma_{i_1}=\gamma_{i_0}^{-1}$ and
$\bu'(p_{i_1})=\varkappa^{j_1}_{i_0}=\bu_{j_0,j_1}(+\infty)$, we
can construct another pair $(\frkc\#_{p_{i_0}}(\R\times
S^1)\#_{p_{i_1}}\frkc',\bu\#_{p_{i_0}}
\bu_{j_0,j_1}\#_{p_{i_1}}\bu' )$. This pair is also called a
soliton $W$-section. In the same way, one can continue to
construct new soliton $W$-sections if the  two glued soliton
$W$-sections satisfy the compatibility conditions at the gluing
point.
\end{df}

There is a natural group action on the soliton $W$-curve. For
example, assume that $(\frkc^1,\bu^1)\in
\MMr_{g_1,k_1,W}(\bgamma_1, \bvkappa_1)$, $(\frkc^2,\bu^2 )\in
\MMr_{g_2,k_2,W}(\bgamma_2, \bvkappa_2)$ and
$(\bu_{j_1,j_2},\gamma_{p_{i_0}})\in
S_{\gamma_{i_0}}(\varkappa^{j_1}_{i_0}, \varkappa^{j_2}_{i_0})$
satisfy the compatibility conditions at marked points $p_{i_0}$ on
$\frkc^1$ and $p_{i_1}$ on $\frkc^2$. Then we can get a soliton
$W$-section $(\frkc^1\#_{p_{i_0}}(\R\times
S^1)\#_{p_{i_1}}\frkc^2,
\bu^1\#_{p_{i_0}}\bu_{j_1,j_2}\#_{p_{i_1}}\bu^2)$. Let $g_i\in
\aut(\frkc^i, \bu^i), i=1,2$ and $(s_0,\theta_0)\in \C$; the
action of $(g_1,(s_0,\theta_0),g_2)$ of the above soliton
$W$-section is defined in an obvious way and we can define its
automorphism group. Obviously the $W$-curve
$\frkc^1\#_{p_{i_0}}(\R\times S^1)\#_{p_{i_1}}\frkc^2$ has genus
$g_1+g_2$, $k_1+k_2-2$ marked points  and may not be a stable
$W$-curve. The $k_1+k_2-2$ marked points are labelled by the set
$\widetilde{\gamma}=:\{\bgamma_1,\bgamma_2\}-\{\gamma_{p_{i_0}},\gamma_{p_{i_1}}\}$
and
$\widetilde{\varkappa}=:\{\bvkappa_1,\bvkappa_2\}-\{\varkappa_{p_{i_0}},\varkappa_{p_{i_1}}\}$.
We say this soliton $W$-section is of type $(\widetilde{\gamma},
\widetilde{\varkappa})$.

\begin{df} The moduli space of $\MMrs_{g,k}(\bgamma, \bvkappa)$\glossary{MMrsgkW@$\MMrs_{g,k}(\bgamma,\bvkappa)$ & The space of soliton $W$-sections of type $(\bgamma,\bvkappa)$}
 is
the space consisting of the isomorphism classes of all  soliton
$W$-sections of type $(\bgamma, \bvkappa)$ under the action of its
automorphism group.
\end{df}

\subsubsection*{Combinatorial types of soliton $W$-sections}

We have the relation $\MMr_{g,k}(\bgamma)=\coprod_i
\MMr_{g,k}(\Gamma_i)$, where the summation is taken over all
possible $G$-decorated dual graphs. Note that each half-edge $\tau$
of a $G$-decorated dual graph is decorated with a group element
$\gamma_\tau\in G$. The number of these $\Gamma$ is finite. Each
$\MMr_{g,k}(\Gamma)$ is a multiple covering of $\MM_{g,k}(\Gamma)$
under the stable map $\str$ and the degree of $\str$ is finite
because of different rigidified $W$-structures. Define
$Comb(g,k,W;\bgamma)$ to be the combinatorial types of $W$-curves in
$\MMr_{g,k}(\bgamma)$. Then this set is a finite set. The partial
order $\succ$ in $Comb(g,k)$, the set of combinatorial types in
$\MM_{g,k}$, induces a more refined partial order $\succ$ (involving
the information of rigidified $W$-structures) in
$\MMr_{g,k}(\bgamma)$ such that it becomes a stratified space.

Now we consider the combinatorial types in $\MMrs_{g,k}(\bgamma,
\bvkappa)$. There are new combinatory types in $\MMrs_{g,k}(\bgamma,
\bvkappa)$ compared to those in $\MMr_{g,k}(\bgamma)$ because of the
existence of soliton components. We always assume that there exists
only one group element $\hat{\gamma}\in G$ such that the perturbed
polynomial $W_{\hat{\gamma}}+W_{\hat{\gamma}, 0}$ has only  two
critical points $\kappa^\pm$ with the property
$\mbox{Im}(W_{\hat{\gamma}}+W_{\hat{\gamma},
0})(\kappa^+)=\mbox{Im}(W_{\hat{\gamma}}+W_{\hat{\gamma},
0})(\kappa^-)$. So if a dual graph $\Gamma\in Comb(g,k,W;\bgamma)$
has a half-edge $\tau$ decorated with $\hat{\gamma}$, then we
replace the edge $\tau$ by the dual graph of a soliton of type
$\hat{\gamma}$. We denote the new graph by $\Gamma^s(\tau)$ and call
it a soliton graph. Define $\Gamma\succ\Gamma^s(\tau)$. Furthermore,
if $\Gamma$ has multiple edges decorated by $\hat{\gamma}$, then we
can replace each edge by the dual graph of a soliton of type
$\hat{\gamma}$ to get a new graph $\Gamma^s$. We define the order
$\Gamma\succ\Gamma^s$. Since the number of half-edges $\tau$
possibly decorated by $\gamma_\tau$ is finite, the number of soliton
graphs is finite. Let $Comb(g,k,W;\bgamma, \bvkappa)$ be the set of
combinatorial types of dual graphs in $\MMrs_{g,k}(\bgamma,
\bvkappa)$. We have a partial order relation $\succ$ in
$Comb(g,k,W;\bgamma,\bvkappa)$. This set gives a stratification to
the moduli space $\MMrs_{g,k}(\bgamma, \bvkappa)$. We will use the
partial order to glue our Kuranishi neighborhoods to obtain a
Kuranishi structure.

Now we begin our construction of the topology in our moduli space.
We first define the topology in $\MMr_{g,k}(\bgamma, \bvkappa)$ when
it is strongly regular.

Let $(\frkc^n, \bu^n)$ be a sequence of isomorphism classes in
$\MMr_{g,k}(\bgamma,\bvkappa)$. Suppose that $\frkc^n \in
\MMr_{g,k}(\bgamma)$ converges to $\frkc \in \MMr_{g,k}(\Gamma)$.
This means that there exists a sequence of $\frkc^n_\Gamma\in
\MMr_{g,k}(\Gamma)$ such that $\frkc^n_\Gamma \to \frkc$ in
$\MMr_{g,k}(\Gamma)$ and there is a sequence of gluing parameters
${\boldsymbol \zeta}^n=(\zeta^n_1,\dots, \zeta^n_{\#E(\Gamma)})$
such that $\frkc^n=(\frkc^n_\Gamma)_{0,{\boldsymbol \zeta}^n}$ and
${\boldsymbol \zeta}^n\to \infty$.

\begin{df}[Neck region] Suppose that $\frkc^n \in
\MMr_{g,k}(\bgamma)$ converges to $\frkc \in \MMr_{g,k}(\Gamma)$.
Let $z$ be any nodal point of the underlying curve $\cC$ of $\frkc$
connecting two components $\cC_\nu$ and $\cC_\mu$. Given $\hat{T}\ge
10\bar{T}_0$ (where $\bar{T}_0$ comes from the definition of cut-off
function $\beta$), the neck region of the underlying curve $\cC$ of
$\frkc$ at the nodal point $z$ is defined as
$$
N_{n,z}(\hat{T}):= ([\hat{T}, {T}^n_z]\times S^1)_\nu \cup
([\hat{T}, {T}^n_z]\times S^1)_\mu.
$$
Here $\zeta^n_z=(s^n_z,\theta_z)$ and $s^n_z=2T^n_z$.
\end{df}

Note that $\cC^n$ is obtained from $\cC^n_\Gamma$ by gluing the
corresponding domains $[\frac{1}{2}T^n_z, \frac{3}{2}T^n_z]$ in
two components connected at $z$. In view of the definition,we can
identify the regions $\cC^n-\cup_{z\in E(\Gamma)}
N_{n,z}(\hat{T})$ in $\cC^n$ with  $\cC^n_\Gamma-\cup_{z\in
E(\Gamma)} N_{n,z}(\hat{T})$ in $\cC_\Gamma^n$ and with
$\cC-\cup_{z\in E(\Gamma)} N_{n,z}(\hat{T})$ in $\cC$. So the
section $\bu^n$ can be viewed as defined on $\frkc^n_\Gamma$, and
by the pull-back of the deformation map we also can assume that
the $ \bu^n$ are defined on the region $\cC-\cup_{z\in E(\Gamma)}
N_{n,z}(\hat{T})$ in $\cC$.

\begin{df}[Gromov convergence]\label{conv-grom-1} We say that $(\frkc^n,\bu^n)\to (\frkc,\bu)$ in
$\MMr_{g,k}(\bgamma,\bvkappa)$ as $n\to \infty$ if
\begin{enumerate}
\item for each $\hat{T}\ge 10\bar{T}_0$, $\bu^n$ converges in the
$C^\infty$ topology to $\bu$ on $\cC-\cup_{z\in E(\Gamma)}
N_{n,z}(\hat{T})$.

\item $\lim_{\hat{T}\to
\infty}\overline{\lim}_{n\to\infty} \mbox{Diam} (\bu^n(\cup_z
N_{n,z}(\hat{T})))=0$.

\end{enumerate}
\end{df}

If $\MMr_{g,k}(\bgamma, \bvkappa)$ is strongly regular, then
$\MMrs_{g,k}(\bgamma, \bvkappa)=\MMr_{g,k}(\bgamma, \bvkappa),$ and
we can prove below that $\MMr_{g,k}(\bgamma, \bvkappa)$ is compact
with respect to Gromov convergence. Now we consider a special case.
We assume that the perturbation polynomials
$W_{\gamma_{p_i}}+W_{0,\gamma_{p_i}}$ are all strongly
$W_{\gamma_{p_i}}$-regular except at the marked (or nodal point)
$p_0$, the perturbation polynomial
$W_{\gamma_{p_0}}+W_{0,\gamma_{p_0}}$ has  only  two critical points
$\varkappa^{j_1}, \varkappa^{j_2}$ such that $\mbox{Im}
(W_{\gamma_{p_0}}+W_{0,\gamma_{p_0}})(\varkappa^{j_1})=\mbox{Im}
(W_{\gamma_{p_0}}+W_{0,\gamma_{p_0}})(\varkappa^{j_2})$. We want to
define the convergence of $\MMrs_{g,k}(\bgamma, \bvkappa)$ in this
case. Without loss of generality, we can assume that $p_0$ is a
broad nodal point.

\begin{df} Take $(\frkc^n,\bu^n), (\frkc,\bu)$ in $\MMrs_{g,k}(\bgamma,
\bvkappa)$. If for sufficiently large $n$, $p_0$ is not any of the
nodal points of $(\frkc^n,\bu^n)$, then we define that
$(\frkc^n,\bu^n)\to (\frkc,\bu)$ in
$\MMrs_{g,k}(\bgamma,\bvkappa)$ as in the Definition
\ref{conv-grom-1}. Now assume that
$(\frkc^n,\bu^n)=(\frkc^{n,1}\#_{p_0}\frkc^{n,2},
\bu^{n,1}\#_{p_0}\bu^{n,2})$, where
$\bu^{n,1}(p_0)=\varkappa^{j_2}=\bu^{n,2}(p_0)$, and
$(\frkc,\bu)=(\frkc^1\#_{p_{0,1}}(\R\times S^1)\#_{p_{0,2}}\frkc^2,
\bu^{1}\#_{p_{0,1}}\#\bu_{j_1,j_2}\#_{p_{0,2}}\bu^{2})$, where
$\bu_{j_1,j_2}(-\infty)=\varkappa^{j_1}=\bu^1(p_{0,1}),
\bu_{j_1,j_2}(+\infty)=\varkappa^{j_2}=\bu^2(p_{0,2})$. We say that
$(\frkc^n,\bu^n)\to (\frkc,\bu)$ in
$\MMrs_{g,k}(\bgamma,\bvkappa)$ if the following conditions hold:
\begin{enumerate}
\item $(\frkc^{n,2},\bu^{n,2})\to (\frkc^2,\bu^2)$ as in
Definition \ref{conv-grom-1}.

\item Let $N(\hat{T})$ be the neck region on
$\frkc^1\#_{p_{0,1}}(\R\times S^1)$ for any $\hat{T}\ge
10\bar{T}$. Then $(\frkc^{n,1}-([\hat{T},\infty)\times S^1),
\bu^{n,1})$ converges to $(\frkc^{1}-([\hat{T},\infty)\times S^1),
\bu^{1})$ in the sense of Definition \ref{conv-grom-1}.

\item There is a sequence of positive numbers $s_n\to
\infty$ such that $\tilde{\bu}^n:=\bu^n(\cdot+s_n, \cdot)$
converges in $\C^\infty_{local}$ to $\bu_{j_1,j_2}$ on $\R\times
S^1$.

\item $\lim_{\hat{T}\to
\infty}\overline{\lim}_{n\to\infty} \mbox{Diam} (\bu^n([\hat{T},
-\hat{T}+s_n]\times S^1))=0$.

\end{enumerate}

\end{df}

\begin{rem} Since the polynomial $W_\gamma$ is only determined by the group action of $\gamma\in G$,
the perturbation polynomial $W_\gamma+W_{0,\gamma}$ can be chosen
to depend only on $\gamma$ but not on the special marked or nodal
points. Hence if there are more than two broad marked or nodal
points  labelled by the same group action and satisfying the same
requirement of critical values as in the above definition of
convergence, we can still define the Gromov convergence in this
case.
\end{rem}

\subsection{The Gromov compactness theorem }\

\
\begin{thm}\label{thm-gromov-comp}[Gromov Compactness theorem] 
Assume that the perturbation polynomials
$W_{\gamma}+W_{0,\gamma}$ are all strongly $W_{\gamma}$-regular
except possibly for only one group element $\gamma_0\in G$. Assume also that the
perturbation polynomial $W_{\gamma_{0}}+W_{0,\gamma_{0}}$ has only
 two critical points $\varkappa^{j_1}, \varkappa^{j_2}$ such that
$\mbox{Im}
(W_{\gamma_{p_0}}+W_{0,\gamma_{p_0}})(\varkappa^{j_1})=\mbox{Im}
(W_{\gamma_{p_0}}+W_{0,\gamma_{p_0}})(\varkappa^{j_2})$. Then the
moduli space $\MMrs_{g,k}(\bgamma,\bvkappa)$ is a compact Hausdorff
space for any $\bgamma, \bvkappa$. In particular, if
$\MMr_{g,k}(\bgamma,\bvkappa)$ is strongly regular, then
$\MMr_{g,k}(\bgamma,\bvkappa)$ is a compact Hausdorff space.
\end{thm}

\begin{proof} Let $(\frkc^n,\bu^n)\in
\MMrs_{g,k}(\bgamma,\bvkappa)$; we need only  find a convergent
subsequence. Since $\frkc^n \in \MMr_{g,k}(\bgamma)$ and
$\MMr_{g,k}(\bgamma)$ is compact, we can assume that $\frkc^n$ has
the same combinatorial type for each $n$ and converges to $\frkc \in
\MMr_{g,k}(\Gamma)$. By Theorem \ref{thm-solu-noda}, for each
$\hat{T}$, $\bu^n$ has a subsequence (we always use the same
notation if we take subsequences) such that it converges to a
$C^\infty$ section $\bu$ on $\cC-\cup_{z\in E(\Gamma)}
N_{n,z}(\hat{T})$. When restricting to each component of $\frkc$,
$\bu$ is a solution of the perturbed Witten equation.

Now we consider the convergence. There are several cases:

\begin{enumerate}

\item  $p$ is a nodal point of $\frkc$ and all $\frkc^n$
connecting two components $\frkc_\nu$ and $\frkc_\mu$ such that
the group element $\gamma_p\neq \gamma_0$. We assume that
$\bu_\nu^n$ converges to $\bu_\nu$ in the interior part of
$\frkc_\nu$.

Now we claim in this case that for any $\varepsilon>0$, there
exist a constant $\hat{T}$ and a critical point $\varkappa_\nu$
such that for any sufficiently large $n$ and $(s,\theta)\in
[\hat{T}, \infty)\times S^1$, the following holds:
\begin{equation}\label{comp-proof-1}
|\bu^n_\nu(s,\theta)-\varkappa_\nu|<\varepsilon,
\end{equation}
and
\begin{equation}\label{comp-proof-2}
|\bu^n_\mu(s,\theta)-\varkappa_\mu|<\varepsilon.
\end{equation}

If the conclusion is not true, then there exists a positive number
$\varepsilon_0$ and a sequence $s_n\to\infty, \theta_n\in S^1$
such that for any critical points $\varkappa$ of
$W_{\gamma_p}+W_{0,\gamma_{p}}$ the following holds
$$
|\bu^n_\nu(s_n,\theta_n)-\varkappa|\ge \varepsilon_0.
$$
Let $\tilde{\bu}^n_\nu(s,\theta)=\bu^n_\nu(s_n+s,\theta_n)$; then
$\tilde{\bu}^n_\nu$ is the solution of the perturbed Witten
equation defined on $[-s_n, \infty)$ satisfying the condition
$|\tilde{\bu}^n_\nu(0,\theta_n)-\varkappa|\ge \varepsilon_0$.

By Theorem \ref{thm-solu-noda}, $\tilde{\bu}^n_\nu$ will
$C^\infty_{local}$ converge to a solution $\bu$ such that for some
$\theta_0\in S^1$,
$$
|\bu(0,\theta_0)-\varkappa|\ge \varepsilon_0,
$$
for any critical point $\varkappa$. Hence $\bu$ is a nontrivial
solution of the perturbed Witten equation connecting two different
critical points. However, we know this is not true because
$W_{0,\gamma_{p}}$ is strongly $W_{\gamma_p}$-regular.

Now by (\ref{comp-proof-1}) and (\ref{comp-proof-2}), we know that
$$
\bu^n_\nu(+\infty)=\varkappa_\nu=\bu^n_\mu(+\infty)=\varkappa_\mu
$$
and
$$
\bu_\nu(+\infty)=\bu_\mu(+\infty)=\varkappa_\nu.
$$
Hence we have proved in this case that $(\frkc^n, \bu^n)$
converges to $(\frkc, \bu)$.

\item  $p$ is a marked point such that $\gamma_p\neq \gamma_0$.
Proof of compactness in this case is the same as in (1) except we
need only consider one component.

\item  $p$ is a nodal point of $\frkc$ and all $\frkc^n$
connecting two components $\frkc_\nu$ and $\frkc_\mu$ satisfy the
relation $\gamma_p=\gamma_0$.

Assume that $\bu^n_\nu(\infty)=\varkappa=\bu^n_\mu(\infty)$.

There are also some cases:

\begin{enumerate}
\item  $\varkappa\neq \varkappa^{j_1},\varkappa^{j_2}$. We know
that in the interior part of each component, $\bu_\nu^n$ and
$\bu_\mu^n$ converge to $\bu_\nu$ and $\bu_\mu$ respectively. In
this case we can show in the same way as (1) that
$\bu_\nu(+\infty)=\bu_\mu(+\infty)$, and $(\frkc^n,\bu^n)\to
(\frkc,\bu)$ in $\WW_{g,k}^s(\bgamma,\bvkappa)$.

\item $\varkappa= \varkappa^{j_2}(\;\text{or}\;\varkappa^{j_1})$.

Assume the sequence $(\frkc^n,\bu^n)$ satisfies the following property: for any
$\varepsilon>0$ there exists a constant $\hat{T}$ and a
critical point $\varkappa_\nu$ such that for any sufficiently
large $n$ and $(s,\theta)\in [\hat{T}, \infty)\times S^1$, the
following holds:
\begin{equation}
|\bu^n_\nu(s,\theta)-\varkappa_\nu|<\varepsilon,
\end{equation}
and
\begin{equation}
|\bu^n_\mu(s,\theta)-\varkappa_\mu|<\varepsilon.
\end{equation}
Then we know that $\varkappa_\nu=\varkappa_\mu=\varkappa$,
$\bu_\nu(\infty)=\bu_\mu(\infty)=\varkappa$ and $(\frkc^n,\bu^n)$
converges to $(\frkc, \bu)$ in $\MMrs_{g,k}(\bgamma,\bvkappa)$.

If the condition described above does not hold, then the argument of (1)
shows that there exists a sequence $s_n\to\infty$ such that the
reparametrized sequence
$\tilde{\bu}^n_\nu(s,\theta):=\bu^n_\nu(s_n+s,\theta_n)$
$C^\infty_{local}$ converges to a nontrivial solution
$\bu_{j_1,j_2}$ on $\R\times S^1$ connecting the uniquely possible
two critical points $\varkappa^{j_1}$ and $\varkappa^{j_2}$. In
fact, it is easy to see that $\bu_\mu(\infty)=\varkappa^{j_2}$ and
$\bu_\nu(+\infty)=\varkappa^{j_1}$, because there is no other kind
of trajectory connecting two critical points. Hence we obtain a
soliton $W$-section $(\frkc_\nu\#_{p_{0,1}}(\R\times
S^1)\#_{p_{0,2}}\frkc_\mu,\bu_\nu\#_{p_{0,1}}\bu_{j_1,j_2}\#_{p_{0,2}}\bu_\mu)$.
Now it is easy to show that $(\frkc^n, \bu^n)\to
(\frkc_\nu\#_{p_{0,1}}(\R\times
S^1)\#_{p_{0,2}}\frkc_\mu,\bu_\nu\#_{p_{0,1}}\bu_{j_1,j_2}\#_{p_{0,2}}\bu_\mu)$
in Gromov convergence.
\end{enumerate}
\item In the convergence of the $W$-curve, one component can
degenerate into a nodal curve. So without loss of generality we
assume that $\frkc^n$ is an irreducible curve and
$\frkc^n\to \frkc_\nu\#_{p_0}\frkc_\mu$ in
$\MMr_{g,k}(\bgamma)$. If $\gamma_{p_0}\neq \gamma_0$, then one can
show as before that $(\frkc^n,\bu^n)$ will converges to a stable
$W$-section $(\frkc_\nu\#_{p_0}\frkc_\mu, \bu_\nu\#_{p_0}\bu_\mu)$.
If $\gamma_{p_0}=\gamma_0$, then $(\frkc^n,\bu^n)$ may converge to a
stable $W$-section or a soliton $W$-section in Gromov convergence.

\end{enumerate}
In summary, we have shown that $\MMrs_{g,k}(\bgamma,\bvkappa)$ is a
compact space in the Gromov topology. This is also a Hausdorff
space, since the limit $W$-section $(\frkc, \bu)$ is uniquely
determined by the approximating sequence.
\end{proof}

\section{Construction of the virtual cycle}\label{sec:construct}

The aim of this section is to construct a Kuranishi structure of the
moduli space $\MMrs_{g,k}(\bgamma,\bvkappa)$, the space of soliton
$W$-section. We know that if this moduli space is a strongly
perturbed space, then there is no soliton $W$-section and
$\MMrs_{g,k}(\bgamma,\bvkappa)=\MMr_{g,k}(\bgamma,\bvkappa)$. As the
first step, we construct the Fredholm theory of the perturbed Witten
map and do some preparations.

\subsection{Fredholm theory and the implicit function theorem}\

\

Note that we have defined for any
$\frkc=(\cC,p_1,\dots,p_k,\LL_1,\dots,\LL_N,\varphi_1,\dots,\varphi_s,\psi_1,\dots,\psi_k)\in
\MMr_{g,k}(\bgamma)$ the corresponding Witten map:
$$
WI_\frkc: C^\infty(\cC,\LL_1)\times \dots \times C^\infty(\cC,
\LL_N)\rTo C^\infty(\cC, \LL_1\otimes\Lambda^{0,1})\times
\dots \times C^\infty(\cC,\LL_N\otimes \Lambda^{0,1}),
$$
which has the following form:
$$
WI_\frkc(\bu)=(\bpat_\cC u_1+\tilde{I}_1\left(\overline{\frac{\pat
(W+W_{0,\beta})}{\pat u_1}}\right),\dots, \bpat_\cC
u_N+\tilde{I}_1\left(\overline{\frac{\pat (W+W_{0,\beta})}{\pat
u_N}}\right)).
$$
Here the perturbation term $W_{0,\beta}$ has the form
$\varpi(\zeta)\beta_i W_{0,\gamma}$ which is determined by the
combinatorial type of $\cC$ and the group element $\gamma$ and the
cut-off section $\beta_i$.

Define
$$
C^\infty_0(\cC,\LL_1\times\dots\LL_N):=\{\bu=(u_1,\dots,u_N)\in
C^\infty(\cC,\LL_1)\times \dots \times C^\infty(\cC, \LL_N)|
u_{j,\nu}(p_\nu)=u_{j,\mu}(p_\mu)=0,\;\text{if}\;\pi_\nu(p_\nu)=\pi_\mu(p_\mu)\},
$$
and
\begin{align}
&L^p_1(\cC,\LL_1\times\dots\times\LL_N)=\oplus_\nu L^p_1(\cC_\nu,
\LL_1\times\dots\times \LL_N)\\
&L^p(\cC, \LL_i\otimes \Lambda^{0,1})=\oplus_\nu L^p(\cC_\nu,
\LL_i\otimes\Lambda^{0,1}), i=1,\dots,N.
\end{align}
The metric used to define the $L^p$ norm is the cylindrical metric
near each marked or nodal point. We will always set $p>2$ in our
discussion.

We have the linearized operator $D_{\frkc,\bu}WI$ of $WI_\frkc$ at
$\bu$:
\begin{align}
&D_{\frkc,\bu}WI(\bphi):=D_{\frkc,\bu}WI(\phi_1,\dots,\phi_N):=\nonumber\\
&\left(\bpat_\cC
\phi_1+\sum_j\tilde{I}_1\left(\overline{\frac{\pat^2
(W+W_{0,\beta})}{\pat u_1\pat u_j}\phi_j}\right),\dots, \bpat_\cC
\phi_N+\sum_j\tilde{I}_1\left(\overline{\frac{\pat^2
(W+W_{0,\beta})}{\pat u_N\pat u_j}\phi_j}\right)\right).
\end{align}

$D_{\frkc,\bu}WI$ is a map from
$C^\infty_0(\cC,\LL_1\times\dots\times\LL_N)$ to $(\oplus_\nu
C^\infty(\cC_\nu,\LL_1))\times \dots \times (\oplus_\nu
C^\infty(\cC_\nu, \LL_N))$ or from
$L^p_1(\cC,\LL_1\times\dots\times\LL_N)$ to $L^p(\cC,
\LL_1\otimes \Lambda^{0,1})\times \dots L^p(\cC, \LL_N\otimes
\Lambda^{0,1})$.

When restricted to the cylindrical neighborhood, the linearized
operator has the local form

\begin{align}
&D_{\frkc,\bu}WI(\bphi):=D_{\frkc,\bu}WI(\phi_1,\cdot,\phi_N):=\nonumber\\
&\left(\bpat_\xi \phi_1-2\sum_j\overline{\frac{\pat^2
(W+W_{0,\beta})}{\pat u_1\pat u_j}\phi_j},\dots, \bpat_\xi
\phi_N-2\sum_j\overline{\frac{\pat^2 (W+W_{0,\beta})}{\pat u_N\pat
u_j}\phi_j}\right),
\end{align}
where $\bu, \phi_i's$ are local sections or can be understood as
functions satisfying the twisted periodic conditions.

\begin{thm}\label{ind-witt-oper} Let $(\frkc,\bu)\in \MMrs_{g,k,w}(\bgamma,\bvkappa)$ and assume that
$\frkc$ is connected. Then its linearized operator
$D_{\frkc,\bu}WI:
L^p_1(\cC,\LL_1\times\dots\times\LL_N)\rTo L^p(\cC,
\LL_1\otimes \Lambda^{0,1})\times \dots \times L^p(\cC,
\LL_N\otimes \Lambda^{0,1})$ is a real linear Fredholm operator of
index $2\hat{c}_W(1-g)-\sum_\tau
2\iota(\gamma_\tau)-\sum_{\tau=1}^k N_{\gamma_\tau}$, where
$\hat{c}_W=\sum_i(1- 2q_i)$,
$\iota(\gamma_\tau)=\sum_{i}(\Theta_i^{\gamma_\tau}-q_i)$ and
$N_{\gamma_\tau}=\dim \C^N_{\gamma_\tau}$ ( if
$\C^N_{\gamma_\tau}=\{0\}$, we set $N_{\gamma_\tau}=0$).
\end{thm}

\begin{proof} Let $p_\tau$ be a marked point. Define the limit matrix of $A(\tau)$ as:
$A_{ij}(\tau)=\lim_{s\to
+\infty}-2\overline{\frac{\pat^2(W+W_{0,\beta})}{\pat u_i\pat
u_j}}$. We define a new real linear operator $D$ such that in the
cylindrical neighborhood $([1,\infty)\times S^1)_{p_\tau}$, $D$
has the form
$$
D\phi=\left(\bpat \phi_1+\sum_j\overline{A_{1j}\phi_j}, \dots,
\bpat \phi_N+\sum_j\overline{A_{Nj}\phi_j}\right).
$$
Note the facts that by Theorem \ref{thm-solu-noda} $\bu$ decays
exponentially to zero, so $D$ is a compact perturbation of
$D_{\frkc,\bu}WI$. To show that $D$ is a Fredholm operator and
compute its index, we need some preparation.\

\subsubsection*{Index of $\bpat^\infty$ in a half tube}

Let $(([0,\infty)\times S^1)\times \C, F)$ be a bundle pair, i.e.,
$\C$ is a trivial line bundle and $F$ is the Lagrangian subbundle
given by the fiber $F_{e^{i\theta}}=e^{i\frac{k}{d}\theta}\R$.
Notice that we take the coordinate $\xi=s+i\theta$ in the half
tube $[0,\infty)\times S^1$. Define the weighted Sobolev spaces of
sections with weight $\delta>0$ and $p>2$ as follows:
\begin{align*}
&L^{p,\delta}_{1,b}:=\{\phi|, \int (|\nabla \phi|^p+|\phi|^p)
e^{\delta s}ds d\theta<\infty, \phi(0,\theta)\in
F_{e^{i\theta}}\}\\
&L^{p,\delta}:=\{\phi d\bar{\xi}|, \int |\phi|^p e^{\delta s}ds
d\theta<\infty\}.
\end{align*}
We have the standard Cauchy-Riemann operator:
$\bpat^\infty:=\bpat: L^{p,\delta}_{1,b}\rTo L^{p,\delta}$.

\begin{lm}\label{ind-inftube} $\bpat^\infty$ is a Fredholm operator. Let $k_{\delta,p}=[\frac{\delta}{p}]$. Then
$$
\begin{cases}
\ind\bpat^\infty=0& if |k| \le k_{\delta,p}\\
\ind\bpat^\infty=-2k+2k_{\delta,p}+1& if k>k_{\delta,p}\\
\ind\bpat^\infty=-2k-2k_{\delta,p}-1& if k<-k_{\delta,p}.
\end{cases}
$$
\end{lm}

\begin{proof} This lemma is possibly not new, since there is a
lot of discussion in the literature (see \cite{MS}, \cite{Do} and
so on). But they are seldom in our required form. We follow Joel
Robbin's method in the appendix of \cite{MS}.\

\

\noindent Step 1: We lift this problem to the $d$-sheeted covering
$\widetilde{([0,\infty)\times S^1)}$ and get the bundle pair
$(\tilde{\C}, \tilde{F})$. Meanwhile, we have the lifting operator
$\tilde{\bpat}$. It is easy to see that $\ind\bpat=\ind
\tilde{\bpat}$. Hence in the following  we always assume that
$d=1$.\

\

\noindent Step 2: Prove $\coker \bpat=\{\psi d\bar{z}| \pat_s \psi
-i \pat_\theta \psi=0, \psi(0,\theta)\in i F,\; \psi\in
L^{p,\delta} \;\text{and is a smooth function}\}$.

In fact, taking any $\phi\in L^{p,\delta}_{1,b}, \psi\in
L^{p,\delta}\subset L^2$, we have
\begin{align*}
&\langle \psi d\bar{z}, \bpat \phi\rangle=\mbox{Re}(
\int_0^\infty\int_{S^1} \bar{\psi}(\pat_s \phi+i\pat_\theta
\phi)ds
d\theta)\\
&=\mbox{Re}( \int_0^\infty\int_{S^1} (\overline{\pat_s
\psi-i\pat_\theta\psi})\phi ds d\theta)+\mbox{Re}(\int_{S^1}
\overline{\psi(0,\theta)}\phi(0,\theta)).
\end{align*}
Therefore we have the conclusion.\

\

\noindent Step 3: If $k\le k_{\delta, p}$, then $\bpat$ is
injective. If $k> k_{\delta, p}$, then $\dim\ker
\bpat=2k-2k_{\delta, p}-1$.

The solution of $\bpat \phi=0$ has the Laurent series
$$
\phi(s,\theta)=\sum_n a_n e^{n(s+i\theta)}.
$$
The coefficient is given by
$$
a_n=\frac{1}{2\pi}\int_{S^1}\phi(0,\theta) e^{-in\theta}d\theta.
$$
By the equality $\overline{\phi(0,\theta)}=\phi(0, \theta)
e^{-2ik\theta}$ we have
\begin{align*}
&\overline{a_n}=\frac{1}{2\pi}\int_{S^1} \overline{\phi(0,\theta)}
e^{i n\theta}d\theta\\
&=\frac{1}{2\pi}\int_{S^1} \phi(0,\theta)e^{-i(2k-n)\theta}\\
&=a_{2k-n}.
\end{align*}
Now by integrability we know that $a_n=0, \forall n\ge
-k_{\delta,p}$. So $\phi\equiv 0$ if $k\ge -k_{\delta,p}$, and
this shows that $\bpat$ is an injection. If $k<-k_{\delta,p}$ we
have
$$
\phi(s,\theta)=a_{2k+k_{\delta,p}+1}e^{(2k+k_{\delta,p}+1)(s+i\theta)}+\dots+
a_{-k_{\delta,p}-1} e^{-(k_{\delta,p}+1)(s+i\theta)},
$$
and so $\dim\ker\bpat=-2k-2k_{\delta,p}-1$.\

\

\noindent Step 4. If $k\le k_{\delta, p}$, then $\bpat$ is
surjective. If $k>k_{\delta,p}$, then $\dim\coker
\bpat=2k-2k_{\delta,p}-1$.

Let $\psi\in \coker \bpat$; then $\bpat(i\overline{\psi})=0$ and
$i\overline{\psi(0,\theta)}\in e^{-ik\theta}\R$. Now using the
result of Step 3, we have the conclusion.

\end{proof}

\subsubsection*{Index of $\bpat^D$ in a disc}

We can also consider the bundle pair $(D\times \C, F)$ on a flat
disc $D$, where the Lagrangian subbundle $F:=e^{ik\theta}\R$. We
have the Fredholm map
$$
\bpat^D:=\bpat: L^p_1(D,\C)\rTo L^p(D, \Lambda^{0,1}).
$$
The following lemma is proved in the appendix of \cite{MS}:

\begin{lm}\label{ind-disc}
$\ind \bpat^D=1+2k$.
\end{lm}

\subsubsection*{Index of $\bpat^{orb}$ in a disc with origin an orbifold point}

Assume $(D\times \C, \Z_d)$ to be a chart of orbifold line bundle
$\C$ at the origin, where the group action is given by
$$
e^{2\pi i/d}(z, w)=(e^{2\pi i/d}z, e^{2\pi k i /d}w).
$$
Since, near the boundary $\pat D$, the orbifold structure gives a
line bundle structure, we can associate a Lagrangian subbundle $F$
to $\pat D$ defined as $F_{e^{\i \theta}}:=e^{i k\theta/d}\R$. We
can define the Sobolev spaces of the orbifold sections:
$L^{p,orb}_1(D, \C)$ and $L^{p,orb}$. We have the Fredholm
operator $\bpat^{orb}:=\bpat:L^{p,orb}_1(D, \C)\rTo
L^{p,orb}$.

\begin{lm}\label{ind-orb}
$$
\begin{cases}
\ind \bpat^{orb}=1& if k=0\\
\ind \bpat^{orb}=2k-1& if k\neq 0.\\
\end{cases}
$$
\end{lm}

\begin{proof} If $k=0$, then the orbifold line bundle is a trivial
non-orbifold bundle, and the index can be obtained from Lemma
\ref{ind-disc}. If $k\neq 0$, then the orbifold line bundle is
nontrivial and the group action forces all continuous sections to
vanish at the origin. Hence the first and the last coefficients in
a Taylor expansion similar to what we computed in the case of the
half-infinite long tube will disappear. So, compared to the case
of the disc, the index will drop by real dimension 2.
\end{proof}

Firstly we assume that $\cC$ has only one component, and let
$\cC_b=\cC\setminus \cup_{\tau=1}^k ([1,+\infty]\times
S^1)_{p_\tau}$, where the $p_\tau$ are marked points. Let
$S^1_\tau=(\{1\}\times S^1)_{p_\tau}$ be the oriented circle with
the induced orientation from $([1,+\infty)\times S^1)_{p_\tau}$.
So $\pat \cC_b=\coprod_\tau (-S^1_\tau)$, where $-$ means the
anti-orientation. Notice that the orbifold structure of $\LL_i$
near the marked point $p_\tau$ induces a monodromy representation
of the line bundle $\LL_i$:
$$
\rho_{\tau,i}: \pi_1(([0,\infty]\times S^1)_{p_\tau})\rTo
G\subset U(1),
$$
which is given by $\rho_{\tau,i}(1)=\Theta^{\gamma_\tau}_i$. Let
$\R^\tau_i=-e^{2\pi \sqrt{-1}\Theta^{\gamma_\tau}_i\theta}\R$ be
the Lagrangian subbundle on $S^1_\tau$. Then we obtain the bundle
pairs $(\cC_b\times \LL_i, \overline{\R^1_i}\times \dots \times
\overline{\R^1_i})$ and $(([1,+\infty)\times S^1)_{p_\tau}\times
\LL_i,\R^\tau_i)$. Define
$$
C^\infty_b(\cC_b, \LL_i):=\{s\in C^\infty_b(\cC_b,
\LL_i)|s|_{S^1_\tau}\in \R^\tau_i\}.
$$
Let $L^p_{1,b}(\cC_b, \LL_i)$ be the closure of $C^\infty_b(\cC_b,
\LL_i)$ under the $L^p_1$ norm and $L^p(\cC_b, \LL_i\times
\Lambda^{0,1})$ be the closure of $ C^\infty(\cC_b, \LL_i\times
\Lambda^{0,1})$. Then it is well-known that
$$
\bpat_i(\cC_b): L^p_{1,b}(\cC_b, \LL_i)\rTo L^p(\cC_b,
\LL_i\times \Lambda^{0,1})
$$
is a self-dual Fredholm operator in this boundary value problem.

Similarly we can define another boundary value problem on
$[1,\infty)\times S^1$ such that
$$
\bpat_i^\infty(\tau): L^p_{1,b}(([1,\infty)\times S^1)_{p_\tau},
\LL_i)\rTo L^p(([1,\infty)\times S^1)_{p_\tau}, \LL_i\times
\Lambda^{0,1})
$$
is also a self-dual Fredholm operator.

Since $D$ is a real linear Cauchy-Riemann type operator, we can
similarly define the following maps:
$$
D(\cC_b):=D: L^p_{1,b}(\cC_b, \LL_1\times \dots \times
\LL_N)\rTo L^p(\cC_b, \times_i(\LL_i\times \Lambda^{0,1}))
$$
and
$$
D^\infty(\tau): L^p_{1,b}(([1,\infty)\times S^1)_{p_\tau},
\LL_1\times \dots \times \LL_t)\rTo L^p(([1,\infty)\times
S^1)_{p_\tau}, \times_i(\LL_i\times \Lambda^{0,1})).
$$

We have the index gluing formula which can be proved in exactly
the same way as in the appendix of \cite{MS}.

\begin{lm}\label{ind-prooflm-1} Assume that $D(\cC_b), D_i^\infty(\tau),
\tau=1,\dots, k$ are defined as above; then we have
$$
\ind(D)=\ind{D(\cC_b)}+\sum_{\tau=1}^k D^\infty(\tau).
$$
\end{lm}

Since $D(\cC_b)$ is a compact perturbation of
$\bpat_1(\cC_b)\oplus\dots\oplus\bpat_i(\cC_b)$, we have
\begin{equation}\label{ind-inner}
\ind D(\cC_b)=\sum_i \ind \bpat_i(\cC_b).
\end{equation}

Notice that the operator
$$
D^\infty(\tau)(\phi)=\left(\bpat
\phi_1+\sum_j\overline{A_{1j}\phi_j}, \dots, \bpat
\phi_N+\sum_j\overline{A_{Nj}\phi_j}\right)
$$
actually split into two parts: the broad part and the
narrow
part. If we write $$ A=\begin{pmatrix} A_R & 0\\
0& A_N\end{pmatrix},
$$
then
$$
D^\infty(\tau)=\begin{pmatrix} D_R(\tau) & 0\\
0& D_N(\tau)\end{pmatrix}= \begin{pmatrix} \bpat\cdot+\overline{A_R\cdot} & 0\\
0& \bpat\cdot+\overline{A_N\cdot}\end{pmatrix},
$$
where $A_R$ is a non-degenerate complex symmetric matrix with
small eigenvalues whose absolute value is less than $1$, and $A_N$
is a small complex symmetric matrix.\

Since the dual operator of $D_R(\tau)$ is
$-\pat\cdot+\overline{A_R\cdot}$, the sections in the kernel and
cokernel are all exponentially decaying and the exponents depend
on the absolute values of the eigenvalues of $A_R$. There exists a
small $\delta>0$ such that the index of $D_R$ in the usual Sobolev
space is the same to the index in the weighted Sobolev space
$L^{p,\delta}_{1,b}$. Since this Sobolev space 
has the compact
embedding theorem,  we can drop  the $0$-term without changing the
index. We have
\begin{equation}
\ind D_R(\tau)=\sum_{i:\Theta^{\gamma_\tau}_i=0}\ind
\bpat^\infty_i(\tau)=0,
\end{equation}
where the last equality is from Lemma \ref{ind-inftube}.

$D_N(\tau)$ is a Fredholm operator in the usual Sobolev space of
sections, since each narrow bundle $\LL_i$ is a nontrivial
line bundle. Actually we can define a desingularization
$$
\eta(\Theta^{\gamma_\tau}_i): \phi(s, e^{i\theta}) \in
C^\infty(([1,\infty)\times S^1)_{p_\tau}, \LL_i)\rTo e^{-i
\Theta^{\gamma_\tau}_i \theta} \phi(s, e^{i\theta})\in
C^\infty(([1,\infty)\times S^1)_{p_\tau}, \C).
$$
This map provides the isomorphisms of Banach spaces from
$L^p_{1,b}(([1,\infty)\times S^1)_{p_\tau}, \LL_i)$ to
$L^p_{1,b}(([1,\infty)\times S^1)_{p_\tau}, \C)$ and
$L^p(([1,\infty)\times S^1)_{p_\tau}, \LL_i)$ to
$L^p(([1,\infty)\times S^1)_{p_\tau}, \C)$. Now the operator
$\eta(\Theta^{\gamma_\tau}_i)\circ D_N(\tau)\circ
\eta(\Theta^{\gamma_\tau}_i)^{-1}
=\bpat\cdot+\diag(\Theta^{\gamma_\tau}_i\dots)\cdot+\overline{A_N\cdot}.$
Hence we know that $D_N(\tau)$ is also a Fredholm operator and
$$
\ind D_N(\tau)=\ind \left(\eta(\Theta^{\gamma_\tau}_i)\circ
D_N(\tau)\circ \eta(\Theta^{\gamma_\tau}_i)^{-1}\right).
$$
For the same reason, the sections in the kernel and cokernel of
$\eta(\Theta^{\gamma_\tau}_i)\circ D_N(\tau)\circ
\eta(\Theta^{\gamma_\tau}_i)^{-1}$ are also exponentially
decaying; we can compute the index in a weighted Sobolev space and
then drop the $0$-term. By Lemma \ref{ind-inftube} and Lemma
\ref{ind-orb}, we have
\begin{equation}
\ind D_N(\tau)=\sum_{i:\Theta^{\gamma_\tau}_i\neq 0}\ind
\bpat^\infty_i(\tau)=\sum_{i:\Theta^{\gamma_\tau}_i\neq 0}\ind
\bpat^{orb}_i(\tau).
\end{equation}

Notice that in Lemma \ref{ind-inftube} we take the coordinate
$\xi=s+i\theta$, but in Lemma \ref{ind-orb} we take the coordinate
$z=e^{-\xi}$. For a different choice of coordinates, there is a
minus sign difference in the corresponding Lagrangian subbundles.

Therefore we can obtain
\begin{equation}\label{ind-bdry-sum}
\ind D^\infty(\tau)=\ind D_R(\tau)+\ind
D_N(\tau)=\sum_{i:\Theta^{\gamma_\tau}_i\neq 0}\ind
\bpat^{orb}_i(\tau).
\end{equation}

By Lemma \ref{ind-prooflm-1}, (\ref{ind-inner}) and
(\ref{ind-bdry-sum}), there holds
\begin{equation}
\ind D=\sum_i \ind \bpat_i(\cC_b)+\sum_{\tau=1}^k
\sum_{i:\Theta^{\gamma_\tau}_i\neq 0}\ind \bpat^{orb}_i(\tau).
\end{equation}

Furthermore, by Lemma \ref{ind-orb} and the index gluing formula,
we have
\begin{align}\label{ind-proof-4}
&\ind D=\sum_i \ind \bpat_i(\cC_b)+\sum_{\tau=1}^k \sum_{i}\ind
\bpat^{orb}_i(\tau)-\sum_{\tau=1}^k \sum_{i:\Theta^{\gamma_\tau}_i=0} 1\nonumber\\
&=\sum_i \ind \bpat_i-\sum_{\tau=1}^k
\sum_{i:\Theta^{\gamma_\tau}_i =0} 1.
\end{align}

Therefore now the index computation is changed to the index
computation of $\bpat_i$ on the closed orbifold with orbifold line
bundle $\LL_i$.

The following lemma is from Proposition 4.2.2. of \cite{CR1}.

\begin{lm}\label{ind-prooflm-3}
\begin{equation}
\ind \bpat_i=\ind_{\R} \bpat_i=2C_1(|\LL_i|)(cl(\cC))+2(1-g),
\end{equation}
where $cl(\cC)$ means the closed orbifold.
\end{lm}

Combining this Lemma (\ref{ind-proof-4}) and the following degree
equality which was obtained before:
$$
C_1(|\LL_i|)(cl(\cC))=\deg(\LL_i)=q_i(2g-2+k)-\sum_{\tau=1}^k
\Theta^{\gamma_\tau}_i,
$$
we have
\begin{align}\label{ind-proof-5}
\ind D=&2\sum_i \left( q_i(2g-2+k)-\sum_{\tau=1}^k
\Theta^{\gamma_\tau}_i\right)+2\sum_i(1-g)-\sum_{\tau=1}^k
\sum_{i:\Theta^{\gamma_\tau}_i =0} 1\nonumber\\
=&2(1-g)(\sum_i(1- 2q_i))-2\sum_\tau (\sum_{i}(
\Theta^{\gamma_\tau}_i-q_i))-\sum_\tau N_{\gamma_\tau}\nonumber\\
=&2\hat{c}_W(1-g)-\sum_\tau 2\iota(\gamma_\tau)-\sum_{\tau=1}^k
N_{\gamma_\tau},
\end{align}
where $\hat{c}_W=\sum_i(1- 2q_i)$ and
$\iota(\gamma_\tau)=\sum_{i}( \Theta_i^{\gamma_\tau}-q_i)$.

If $\frkc$ has more than one connected component, i.e., nodal
points appear, than we use the   relation
$$
\iota(\gamma_\tau)+\iota(\gamma_\tau^{-1})+N_{\gamma_\tau}=\hat{c}_W
$$
to obtain the general result. We have  finished the proof of
Theorem \ref{ind-witt-oper}.
\end{proof}

It is easy to get the following conclusion:
\begin{crl} Let $(\bu_{j_1,j_2},\gamma)\in
S_\gamma(\kappa^{j_1},\kappa^{j_2})$. Then the linearized operator
$D_{\bu_{j_1,j_2}}(WI)$ is a real linear Fredholm operator of
index $0$ on $\R\times S^1$.
\end{crl}

In the following, we introduce the implicit function theorem in
our required form, which is called the "Kuranishi model" in
\cite{CR2}. Here we have done a small modification.

Suppose $F$ is a $C^1$ Fredholm map from a neighborhood of $0$ in
the Banach space $X$ to a neighborhood of $0$ in the Banach space
$Y$. Then $DF(0)$ has a finite-dimensional kernel and cokernel.
Write $F(x)=F(0)+DF(0)x+G(x)$.

Assume that $E$ is a finite-dimensional space such that $Y=E+\im
DF(0)$, and let $V=\{x\in X|DF(0)\cdot x\in E\}$. Then $V$ is a
finite-dimensional subspace in $X$, and $\dim V-\dim E=\ind
DF(0)$. There exist subspaces $V'$ and $E'$ in $X$ and $Y$
respectively such that $X=V\oplus V'$ and $Y=E\oplus E'$. Let
$P_E:Y\rTo E'$ and $P_V:X\rTo V'$ be the projection
map. The operator $DF(0): V'\rTo E'$ is invertible. We
denote its inverse by $Q$.

\begin{lm}\label{lm:model}

Consider the ball $U_{2r}$ in $X$ of radius $2r$ such that $x\in
U_{2r}$ satisfies the condition $\forall x_1, x_2\in U_{2r}$,
$$
||P_E\cdot G(x_1)-P_E\cdot G(x_2)||\le
C(||x_1||+||x_2||)||x_1-x_2||.
$$
Let $0<r<\frac{1}{8C||Q||}$. If $||P_E\cdot F(0)||\le \epsilon r$,
where $\epsilon<\frac{1}{2||Q||}$, then for any $v\in V\cap U_r$,
there is a unique $v'(v)\in V'\cap U_r$ such that $F(v+v'(v))\in
E$, and $\Psi: v\mapsto v+v'(v)$ is continuous and one to one.
On the other hand, for any $x\in U_{r/||I-P_V||}$ such that
$F(x)\in E$, there is a unique $v\in V\cap U_r$ such that
$x=v+v'(v)$. In particular, let $s: V\cap
\Psi^{-1}(U_{r/||I-P_V||})\rTo E$ be defined by
$s(v)=F(v+v'(v))$; then $s$ is continuous and $s^{-1}(0)$ is
homeomorphic to the zero set $F^{-1}(0)\cap U_{r/||I-P_V||}$ by
$v\mapsto v+v'(v)$. $(U_r, E, s, \Psi)$ is called the
Kuranishi model of $F$ at $0$.

\end{lm}

\

\begin{proof}

The proof is similar to the proof of Lemma 3.2.1.in \cite{CR2}, if
we put $B_v(x)=-Q\cdot P(F(0)+G(v+x))$ there.

\end{proof}

\begin{rem}

In the construction of the local charts, we will see that $F(x)$
represents the nonlinear Witten map and $F(0)$ represents the
approximating map.
\end{rem}

\subsection{Multisections and Kuranishi theory}\label{subsec:MK}

\

This section provides the machinery, multisections and Kuranishi
theory, to produce the virtual cycles. All of the contexts we
summarize below can be found in \cite{FO}. We cite the context
here only for the convenience of the reader.

\subsubsection*{Multisection} For a space $Z$, let $\mathscr{S}^k(Z):=Z^k/\mathscr{S}_k$
be the $k$-th symmetric power of $Z$, where $\mathscr{S}_k$ is the
permutation group of order $k!$. If $Z$ is an orbifold, then
$\mathscr{S}^k(Z)$ is also an orbifold.

\begin{df} Let $(V\times \R^k, \Gamma, pr)$ be a local model of
smooth orbibundle of rank $k$ over $(V,\Gamma)$ and $l$ be a
positive integer. An \emph{$l$-multisection} of $(V\times \R^k,
\Gamma, pr)$ is a continuous map $s:V\rTo
\mathscr{S}^l(\R^k)$ which is $\Gamma$-equivariant.
\end{df}

There is a canonical map $tm_{l'}: \mathscr{S}^l(Z)\rTo
\mathscr{S}^{ll'}(Z)$ for each $l,l'$ defined by
$$
tm_{l'}[x_1,\dots,x_l]=[{x_1,\dots,x_1},\dots,{x_l,\dots,x_l}],
$$
where we write $x_i$ $l'$ times.

It is easy to see if $s$ is a $l$-multisection, then $tm_{l'}\circ
s$ is an $l'l$-multisection and the restriction of $s$ is still an
$l$-multisection.

\begin{df} Let $pr: E\rTo X$ be an orbibundle. A multisection
is an isomorphism class of the following objects $(\{(V_i\times
\R^k, \Gamma_i,pr)\}, \{s_i\})$ such that
\begin{enumerate}
\item[(1)] $\{(V_i\times \R^k, \Gamma_i,pr)\}$ is a family of
charts of $E$ such that $\cup_i V_i/\Gamma_i=X$. \item[(2)] $s_i$
is an $l_i$-multisection of $(V_i\times \R^k, \Gamma_i,pr)$.
\item[(3)] In the overlap $V_i/\Gamma_i\cap V_j/\Gamma_j$, the
section $tm_{l_j}\circ s_i$ equals  $tm_{l_i}\circ s_j$ under the
transition map.
\end{enumerate}

We say two multisections $$(\{(V'_i\times \R^k, \Gamma'_i,pr)\},
\{s'_i\})$$ and $$(\{(V_i\times \R^k, \Gamma_i,pr)\}, \{s_i\})$$
are equivalent if the sections $tm_{l_j}\circ s'_i$ and
$tm_{l_i}\circ s_j$ satisfy the compatibility conditions on the
overlaps.
\end{df}
For a locally liftable multisection $(\{(V_i\times \R^k,
\Gamma_i,pr)\}, \{s_i\})$, $s_{i_j}$ is called a branch of
$s_{i}=[s_{i_j},1\le j\le l_i]$. The sum of an $l$-multisection
$s^{(1)}$ and $l'$-multisection $s^{(2)}$ is an $l'l$-multisection
$s^{(1)}+s^{(2)}$ and its branches consist of all the possible
sums of branches of $s^{(1)}$ and $s^{(2)}$. We can also naturally
define the multiplication of a function with a multisection. For
an open set $\Omega\subset X$, we can define the space
$C^0_m(\Omega,E)$ of continuous liftable multisections. It is a
$C^0(\Omega)$-module. Similarly, one can define the
$C^k$-differentiable space $C^k_m(\Omega,E)$.

Using the concept of multisection, Fukaya and Ono have constructed
the transversality approximation theorems for orbifold
sections (see Theorem 1 and Lemma 3.14 in \cite{FO}).

\subsubsection*{Kuranishi theory}
Let $X$ be a compact, metrizable topological space.

\begin{df} Let $V$ be an open subset of $X$. A Kuranishi or virtual
neighborhood of $V$ is a system $(U,E,G, s, \Psi)$ where

\begin{enumerate}
\item $\tilde{U}=U /G$ is an orbifold, and $E\rTo U$ is a
$G$-equivariant bundle.

\item $s$  is a $G$-equivariant continuous section of $E$. \item
$\Psi$ is a homeomorphism from $s^{-1}(0)$ to $V$ in $X$.
\end{enumerate}
We call $E$ the {\em obstruction bundle} and $s$ the {\em
Kuranishi map}. We say  $(U,E,G, s, \Psi)$ is a Kuranishi
neighborhood of a point $p\in X$ if $p$ has a neighborhood $V$
carrying a Kuranishi neighborhood.

\end{df}

\begin{df}\label{kura-stru}
Let $(U_i,E_i,G_i,s_i,\Psi_i)$ be a Kuranishi neighborhood of
$V_i$ and $f_{21}: V_1\rTo V_2$ be an open embedding. A
morphism
$$\{\phi\}: (U_1, E_1, G_1, s_1,\Psi_1)\rTo (U_2, E_2, G_2,
s_2,\Psi_2)$$ covering $f$ is a family of open embeddings
$$\phi_{21}: U_1\rTo U_2, \hat{\phi}_{21}:
E_1\rTo E_2, \lambda_{21}: G_1\rTo G_2, \Phi_{21}:
\phi_{21}^*TU_2/ TU_1\rTo \phi_{21}^*E_2/E_1
$$ (called injections) such that
\begin{enumerate}
\item $\phi_{21}, \hat{\phi}_{21}$ are $\lambda_{21}$-equivariant
and commute with bundle projection.

\item $\lambda_{21}$ induces an isomorphism from $\ker(G_1)$ to
$\ker(G_2)$, where $\ker(G)$ is the subgroup acting trivially.

\item $s_2\phi_{21}=\hat{\phi}_{21}s_1$ and $\phi_{21}$ covers
$f_{21}: V_1\rTo V_2$; $\;\Psi_2\phi_{21}=\Psi_1$.

\item If $g\phi_{21}(U_1)\cap \phi_{21}(U_1)\neq \emptyset$ for
some $g\in G_2$, then $g$ is in the image of $\lambda_{21}$.

\item $G_2$ acts on the set $\{\phi_{21}\}$ transitively, where
$g(\phi_{21}, \hat{\phi}_{21}, \lambda_{21})=(g\phi_{21},
g\hat{\phi}_{21}, g\lambda_{21} g^{-1})$.

\item $\Phi_{21}$ is an $G$-equivariant bundle isomorphism.
\end{enumerate}
\end{df}

\begin{df}\label{df-kran-stru}
A Kuranishi structure of dimension $n$ on $X$ is an open cover
$\V$ of $X$ such that
\begin{enumerate}
\item Each $V\in \V$ has a Kuranishi neighborhood $(U, E, G,
s,\Psi)$ such that $\dim U-\dim E=n$.

\item If $V_2\subset V_1$, the inclusion map $i_{12}:
V_2\rTo V_1$ is covered by a morphism between their
Kuranishi neighborhoods.

\item For any $x\in V_1\cap V_2, V_1, V_2\in \V$, there is a
$V_3\in \V$ such that $x\in V_3\subset V_1\cap V_2$.

\item The composition of injections is an injection.
\end{enumerate}
\end{df}

We can also define morphisms between two spaces $X$ and $Y$
carrying Kuranishi structures. However the morphism we prefer is a
special map from $X$ which carries a Kuranishi structure to $Y$
which is an orbifold. The following definition is given in
Definition 6.6 of \cite{FO}.

\begin{df} Let $X$ be a space carrying a Kuranishi structure $\V$ and
$Y$ be an orbifold. $f:X\rTo Y$ is called a strongly
continuous map if $f$ is a family of continuous maps $\{f_{V}\}$
for $V\in \V$ such that $f_{V_2}\circ \phi_{V_2,V_1}=f_{V_1}$. $f$
is said to have maximal rank if for each $V\in \V$, $f_V$ is
smooth and $\forall p\in V$ we have $\rk d f_{V}|_p=\min \{\dim V,
\dim Y\}$.
\end{df}

\begin{rem} The concept of Kuranishi structure was introduced in symplectic geometry by
Fukaya and Ono (\cite{FO}) to define the Gromov-Witten invariants
and prove the Arnold conjecture in general symplectic manifolds.
The reason is that the moduli space of  stable maps is not an
orbifold in general but it can still carry a Kuranishi structure.
The existence of this structure is sufficient for defining the
virtual cycle.

Fukaya-Ono used the concept of a germ to define the Kuranishi
structure; then they proved that there is a good coordinate
system, i.e., a finite covering of Kuranishi neighborhoods. Here
we define the Kuranishi structure in a different but equivalent
style. In the definition of morphism, we included the map $\Phi$
as our data and required $\Phi: \phi^*TU_2/ TU_1\rTo
\phi^*E_2/E_1$ to be an isomorphism. In definition
\ref{df-kran-stru}, $\Phi$ is required to satisfy the injection
composition rule. If we assume that $U_3\subset U_2\subset U_1$,
this is equivalent to the following commutative diagram:
$$
\begin{CD}
0@>>>\frac{\phi^*_{23}TU_2}{TU_3}@>>>\frac{\phi^*_{13}TU_1}{TU_3}@>>>\phi^*_{23}(\frac{\phi^*_{12}TU_1}{TU_2})@>>>0\\
@. @VV{\Phi_{23}}V @VV{\Phi_{13}}V @VV{\Phi_{12}}V\\
0@>>>\frac{E_2}{E_3}@>>>\frac{E_1}{E_3}@>>>\frac{E_1}{E_2}@>>>0.
\end{CD}
$$
This is just the definition of ``tangent bundle" in \cite{FO}. In
our definition, we require the existence of tangent bundle to be a
part of the definition of Kuranishi structure.
\end{rem}

To treat the orientation problem of the moduli problem, we follow
Fukaya-Ono's construction of the ``bundle system."

\begin{df} Assume that the space $X$ has a Kuranishi structure
$\{(U, E, G, s,\Psi)\}$. We change these data into a new
composition $\{(U, F_1, F_2, G, \Psi^F)\}$ (forget the information
about $E,s,\Psi$) such that they satisfy the following conditions:

\begin{enumerate}
\item For each open set $V_q$ carrying a Kuranishi neighborhood,
$F_{1,q}$ and $F_{2,q}$ are orbifold bundles over $U_q/G$. \item
For any embedding $\phi_{pq}: U_q\rTo U_p$ covering
$f_{pq}:V_q\rTo V_p$, there are embeddings of orbibundles:
$\Psi^F_{i,pq}: F_{i,q}\rTo \phi^*_{pq}F_{i,p}, i=1,2$, and
an isomorphism
$$
\Psi^F_{pq}: \frac{\phi^*_{pq}F_{1,p}}{F_{1,q}}\rTo
\frac{\phi^*_{pq}F_{2,p}}{F_{2,q}}.
$$
\item Let $V_r\subset V_q\cap V_p$ be three open sets in $X$ all
of which carry the Kuranishi neighborhoods, then we have the
commutative diagram:
$$
\begin{CD}
0@>>>\frac{\phi^*_{qr}F_{1,q}}{F_{1,r}}@>>>\frac{\phi^*_{pr}F_{1,p}}{F_{1,r}}@>>>\phi^*_{qr}(\frac{\phi^*_{pq}F_{1,p}}
{F_{1,q}})@>>>0\\
@. @VV{\Psi^F_{qr}}V @VV{\Psi^F_{pr}}V @VV{\Psi^F_{pq}}V\\
0@>>>\frac{\phi^*_{qr}F_{2,q}}{F_{2,r}}@>>>\frac{\phi^*_{pr}F_{2,p}}{F_{2,r}}@>>>\phi^*_{qr}(\frac{\phi^*_{pq}F_{2,p}}
{F_{2,q}})@>>>0.
\end{CD}
$$
\end{enumerate}
$(F_1,F_2, \Psi^F)$ is called a bundle system related to the
Kuranishi structure $\{(U, E, G, s,\Psi)\}$ of $X$.
\end{df}

We say two bundle systems $(F_1,F_2, \Psi^F)$ and $(F'_1,F'_2,
\Psi^{F'})$ are isomorphic if there exist isomorphisms between
$F_1$ and $F'_1$, $F_2$ and $F'_2$ in each chart of $X$ such that
they commute with the morphisms $\Psi^F$'s and $\Psi^{F'}$'s.

\begin{ex} $(TU, E, \Phi)$ is an intrinsic bundle system related
to the Kuranishi structure  $\{(U, E, G, s, \Psi)\}$.
\end{ex}

\begin{df} A bundle system is said to be orientable if for any point $p\in X$,
$F_{1,p}$ and $F_{2,p}$ are orientable as orbifold bundles on a
small Kuranishi neighborhood $U_p$ of $p$ and if $U_q\subset U_p$,
then
$$
\Psi_{pq}: \frac{\phi^*_{pq}F_{1,p}} {F_{1,q}}\rTo
\frac{\phi^*_{pq}F_{2,p}} {F_{2,q}}
$$
is orientation preserving.
\end{df}

Fukaya-Ono \cite{FO} then defined the K-theory of the bundle
system related to a Kuranishi structure of a space $X$. They
defined the concepts of oriented bundle system, complex bundle
system, stably oriented and stably almost complex. They showed
that a bundle system is stably oriented iff it is oriented. Hence
if one proves that a bundle system is stably almost complex, then
it is stably orientable and so an orientable system.

Once the definitions of Kuranishi structure and orientation are
constructed, Fukaya-Ono have the following theorem:

\begin{thm}\label{thm-vir-cycl} Let $X$ be a topological space carrying an $n$-dimensional
orientable Kuranishi structure $\{(U, E, G, s,\Psi)\}$. Then the
section $s$ can be perturbed to a multisection $\tilde{s}$ which
is transverse to the zero section, and the zero set
$\tilde{s}^{-1}(0)$ is an oriented cycle. The cobordism class
$[\tilde{s}^{-1}(0)]$ of this cycle is independent of the choice
of the multisection $\tilde{s}$. Furthermore, if $Y$ is an
orbifold and there is a strongly continuous map $f: X\rTo
Y$, then $f_*[\tilde{s}^{-1}(0)]$ is a homology class in
$H_n(Y;\Q)$.
\end{thm}

\subsubsection*{Kuranishi structure with boundary}\

Let $X$ be a compact metric space. We can also define an
$n$-dimensional Kuranishi structure with boundary on $X$. We only
need a minor change in Definition \ref{kura-stru}. We modify the
definition of Kuranishi neighborhood $(U, E, G, s,\Psi)$ and don't
change the other conditions. $U$ is required to be a $G$-invariant
open neighborhood of $0$ in $\R^n$ or $\R^n\times [0,\infty)$ and
$G$ is required to be a finite group with a linear representation
to $\R^n$ or $\R^{n-1}$ respectively. A point $p$ is said to be a
boundary point of $X$ if there is a chart $(U_p^{n-1}\times
[0,\varepsilon_p), E_p, G_p,s_p,\Psi_p)$, where $U_p^{n-1}$ is a
small neighborhood of $0$ in $\R^{n-1}$, such that $p\in
U_p^{n-1}/G$. Let $\pat X$ be the boundary of $X$, which consists
of all the boundary points of $X$. Obviously $\pat X$ is a space
with $(n-1)$-dimensional Kuranishi structure.

Similarly, we can define the notions of bundle systems,
orientation, etc. on $X$ carrying a Kuranishi structure with
boundary.

\subsection{Construction of the local chart of an inner point in {$\MMrs_{g,k}
$}}\

If a soliton $W$-section $(\frkc_\sigma, \bu_\sigma)$ in
$\MMrs_{g,k}(\bgamma,\bvkappa)$ has no BPS soliton component, then
we show that it is actually an interior point of the moduli space.
In this part, we will construct the local chart of an interior
point.

A non-BPS $W$-section $(\frkc_\sigma,\bu_\sigma)$ can still have a
soliton component, for example, we denote one of them by
$(\bu_{j_1,j_2}, \gamma_p)$, where $p$ is a marked or nodal point
of $\frkc$. Since this soliton is not a BPS soliton,
$\bu_{j_1,j_2}$ has at least one regular value and this guarantees
that the automorphism group $\aut(\frkc_\sigma, \bu_\sigma)$ is a
finite group.

For simplicity, we define
$$
D_\sigma(WI):=D_{\bu_\sigma}((WI)_{\frkc_\sigma});
\aut(\sigma):=\aut(\frkc_\sigma, \bu_\sigma).
$$

Using the standard method one can easily find a finite dimensional
space $E_\sigma\subset L^p(\cC,\times_i (\LL_i\otimes
\Lambda^{0,1}))$ satisfying the following properties:
\begin{enumerate}
\item $E_\sigma+\im D_\sigma(WI)=L^p(\cC,\times_i (\LL_i\otimes
\Lambda^{0,1}))$.

\item $E_\sigma$ is complex linear and $\aut(\sigma)$-invariant.

\item There exists a compact set $K_{obst}$ away from the marked
or nodal points containing the support of all elements in
$E_\sigma$.

\item $E_\sigma$ is a finite-dimensional space consisting of
smooth sections.

\end{enumerate}

Let $E'_\sigma\oplus E_\sigma=L^p(\cC,\times_i (\LL_i\otimes
\Lambda^{0,1}))$, and let $\Vm\oplus V'_\sigma=L^p_1(\cC, \times_i
\LL_i)$, where $\Vm=(D_\sigma(WI))^{-1}(E_\sigma)$. Thus
$D_\sigma(WI): V'_\sigma\rTo E'_\sigma$ is a bounded
invertible operator. Its inverse is denoted by $Q_\sigma$.

Let $\frkc_\sigma=(\cC_\sigma, p_1,\dots, p_k,
\LL_1,\dots,\LL_N,\varphi_1,\dots,\varphi_s,\psi_1,\dots,\psi_k
)$. Then a neighborhood of $[\cC_\sigma,p_1,\dots, p_k]$ in
$\MM_{g,k}$ can be parametrized by the following neighborhood of
${0\times\infty}$:
$$
\frac{\Vd\times\Vr}{\aut{\cC}},
$$
where $\Vd$ is a small neighborhood of $0$ in $\prod_{\cC_\nu
\text{is stable}} \C^{3g_\nu-3+k_\nu}$ and
$$\Vr=\oplus_{\text{$z$ is
nodal point}}([T_0, \infty]\times S^1)_z.
$$
$\Vd$ and $\Vr$ are the deformation domain and resolution domain
respectively. Now a uniformizing system of $\frkc_\sigma$ in
$\MMr_{g,k}(\bgamma)$ is given by
$$
\frac{\Vd\times\Vr}{\aut{\frkc}}.
$$
So for each data $(y,\zeta)\in \Vd\times \Vr$, we can obtain a
nearby curve $\frkc_{y,\zeta}\in \MMr_{g,k}(\bgamma)$. They are also
$\aut(\sigma)$-invariant, i.e., for any $g\in \aut(\sigma)$ we have
$g\cdot\frkc_{y,\zeta}=\frkc_{g\cdot(y,\zeta)}$.\

Now we begin the construction of the Kuranishi neighborhood of
$[\frkc_\sigma, \bu_\sigma]$ in $\MMrs_{g,k}(\bgamma)$.

\subsubsection*{The approaximation solutions}

Note that there are some parameters used in the neighborhood of a
nodal point. They are fixed numbers $T_0, \bar{T}_0, \hat{T}$
which have the relation $T_0< 10\bar{T}_0<\hat{T}$. Recall the
meaning of these parameters: $T_0$ gives the range of the gluing
parameter $\zeta_z=(s_z, \theta_z)\in ([T_0,\infty)\times S^1)$.
The supports of the derivatives of the cut-off function $\beta,
\varpi$ which are used to define the perturbed Witten map
 is in $[\bar{T}_0, \bar{T}_0+1]$ and in $[8\bar{T}_0, 9\bar{T}_0]$ respectively.
$\hat{T}$ appears in the definition of the neck region
$N_{n,z}(\hat{T})=([\hat{T}, T^n_z]\times S^1)_\nu\cup ([\hat{T},
T^n_z]\times S^1)_\mu$, where $\zeta^n_z=(s^n_z,\theta_z)$ is the
gluing parameter and $s^n_z=2T^n_z$.

Now we define some sets near a nodal point $z$. Let
$1/2<\delta<1$. Take the gluing parameter $\zeta_z=(s_z=2T_z,
\theta_z)$ and $T_z>4\hat{T}$.

\begin{align*}
&N_{z_\nu}(\delta, T_z)=([(1-\delta)T_z, T_z]\times S^1)_\nu\\
&N_z(\delta,T_z)=N_{z_\nu}(\delta, T_z)\cup N_{z_\mu}(\delta,
T_z)\\
&N_z(\delta, T_z, \infty)=([(1-\delta)T_z, \infty]\times
S^1)_\nu\cup ([(1-\delta)T_z, \infty]\times S^1)_\mu\\
&N_z(T_z,\infty)=([T_z, \infty]\times S^1)_\nu\cup
([T_z,\infty]\times S^1)_\mu
\end{align*}

Notice that $([\frac{T_z}{2}, \frac{3T_z}{2}]\times S^1)_\nu\cup
([\frac{T_z}{2}, \frac{3T_z}{2}]\times S^1)_\mu$ is the gluing
domain. We can take $T_z$ to be sufficiently large such that
$K_{obst}$ is disjoint from the gluing domain but contains the
definition domain of $\beta,\varpi$.

Take a section $\bphi_\sigma\in
C^\infty(\cC_\sigma,\LL_1)\times\dots
C^\infty(\cC_\sigma,\LL_t)$; we want to define an approximating
section $\papp$ on the nearby curve $\frkc_{y,\zeta}$. We will
only modify the solution $\bphi_\sigma$ near each nodal point $z$.
So without loss of generality, we assume that $\cC$ has only one
nodal point $z$ connecting the components $\cC_\nu$ and $\cC_\mu$.
Assume that $\bphi_\nu(\infty)=\bvkappa=\bphi_\mu(\infty)$ (under
the same rigidification, i.e., viewed in $\C^N$).

Let $\beta_{\delta}$ be a smooth function satisfying

$$
\beta_{\delta}(s, \theta)=
\begin{cases}
0&, if \;s\le 0\\
1&, if \; s\ge \delta,
\end{cases}
$$
and $|\nabla \beta_{\delta}|\le C/\delta$.

Let $(s_\nu,\theta_\nu)$ and $(s_\mu,\theta_\mu)$ be the
cylindrical coordinates of the two components respectively. On the
gluing domain of $\cC_{y,\zeta}$ they satisfy the relation
$(s_\nu,\theta_\nu)=(s_z,\theta_z)-(s_\mu,\theta_\mu)$.

Define a section $\papp$ on $\cC_{y,\zeta}$. When outside the
gluing domain in $\cC_{y,\zeta}$, $\papp\equiv \bphi$. On the
gluing domain, we define
$$
\papp(s_\mu,\theta_\mu):=\bphi_\mu(s_\mu,\theta_\mu)+(1-\beta_{\delta}(\frac{T_z-s_\mu}{T_z}))
(\bphi_\nu(2T_z-s_\mu, \theta_z-\theta_\mu)-\bvkappa)
$$
in the $\mu$ coordinate or in the $\nu$ coordinate in the
following form:
$$
\papp(s_\nu,\theta_\nu):=\bphi_\nu(s_\nu,\theta_\nu)+(1-\beta_{\delta}(\frac{T_z-s_\nu}{T_z}))
(\bphi_\mu(2T_z-s_\nu, \theta_z-\theta_\nu)-\bvkappa).
$$
One can show this section is well defined since on $\{T_z\}\times
S^1$ there holds $\papp(s_\nu,\theta_\nu)=
\bphi_\nu(T_z,\theta_\nu)+\bphi_\mu(T_z,
\theta_z-\theta_\nu)-\bvkappa=\papp(s_\mu,\theta_\mu)$.

Define
\begin{equation}
Glue_{y,\zeta}(\bphi):=\papp.
\end{equation}

$Glue_{y,\zeta}$ is a map from $C^\infty(\cC_\sigma)$ to
$C^\infty(\cC_{y,\zeta})$ or from $L^p_1(\cC_\sigma)$ to
$L^p_1(\cC_{y,\zeta})$.

Let $\bu_\sigma$ be a solution of the perturbed Witten equation
$(WI)_{\cC_\sigma}(\bu)=0$;  we can obtain the approximate
solution $\uapp=Glue_{y,\zeta}(\bu)$.

If we consider the map $Glue_{y,\zeta}$ as a map between Sobolev
spaces $L^p_1$, we have the following lemma.

\begin{lm}\label{lm-appro-1} Assume that $\bphi\in \Vm$. For any $\varepsilon,
0<\varepsilon<1$, the following holds for sufficiently large
$\zeta$ and sufficiently small $y$:
\begin{equation}
(1+\varepsilon)||\bphi||_{L^p_1(\cC_\sigma)}\ge
||Glue_{y,\zeta}(\phi)||_{L^p_1(\cC_{y,\zeta})}\ge
(1-\varepsilon)||\bphi||_{L^p_1(\cC_\sigma)}.
\end{equation}
\end{lm}

\begin{proof} It suffices to prove it on one component near a nodal point $z$.
Since $\bphi\in \Vm$, it satisfies the equation
\begin{equation}
D_\sigma(WI)(\bphi)=0 \mod E_\sigma.
\end{equation}
Note that when $\zeta$ is sufficiently large, the gluing domain
does not intersect  $K_{obst}$. So in $N_z(\delta, T, \infty)$, we
have
\begin{equation}\label{glui-equa}
D_\sigma(WI)(\bphi)=\left(\bpat_{\cC_\sigma}
\phi_1+\sum_j\tilde{I}_1\left(\overline{\frac{\pat^2
(W+W_{0,\beta})}{\pat u_1\pat u_j}\phi_j}\right),\dots,
\bpat_{\cC_\sigma}
\phi_N+\sum_j\tilde{I}_1\left(\overline{\frac{\pat^2
(W+W_{0,\beta})}{\pat u_N\pat u_j}\phi_j}\right) \right)=0.
\end{equation}

Note that all the sections here are smooth orbifold sections. Let
$e_i$ be the basis of $\LL_i$ and let $\tilde{u}_i,\phi_i$ be the
coordinate functions of $u_i,\phi_i$ with respect to this basis.
If we take the transformation $\hat{\bu}_\sigma=(\tilde{u}_1
e^{-\Theta^{\gamma_z}_1(s+i\theta)}), \dots, \tilde{u}_N
e^{-\Theta^{\gamma_z}_N(s+i\theta)})$ and
$\hat{\bphi}_\sigma=(\tilde{\phi}_1
e^{-\Theta^{\gamma_z}_1(s+i\theta)}, \dots, \tilde{\phi}_N
e^{-\Theta^{\gamma_z}_N(s+i\theta)})$, then the equation
(\ref{glui-equa}) becomes the equation on the resolved $W$-curve
$|\frkc_\sigma|$:
\begin{equation}
\left(\bpat_{|\cC|_\sigma} \hat{\phi}_1+\sum_j
I_1\left(\overline{\frac{\pat^2 (W+W_{0,\beta})}{\pat
\hat{u}_1\pat \hat{u}_j}\hat{\phi}_j}\right),\dots,
\bpat_{|\cC|_\sigma} \hat{\phi}_N+\sum_j
I_1\left(\overline{\frac{\pat^2 (W+W_{0,\beta})}{\pat
\hat{u}_N\pat \hat{u}_j}\hat{\phi}_j}\right) \right)=0.
\end{equation}
By the analysis in Section 6.4, we know that $\hat{\phi}_i
e^{\Theta^{\gamma_z}_1(s+i\theta)}$ and any of its derivatives are
of exponential decay. This just means that $\phi_i$ and its
derivatives are of exponential decay. Thus for sufficiently large
$\zeta_z$, we have
\begin{equation}\label{glui-isom-proof1}
||\bphi||_{L^p_1(N_z(\delta, T, \infty))}\le \frac{\varepsilon}{6}
||\bphi||_{L^p_1(\cC_\sigma)}.
\end{equation}
Since $\papp=\bphi$ in $\cC_{y,\zeta}\setminus N_z(\delta,T_z)$,
for sufficiently small deformation parameter $y$, we have
\begin{equation}\label{glui-isom-proof2}
(1-\frac{\varepsilon}{2})||\bphi||_{L^p_1(\cC_\sigma\setminus
N_z(\delta,T_z,\infty))}\le
||\papp||_{L^p_1(\cC_{y,\zeta}\setminus N_z(\delta,T_z))}\le
(1+\frac{\varepsilon}{2})||\bphi||_{L^p_1(\cC_\sigma\setminus
N_z(\delta,T_z,\infty))}.
\end{equation}
On the other hand, we have
\begin{align*}
&||\papp||_{L^p_1(N_z(\delta,T_z))}=||\papp||_{L^p}+||\nabla\papp||_{L^p}\nonumber\\
&\le 2||\bphi||_{L^p(N_z(\delta,T_z,\infty))}
+2||\nabla\bphi||_{L^p(N_z(\delta,T_z,\infty))}+2||\nabla\beta_\delta
\frac{1}{T_z}\bphi||_{L^p(N_z(\delta,T_z,\infty))}\nonumber\\
&\le 3||\bphi||_{L^p_1(N_z(\delta,T_z,\infty))}.
\end{align*}
By (\ref{glui-isom-proof1}), we obtain for sufficiently large
$\zeta$:
\begin{equation}\label{glui-isom-proof3}
||\papp||_{L^p_1(N_z(\delta,T_z))}\le \frac{\varepsilon}{2}
||\bphi||_{L^p_1(\cC_\sigma)}.
\end{equation}
Combining the results (\ref{glui-isom-proof2}) and
(\ref{glui-isom-proof3}), we get the conclusion.
\end{proof}

\subsubsection*{Obstruction bundle on $\cC_{y,\zeta}$}

The deformation map from $\cC_\sigma$ to $\cC_{y,\zeta}$ provides
a bundle isomorphism when restricted to the domain $K_{obst}$:
$$
\theta_{y,\zeta}: \times_i(\LL_{i,\sigma}\otimes
\Lambda^{0,1}|_{K_{obst}})\rTo
\times_i(\LL_{i,y,\zeta}\otimes \Lambda^{0,1}|_{K_{obst}}),
$$
which induces the isomorphism of sections:
$$
\theta_{y,\zeta}: C^\infty(K_{obst},\times_i(\LL_{i,\sigma}\otimes
\Lambda^{0,1}|_{K_{obst}}))\rTo
C^\infty(K_{obst},\times_i(\LL_{i,y,\zeta}\otimes
\Lambda^{0,1}|_{K_{obst}})).
$$

Since the support of each section in $E_\sigma$ is contained in
$K_{obst}$, we define $E_{y,\zeta}:=\theta_{y,\zeta}(E_\sigma)$.

Set $D_{y,\zeta}(WI):=D_{\uapp}((WI)_{y,\zeta})$. Our aim is to
find the solution of the following equations:
\begin{equation}\label{glui-equa-yz}
D_{y,\zeta}(WI)(\bphi)=0 \mod E_{y,\zeta}.
\end{equation}

\begin{lm}\label{lm-localchart1} Let $\bphi\in \Vm$; then for sufficiently large $\zeta$ we have
\begin{equation}
||D_{y,\zeta}(WI)\circ Glue_{y,\zeta}(\bphi)-\theta_{y,\zeta}\circ
D_\sigma(WI)(\bphi)||_{L^p(\cC_{y,\zeta})}\le
C(|y|+\frac{1}{T_z}+e^{-\delta_0
 T_z})||\bphi||_{L^p_1(\cC_{y,\zeta})},
\end{equation}
where $C$ is a constant depending on $\bu_\sigma, \delta$ and the
decay exponent $\delta_0$ which is from Theorem
\ref{thm-solu-noda}.
\end{lm}

\begin{proof}
We discuss the integral on $\cC_{y,\zeta}\setminus N_z(\delta,
T_z)$ and on $N_z(\delta, T_z)$ respectively. For simplicity, we
write the operator as
$$
D_{y,\zeta}(WI)(\bphi):=\bpat_{y}\bphi+\overline{A(\uapp,y)\cdot
\bphi},
$$
where $A(\uapp,y)$ is the corresponding matrix depending on
$\uapp$ and the deformation parameter $y$, since the
metric-preserving isomorphism $\tilde{I}_1$ also depends on $y$
which is induced by the $W$-spin structure. This shows that the
operator only depends on the deformation parameter if the
resolution parameter is sufficiently large (because the function
$\varpi\equiv 1$ for large $\zeta$).

Let $\bphi_\nu$ and $\bphi_\mu$ represent the $\nu$ and $\mu$
component of $\bphi$. On $N_z(\delta, T_z)$ we have
\begin{align}
&||D_{y,\zeta}(WI)\circ Glue_{y,\zeta}(\bphi)-0||_{L^p(N_z(\delta,
T_z))}\nonumber\\
&=||\bpat_y(\bphi_\nu+(1-\beta_\delta)\bphi_\mu)+
\overline{A(\uapp,y)\cdot(\bphi_\nu+(1-\beta_\delta)\bphi_\mu)}
||_{L^p}\nonumber\\
&\le
||(\bpat_y-\bpat_\sigma)\bphi_\nu+(1-\beta_\delta)(\bpat_y-\bpat_\sigma)\bphi_\mu||+
||\bpat_y(1-\beta_\delta)\bphi_\mu||\nonumber\\
&+||(A(\uapp,y)-A(\bu,\sigma))(\bphi_\nu+(1-\beta_\delta)\bphi_\mu)||\nonumber\\
&\le C|y|
||\nabla\bphi||_{L^p}+\frac{C}{T_z}||\bphi||_{L^p}+(||A(\uapp,y)-A(\bu,y)||+||A(\bu,y)-A(\bu,\sigma)||)
||\bphi||_{L^p}\nonumber\\
&\le C(|y|+\frac{1}{T_z}+e^{-\delta_0 T_z})||\bphi||_{L^p_1}.
\end{align}
Here we have used the definition of $\uapp$ and the property of
exponential decay of $\bu$ on $[T_z,\infty)\times S^1$ when $T_z$
is large enough.

On $\cC_{y,\zeta}\setminus N_z(\delta, T_z)$, we have
\begin{align}
&||D_{y,\zeta}(WI)\circ
Glue_{y,\zeta}(\bphi)-\theta_{y,\zeta}\circ
D_\sigma(WI)(\bphi)||_{L^p(\cC_{y,\zeta}\setminus N_z(\delta,
T_z))}\nonumber\\
&\le
||(D_{y,\zeta}(WI)-D_\sigma(WI))\bphi||+||(I-\theta_{y,\zeta})D_\sigma(WI)\bphi||\nonumber\\
&\le  C(|y|+\frac{1}{T_z})||\bphi||_{L^p_1}.
\end{align}

Combining the above two inequalities, we obtain the result.
\end{proof}

\subsubsection*{Existence of right inverse and its uniform upperbound}

Let $V_{y,\zeta}=(D_{y,\zeta}(WI))^{-1}(E_{y,\zeta})$. Define
$E'_{y,\zeta}$ to be the complementary subspace of $E_{y,\zeta}$
and $V'_{y,\zeta}$ to be the complementary subspace of
$V_{y,\zeta}$. To solve the equation (\ref{glui-equa-yz}) is
equivalent to proving the existence of the right inverse of
$D_{y,\zeta}(WI)$. Define a map
$$
I_{y,\zeta}: L^p(\cC_{y,\zeta}, \times_i(\LL_{i,y,\zeta}\otimes
\Lambda^{0,1}(\cC_{y,\zeta})))\rTo L^p(\cC_\sigma,
\times_i(\LL_{\sigma}\otimes \Lambda^{0,1}(\cC_{\sigma})))
$$
as
$$
I_{y,\zeta}(\bphi)(z):=
\begin{cases}
\theta_{y,\zeta}^{-1}\circ\phi(z)&,
\;if\;z\in \cC_{\sigma}\setminus N_z(T_z,\infty)\\
0,&\;if\;z\in N_z(T_z,\infty).
\end{cases}
$$

We claim that the composition map
$Q'_{app,y,\zeta}:=Glue_{y,\zeta}\circ Q_\sigma\circ I_{y,\zeta}:
E'_{y,\zeta}\rTo L^p_1$ is an approximating right inverse
of $D_{y,\zeta}(WI): V'_{y,\zeta}\rTo E'_{y,\zeta}$. This
is known from the following lemma.

\begin{lm} If the gluing parameter $\zeta$ is sufficiently large
and the deformation parameter $y$ is sufficiently small, then
\begin{equation}
||D_{y,\zeta}(WI)\circ
Q'_{app,y,\zeta}(\bphi)-\bphi||_{L^p(\cC_{y,\zeta})}\le
\frac{1}{2}||\bphi||_{L^p(\cC_{y,\zeta})}.
\end{equation}
\end{lm}

\begin{proof} Note that $\bphi=\theta_{y,\zeta}\circ
D_\sigma(WI)\circ Q_\sigma\circ I_{y,\zeta}(\bphi)$; we only need
to prove that for $\psi=Q_\sigma\circ I_{y,\zeta}(\bphi)$ the
following inequality holds:
\begin{equation}
||D_{y,\zeta}(WI)\circ Glue_{y,\zeta}(\psi)-\theta_{y,\zeta}\circ
D_\sigma(WI)(\psi)||_{L^p}\le C(|y|+\frac{1}{T_z}+e^{-\delta_0
T_z})||\psi||_{L^p_1}.
\end{equation}
Now the proof is the same as the proof of Lemma
\ref{lm-localchart1} while observing that
$$
D_\sigma(WI)(\psi)=0
$$
on $([T_z, (1+\delta)T_z]\times S^1)$.
\end{proof}

This lemma implies that the right inverse $Q_{y,\zeta}$ exists and
$Q_{y,\zeta}=Q'_{app,y,\zeta}\circ(D_{y,\zeta}(WI)\circ
Q'_{app,y,\zeta})^{-1}$.

Now it is easy to obtain the following lemma.
\begin{lm}\label{lm-kuramodel-0} For sufficiently large $\zeta$ and sufficiently small
$y$, $Q_{y,\zeta}$ has uniform upper bound:
$$
||Q_{y,\zeta}||\le \tilde{C}_1.
$$
\end{lm}

\subsubsection*{Kuranishi model on $\frkc_{y,\zeta}$}

Define the map
$F_{y,\zeta}(\bphi):=WI_{y,\zeta}(\uapp+\bphi):L^p_1(\cC_{y,\zeta},
\times_i \LL_i)\rTo L^p(\cC_{y,\zeta}, \times_i
(\LL_i\otimes \Lambda^{0,1}))$. We want to apply Lemma
\ref{lm:model} to the nonlinear operator $F_{y,\zeta}$ and hope to
get a Kuranishi model centered at $\uapp$.

We have
$$
F_{y,\zeta}(0)=WI_{y,\zeta}(\uapp),\;DF_{y,\zeta}(0)=D_\uapp((WI)_{y,\zeta})=D_{y,\zeta}(WI).
$$

Let
$G_{y,\zeta}(\bphi):=F_{y,\zeta}(\bphi)-F_{y,\zeta}(0)-DF_{y,\zeta}(0)\bphi$.

Define the projection maps $P_{E_{y,\zeta}}:
L^p(\cC_{y,\zeta})\rTo E'_{y,\zeta}$ and $P_{V_{y,\zeta}}:
L^p_1(\cC_{y,\zeta}) \rTo V'_{y,\zeta}$. We need two lemmas
when applying Lemma \ref{lm:model}.

\begin{lm}\label{lm-kuramodel-1} For sufficiently large $\zeta$ and sufficiently small
$y$, we have
\begin{equation}
|| P_{E_{y,\zeta}}\circ F_{y,\zeta}(0)||_{L^p}\le
C(|y|+\frac{1}{T_z}+e^{-\delta_0 T_z})||\bu||_{L^p_1}.
\end{equation}
\end{lm}

\begin{proof} We have
\begin{align*}
&|| P_{E_{y,\zeta}}\circ F_{y,\zeta}(0)||_{L^p}\le
||WI_{y,\zeta}(\uapp)-WI_\sigma(\bu)||_{L^p}\\
&=||\bpat_y \uapp+\overline{B(\uapp,y)}-\bpat_\sigma
\bu-\overline{B(\bu,\sigma)}||,
\end{align*}
where $B(\bu, \sigma)=(\tilde{I}_1(\overline{\frac{\pat
W+W_\beta}{\pat u_1}}),\dots, \tilde{I}_1(\overline{\frac{\pat
W+W_\beta}{\pat u_N}}))$ is a $t$-dimensional vector.

Then using the decay property of $\bu$ on $[T_z,\infty)\times
S^1$, we obtain
$$
|| P_{E_{y,\zeta}}\circ F_{y,\zeta}(0)||_{L^p}\le
C(|y|+\frac{1}{T_z}+e^{-\delta_0 T_z})||\bu||_{L^p_1}.
$$
\end{proof}

\begin{lm}\label{lm-kuramodel-2}For sufficiently large $\zeta$ and sufficiently small
$y$, we have
\begin{equation}
||P_{E_{y,\zeta}}\circ G_{y,\zeta}(\bphi_1)-P_{E_{y,\zeta}}\circ
G_{y,\zeta}(\bphi_2)||_{L^p}\le
\tilde{C}_2(||\bphi_1||_{L^p_1}+||\bphi_2||_{L^p_1})||\bphi_1-\bphi_2||_{L^p_1},
\end{equation}
where $\tilde{C}_2$ depends only on $\bu$.
\end{lm}

\begin{proof} This is a direct computation for which, among other things, the Sobolev
embedding theorem and the interpolation formula of $L^p$ spaces
are used.
\end{proof}

Now the hypothesis of Lemma \ref{lm:model} is satisfied by Lemma
\ref{lm-kuramodel-0}, \ref{lm-kuramodel-1}, \ref{lm-kuramodel-2}.
Therefore we have the following lemma.

\begin{lm} Take $y$ small enough and $\zeta$ large enough. Let
$\tilde{C}_1, \tilde{C}_2$ be the constants from Lemma
\ref{lm-kuramodel-0} and Lemma \ref{lm-kuramodel-2} respectively.
Choose $0<r<\frac{1}{8\tilde{C}_1\tilde{C}_2}$ and let
$U_{y,\zeta}(r)$ be a ball centered at the origin with radius $r$
in $L^p_1(\cC_{y,\zeta})$. Then for any $\bphi\in V_{y,\zeta}\cap
U_{y,\zeta}(r)$, there is a unique $v'(\bphi)\in V'_{y,\zeta}\cap
U_{y,\zeta}(r)$ such that $F_{y,\zeta}(\bphi+v'(\bphi))\in
E_{y,\zeta}$. On the other hand, for any $\tilde{\bphi}\in
U_{y,\zeta}(r/||I-P_{V_{y,\zeta}}||)$ such that
$F_{y,\zeta}(\tilde{\bphi})\in E_{y,\zeta}$, there is a unique
$\bphi\in V_{y,\zeta}$ such that $\tilde{\bphi}=\bphi+v'(\bphi)$.
Define $\Psi_{y,\zeta}: \bphi\to \bphi+v'(\bphi)$, and let
$s_{y,\zeta}: V_{y,\zeta}\cap
\Psi^{-1}(U_{y,\zeta}(r/||I-P_{V_{y,\zeta}}||))\to E_{y,\zeta}$ be
defined by $s_{y,\zeta}(\bphi):=F(\bphi+v'(\bphi))$. Then
$(U_{y,\zeta}(r), E_{y,\zeta}, s_{y,\zeta}, \Psi_{y,\zeta})$ forms
a Kuranishi model, where $s_{y,\zeta},\Psi_{y,\zeta}$ are
continuous and $\Psi_{y,\zeta}$ is a one to one map.
\end{lm}

\subsubsection*{Kuranishi neighborhood at $(\frkc_\sigma, \bu_\sigma)$}

Now we know that starting from the point
$\sigma=(\frkc_\sigma,\bu_\sigma)$ in the interior of
$\MMrs_{g,k}(\bgamma,\bvkappa)$ we can construct the Kuranishi model
$(U_{y,\zeta}(r), E_{y,\zeta}, s_{y,\zeta}, \Psi_{y,\zeta})$ on any
nearby curve $\frkc_{y,\zeta}$. To construct the Kuranishi
neighborhood, we only need to construct a family of isomorphisms
$\eta_{y,\zeta}: \Vm \rTo V_{y,\zeta}$ with uniformly bounded
norms. For $\bphi\in \Vm$, we define
$$
\eta_{y,\zeta}(\bphi)=Glue_{y,\zeta}(\bphi)-Q_{y,\zeta}\circ
D_{y,\zeta}(WI)\circ Glue_{y,\zeta}(\bphi).
$$
\begin{lm} For sufficiently large $\zeta$ and sufficiently small
$y$, $\eta_{y,\zeta}$ is an isomorphism.
\end{lm}

\begin{proof} By Lemmas \ref{lm-appro-1}, \ref{lm-kuramodel-2},
there exists a uniform constant $C$ independent of $y,\zeta$ such
that
$$
||\eta_{y,\zeta}(\bphi)||_{L^p_1}\le C||\phi||.
$$
On the other hand, we have
\begin{align*}
||\eta_{y,\zeta}(\bphi||_{L^p_1}\ge&
||Glue_{y,\zeta}(\bphi)||_{L^p_1}-||Q_{y,\zeta}\circ(
D_{y,\zeta}(WI)\circ Glue_{y,\zeta}(\bphi)-\theta_{y,\zeta}\circ
D_\sigma(WI)(\bphi))||\\
&\ge (1-\varepsilon)||(\bphi)||_{L^p_1}-C\varepsilon ||
\bphi||_{L^p_1},
\end{align*}
where we have used Lemma \ref{lm-localchart1} and
\ref{lm-kuramodel-0}. Hence if the parameters $\zeta,y$ satisfy
our requirement, then $\eta_{y,\zeta}$ is an isomorphism.
\end{proof}

For convenience, we identify the gluing parameter space $\Vr$ with
a small neighborhood of $\prod_{\text{$z$ is nodal point}}\C_z$ by
the map $e^{-z}$.

Before formulating the main result, we have to consider the action
of automorphisms of curves. Because of the existence of unstable
soliton components, the automorphism group is of positive
dimension. It is complex $1$-dimensional for each unstable
component. For example, if $(\R\times S^1,\bu_{j_1,j_2})$ is a
soliton component (non-BPS soliton by assumption), then the field
$\lambda_s\frac{\pat}{\pat s}+\lambda_\theta\frac{\pat}{\pat
\theta}$ generates the automorphism group, the transition group.
Also, $d\bu(\frac{\pat}{\pat s}), d\bu(\frac{\pat}{\pat \theta})$
generates a complex one-dimensional subspace in $\Vm$.

To eliminate the action of the transition group in the unstable
component, we use a normalization technique used in \cite{FO}.

We add one narrow marked point $z_\nu$ with the trivial
orbifold structure (i.e., the group action is given by $\exp 2\pi
i$) in the unstable component $(\R\times S^1)_\nu$ such that if
there exists a map in $\aut(\sigma)$ that maps a unstable
component $\cC_{\nu_1}$ to $\cC_{\nu_2}$, then this map will send
the extra marked point $z_{\nu_1}$ to $z_{\nu_2}$. Let
$\z'_0=\{z'_1,\dots, z'_l\}$ be the set of extra marked points on
$\cC_\sigma$ chosen in this way. We can also assume that these
marked points are chosen such that $\bu_\sigma$ is an immersion
near these points.

For each new marked point $z'_j$, take an $(2N-2)$-dimensional disk
$\D_{z'_j}$ in $\C^N$ which is transversal to $\im (\bu)$ at
$\bu(z'_j)$. We assume that $\D_{g\cdot z'_j}=\D_{z'_j}$ when
$z'_j$ and $g\cdot z'_j$ are marked points when $g\in
\aut(\sigma)$.

Define the parameter space:
\begin{equation*}
V^+_\sigma=\Vd\times \Vr\times \Vm.\\
\end{equation*}

\begin{thm}\label{thm:knbhd} Let $(\frkc_\sigma, \bu_\sigma)$ be a non-BPS soliton $W$-section in
$\MMrs_{g,k}(\bgamma,\bvkappa)$. Let the set $\Vd\times \Vr$ be
small enough such that for any $(y,\zeta)\in \Vd\times
\Vr$($(0,0)\equiv \sigma$) the operators
$\theta_{y,\zeta},\eta_{y,\zeta},s_{y,\zeta},\Psi_{y,\zeta}$ are
well defined on $\frkc_{y,\zeta}$. Define the set
$$
Z_r=\cup_{(y,\zeta)\in \Vd\times \Vr}\{\bphi|F_{y,\zeta}(\bphi)\in
E_{y,\zeta}, ||\bphi||_{L^p_1(\cC_{y,\zeta}, \times_i \LL_i)}\le r
\}.
$$
Define the map $\Psi_\sigma: V^+_\sigma\rTo Z_r$ by
$$
\Psi_\sigma(y,\zeta,\bphi):=\Psi_{y,\zeta}(\eta_{y,\zeta}(\bphi)),
$$
and the map $\tilde{s}_\sigma: V^+_\sigma\rTo E_\sigma$ by
$$
\tilde{s}_\sigma(y,\zeta,\bphi):=\theta^{-1}_{y,\zeta}\circ
s_{y,\zeta}\circ \eta_{y,\zeta}(\bphi).
$$
Then for sufficiently small $r>0$, we can obtain a
$\aut(\sigma)$-invariant open neighborhood $U^+_\sigma$ of $0$
which is contained in $V^+_\sigma$ such that when restricting to
the domain $U^+_\sigma$ the following conclusions hold:
\begin{enumerate}
\item $\Psi_\sigma$ is an $\aut(\sigma)$-equivariant continuous
map, which is one to one and onto its image, and for fixed
$(y,\zeta)$, $\Psi(y,\zeta,\cdot)$ is a homeomorphism.

\item The map $s_\sigma: U^+_\sigma/\aut{\sigma}\rTo
E_\sigma/\aut(\sigma)$ defined by its lifting map
$\tilde{s}_\sigma$ is continuous.

\item Define a closed set in $V^+_\sigma$:
\begin{align*}
&\Vmt:=\{(y,\zeta,\bphi)\in V^+_\sigma| (y,\zeta)\in \Vd\times
\Vr,
\bphi\in \Vm, \\
&\Psi_\sigma(y,\zeta,\phi)(z'_j) \;\text{is tangential to
}\;\D_{z'_j},\forall j=1,\dots, l\}.
\end{align*}
Then there is an $\aut(\sigma)$ action on $\Vmt$, and if we let
$U_\sigma=U^+_\sigma\cap \Vmt$, $\Psi_\sigma$ induces a
homeomorphism between $s_\sigma^{-1}(0)\subset
U_\sigma/\aut{\sigma}$ and a neighborhood of $\sigma$ in a branch
containing $\sigma$ in $\MMrs_{k,g,W}(\bgamma,\bvkappa)$.
\end{enumerate}
The data $(U_\sigma, E_\sigma, s_\sigma, \Psi_\sigma)$ forms a
Kuranishi neighborhood of $(\frkc_\sigma,\bu_\sigma)$ in
$\MMrs_{g,k}(\bgamma,\bvkappa)$ of real dimension
$6g-6+2k-2D-\sum_{i=1}^k N_{\gamma_i}$, where
$D=\hat{c}_W(g-1)+\sum_\tau \iota(\gamma_\tau)$.
\end{thm}

\begin{rem} The closed set $\Vmt\subset V^+_\sigma$ is actually a fiber
bundle over $\Vd\times \Vr$, and each fiber is homeomorphic to the
fiber at zero.
\end{rem}

\begin{proof} Since in our construction of $\cC_{y,\zeta}$ and
$E_{y,\zeta}$ they can be required to be
$\aut(\sigma)$-invariant, so the operators
$\theta_{y,\zeta},\eta_{y,\zeta},s_{y,\zeta},\Psi_{y,\zeta}$ can
also be required to be $\aut(\sigma)$-equivariant maps. Thus
$\tilde{s}_\sigma,\Psi_\sigma$ are $\aut(\sigma)$-equivariant
continuous maps. Continuity comes from the implicit function
theorem, Lemma \ref{lm:model}. Furthermore, $\Vmt$ is an
$\aut(\sigma)$-invariant closed set. The dimension is given by
Theorem \ref{ind-witt-oper}.

We need only to prove that $\Psi_\sigma$ indeed induces a
homeomorphism between $s^{-1}_\sigma(0)\cap U_\sigma/\aut(\sigma)$
and a neighborhood of $\sigma=(\frkc_\sigma,\bu_\sigma)\in
\MMrs_{g,k}(\bgamma,\bvkappa)$.

First we prove the injectivity of $\Psi_\sigma$ when restricting
to $s^{-1}_\sigma(0)\cap U_\sigma/\aut(\sigma)$.

Denote $\sigma'=(\cC_\sigma, (\z_\sigma, \z'_0))\in \MM_{g, k+l}$.
We remark that $\Vd=V_{deform,\sigma'}, \Vr=V_{resol,\sigma'}$, and
there is an open embedding
$$
\frac{\Vd\times \Vr}{\aut(\cC_{\sigma'},
(\z_\sigma,\z'_\sigma))}=\frac{V_{deform,\sigma'}\times
V_{resol,\sigma'}}{\aut(\cC_{\sigma'},
(\z_\sigma,\z'_\sigma))}\rTo \MM_{g, k+l}.
$$

Let $\gamma_0=\oplus_{\nu}(s_\nu, \theta_\nu)$ be a transition
moving any point $(s,\theta)$ on the component $(\R\times
S^1)_\nu$ to $(s+s_\nu,\theta+\theta_\nu)$. We can define an
action $\gamma_0\cdot (\cC_\sigma, (\z_\sigma,
\z'_0))=(\cC_{\gamma_0\cdot \sigma}, (\z_\sigma,
\gamma_0(\z'_0)))$, i.e., fixing the marked points $\z_k$ in
stable components but moving the extra marked points $\z'_0$ by
$\gamma_0$.

Similarly, we can define
$$
\gamma_0\cdot (\cC_{y,\zeta}, (\z_\sigma,
\z'_0))=(\cC_{\gamma_0\cdot (y,\zeta)}, (\z_\sigma,
\gamma_0(\z'_0))).
$$
Here one has to do some small modification such that the
transition fixes the gluing domain and only moves the interior
points of the unstable components.

If $\gamma_0$ is small enough, then the action of $\gamma_0$ will
induce an action on $\Vd\times \Vr$. This action is described in
\cite{FO}, and we follow their description. If $\gamma_0$ is small
enough, the surface $\cC_{\gamma_0\cdot (y,\zeta)}$ is still in
the neighborhood of $(\cC_\sigma, (\z_\sigma, \z'_0))\in
\MM_{g,k+l}$, hence there exists $(y_0,\zeta_0)\in
V_{deform,\sigma'}\times V_{resol,\sigma'}=\Vd\times \Vr$ such
that $(\cC_{\gamma_0\cdot (y,\zeta)},(\z_\sigma, \gamma_0(\z'_0)))
$ is equivalent to $(\cC_{y_0,\zeta_0}, (\z_\sigma,
\gamma_0(\z'_0)))$. This map is unique modulo the finite group\\
$\aut((\cC_{y_0,\zeta_0}, (\z_\sigma, \gamma_0(\z'_0))))$. We
define $\gamma_0\cdot(y,\zeta):=(y_0,\zeta_0)$.

Now assume that there are $(y',\zeta',\bphi'), (y,\zeta,\bphi)\in
\tilde{s}^{-1}_\sigma(0)\cap U_\sigma$ such that $(\cC_{y,\zeta},
\bu_{y,\zeta,\bphi})$ is equivalent to $(\cC_{y',\zeta'},
\bu_{y',\zeta',\bphi'})$, where
$\bu_{y,\zeta,\bphi}=\uapp+\Psi_\sigma(y,\zeta,\bphi)$ and
$\bu_{y',\zeta',\bphi'}=\bu_{app,y',\zeta'}+\Psi_\sigma(y',\zeta',\bphi')$
are assumed to be solutions of the perturbed Witten equation
$WI_{y,\zeta}(\bu)=0$. So there exists a biholomorphic map $\gamma
:\cC_{y,\zeta}\rTo \cC_{y',\zeta'}$ such that
$\bu_{y,\zeta,\bphi}=\gamma\cdot\bu_{y',\zeta',\bphi'}$. We want
to prove that $\gamma\in \aut(\sigma)$. Now we have
$(\cC_{y,\zeta}, (\z_\sigma, \z'_0))\in \MM_{g,k+l}$ and
$\gamma\cdot (\cC_{y,\zeta}, (\z_\sigma, \z'_0))=(\cC_{y',\zeta'},
(\z_\sigma, \gamma\cdot\z'_0))$. So $(\cC_{\gamma(y,\zeta)},
(\z_\sigma, \gamma\cdot\z'_0))=(\cC_{y',\zeta'}, (\z_\sigma,
\gamma\cdot\z'_0))$.  We can take $\gamma_1^{-1}\in \aut(\sigma)$
such that $\gamma_1^{-1}\cdot \gamma(\z'_i)$ is close to $\z'_i$,
for $i=1,\dots,l$. Then there exists a unique $\gamma_0=
\oplus_\nu (s_\nu,\theta_\nu)$ such that
$(\cC_{\gamma^{-1}_1(y',\zeta')}, (\z_\sigma,
\gamma_1^{-1}\gamma(\z'_0)))=(\cC_{\gamma_0(y,\zeta)},(\z_\sigma,\gamma_0\cdot
(\z'_0)))$ modulo the group $\aut(\cC_\sigma,(\z_\sigma,\z'_0))$.
Therefore $\gamma_1^{-1}\gamma=\gamma_0$ and $\gamma_1\cdot
\gamma_0(y,\zeta)=(y',\zeta')$. We claim that $\gamma_0=1$ under
our normalization choice. In fact, from the relation
$$
\bu_{y,\zeta,\bphi}=\gamma_1\cdot\gamma_0(\bu_{y',\zeta',\bphi'})
$$
we have at the extra marked point $z'_j$
\begin{align*}
&\bu_\sigma(z'_j)-\bu_\sigma(\gamma_1\cdot\gamma_0\cdot z'_j) =
\gamma_1\cdot\gamma_0(\Psi_\sigma(y',\zeta',\bphi'))-\Psi_\sigma(y,\zeta,\bphi)\\
&=\Psi_\sigma(y,\zeta,\gamma_1\cdot\gamma_0(\bphi'))-\Psi_\sigma(y,\zeta,\bphi')+
(\Psi_\sigma(y,\zeta,\bphi')-\Psi_\sigma(y,\zeta,\bphi)).
\end{align*}
Now if we consider the projection to the perpendicular direction
of $\D_{z'_j}$, we find that the term in the last bracket
vanishes. Since $\bu_\sigma$ is assumed to be an immersion at
$z'_j$, so if our neighborhood $U_\sigma$ is small enough, the
projection equality holds if and only if $\gamma_0=0$ and
$\gamma_1(\bphi')=\bphi$. So we have proved the injectivity.

We begin the proof of surjectivity. Suppose that there is a sequence
$(\frkc^n, \bu^n)\rTo (\frkc, \bu)$ in
$\MMrs_{g,k}(\bgamma,\bvkappa)$, where $(\frkc, \bu)$ is a non-BPS
soliton. To fix the neighborhood of non-BPS solitons in the moduli
space, we add an extra marked point to each soliton component just
as before. There is a sequence of parameters $(y_n,\zeta_n)$
converging to zero as $n\rTo \infty$ such that we have the
convergence of the underlying rigidified $W$-curve
$\frkc^n\rTo \frkc$ in $\MMr_{g,k}(\bgamma)$. On the other
hand, since the section $\bu^n$ converges weakly to $\bu$, if $n$ is
large then $\bu^n$ will lie in the tubular neighborhood of $\bu$.
Hence there is a sequence of sections $\phi^n\in
L^p_1(\cC_{y_n,\zeta_n}, \times_i \LL_i)$ such that
$\bu^n=\uapp+\phi^n$. Given arbitrarily small $r>0$, if we can prove
that $||\phi^n||_{L^p_1}<r$, then $\phi^n$ will lie in the image of
$\Psi_\sigma$. Since $\phi^n$ converges to $0$ in the $C^\infty$
topology outside the neck region, say $N_{n,z}(\hat{T})$, so we need
only to show that $||\phi^n||_{L^p_1(N_{n,z}(\hat{T}))}$ can be
arbitrarily small for sufficiently large $n$.

In this case, we have (without loss of generality, we consider one
nodal point and the two-component case)
$$
\lim_{\hat{T}\to \infty}\mbox{Diam}(\bu^n (N_{n,z}(\hat{T})))=0.
$$
Here $N_{n,z}(\hat{T})=([\hat{T},T^n_z]\times S^1)_\nu\cup
([\hat{T},T^n_z]\times S^1)_\mu$ and $T^n_z$ is a gluing
parameter.

On the neck region we have two equations:
\begin{align*}
&\bpat \uapp-2\overline{\frac{\pat (W+W_\gamma)(\uapp)}{\pat
u_i}}=0\\
&\bpat (\uapp+\phi^n)-2\overline{\frac{\pat
(W+W_\gamma)(\uapp+\phi^n)}{\pat u_i}}=0.
\end{align*}

When taking the difference of the two equations, we obtain the
equation of $\phi^n$:
\begin{equation}
\bpat \phi^n+\overline{A\cdot \phi^n}=0,
\end{equation}
where $A$ is the coefficient matrix which is differentiable and
all of its derivatives are bounded (Notice that the $C^m$
estimates of any solutions are uniformly bounded). Since the
diameter of the image $\bu^n(N_{n,z}(\hat{T}))$ is sufficiently
small if $\hat{T}$ and $n$ are large enough, the $C^0$ norm of
$\phi^n$ is uniformly small on the neck region. So the matrix
$A(\uapp+\phi^n)=A(\bu_\sigma+\phi^n)$ will tend to
$A(\bu(\infty))$, which is a nondegenerate matrix. Thus, for large
$\hat{T}$ and $n$, the absolute value of the eigenvalues of the
matrix $A(\uapp+\phi^n)$ has uniformly positive lower bound. This
means $\phi^n$ have uniformly exponential decay, hence
$||\phi^n||_{L^p_1(N_{n,z}(\hat{T}))}<r$ for large $n$ and
$\hat{T}$. We have to fix $\phi^n$ on an unstable component of
$\frkc$. Since $\bu^n$ is close to $\bu$, there is another unique
point $z''_j$ near $z'_j$ such that $\bu^n(z''_j)-\uapp(z'_j)$ is
tangential to $D_{z'_j}$. So we can redefine
$\phi^n(z)=\bu^n(z-z'_j+z''_j)-\uapp(z)$ on the unstable part.
\end{proof}

\subsection{Interior gluing and
orientation}\label{sec-inter-gluing}\

In this part, we will glue the Kuranishi neighborhoods of the
interior points of $\MMrs_{g,k}(\bgamma,\bvkappa)$ into an open
global Kuranishi structure. If our moduli space is strongly
perturbed, i.e., there is no soliton $W$-section, then this gives a
Kuranishi structure to the moduli space and hence obtains a virtual
cycle by Fukaya-Ono theory if it is orientable. If the moduli space
contains a soliton $W$-section, then we need to define the Kuranishi
neighborhoods near the boundary and then glue the interior Kuranishi
structure with the boundary neighborhoods to form a global structue.
This will be discussed in latter section.

Since there is no canonical way to construct the Kuranishi
structure, we will choose the obstruction bundle step by step and
then patch up these charts by induction. Though the Kuranishi
structure is dependent on the choice of such spaces, the virtual
cycle is well-defined. We will use the partial order $\succ$ in
$\MMrs_{g,k}(\bgamma,\bvkappa)$ defined before to do our induction.

Let $\MMrs_{g,k}(\Gamma;\bgamma,\bvkappa)$ be a minimal (and
nonempty) stratum in $\MMrs_{g,k}(\bgamma,\bvkappa)$ with respect to
the partial order $\succ$. Notice that this minimal stratum is not
unique.

Unlike the stratum in $\MM_{g,k}$,
$\MMrs_{g,k}(\Gamma;\bgamma,\bvkappa)$ is not in general an
orbifold. So we can't construct a universal orbifold obstruction
bundle on it. We have to construct it in a more direct way.

Let $(U_\sigma, E_\sigma, \aut(\sigma), s_\sigma, \Psi_\sigma)$ be
the Kuranishi neighborhood of $\sigma\in
\MMrs_{g,k}(\bgamma,\bvkappa)$ constructed in
Theorem~\ref{thm:knbhd} such that $(s_\sigma)^{-1}(0)$ is
homeomorphic to a neighborhood of $\sigma\in
\MMrs_{g,k}(\Gamma;\bgamma,\bvkappa)$. Let
$\Psi'_{\sigma}:s^{-1}_{\sigma}(0)\cap U_\sigma\rTo
\MMrs_{g,k}(\Gamma;\bgamma,\bvkappa)$ be the restriction of
$\Psi_\sigma$ to $s^{-1}_\sigma(0)\cap U_\sigma$. $\Psi'_\sigma$ is
a homeomorphism by Theorem \ref{thm:knbhd}. Since
$\MMrs_{g,k}(\Gamma;\bgamma,\bvkappa)$ is compact, there are
finitely many points $\tau_i$ in it such that the union of
$\hat{\Omega}_{\tau_i}:=\im \Psi'_{\tau_i}\cap
\MMrs_{g,k}(\Gamma;\bgamma,\bvkappa)$ covers
$\MMrs_{g,k}(\Gamma;\bgamma,\bvkappa)$. Assume
$\Omega_{\tau_i}\subset\hat{\Omega}_{\tau_i}$ are closed subsets and
their interior; also cover $\MMrs_{g,k}(\Gamma;\bgamma,\bvkappa)$.

Fix a representative $(\frkc_{\tau_i}, \bu_{\tau_i})\in \tau_i$,
then we have the bundle $E_{\tau_i}\subset L^p(\cC_{\tau_i};\times_i
(\LL_i\otimes \Lambda^{0,1}(\cC_{\tau_i})))$. For a point $(\frkc,
\bu)$ near $\MMrs_{g,k}(\Gamma;\bgamma,\bvkappa)$ in the big Banach
bundle, we want to embed $E_{\tau_i}$ in a more canonical way in
$L^p(\cC, \times_i(\LL_i\otimes \Lambda^{0,1}(\cC)))$.

\begin{df}

Let $(\frkc, \bu)$ be a point consisting of a map $\bu\in L^p_1(\cC,
\LL_1\times\dots\times \LL_t)$ and a $W$-curve $\frkc\in
\MMr_{g,k}(\bgamma)$. $(\frkc, \bu)$ is said to be closed to
$\sigma$ iff there exist a representative $(\frkc_\sigma,
\bu_\sigma)\in \sigma$, $(y,\zeta)\in \Vd\times\Vr$ and a
biholomorphic map $\theta:\cC\rTo\cC_{\sigma,y,\zeta}$ such
that $\theta$ preserves marked points and is an isomorphism of
$W$-structures, and $\bu\theta^{-1}$ is close to $\bu_{y,\zeta}$ in
the smooth topology on each irreducible component of
$\cC_{\sigma,y,\zeta}$, where $(\frkc_{y,\zeta},
\bu_{y,\zeta})=\Psi'_\sigma(y,\zeta,0)$. This actually gives  a
topology on the space of ``perturbed $W$-sections."
\end{df}

Let $(\frkc, \bu)$ be close to $\sigma\in \Omega_{\tau_i}$. By
taking $\Omega_{\tau_i}$ sufficiently small, we can obtain
$(y,\zeta)\in V_{deform,\tau_i}\times V_{resolv,\tau_i}$ and an
automorphism $\theta:\frkc\rTo \frkc_{\tau_i,y,\zeta}$ such
that
\begin{equation}
\sup_p |\bu_{y,\zeta}\circ \theta(p)-\bu(p)|<\varepsilon.
\end{equation}
So if $\varepsilon$ is small, we can use $\theta$ to define
\begin{equation}
\mbox{emb}_{y,\zeta,\theta}:E_{\tau_i}\rTo
C^\infty(\cC;\times_i(\LL_i\otimes \Lambda^{0,1}(\cC))).
\end{equation}
Therefore, the perturbed Witten equation we should discuss is
\begin{equation}
(WI)_\frkc\equiv 0\;\mbox{mod} \;\oplus_{\sigma\in \Omega{\tau_i}}
\mbox{emb}_{y,\zeta,\theta}(E_{\tau_i})\label{cr}.
\end{equation}

\begin{rem} Note that if $\beta$ is the cut-off section on
$\frkc_{y,\zeta}$ which is used to define the perturbed Witten
map, then we use $\theta^*\beta$ as the corresponding quantity to
define the perturbed Witten map on $\frkc$.
\end{rem}

This means  if $(\frkc, \bu)$ is a $W$-section, then the
obstruction bundle on it is $\oplus_{\sigma\in
\Omega_{\tau_i}}\\
emb_{y,\zeta,\theta}(E_{\tau_i}) .$

However, there is an ambiguity in choosing the embedding
$\mbox{emb}_{y,\zeta,\theta}$, because of the possible existence
of unstable (soliton) components of $\frkc$. To get rid of the
ambiguity, we choose a similar normalization condition as in (3)
of Theorem \ref{thm:knbhd}. We choose the minimal extra marked
points as done in the proof of Theorem \ref{thm:knbhd}. We also
let $\bu_\tau$ (where $\tau$ is some $\tau_i$) be an immersion
near these extra marked points.

For each new marked point $p$, take an embedded
$(2N-2)$-dimensional disk $\D_p$ in $\C^N$, which is transversal
to $\im(\bu_\tau)$ at $\bu_\tau(p)$. We assume that
$\D_{\phi(p)}=\D_p$ when $p$ and $\phi(p)$ are marked points, and
$\phi\in \aut(\tau)$. We choose those $\theta$ such that
\begin{equation}
\bu\circ\theta^{-1}(p)\in \D_p\label{norm2}.
\end{equation}
Now there are only finitely many $\theta$ satisfying
(\ref{norm2}), which is $\aut(\tau)$-invariant and the action of
$\aut(\tau)$ is described in the appendix of \cite{FO}.

Let $\overline{\mbox{emb}_{y,\zeta}}:E_{\tau_i}\rTo
C^\infty(\cC;\times_i(\LL_i\otimes \Lambda^{0,1}(\cC)))$ be the
average of the map $\mbox{emb}_{y,\zeta,\theta}$ defined for
$s\in E_{\tau_i}$, as 
$$
\overline{\mbox{emb}_{y,\zeta}}(s)=\frac{\sum_\theta
\mbox{emb}_{y,\zeta,\theta}(s)}{|\aut(\tau_i)|}.
$$
So the equation (\ref{cr}) can be modified as
\begin{equation}
(WI)_\frkc\equiv 0\;\text{mod}\;E_\frkc\label{norm-eq},
\end{equation}
where
$E_\frkc:=\oplus_{\tau_i}\overline{\mbox{emb}_{y,\zeta}}(E_{\tau_i})$.
Now the definition of the equation (\ref{norm-eq}) is independent
of the choice of $\sigma$.

Using Theorem \ref{thm:knbhd}, we can construct for any point
$\sigma\in \MMrs_{g,k}(\Gamma;\bgamma,\bvkappa)$ a Kuranishi
neighborhood $(U_\sigma,E_\sigma,s_\sigma,\Psi_\sigma)$ with respect
to the equation (\ref{norm-eq}).

We next construct the coordinate change. We choose $\{U_\sigma\}$
satisfying the following condition:
\begin{equation}
\text{if}\;\rho\in \im \Psi'_\sigma \;\text{and if}\;\rho\in
\Omega_{\tau_i},\;\text{then}\;\sigma\in \Omega_{\tau_i}.
\label{inside}
\end{equation}

(\ref{inside}) is true since $\Omega_{\tau_i}$ is closed.
(\ref{inside}) implies if $\rho\in \im\Psi'_\sigma\cap
\MMrs_{g,k}(\Gamma;\bgamma,\bvkappa)$, then $E_\rho\subset
E_\sigma$. Hence if $(\frkc, \bu)$ is a solution of (\ref{norm-eq})
for $\rho$ with the condition (\ref{norm2}), then it solves
(\ref{norm-eq}) for $\sigma$. Thus we find the required embeddings
$\phi_{\sigma\rho}:U_\rho\to U_\sigma$ and
$\hat{\phi}_{\sigma\rho}:E_\rho\rTo E_\sigma$. So
$\{\phi_{\sigma\rho}\}:(U_\rho,E_\rho,\aut(\rho),s_\rho,\Psi_\rho)\to
(U_\sigma,E_\sigma,\aut(\sigma),s_\sigma,\Psi_\sigma)$ becomes the
morphism defined in Definition \ref{df-kran-stru}. It is easy to
check that those data construct a Kuranishi structure of
$\MMrs_{g,k}(\Gamma;\bgamma,\bvkappa)$. By induction, we can assume
we have constructed a Kuranishi structure in a neighborhood of those
strata $\MMrs_{g,k}(\Gamma';\bgamma,\bvkappa)$ for $\Gamma'\prec
\Gamma$.
 Now we want to construct the Kuranishi structure near
 $\MMrs_{g,k}(\Gamma;\bgamma,\bvkappa)$. We already have finitely many $\tau_i$ contained in some
 $\MMrs_{g,k}(\Gamma';\bgamma,\bvkappa)$ with
 $\Gamma'\prec\Gamma$ and maps
 $\Psi'_{\tau_i}: s^{-1}_{\tau_i}(0)\cap U_{\tau_i} \to
 \MMrs_{g,k}(\bgamma,\bvkappa)$ such that
 $\MMrs_{g,k}(\Gamma;\bgamma,\bvkappa)$ minus the union of
 images of $\Psi'_{\tau_i}$ is compact. We then choose finitely
 many $\tau'_i$ on $\MMrs_{g,k}(\Gamma;\bgamma,\bvkappa)$
 such that
 $$
 \cup_i\im \Psi'_{\tau_i}\cup_i\im \Psi'_{\tau'_i}\supset
 \MMrs_{g,k}(\Gamma;\bgamma,\bvkappa).
$$
Choose closed subsets $\Omega_{\tau_i}\subset \im
\Psi'_{\tau_i},\Omega_{\tau'_i}\subset\im\Psi'_{\tau_i}$ such that
their interiors cover $\MMrs_{g,k}(\Gamma;\bgamma,\bvkappa)$. For
each $\sigma\in \MMrs_{g,k}(\Gamma;\bgamma,\bvkappa)$, we put
\begin{equation}
E_\sigma=\oplus_{\sigma\in \Omega_{\tau_i}}E_{\tau_i}\oplus
\oplus_{\sigma\in \Omega_{\tau'_i}}E_{\tau'_i}.
\end{equation}

On this this space we define a similar equation to (\ref{norm-eq})
and also require that condition (\ref{norm2}) holds. So by the same
argument used to study the first stratum, we can extend the
Kuranishi structure to a neighborhood of
$\MMrs_{g,k}(\Gamma;\bgamma,\bvkappa)$.

Finally, we prove that the Kuranishi structure in the interior of
$\MMrs_{g,k}(\Gamma;\bgamma,\bvkappa)$ is orientable. To prove that
it is orientable, Fukaya-Ono showed it is enough to show the tangent
bundle is stably almost complex. The key point in their proof is
that the symbol of the linearization of the Cauchy-Riemann equation
is complex linear and the moduli space $\MM_{g,k}$ is a complex
orbifold which has a complex tangent bundle. Then they use a family
of operators connecting the Cauchy-Riemann operator and its complex
linear part and change the orientation problem of the Cauchy-Riemann
operator to the orientation problem of its complex linear part. In
this way, they proved that the Kuranishi structure is orientable. It
is the same situation in our case. The linearization operator is
only real linear, not complex linear, but the symbol of the first
order term is complex linear. We can prove that the interior of our
moduli space is orientable using their proof almost word for word.
The only difference is that we don't connect our operator directly
to $\bpat$ but to some complex linear operator having $0$-order
terms, since the Sobolev spaces we use are based on a noncompact
surface. The detailed proof can be seen in section 16 of \cite{FO}.

Now if the moduli space $\MMrs_{g,k}(\bgamma,\bvkappa)$ is strongly
regular, we have
$\MMrs_{g,k}(\bgamma,\bvkappa)=\WW_{g,k}(\bgamma,\bvkappa)$. By
abstract Kuranishi theory, we have the following conclusion.

\begin{thm}\label{thm-stro-kura} Suppose that the moduli space $\MMr_{g,k}(\bgamma,\bvkappa)$
is strongly regular, then $\MMr_{g,k}(\bgamma,\bvkappa)$ is a
compact Hausdorff space carrying an orientable Kuranishi structure
$\{U_\sigma, E_\sigma, \aut(\sigma), s_\sigma, \Psi_\sigma\}$ with
(real) dimension $6g-6+2k-2D-\sum_{i=1}^k N_{\gamma_i}$, where
$D=\hat{c}_W(g-1)+\sum_\tau \iota(\gamma_\tau)$, where
$\hat{c}_W=\sum_i(1- 2q_i)$ and $\iota(\gamma_\tau)=\sum_{i}(
\Theta_i^{\gamma_\tau}-q_i)$. $[\MMr_{g,k}(\bgamma,\bvkappa)]^{vir}$
becomes the virtual fundamental cycle of its Kuranishi structure.
\end{thm}

\
\subsection{The neighborhood around a BPS soliton section}\label{kura-neig-sing}\

We always assume that there is only one group elment
$\tilde{\gamma}\in G$ such that the perturbed polynomial
$W_{\tilde{\gamma}}+W_{0,\tilde{\gamma}}$ has only two critical
points $\kappa^\pm$ such that
$\mbox{Im}(W_{\tilde{\gamma}}+W_{0,\tilde{\gamma}})(\kappa^+)=\mbox{Im}(W_{\tilde{\gamma}}+W_{0,\tilde{\gamma}})(\kappa^-
)$. We consider a BPS soliton $W$-section $(\frkc,\bu)$ in
$\MMrs_{g,k}(\bgamma,\bvkappa)$. We say a nodal point $p$ of
$(\frkc,\bu)$ is decorated by $(\tilde{\gamma}, \kappa)$ if there
exists a rigidification $\psi_p$ such that $\psi_p(\bu(p))=\kappa$.
Note that the combinatorial type of this kind of soliton section is
finite. Several cases may occur for a soliton $W$-section in
$\MMrs_{g,k}(\bgamma,\bvkappa)$:
\begin{enumerate}
\item There is only one marked point decorated by
$(\tilde{\gamma}, \kappa^-)$ and there is no nodal point decorated
by $(\tilde{\gamma}, \kappa^-)$.

\item There is no marked point decorated by $(\tilde{\gamma},
\kappa^-)$ and there is only one nodal point decorated by
$(\tilde{\gamma}, \kappa^-)$.

\item There is no marked point decorated by $(\tilde{\gamma},
\kappa^-)$ and there are several nodal points decorated by
$(\tilde{\gamma}, \kappa^-)$.

\item There is one marked point decorated by $(\tilde{\gamma},
\kappa^-)$ and there are several nodal points decorated by
$(\tilde{\gamma}, \kappa^-)$.
\end{enumerate}

To construct the neighborhoods of BPS soliton sections of the
above cases, we need to know more details about a BPS soliton
$\bu_{+-}\in S_{\tilde{\gamma}}(\kappa^+,\kappa^-)$.

\subsubsection*{Kernel and cokernel of $D_{\bu_{+-}}(WI)$}

The linearized operator
$$
D_{\bu_{+-}}(WI): L^p_1(\R\times S^1, \LL_1)\times\dots\times
L^p_1(\R\times S^1\LL_N)\rTo L^p(\R\times S^1,
\LL_1\otimes\Lambda^{0,1})\times\dots\times L^p(\R\times S^1,
\LL_N\otimes\Lambda^{0,1})
$$
is split into the direct sum of two operators:
$D_{\bu_{+-}}(WI)=D^R_{\bu_{+-}}(WI)\oplus D^N_{\bu_{+-}}(WI)$,
where
$$
D^R_{\bu_{+-}}(WI): L^p_1(\R\times S^1,
\C^{N_{\tilde{\gamma}}})\rTo L^p(\R\times S^1,
\C^{N_{\tilde{\gamma}}}\otimes\Lambda^{0,1})
$$
has the form
$$
D^R_{\bu_{+-}}(WI)(\phi_1,\dots,\phi_{N_{\tilde{\gamma}}})
=(\bpat_\xi
\phi_1-2\sum_{j=1}^{N_{\tilde{\gamma}}}\overline{\frac{\pat^2(W_{\tilde{\gamma}}+W_{0,\tilde{\gamma}})}{\pat
u_1\pat u_j}\phi_j},\dots, \bpat_\xi
\phi_{N_{\tilde{\gamma}}}-2\sum_{j=1}^{N_{\tilde{\gamma}}}
\overline{\frac{\pat^2(W_{\tilde{\gamma}}+W_{0,\tilde{\gamma}})}{\pat
u_{N_{\tilde{\gamma}}}\pat u_j}}\phi_j),
$$
and
$$
D^N_{\bu_{+-}}(WI): L^p_1(\R\times
S^1,\LL^{N_{\tilde{\gamma}}+1})\times\dots\times L^p_1(\R\times
S^1,\LL^N)\rTo L^p(\R\times S^1,
\LL^{N_{\tilde{\gamma}}+1}\otimes\Lambda^{0,1})\times L^p(\R\times
S^1, \LL^{N}\otimes\Lambda^{0,1})
$$
has the form
$$
D^N_{\bu_{+-}}(WI)(\phi_{N_{\tilde{\gamma}}+1},\dots,
\phi_t)=(\bpat_\xi
\phi_1-2\sum_{j=N_{\tilde{\gamma}}+1}^{t}\overline{\frac{\pat^2
W_{N}}{\pat u_{N_{\tilde{\gamma}}+1}\pat u_j}\phi_j},\dots,
\bpat_\xi \phi_{t}-2\sum_{j=N_{\tilde{\gamma}}+1}^{N}
\overline{\frac{\pat^2 W_{N}}{\pat u_{N}\pat u_j}\phi_j}).
$$
Note that we have the decomposition $W=W_N+W_{\tilde{\gamma}}$
according to the action of $\tilde{\gamma}$.

\begin{prop}\label{prop-BPS-C0norm} Suppose $\bu_{+-}$ is a BPS soliton. If the perturbation parameters
$b_i$ in $W_{0,\tilde{\gamma}}$ are sufficiently small, then the
$C^0$ norm of $\bu_{+-}$ is also sufficiently small.
\end{prop}
\begin{proof} Since $(WI)(\bu_{+-})=0$, we have
$$
\mbox{Re}(W_{\tilde{\gamma}}+W_{0,\tilde{\gamma}})(\kappa^-)-\mbox{Re}(W_{\tilde{\gamma}}+W_{0,\tilde{\gamma}})(\kappa^+)
=\int^\infty_{-\infty} \sum_i |\frac{\pat
(W_{\tilde{\gamma}}+W_{0,\tilde{\gamma}})}{\pat u_i}|^2.
$$
Hence if the perturbation is small enough, the integral of the
right hand side is small enough, which implies that the pointwise
norm $|\frac{\pat (W_{\tilde{\gamma}}+W_{0,\tilde{\gamma}})}{\pat
u_i}|$ is small enough. This is because  the $C^1$ norm of each
section $u_i$ can be uniformly bounded by a constant which is
independent of the perturbation parameters $b_i$ in
$W_{0,\tilde{\gamma}}$ when all $b_i$ are less than $1$ (cf.
Theorem \ref{thm-solu-noda}). Therefore $|\frac{\pat W}{\pat
u_i}|=|\frac{\pat W_{\tilde{\gamma}}}{\pat u_i}|$ is sufficiently
small. By nondegeneracy of $W$, the absolute value of each section
$u_i$ should be sufficiently small.
\end{proof}

\begin{lm}\label{lm-soli-1} If the perturbation parameters
$b_i$ in $W_{0,\tilde{\gamma}}$ are sufficiently small, then the
linearized operator $D^N_{\bu_{+-}}(WI)$ is an isomorphism.
\end{lm}

\begin{proof} If we set $\bphi^N:=(\phi_{N_{\tilde{\gamma}}+1},\dots,
\phi_N)^T$ and $A^N(s):=(-2\frac{\pat^2 W_{N}}{\pat u_{i}\pat
u_j})_{N_{\tilde{\gamma}}+1\le i,j\le N}$, then
$$
D^N_{\bu_{+-}}(WI)(\bphi^N)=\bpat_\xi
\bphi^N+\overline{A^N(s)\cdot \bphi^N}.
$$
Here $\xi=s+i\theta$ and the matrix $A^N$ depends on $\bu_{+-}$.
Define matrices
\begin{align*}
&\Theta_N=\diag(\Theta^{\tilde{\gamma}}_{N_{\tilde{\gamma}}+1},\dots,
\Theta^{\tilde{\gamma}}_N)\\
&\eta(\Theta_N)=\diag(e^{-i\theta\Theta^{\tilde{\gamma}}_{N_{\tilde{\gamma}}+1}},\dots,
e^{-i\theta\Theta^{\tilde{\gamma}}_N}).
\end{align*}
Then the multiplication $\eta(\Theta_N)$ from $L^p_1(\R\times
S^1,\LL^{N_{\tilde{\gamma}}+1})\times\dots\times L^p_1(\R\times
S^1,\LL^N)$ to $L^p_1(\R\times S^1, \C^{N-N_{\tilde{\gamma}}})$
and from $L^p(\R\times S^1,
\LL^{N_{\tilde{\gamma}}+1}\otimes\Lambda^{0,1})\times L^p(\R\times
S^1, \LL^{N}\otimes\Lambda^{0,1})$ to $L^p(\R\times S^1,
\C^{N-N_{\tilde{\gamma}}}\otimes\Lambda^{0,1}) $ are isomorphisms.
The operator $\eta(\Theta_N)\circ D^N_{\bu_{+-}}(WI)\circ
\eta(\Theta_N)^{-1}$ is changed to the form $\bpat_\xi \cdot+
\Theta_N\cdot+\overline{A^N\cdot}$, which is a small perturbation
of the operator $\bpat_\xi \cdot+ \Theta_N\cdot$. Since $\bpat_\xi
\cdot+ \Theta_N\cdot$ is an isomorphism, then $\bpat_\xi \cdot+
\Theta_N \cdot+\overline{A^N\cdot}$ is also an isomorphism by
Proposition \ref{prop-BPS-C0norm} if the perturbation parameters
$b_i$ are sufficiently small. Thus we know that
$D^N_{\bu_{+-}}(WI)$ is an isomorphism.

\end{proof}

Now we study the transversality of $D^R_{\bu_{+-}}(WI)$. Define
$$
H_{+-}:=\{(u_1,\dots,u_{N_{\tilde{\gamma}}})\in \C^{n_{\gamma}}|
\mbox{Im}(W_{\tilde{\gamma}}+W_{0,\tilde{\gamma}})(u_1,\dots,u_{N_{\tilde{\gamma}}})
=\mbox{Im}(W_{\tilde{\gamma}}+W_{0,\tilde{\gamma}})(\kappa^+)\}.
$$
This is a real codimension 1 submanifold in
$\C^{N_{\tilde{\gamma}}}$.

\begin{lm}\label{lm-mors-smal} Suppose the BPS soliton $\bu_{+-}$ to be a Morse-Smale
flow on the manifold $H_{+-}$. If the perturbation parameters
$b_i$ in $W_{0,\tilde{\gamma}}$ are sufficiently small, then the
kernel $V_{+-}$ and cokernel $E_{+-}$ of the linearized operator
$D^R_{\bu_{+-}}(WI): L^p_1(\R\times S^1,
\C^{N_{\tilde{\gamma}}})\rTo L^p(\R\times S^1,
\C^{N_{\tilde{\gamma}}}\otimes\Lambda^{0,1})$ are $1$-dimensional
and are generated by $\bu_s$ and $i\bu_s$ respectively.
\end{lm}

\begin{proof} Set $\bphi^R:=(\phi_1,\dots,
\phi_{N_{\tilde{\gamma}}})^T$ and $A^R(s):=(-2\frac{\pat^2
(W_{\tilde{\gamma}}+W_{0,\tilde{\gamma}}) }{\pat u_{i}\pat
u_j})_{1\le i,j\le N_{\tilde{\gamma}}}$, then
$$
D^R_{\bu_{+-}}(WI)(\bphi^R)=\bpat_\xi
\bphi^R+\overline{A^R(s)\cdot \bphi^R}.
$$
By Proposition \ref{prop-BPS-C0norm}, the $C^0$ norm of $A^R(s)$
is sufficiently small if the $b_i$ in $W_{0,\tilde{\gamma}}$ are
sufficiently small. Now it is a well-known fact in symplectic
geometry that if the $C^0$ norm of $A^R(s)$ is small enough, then
the kernel and cokernel of $D^R_{\bu_{+-}}(WI)$ are the same as
the kernel and cokernel of the operator
$$
D^{R,\R}_{\bu_{+-}}(WI): L^p_1(\R,
\bu_{+-}^*T\C^{N_{\tilde{\gamma}}})\rTo L^p(\R,
\bu_{+-}^*T\C^{N_{\tilde{\gamma}}}).
$$
(see the proof on Page 1038 of \cite{FO} or \cite{CZ, On}).

Since all the BPS solitons connecting $\kappa^+$ and $\kappa^-$
should lie in the hypersurface $H_{+-}$, we have $\ker
D^{R,\R}_{\bu_{+-}}(WI)\subset L^p_1(\R, \bu_{+-}^*T H_{+-})$.
Since we assume that $\bu_{+-}$ is a Morse-Smale flow on $H_{+-}$,
$\ker D^{R,\R}_{\bu_{+-}}(WI)$ is just the $1$-dimensional space
generated by the field $\frac{\pat \bu_{+-}}{\pat s}$. On the
other hand, the dual operator
$(D^{R,\R}_{\bu_{+-}}(WI))^*=-\pat_s\cdot+\overline{A^R(s)\cdot}$.
It is easy to see that $i\frac{\pat \bu_{+-}}{\pat s}$ satisfies
$(D^{R,\R}_{\bu_{+-}}(WI))^*(\bu)=0$. Therefore it generates the
$1$-dimensional cokernel.
\end{proof}

Let $H^{\tilde{\gamma}}_{para}\in \C^{N_{\tilde{\gamma}}}$ be the
set in the parameter space of $b_i$ such that there exist two
critical points $\kappa^+$ and $\kappa^-$ of
$W_{\tilde{\gamma}}+W_{0,\tilde{\gamma}}$ satisfying
$\mbox{Im}(W_{\tilde{\gamma}}+W_{0,\tilde{\gamma}})(\kappa^-)
=\mbox{Im}(W_{\tilde{\gamma}}+W_{0,\tilde{\gamma}})(\kappa^+)$. By
Theorem \ref{thm-para-chamber}, $H^{\tilde{\gamma}}_{para}$ is a
union of finitely many real hypersurfaces.

\begin{thm}\label{thm-soli-kura-nbhd} For generic points on $H^{\tilde{\gamma}}_{para}$, the
kernel and cokernel spaces of $D_{\bu_{+-}}(WI)$ are just the
$1$-dimensional spaces $V_{+-}$ and $E_{+-}$. We can choose a
Kuranishi neighborhood of $(\R\times S^1, \bu_{+-})$ as $(\{pt\},
E_{+-}, s\equiv 0)$.
\end{thm}

\begin{proof} For generic points on $H^{\tilde{\gamma}}_{para}$,
the function $-2Re(W_{\tilde{\gamma}}+W_{0,\tilde{\gamma}})$ is a
Morse function on $H_{+-}$ such that $\bu_{+-}$ is a Morse-Smale
flow on $H_{+-}$. By Lemmas \ref{lm-soli-1} and
\ref{lm-mors-smal}, we finish the proof.
\end{proof}

\begin{rem} The cokernel space $E_{+-}$ can be generated by a
compactly supported element in $L^p(\R,
\bu_{+-}^*T\C^{N_{\tilde{\gamma}}})$, since we can multiply the
section $i\frac{\pat \bu}{\pat s}$ by a cut-off function with
sufficiently large compact support.
\end{rem}

Now we begin the construction of the neighborhoods of BPS soliton
sections case by case.

Case (1).  Suppose that $\bvkappa=(\bvkappa',\kappa^+)$. Take
$(\frkc_1,\bu_1)\in \MMrs_{g,k}(\bgamma,\bvkappa)$ and a BPS soliton
$\bu_{+-}\in S_{\tilde{\gamma}}(\kappa^+,\kappa^-)$. Then
$(\frkc_1\#(\R\times S^1), \bu_1\#\bu_{+-})\in
\MMrs_{g,k}(\bgamma,\bvkappa',\kappa^-)$. We define a map:
$$
Glue^m_{\bu_{+-}}:
\MMrs_{g,k}(\bgamma,\bvkappa',\kappa^+)\rTo
\MMrs_{g,k}(\bgamma,\bvkappa',\kappa^-)
$$
by
$$
Glue^m_{\bu_{+-}}(\bu_1):=\bu_1\#\bu_{+-}.
$$
Here we want to say words about the gluing. Note that $\bu_{+-}$
is defined in $\C^N$ space. To do the gluing, we firstly map
$\bu_1$ to $\C^N$ by the rigidification and then do gluing in
$\C^N$ space. The obtained element inherits the same
rigidification.

By Theorem \ref{thm-stro-kura}, we can give an oriented Kuranishi
structure to each strongly regular moduli space
$\MMr_{g,k}(\bgamma,\bvkappa)$. We repeat the procedures in the last
section to construct the neighborhood of $Glue^m_{\bu_{+-}}(\bu_1)$.
Let $(U_\sigma, E_\sigma, s_\sigma)$ be a Kuranishi neighborhood of
$\bu_1$ in $\MMrs_{g,k}(\bgamma,\bvkappa',\kappa^+)$. By Theorem
\ref{thm-soli-kura-nbhd}, we can take the Kuranishi neighborhood of
$(\R\times S^1, \bu_{+-})$ to be $(\{pt\}, E_{+-}, s\equiv 0)$. Now
we can use our gluing construction again:
\begin{enumerate}
\item Assume that $U_\sigma=\Vd\times \Vr\times \Vm$. Let
$\beta\in \Vd\times \Vr$ and $\zeta\in D_\varepsilon(0)\in \C$. We
can get the nearby curve $\frkc_{\beta,\zeta}$.

\item Take the obstruction bundle $E_{\beta,\zeta}=E_\sigma\oplus
E_{+-}$.

\item Use the implicit function theorem to construct the Kuranishi
map $s_{\sigma, \zeta}$ on $U_\sigma\times D_\varepsilon(0)$.
Because of the $S^1$ action, the ``Kuranishi" neighborhood of the
boundary point $(\frkc_{1,\sigma}\#(\R\times S^1),
\bu_{1,\sigma}\#\bu_{+-})$ is $((U_\sigma\times
[0,\varepsilon_\sigma]), E_\sigma\oplus E_{+-}, s_{\sigma,
\zeta})$, where the length $\varepsilon_\sigma$ of the cone may
depend on the point $\sigma$. We understand that the section
$s_{\sigma, \zeta}$ depends only on the first coordinate of
$\zeta$.
\end{enumerate}

Now we have the following lemma.

\begin{lm}\label{lm-BPS-kura1} Let $(U_\sigma, E_\sigma, s_\sigma)$ be a Kuranishi neighborhood of
 $(\frkc_1,\bu_1)\in \MMrs_{g,k}(\bgamma,\bvkappa',\kappa^+)$,
then $(U_\sigma\times [0,\varepsilon_\sigma], E_\sigma\oplus
E_{+-}, s_{\sigma, \zeta})$ is a Kuranishi neighborhood of
$Glue^m_{\bu_{+-}}(\bu_1)$ in
 $\MMrs_{g,k}(\bgamma,\bvkappa',\kappa^-)$. 
There exists a
 sequence of smooth
 multisections $s_{\sigma, \zeta, n}$, transversal to the zero
 section, such that it converges to $s_{\sigma, \zeta}$ in the $C^0$
 topology. For sufficiently large $n$, the zero set $s_{\sigma, \zeta,
 n}^{-1}(0)$ is diffeomorphic to $s_{\sigma, 0, n}^{-1}(0)\times
 [0,\varepsilon]$.
\end{lm}

\begin{proof}  The approximating sequence of multisections $s_{\sigma, \zeta,
n}$ is given by Lemma 17.4 in \cite{FO}.
\end{proof}

Case (2). There are still two types of  gluing operations called
tree gluing and loop gluing in this case.

\subsubsection*{Tree gluing}\quad Suppose that $g=g_1+g_2,
k=k_1+k_2, \bgamma=(\bgamma_1,\bgamma_2)$ and
$\bvkappa=(\bvkappa_1,\bvkappa_2)$. Let $(\frkc_1,\bu_1)\in
\MMr_{g_1,k_1+1,W}(\bgamma_1,\tilde{\gamma};\bvkappa_1,\kappa^+),
(\frkc_2,\bu_2)\in
\MMr_{g_2,k_2+1,W}(\bgamma_2,\tilde{\gamma}^{-1};\bvkappa_2,\kappa^-)$
and let $\bu_{+-}\in S_{\tilde{\gamma}}(\kappa^+,\kappa^-)$ be a BPS
soliton. Then $(\frkc_1\#(\R \times S^1)\#\frkc_2),
\bu_1\#\bu_{+-}\#\bu_2)$ is a BPS soliton $W$-section in
$\MMrs_{g,k}(\bgamma;\bvkappa)$. We define the map from
$\MMr_{g_1,k_1+1,W}(\bgamma_1,\tilde{\gamma};\bvkappa_1,\kappa^+)
\times\MMr_{g_2,k_2+1,W}(\bgamma_2,\tilde{\gamma}^{-1};\bvkappa_2,\kappa^-)$
to $\MMrs_{g,k}(\bgamma,\bvkappa)$ to be
$Glue^n_{\bu_{+-}}(\bu_1,\bu_2):=\bu_1\#\bu_{+-}\#\bu_2$.

Let $A$ be a space; we define the cone based on $A$ with length
$\varepsilon$ to be $C_\varepsilon(A)$.

Let $(U_{\sigma_i}, E_{\sigma_i}, s_{\sigma_i})$ be a Kuranishi
neiborhood of $(\frkc_i,\bu_i)$ for $i=1,2$. Then
$(U_{\sigma_1}\times U_{\sigma_2}, E_{\sigma_1}\oplus
E_{\sigma_2}, s_{\sigma_1}\oplus s_{\sigma_2})$ is a neighborhood
of $\bu_1\times \bu_2$. Let $U_{\sigma_i}=V_{deform,i}\times
V_{resol,i}\times V_{map,i}$. We can begin our gluing procedures:

\begin{enumerate}
\item The parameter space for the curves is $V_{deform,1}\times
V_{resol,1}\times V_{deform,2}\times V_{resol,2}\times
D_\varepsilon(0)\times D_\varepsilon(0)$. We take a point
$(\beta_1,\beta_2,\zeta_1,\zeta_2)$, then we can get a nearby curve
$\frkc_{\beta_1,\beta_2,\zeta_1,\zeta_2}\in \WW_{g,k}(\bgamma)$,
where $\beta_i\in V_{deform,i}\times V_{resol,i}$ and
$(\zeta_1\times \zeta_2)\in D_\epsilon(0)\times D_\epsilon(0)$.

\item Set the obstruction bundle on
$\cC_{\beta_1,\beta_2,\zeta_1,\zeta_2}$ to be $E_{\sigma_1}\oplus
E_{\sigma_2}\oplus E_{+-}$.

\item From $\bu_1,\bu_{+-},\bu_2$ we get the approximating
solution $\bu_{app,\beta_1,\beta_2,\zeta_1,\zeta_2}$. Use the
implicit function theorem to obtain the Kuranishi map
$s_{\sigma_1,\sigma_2,\zeta_1,\zeta_2}$ on $U_{\sigma_1}\times
U_{\sigma_2}\times D_\varepsilon(0)\times D_\varepsilon(0)$. Here
the radius of these domains may shrink. Since $\bu_{+-}$ is
$S^1$-invariant, the approximating solution is also $S^1$-invariant.
Hence the moduli space we need should be modulo the $S^1$ action.
The $S^1$ group gives a diagonal action to the neighborhood
$D_\varepsilon(0)\times D_\varepsilon(0)$ while keeping the  other
parameter space fixed. We take the neighborhood
$C_\varepsilon(S^3)\subset D_\varepsilon(0)\times D_\varepsilon(0)$.
We have $C_\varepsilon(S^3)/S^1=C_\varepsilon(S^2)$.  Hence we
obtain a neighborhood of the point $\bu_1\#\bu_{+-}\#\bu_2$ in
$\MMrs_{g,k}(\bgamma,\bvkappa)$:
$$
(U_{\sigma_1}\times U_{\sigma_2}\times
C_{\varepsilon_\sigma}(S^2), E_{\sigma_1}\oplus E_{\sigma_2}\oplus
E_{+-}, s_{\sigma_1,\sigma_2,[\zeta_1,\zeta_2]}),
$$
where $\varepsilon_\sigma$ depends on the point
$\sigma_1\times\sigma_2$, and $[\zeta_1,\zeta_2]$ is the element
in the quotient space
$C_{\varepsilon_\sigma}(S^2)=C_{\varepsilon_\sigma}(S^3)/S^1\subset
\C\times \C$.
\end{enumerate} Using the same method  as Lemma \ref{lm-BPS-kura1}, we
can chose a uniform $\varepsilon>0$ for the Kuranishi structure on
$\MM_{g_1,k_1+1,W}(\bgamma_1,\tilde{\gamma};\bvkappa_1,\kappa^+)
\times\MM_{g_2,k_2+1,W}(\bgamma_2,\tilde{\gamma}^{-1};\bvkappa_2,\kappa^-)$.

Similar to Lemma \ref{lm-BPS-kura1}, we have the following lemma.

\begin{lm}[\textbf{Tree}]\label{lm-BPS-kura2-Tree} Let $(U_{\sigma_1}\times U_{\sigma_2}, E_{\sigma_1}\oplus E_{\sigma_2},
s_{\sigma_1}\oplus s_{\sigma_2})$ be a Kuranishi neighborhood of
$(\bu_1,\bu_2)\in
\MMr_{g_1,k_1+1,W}(\bgamma_1,\tilde{\gamma};\bvkappa_1,\kappa^+)
\times\MMr_{g_2,k_2+1,W}(\bgamma_2,\tilde{\gamma}^{-1};\bvkappa_2,\kappa^-)$.
Then $(U_{\sigma_1}\times U_{\sigma_2}\times C_{\varepsilon}(S^2),
E_{\sigma_1}\oplus E_{\sigma_2}\oplus E_{+-},
s_{\sigma_1,\sigma_2,[\zeta_1,\zeta_2]})$ is a Kuranishi
neighborhood of $Glue^n_{\bu_{+-}}(\bu_1,\bu_2)$ in
 $\MMrs_{g,k}(\bgamma,\bvkappa',\kappa^-)$. There exists a
 sequence of smooth multisections $s_{\sigma_1,\sigma_2,[\zeta_1,\zeta_2],n}$ which is transversal to the zero
 section such that it converges to $s_{\sigma_1,\sigma_2,[\zeta_1,\zeta_2]}$ in the $C^0$
 topology. For sufficiently large $n$, the zero set $s_{\sigma_1,\sigma_2,[\zeta_1,\zeta_2],n}^{-1}(0)$
 is diffeomorphic to $s_{\sigma_1,\sigma_2,[0,0],n}^{-1}(0)\times
 C_\varepsilon(S^2)$.
\end{lm}

\begin{proof} Since the cone $C_\varepsilon(S^2)$ is homeomorphic to
a closed small neighborhood $N_\varepsilon$ of $\R^3$, we can use
the approximation theorem on the  section
$s_{\sigma_1,\sigma_2,[\zeta_1,\zeta_2]}$ over $U_{\sigma_1}\times
U_{\sigma_2}\times N_\varepsilon$.
\end{proof}

\subsubsection*{Loop gluing}\quad Let $(\frkc_1,\bu_1)\in
\MMr_{g,k+2,W}(\bgamma,\tilde{\gamma},\tilde{\gamma}^{-1};\bvkappa,
\kappa_-,\kappa_+)$ and $(\R\times S^1,\bu_{+-})\in
S^{\tilde{\gamma}}(\kappa_+,\kappa_-)$. Then we can glue
$(\frkc_1,\bu_1)$ and $(\R\times S^1,\bu_{+-})$ between the two
marked points on $\frkc_1$ decorated by $(\tilde{\gamma}, \kappa_-)$
and $(\tilde{\gamma}^{-1},,\kappa_+)$ to obtain an element in
$\MMrs_{g,k}(\bgamma,\bvkappa)$. Denote this element by
$Glue^n_{\bu_{+-}}(\bu_1)$. In the same way, we have the following:

\begin{lm}[\textbf{Loop}]\label{lm-BPS-kura2-loop} Let $(U_{\sigma_1}, E_{\sigma_1},
s_{\sigma_1})$ be a Kuranishi neighborhood of
$\MMr_{g,k+2,W}(\bgamma,\tilde{\gamma},\tilde{\gamma}^{-1};\bvkappa,
\kappa_-,\kappa_+)$. Then $(U_{\sigma_1}\times
C_{\varepsilon}(S^2), E_{\sigma_1}\oplus E_{+-},
s_{\sigma_1,[\zeta_1,\zeta_2]})$ is a Kuranishi neighborhood of
$Glue^n_{\bu_{+-}}(\bu_1)$ in
 $\MMrs_{g,k}(\bgamma,\bvkappa',\kappa^-)$. There exists a
 sequence of smooth multisections $s_{\sigma_1,[\zeta_1,\zeta_2],n}$ which is transversal to the zero
 section such that it converges to $s_{\sigma_1,[\zeta_1,\zeta_2]}$ in the $C^0$
 topology. For sufficiently large $n$, the zero set $s_{\sigma_1,[\zeta_1,\zeta_2],n}^{-1}(0)$
 is diffeomorphic to $s_{\sigma_1,[0,0],n}^{-1}(0)\times
 C_\varepsilon(S^2)$.
\end{lm}

Case (3) The gluing operation will become more complicated because
of the possible tree gluing and the possible loop gluing. Here we
only consider the simplest gluing which does not contain any loop
gluing.

Let $\bu_{+-}^i, i=1,\dots, l$ be $l$ BPS solitons in
$S^{\tilde{\gamma}}(\kappa^+,\kappa^-)$. Assume that
$g=g_1+\dots+g_{l+1}, k=k_1+\dots+k_{l+1}$ and the index group $
(\bgamma, \bvkappa)$ is divided into $l+1$ parts:
$(\bgamma_1,\bvkappa_1),\dots, (\bgamma_{l+1},\bvkappa_{l+1})$.
Take $(\frkc_1,\bu_1)\in
\MMr_{g_1,k_1+1,W}(\bgamma_1,\tilde{\gamma};\bvkappa_1,\kappa^+),
(\frkc_2,\bu_2)\in \MMr_{g_2,k_2+2,W}(\bgamma_2,\tilde{\gamma}^{-1},
\tilde{\gamma};\bvkappa_2,\kappa^-,\kappa^+),\dots,
(\frkc_{l+1},\bu_{l+1})\in
\MMr_{g_{l+1},k_{l+1}+1,W}(\bgamma_{l+1},\tilde{\gamma}^{-1};\bvkappa_{l+1},\kappa^-)$;
we can define the gluing map from
$\MMr_{g_1,k_1+1,W}(\bgamma_1,\tilde{\gamma};\bvkappa_1,\kappa^+)\times
\dots\times
\MMr_{g_{l+1},k_{l+1}+1,W}(\bgamma_{l+1},\tilde{\gamma}^{-1};\bvkappa_{l+1},\kappa^-)$
to $ \MMr_{g,k}(\bgamma,\bvkappa)$ as
$$
Glue^n_{\bu_{+-}^i}(\bu_1,\dots,\bu_{l+1}):=(\frkc_1\#(\R\times
S^1)\#\dots\#\frkc_{l+1},
\bu_1\#\bu^1_{+-}\#\dots\#\bu^l_{+-}\#\bu_{l+1}).
$$

Similarly, we have the following conclusion.

\begin{lm}\label{lm-BPS-kura3} Let $(U_{\sigma_1}\times\dots\times U_{\sigma_{l+1}}, E_{\sigma_1}\oplus\dots\oplus E_{\sigma_{l+1}},
s_{\sigma_1}\oplus\dots\oplus s_{\sigma_{l+1}})$ be a Kuranishi
neighborhood of
$(\bu_1,\dots,\bu_{l+1})\in\MMr_{g_1,k_1+1,W}(\bgamma_1,\tilde{\gamma};\bvkappa_1,\kappa^+)
\times \dots\times
\MMr_{g_{l+1},k_{l+1}+1,W}(\bgamma_{l+1},\tilde{\gamma}^{-1};\bvkappa_{l+1},\kappa^-)$.
Then $(U_{\sigma_1}\times\dots\times
U_{\sigma_{l+1}}\times\\
\substack{\underbrace{C_{\varepsilon}(S^2)\times C_{\varepsilon}(S^2)}\\
l}, E_{\sigma_1}\oplus\dots\oplus E_{\sigma_{l+1}}\oplus
\substack{\underbrace{E_{+-}\oplus\dots\oplus E_{+-}}\\ l},
s_{\sigma_1,\dots,\sigma_{l+1},[\zeta_1,\zeta_2]_1,\dots,[\zeta_1,\zeta_2]_l})$
is a Kuranishi neighborhood of
$Glue^n_{\bu_{+-}^i}(\bu_1,\dots,\bu_{l+1})$ in
$\MMrs_{g,k}(\bgamma,\bvkappa)$. There exists a
 sequence of smooth multisections $s_{\sigma_1,\dots,\sigma_{l+1},[\zeta_1,\zeta_2]_1,\dots,[\zeta_1,\zeta_2]_l,n}$ which are transversal to the zero
 section such that it converges to $s_{\sigma_1,\dots,\sigma_{l+1},[\zeta_1,\zeta_2]_1,\dots,[\zeta_1,\zeta_2]_l}$ in the $C^0$
 topology. For sufficiently large $n$, the zero set $s_{\sigma_1,\dots,\sigma_{l+1},[\zeta_1,\zeta_2]_1,\dots,[\zeta_1,\zeta_2]_l,n}^{-1}(0)$
 is diffeomorphic to $s_{\sigma_1,\dots,\sigma_{l+1},[0,0]_1,\dots,[0,0]_l,n}^{-1}(0)\times
 \substack{\underbrace{C_{\varepsilon}(S^2)\times C_{\varepsilon}(S^2)}\\ l}$.
\end{lm}

Case (4). Like in Case (3), we only give the description of the
simplest treegluing.

Take the same notations as in Case (3) except that we require
$$
\bu_{l+1}\in\MMr_{g_{l+1},k_{l+1}+2,W}(\bgamma_{l+1},\tilde{\gamma}^{-1},
\tilde{\gamma};\bvkappa_{l+1},\kappa^-, \kappa^+).
$$
Let $\bu_{+-}^{l+1}\in S^{\tilde{\gamma}}(\kappa^+,\kappa^-)$.
Then we define the gluing operation:
$$
Glue^m_{\bu_{+-}^i}(\bu_1,\dots,\bu_{l+1}):=Glue^m_{\bu^{l+1}_{+-}}((Glue^n_{\bu_{+-}^i}(\bu_1,\dots,\bu_{l+1}))).
$$

\begin{lm}\label{lm-BPS-kura4}Let $(U_{\sigma_1}\times\dots\times U_{\sigma_{l+1}},
E_{\sigma_1}\oplus\dots\oplus E_{\sigma_{l+1}},
s_{\sigma_1}\oplus\dots\oplus s_{\sigma_{l+1}})$ be a Kuranishi
neighborhood of
$(\bu_1,\dots,\bu_{l+1})\in\MMr_{g_1,k_1+1,W}(\bgamma_1,\tilde{\gamma};\bvkappa_1,\kappa^+)
\times \dots\times
\MMr_{g_{l+1},k_{l+1}+1,W}(\bgamma_{l+1},\tilde{\gamma}^{-1};\bvkappa_{l+1},\kappa^-)$.
Then $(U_{\sigma_1}\times\dots\times
U_{\sigma_{l+1}}\times\\
\substack{\underbrace{C_{\varepsilon}(S^2)\times C_{\varepsilon}(S^2)}\\
l}\times [0, \varepsilon_\sigma], E_{\sigma_1}\oplus \dots\oplus
E_{\sigma_{l+1}}\oplus
\substack{\underbrace{E_{+-}\oplus\dots\oplus E_{+-}}\\ l+1},
s_{\sigma_1,\dots,\sigma_{l+1},[\zeta_1,\zeta_2]_1,\dots,[\zeta_1,\zeta_2]_l,
\zeta_{l+1}})$ is a Kuranishi neighborhood of
$Glue^m_{\bu_{+-}^i}(\bu_1,\dots,\bu_{l+1})\in
\MMrs_{g,k}(\bgamma,\bvkappa',\kappa^-)$. There exists a
 sequence of smooth multisections $s_{\sigma_1,\dots,\sigma_{l+1},[\zeta_1,\zeta_2]_1,\dots,[\zeta_1,\zeta_2]_l,\xi_{l+1},n}$ which are transversal to the zero
 section such that it converges to $s_{\sigma_1,\dots,\sigma_{l+1},[\zeta_1,\zeta_2]_1,\dots,[\zeta_1,\zeta_2]_l,\xi_{l+1}}$ in the $C^0$
 topology. For sufficiently large $n$, the zero set
 $s_{\sigma_1,\dots,\sigma_{l+1},[\zeta_1,\zeta_2]_1,\dots,[\zeta_1,\zeta_2]_l,\xi_{l+1},n}^{-1}(0)$
 is diffeomorphic to $s_{\sigma_1,\dots,\sigma_{l+1},[0,0]_1,\dots,[0,0]_l,0,n}^{-1}(0)\times
 \substack{\underbrace{C_{\varepsilon}(S^2)\times C_{\varepsilon}(S^2)}\\ l}\times [0,\varepsilon].$
\end{lm}

\begin{df} If the BPS soliton section $(\frkc,\bu)\in
\MMrs_{g,k}(\bgamma,\bvkappa)$ satisfies Case 2 or Case 3, we call
it a cone point in $\MMrs_{g,k}(\bgamma,\bvkappa)$.
\end{df}

By Lemma \ref{lm-BPS-kura2-Tree}, \ref{lm-BPS-kura2-loop}, and Lemma
\ref{lm-BPS-kura3}, we know that the cone point also carries a
Kuranishi neighborhood and hence is actually an interior point of
the moduli space $\MMrs_{g,k}(\bgamma,\bvkappa)$.

\begin{thm}\label{thm-sing-neib} Suppose that the moduli space $\MMrs_{g,k}(\bgamma,\bvkappa)$
is regular but not strongly regular. Then we have the following
conclusions:
\begin{enumerate}
\item if $(\bgamma, \bvkappa)$ does not contain the pair
$(\tilde{\gamma},\kappa^-)$, then $\MMrs_{g,k}(\bgamma,\bvkappa)$ is
a compact Hausdorff space carrying an orientable Kuranishi
structure. Its (real) dimension is $6g-6+2k-2D-\sum_{i=1}^k
N_{\gamma_i}$, where $D=\hat{c}_W(g-1)+\sum_\tau
\iota(\gamma_\tau)$, where $\hat{c}_W=\sum_i(1- 2q_i)$ and
$\iota(\gamma_\tau)=\sum_{i}( \Theta_i^{\gamma_\tau}-q_i)$.

\item if $(\bgamma, \bvkappa)$ contains the pair
$(\tilde{\gamma},\kappa^-)$, then $\MMrs_{g,k}(\bgamma,\bvkappa)$ is
a compact Hausdorff space carrying an orientable Kuranishi structure
with boundary. Its (real) dimension is $6g-6+2k-2D-\sum_{i=1}^k
N_{\gamma_i}$, where $D$ is given as above. The boundary points
consist of those BPS soliton $W$-sections satisfying Case (1) and
(4). Their neighborhoods in $\MMrs_{g,k}(\bgamma,\bvkappa)$ are
characterized by Lemma \ref{lm-BPS-kura1} and \ref{lm-BPS-kura4}.
The boundary of $\MMrs_{g,k}(\bgamma,\bvkappa)$ together with its
Kuranishi structure is diffeomorphic to that of
$\MMrs_{g,k}(\bgamma',\tilde{\gamma}^{-1},\bvkappa', \kappa^+)\#
(S^{\tilde{\gamma}}(\kappa^+,\kappa^-)/\R)^{S^1}$, where
$(S^{\tilde{\gamma}}(\kappa^+,\kappa^-)/\R)^{S^1}$ is the
$S^1$-invariant set of $(S^{\tilde{\gamma}}(\kappa^+,\kappa^-))/\R$,
i.e., the geometrical set of BPS solitons connecting $\kappa^+$ and
$\kappa^-$.
\end{enumerate}
\end{thm}

\begin{proof}
Proof of (1). $\MMrs_{g,k}(\bgamma,\bvkappa)$ is a compact Hausdorff
space, which is proved by Theorem \ref{thm-gromov-comp}. It is a
stratified space stratified by the partial order $\succ$. Note that
the decoration $(\tilde{\gamma},\kappa^-)$ may still occur at some
nodal points. The minimum strata are composed of soliton sections.
To construct the Kuranishi structure of
$\MMrs_{g,k}(\bgamma,\bvkappa)$, we first construct the Kuranishi
neighborhoods of those soliton $W$-sections. If the soliton section
is a non-BPS soliton, then we have already constructed its
neighborhood. If the soliton is a BPS soliton section, then its
neighborhood has been constructed by Lemmas \ref{lm-BPS-kura2-loop},
\ref{lm-BPS-kura2-Tree} and \ref{lm-BPS-kura3}. Subsequently, we
extend them to the interior part of the moduli space. We use the
gluing method described in the last section to get a global
Kuranishi structure. Notice that those cone points are also interior
points of the moduli space. The dimension is given by Theorem
\ref{ind-witt-oper}.

One can argue as in Section \ref{sec-inter-gluing} that this
Kuranishi structure is also stably almost complex and hence is
orientable.

Proof of (2). In this case, the moduli space
$\MMrs_{g,k}(\bgamma,\bvkappa)$ is a compact Hausdorff space with
boundary $\MMrs_{g,k}(\bgamma',\tilde{\gamma}^{-1},\bvkappa',
\kappa^+)\# (S^{\tilde{\gamma}}(\kappa^+,\kappa^-)/\R)^{S^1}$.
Boundary points consists of those BPS soliton sections satisfying
Case (1) and (4). To construct the Kuranishi structure with boundary
over $\MMrs_{g,k}(\bgamma,\bvkappa)$, we begin our construction of
neighborhoods of the cone points and the boundary points.
Neighborhoods of cone points are constructed by Lemma
\ref{lm-BPS-kura2-loop}, \ref{lm-BPS-kura2-Tree} and
\ref{lm-BPS-kura3}. By (1), we can choose an oriented Kuranishi
structure $\{(U_\sigma, E_\sigma, s_\sigma)\}$ of
$\MMrs_{g,k}(\bgamma,\bvkappa',\kappa^+)$ such that the induced
orientation of Kuranishi structure $\{(U_\sigma\times
[0,\varepsilon_\sigma], E_\sigma\oplus E_{+-}, s_{\sigma, \zeta})\}$
is compatible with the orientation of the neighborhoods of the
interior points. After construction of neighborhoods of these
special points, we can use the same gluing method shown in Section
\ref{sec-inter-gluing} to build an oriented Kuranishi structure of
the whole moduli space.
\end{proof}

\section{Proof of the axioms}\label{sec:proveAx}

\subsection{The virtual cycle and quantum Picard-Lefschetz theory}\

\

In this part, we will define the virtual cycle
$[\MMr_{g,k}(\bgamma,\bvkappa)]^{vir}$ in $H_*(\MMr_{g,k}(\bgamma))$
which is given by its oriented Kuranishi structure. The Kuranishi
structure may depend on many choices, but we will show that the
virtual cycle depends only on $g,k,W$ and the Lefschetz thimbles and
is independent of the other choices used to construct the Kuranishi
structure. The interesting thing is that we can see another version
of Picard-Lefschetz theory at the level of moduli space.

Let $\MMr_{g,k}(\bgamma,\bvkappa)$ be a strongly regular moduli
space, then by Theorem \ref{thm-stro-kura}
$\MMr_{g,k}(\bgamma,\bvkappa)$ carries an oriented Kuranishi
structure $\{(U_\sigma, E_\sigma, s_\sigma)\}$(even almost stably
complex). By Theorem \ref{thm-vir-cycl}, the section $s_\sigma$ can
be perturbed to a transversal smooth multisection
$\{\tilde{s}_\sigma\}$ such that
$\cup_\sigma\tilde{s}_\sigma^{-1}(0)$ forms a cycle which is
independent of the choice of the perturbed multisection
$\{\tilde{s}_\sigma\}$.

We define a map $\Pi:\MMr_{g,k}(\bgamma,\bvkappa)\rTo
\MMr_{g,k}(\bgamma)$ by
$$
[(\frkc, \bu)]\longrightarrow [(\frkc)].
$$
Note that $\Pi$ is a system of map germs $\Pi_\sigma:
U_\sigma\rTo \MMr_{g,k}(\bgamma)$.

It is easy to see that this map is strongly continuous. By Theorem
\ref{thm-stro-kura} and Theorem \ref{thm-vir-cycl}, we have the
following definition.

\begin{df} We define the homology class
$[\MMr_{g,k}(\bgamma,\bvkappa)]^{vir}:= [\Pi(\tilde{s}^{-1}(0))]\in
H_*(\MMr_{g,k}(\bgamma))$. Its (real) dimension is
$6g-6+2k-2D-\sum_{i=1}^k
N_{\gamma_i}=2((\hat{c}_W-3)(1-g)+k-\sum_{\tau=1}^k\iota(\gamma_\tau))-\sum_{i=1}^k
N_{\gamma_i}$. Furthermore, if $\Gamma$ is a decorated stable
$W$-graph, then we have the homology class $[\MMr_{g,k}(\Gamma;
\bgamma,\bvkappa)]^{vir}:=[\Pi(\tilde{s}_\Gamma^{-1}(0)))]\in
H_*(\MMr_{g,k}(\Gamma;\bgamma))$, where $\tilde{s}_\Gamma$ is the
perturbed multisection over the Kuranishi structure of
$\MMr_{g,k}(\Gamma;\bgamma,\bvkappa)$. Its real dimension is
$6g-6+2k-2D_\Gamma-\sum_{i=1}^k
N_{\gamma_i}-2\#E(\Gamma)=2((\hat{c}_W-3)(1-g)+k-\sum_{\tau\in
T(\Gamma)}\iota(\gamma_\tau)-\#E(\Gamma))-\sum_{i=1}^k
N_{\gamma_i}$, where $E(\Gamma)$ is the set of  edges of $\Gamma$.
\end{df}

For each $\gamma\in G $, we have a perturbed polynomial
$W_{\gamma}+W_{0,\gamma}$, where $W_{0,\gamma}$ is assumed to be
$W_\gamma$-regular. Recall that
$W_{0,\gamma}(\bu)=\sum_{i=1}^{N_\gamma} \beta_j u_j$, where
$N_\gamma$ is the number of  broad sections with respect to the
action of $\gamma$. Now we consider a path of the perturbation
$\lambda: [-1,1]\rTo
W_{0,\gamma}^\lambda(\bu)=\sum_{i=1}^{N_\gamma} (\lambda) u_j$
such that the path of the perturbation is still
$W_\gamma$-regular. Such a path is generic in the path space of
the perturbation parameter space. Assume that the $i$-th critical
point of the path is
$\kappa^i(\lambda)=(\kappa^i_1(\lambda),\dots,\kappa_{N_\gamma}^i(\lambda))$,
for $i=1,\dots, \mu_\gamma$, where $\mu_\gamma$ is the
multiplicity of $W_{\gamma}+W_{0,\gamma}$. These critical points
are all nondegenerate critical points. We know from Section 3 that
there are real hypersurfaces in the perturbed coefficient space
$\C^{N_\gamma}$ separating $\C^{N_\gamma}$ into a system of
chambers. When the path $(b_1(\lambda),\dots,
b_{N_\gamma}(\lambda))$ crosses one hypersurface, e.g., at
$\lambda=0$, then correspondingly there exist two critical points
$\kappa^{+}$ and $\kappa^{-}$ such that
\begin{equation}\label{axim-proof-1}
\mbox{Im}(W_{\gamma}+W_{0,\gamma})(\kappa^{+})=\mbox{Im}(W_{\gamma}+W_{0,\gamma})(\kappa^{-}).
\end{equation}
Since for different $\gamma\in G$, we can take different
perturbations,  in the following discussion, we always fix the
perturbations about all group elements $\gamma$ except for one
group element $\tilde{\gamma}$. For $\tilde{\gamma}$ we choose a
path of perturbation such that all of its critical points are
fixed except for one depending on $\lambda$. We denote it by
$\kappa^+(\lambda)$. We assume that the perturbation crosses only
one chamber at $\lambda=0$. Namely, if $\lambda\neq 0$, the
perturbation is always strongly regular, and at $\lambda=0$, there
exists only one critical point $\kappa^{-}$ satisfying
(\ref{axim-proof-1}). The existence of this perturbation path is
generic in the path space. We call this path of perturbation the \emph{
generic crossing path.}

Now our Kuranishi structure depends on (1) the metric we chose near
the marked points (we choose the cylindrical metric), (2) the cut-off
functions $\beta_j, \varpi$ used to define the perturbed Witten
map, (3) the obstruction bundle $E_\sigma$ and (4) the partition of unity.
There are two natural ways to choose the metrics near marked
points which correspond to the ``Smooth theory" and the ``Cylindrical theory,"
as mentioned in the introduction. The different choice of the
metrics will significantly change the Witten map and the Witten
equation. In physics, people think these two theories are
equivalent in some sense. This is just the conjecture we proposed
in the introduction. Here we fix the cylindrical metric, and
consider the influence of other choices.

Let $\bgamma=(\bgamma',\tilde{\gamma})$ and
$\bvkappa^\pm=(\bvkappa',\kappa^\pm)$. We fix $\bvkappa=(\bvkappa',
\kappa^-)$ and choose a generic crossing path of perturbation with
the crossing point at $\lambda=0$ and satisfying
$\kappa^+(\lambda)<\kappa^-$, i.e.,
$\mbox{Re}(W_{\gamma}+W_{0,\gamma})(\kappa^{+})<\mbox{Re}(W_{\gamma}+W_{0,\gamma})(\kappa^{-}).
$ Then we get a family of perturbed Witten equations
$(WI)(\lambda)(\bu)=0$ , and this family induces a family of moduli
spaces $\MMr_{g,k}(\bgamma',\tilde{\gamma};
\bvkappa',\kappa^-;\lambda)$ for $\lambda\in
[-1+\varepsilon,1-\varepsilon]$.

Since $\MMr_{g,k}(\bgamma',\tilde{\gamma};
\bvkappa',\kappa^-;\pm(1-\varepsilon))$ are strongly regular moduli
spaces, they have oriented (stable almost complex) Kuranishi
structures $\V_\pm=\{(U_{\sigma}(\pm), E_{\sigma}(\pm),
s_{\sigma}(\pm))\}$. The Kuranishi structure can be trivially
extended to the spaces $\MMr_{g,k}(\bgamma',\tilde{\gamma};
\bvkappa',\kappa^-; -1+\varepsilon)\times [-1,-1+\varepsilon]$ and
$\MMr_{g,k}(\bgamma',\tilde{\gamma}; \bvkappa',\kappa^-;
1-\varepsilon)\times [1-\varepsilon,1]$.

Define

$$
\WW_{g,k}^\lambda(\kappa^-)=\cup_{\lambda\in
[-1,1]}\{\lambda\}\times \MMrs_{g,k}(\bgamma',\tilde{\gamma};
\bvkappa',\kappa^-;\lambda).
$$

We can define the Gromov convergence on
$\WW_{g,k}^\lambda(\kappa^-)$ in the same way as in Section 7. The
sole difference is that the sequence $(\frkc^n, \bu^n)$ may depend
on the extra parameter $\lambda$. Similarly, we can prove that
$\WW_{g,k}^\lambda(\kappa^-)$ is a compact Hausdorff space.

Our aim is to examine the change of the virtual cycle
$\MMr_{g,k}(\bgamma',\tilde{\gamma}; \bvkappa',\kappa^-;\lambda)$
when $\lambda$ changes from positive to negative.

\begin{thm}\label{thm-cobo-sing} $\WW_{g,k}^\lambda(\kappa^-)$ is a compact Hausdorff space
carrying an orientable Kuranishi structure with boundaries. The
boundaries appear at $\lambda=\pm 1$ and $\lambda=0$. When
$\lambda=\pm 1$, the boundaries correspond to
$\MMr_{g,k}(\bgamma',\tilde{\gamma}; \bvkappa',\kappa^-;\pm 1)$.
When $\lambda=0$, the boundary is
$\MMrs_{g,k}(\bgamma',\tilde{\gamma}; \bvkappa',\kappa^+)\#
S^{\tilde{\gamma}}(\kappa^+,\kappa^-)$, where
$S^{\tilde{\gamma}}(\kappa^+,\kappa^-)$ is the space of $BPS$
solitons flowing from the critical point $\kappa^+$ to $\kappa^-$.
\end{thm}

\begin{proof} The proof of $\WW_{g,k}^\lambda(\kappa^-)$ to be a compact Hausdorff space
under Gromov convergence is similar to the proof of Theorem
\ref{thm-gromov-comp}. The next thing is to construct an oriented
Kuranishi structure over $\WW_{g,k}^\lambda(\kappa^-)$.

We  already have the Kuranishi structure at the boundary points
corresponding to $\lambda=\pm 1$. When $\lambda=0$, the
perturbation of $W_{\tilde{\gamma}}$ is regular but not strongly
regular. The BPS soliton $W$-section may occur. We now consider
the neighborhoods near those BPS soliton sections. We still use
the notation from Section \ref{kura-neig-sing}.

By Theorem \ref{thm-sing-neib} we can give an oriented Kuranishi
structure to the regular but not strongly regular moduli space
$\MMr_{g,k}(\bgamma,\tilde{\gamma};\bvkappa', \kappa^+)$.

The BPS soliton sections in $\WW_{g,k}^\lambda(\kappa^-)$ can also
be divided into four cases according to Section
\ref{kura-neig-sing}. We only construct the Kuranishi neighborhoods
for those BPS soliton sections satisfying Case 1, and the other
cases can be treated in the same way.

Case 1. Let $(\frkc_1,\bu_1)\in
\MMrs_{g,k}(\bgamma,\bvkappa',\kappa^+)$ be a non-soliton section
and let $\bu_{+-}$ be a BPS soliton. Then $(\frkc_1\#(\R\times S^1),
\bu_1\#\bu_{+-})\in \WW_{g,k}^\lambda(\kappa^-)$ is a BPS soliton
section satisfying Case 1 of Section \ref{kura-neig-sing}.

Let $(U_\sigma, E_\sigma, s_\sigma)$ be a Kuranishi neighborhood of
$\bu_1$ in $\MMrs_{g,k}(\bgamma,\bvkappa',\kappa^+)$. By Theorem
\ref{thm-soli-kura-nbhd}, we can take the Kuranishi neighborhood of
$(\R\times S^1, \bu_{+-})$ to be $(\{pt\}, E_{+-}, s\equiv 0)$. Now
we can use our gluing construction again:
\begin{enumerate}
\item Assume that $U_\sigma=\Vd\times \Vr\times \Vm$. Let
$\beta\in \Vd\times \Vr$ and $\zeta\in D_\varepsilon(0)\in \C$. We
can get the nearby curve $\frkc_{\beta,\zeta}$ and the
approximating solution $\bu_{app, \beta,\zeta,\lambda}\equiv
\bu_{app, \beta,\zeta},$ which is defined as before since the
boundary values at marked points are fixed.

\item Take the obstruction bundle
$E_{\beta,\zeta}=\theta_{\beta,\zeta}(E_\sigma\oplus E_{+-})$,
where $\theta_{\beta,\zeta}$ is defined as before. Notice that the
obstruction bundle is independent of $\lambda$. Now the equation
we study is
$$
D_{y,\zeta}(WI(\lambda))(\bphi)=0 \mod E_{\beta,\zeta}.
$$
However, we find that the linearized operator
$D_{y,\zeta}(WI(\lambda))$ is independent of the perturbation
term, i.e., independent of $\lambda$. Thus all the Lemmas from
\ref{lm-appro-1} to \ref{lm-kuramodel-0} hold without any change.

\item Now we want to apply the implicit function theorem to the
operator $WI(\lambda)$ on $\frkc_{y,\zeta}$. We define
$$
F_{\beta,\zeta,\lambda}(\bphi):=WI_{\beta,\zeta,\lambda}(\uapp+\bphi).
$$
We have
$$
F_{\beta,\zeta,\lambda}(0)=WI_{\beta,\zeta,\lambda}(\uapp),\;DF_{\beta,\zeta,\lambda}(0)=
D_\uapp((WI)(0)_{\beta,\zeta})=D_{\beta,\zeta}(WI(0)).
$$
Actually, we only need to modify Lemmas \ref{lm-kuramodel-1} and
\ref{lm-kuramodel-2}. After a simple computation, we find that we
can obtain Lemmas \ref{lm-kuramodel-1} and \ref{lm-kuramodel-2}
with the constants there independent of $\lambda$ near $0$. Then
we can construct our Kuranishi neighborhood $(\Vd\times \Vr\times
D_\varepsilon(0)\times \Vm, E_\sigma\oplus E_{+-},
\hat{s}_\sigma)$ by the same technique. Here the set $\Vm$ is
obtained by using the implicit function theorem and hence should
depend on $\lambda$. However, because of uniform estimates, we can
take the same $\Vm$. Since $S^1$ only acts on $D_\varepsilon(0)$,
when modulo the $S^1$ action, we can obtain the ``Kuranishi"
neighborhood $((U_\sigma\times [0,\varepsilon_\sigma]),
E_\sigma\oplus E_{+-}, \hat{s}_{\sigma})$, where
$U_\sigma=\Vd\times \Vr\times
\Vm,\hat{s}_{\sigma}(\beta,\bphi,0)=s_\sigma(\bphi)$, and the
length $\varepsilon_\sigma$ of the cone may depend on the point
$\sigma$.
\end{enumerate}

Note that the Kuranishi structure $(U_\sigma, E_\sigma, s_\sigma)$
and orientation on $\MMrs_{g,k}(\bgamma,\bvkappa',\kappa^+)$ are
well defined; we get a Kuranishi structure $\{(U_\sigma\times
[0,\varepsilon_\sigma], E_\sigma\oplus E_{+-}, s_{\sigma, \zeta})\}$
with natural orientation of a neighborhood of
$\mbox{Im}(Glue^m_{\bu_{+-}})\in
\MMrs_{g,k}(\bgamma,\bvkappa',\kappa^-)$.

Case 3. If
$(\frkc_1,\bu_1)\in\MMrs_{g,k}(\bgamma,\bvkappa',\kappa^+)$ is a BPS
soliton section , then $(\frkc_1\#(\R\times S^1), \bu_1\#\bu_{+-})$
satisfies Case 3 of Section \ref{kura-neig-sing}. We can still
construct the neighborhoods near those points.

Now we can take a good coordinate system of
$\MMrs_{g,k}(\bgamma,\bvkappa',\kappa^+)$ (see Lemma 6.3 of
\cite{FO} for the existence) which is a finite covering of Kuranishi
neighborhoods. Therefore, we can take the length $\varepsilon_\sigma$
to be the minimal length $\varepsilon$.

Thus one collar of the boundary of $\WW_{g,k}^\lambda(\kappa^-) $ at
$\lambda=0$ is $\MMrs_{g,k}(\bgamma,\bvkappa',\kappa^+)\times
[0,\varepsilon]$.

The cone points of $\WW_{g,k}^\lambda(\kappa^-)$ at $\lambda=0$ can
also occur; one can construct their neighborhoods as done in Case 1,
which is also characterized by Lemma \ref{lm-BPS-kura4}. Now we
constructed the Kuranishi neighborhoods of the boundary points and
the cone points occurring at $\lambda=0$. Using the gluing argument
as before, we can extend the Kuranishi structure covering the
compact set of  boundary points and cone points to the whole space
$\WW_{g,k}^\lambda(\kappa^-)$ and construct an oriented Kuranishi
structure (See the proof of Theorem 17.11 of \cite{FO} for the
extension reason).
\end{proof}

We still take the above special perturbation path. Define
$$
\WW_{g,k}^\lambda(\kappa^+)=\cup_{\lambda\in
[-1,1]}\{\lambda\}\times \MMrs_{g,k}(\bgamma',\tilde{\gamma};
\bvkappa',\kappa^+(\lambda);\lambda).
$$

\begin{thm}\label{thm-cobo-nonsing} $\WW_{g,k}^\lambda(\kappa^+)$ is a compact Hausdorff space
carrying an orientable Kuranishi structure with boundaries. The
boundaries appear only at $\lambda=\pm 1$. When $\lambda=\pm 1$, the
boundaries correspond to $\MMr_{g,k}(\bgamma',\tilde{\gamma};
\bvkappa',\kappa^+(\pm 1))$.
\end{thm}

\begin{proof} The proof is analogous to the proof of Theorem
\ref{thm-cobo-sing} except there is no soliton $W$-section
satisfying Case 1 and 3 of Section \ref{kura-neig-sing}. Therefore
when $\lambda=0$, only the cone points appear. We need to
construct the local charts of the ordinary interior points as well
as the local charts of cone points. There is a subtle difference
in the present construction when compared to the previous
construction of local charts: the critical point
$\kappa^+(\lambda)$ is movable. We should modify our previous
argument.

For example, we consider $\lambda\in (-\varepsilon,\varepsilon)$ and
$\bu_\sigma \in \MMrs_{g,k}(\bgamma',\tilde{\gamma};
\bvkappa',\kappa^+)$  a non-solitoon section. The approximating
solution $\bu_{app,y,\zeta,\lambda}$ on $\cC_{y,\zeta}$ is defined
as follows: near the nodal points we define
$\bu_{app,y,\zeta,\lambda}$ as before; near the marked point
labelled by $\tilde{\gamma}$, we let
$$
\bu_{app,y,\zeta,\lambda}=\bu_\sigma-\kappa^+ +\kappa^+(\lambda).
$$

The linearized operator at the point $(y,\zeta,\lambda)$ is
$$
D_{y,\zeta,\lambda}(WI)(\bphi):=D_{\bu_{app,y,\zeta,\lambda}}((WI)_{y,\zeta})(\bphi)=
\bpat_y \bphi+\overline{A(\bu_{app,y,\zeta,\lambda},y)\bphi}.
$$

The parameter $\lambda$ appears in the nonlinear term. Then we
modify Lemmas \ref{lm-appro-1}--\ref{lm-kuramodel-2} by a trick:
we replace the symbol $y$ representing the deformation parameter
by $y,\lambda$ and $|y|$ by $|y|+|\lambda|$. Then all those lemmas
hold if $y,\zeta,\lambda$ are sufficiently small. Using the
implicit function theorem, we can find a Kuranishi chart for
$\bu_\sigma$.

Similarly, after a small modification, we can construct the
Kuranishi neighborhoods of cone points as done in Lemma
\ref{lm-BPS-kura4}.

Then we have constructed the local chart of each interior point
(including cone points).

We can extend the given Kuranishi structure of a collar of the
boundaries of $\WW_{g,k}^\lambda(\kappa^+)$ at $\lambda=\pm 1$ to
the whole space as done  before to obtain an oriented Kuranishi
structure.
\end{proof}

\subsubsection*{Classical Picard-Lefschetz theory} We will give a
simple description of vanishing cycles, Lefschetz thimbles and
related transformation groups. It has already become a classical
theory. The interested reader can see \cite{Ar, E}, etc.

Now we assume that the perturbation polynomial
$W_{0,\tilde{\gamma}}$ is strongly $W_{\tilde{\gamma}}$-regular
and sufficiently small such that there are exactly
$\mu_{\tilde{\gamma}}$ critical points of
$W_{\tilde{\gamma}}+W_{0,\tilde{\gamma}}$ inside a small ball $B$
centered at $0$. Let $\alpha^i$ be the critical value of
$W_{\tilde{\gamma}}+W_{0,\tilde{\gamma}}$ which lies inside a
small neighborhood $T\subset \C$ corresponding to $B$.
Furthermore, we can require the order of these critical values to
satisfy $\mbox{Im}(\alpha_i)<\mbox{Im}(\alpha_j)$ if $i<j$. Let
$\alpha_*\in \pat T$ be a regular value. We take $T$ small enough
and define
$Y=(W_{\tilde{\gamma}}+W_{0,\tilde{\gamma}})^{-1}(T)\cap B$ and
$Y_*=(W_{\tilde{\gamma}}+W_{0,\tilde{\gamma}})^{-1}(\alpha_*)$.

Take a system of paths $l_i(\tau): [0,1]\rTo \C$ connecting
$\alpha_*$ and $\alpha_i$ such that
\begin{enumerate}
\item the paths $l_i$ have no self-intersections; \item the paths
$l_i$ and $l_j$ intersect only for $\tau=0$, at the point
$\alpha_*=l_i(0)=l_j(0)$; \item the paths $l_i$ are ordered by the
requirement that $\arg l'_i(0)<\arg l'_j(0)$ if $i<j$.
\end{enumerate}

For each path $l_i$, there exists a corresponding {\em simple
loop} $\beta_i$ which goes along $l_i$ from $\alpha_*$ to the
critical value $\alpha_i$, then goes 
counterclockwise around
$\alpha_i$ and finally returns to $\alpha_*$ along $l_i$. The
system of these paths $l_i$'s is called distinguished if the
corresponding system of simple loops $\beta_i$ generates the group
$\pi_1(T',\alpha_*)$, where
$T'=T-\{\alpha_1,\dots,\alpha_{\mu_{\tilde{\gamma}}}\}$.

Each path $l_i$ induces a unique vanishing cycle $\Delta_i\in
H_{N_{\tilde{\gamma}}-1}(Y_*)$ or a Lefschetz thimble $S_i\in
H_{N_{\tilde{\gamma}}}(Y,Y_*)$ up to orientation. In singularity
theory, the set of these cycles or thimbles forms a basis of the
homology group $H_{N_{\tilde{\gamma}}-1}(Y_*)\cong
\Z^{\mu_{\tilde{\gamma}}}$, or the relative homology group
$H_{N_{\tilde{\gamma}}}(Y,Y_*)$, which is called a distinguished
basis of vanishing cycles or thimbles respectively. Let $D_*$
represent the set of all the distinguished bases of vanishing
cycles (or thimbles).

Assume that the boundary of the relative cycle $S_i$ is just
$\Delta_i$; then when taking compatible orientations we have the
connecting isomorphism:
$$
\pat_*: H_{N_{\tilde{\gamma}}}(Y,Y_*)\rTo
H_{N_{\tilde{\gamma}}-1}(Y_*)
$$
such that $\pat_*(S_i)=\Delta_i$.

Each simple loop $\beta_i$ induces the monodromy operators
$$
h_{\Delta_i}: H_{N_{\tilde{\gamma}}-1}(Y_*)\rTo
H_{N_{\tilde{\gamma}}-1}(Y_*)
$$
and
$$
h_{S_i}: H_{N_{\tilde{\gamma}}}(Y,Y_*)\rTo
H_{N_{\tilde{\gamma}}}(Y,Y_*),
$$
which have action given by Picard-Lefschetz theory as follows:
\begin{equation}\label{monodromy-action}
h_{\Delta_i}(\Delta_j)=\Delta_j+R_{j,i}\Delta_i, \forall j
\end{equation}
and
\begin{equation}
h_{S_i}(S_j)=S_j+R_{j,i}S_i, \forall j,
\end{equation}
where
$R_{j,i}=(-1)^{N_{\tilde{\gamma}}(N_{\tilde{\gamma}}+1)/2}\Delta_j\circ
\Delta_i$ and $\Delta_j\circ \Delta_i$ is the intersection number.

The map $\beta_i\mapsto h_{\Delta_i} (h_{S_i})$ induces a
homomorphism $\pi_1(T',\alpha_*)\rTo \aut
H_{N_{\tilde{\gamma}}-1}(Y_*) (\aut
H_{N_{\tilde{\gamma}}}(Y,Y_*))$. The image of the homomorphism is
called the (relative) monodromy group of the singularity
$W_{\tilde{\gamma}}$. It can be proved (shown in \cite{Ar}) that
the (relative) monodromy group only depends on the type of the
singular point of $W_{\tilde{\gamma}}$. The set $\{h_{\Delta_i},
i=1,\dots,\mu_{\tilde{\gamma}}\}$ ($\{h_{S_i},
i=1,\dots,\mu_{\tilde{\gamma}}\}$) generates the (relative)
monodromy group.

There are several groups acting on the set $D_*$(cf. \cite{E}).
Let $(\delta_1,\dots,\delta_{\mu_{\tilde{\gamma}}})$ be a
distinguished basis of vanishing cycles or thimbles. We have
\begin{enumerate}
\item action of $(\Z/2\Z)^{\mu_{\tilde{\gamma}}}$ (change of
orientation)
$$
Or_j(\delta_1,\dots,\delta_{\mu_{\tilde{\gamma}}})=(\delta_1,\dots,\delta_{j-1},-\delta_j,\delta_{j+1},\dots,
\delta_{\mu_{\tilde{\gamma}}}).
$$
\item action of monodromy groups
$$
h_i(\delta_1,\dots,\delta_{\mu_{\tilde{\gamma}}})=(h_i(\delta_1),\dots,h_i(\delta_{\mu_{\tilde{\gamma}}})).
$$
\item action of braid group $Br(\mu_{\tilde{\gamma}})$. Let
$br_1,\dots,br_{\mu_{\tilde{\gamma}}-1}$ be the standard
generators of $Br((\mu_{\tilde{\gamma}}))$, then
\begin{equation}
br_j(\delta_1,\dots,\delta_{\mu_{\tilde{\gamma}}})=(\delta_1,\dots,\delta_{j-1},
h_j(\delta_{j+1}),\delta_j,\delta_{j+2},\dots,\delta_{\mu_{\tilde{\gamma}}}).
\end{equation}

\item action of the symmetric group $Sym_{\mu_{\tilde{\gamma}}}$
$$
\sigma(\delta_1,\dots,\delta_{\mu_{\tilde{\gamma}}})=(\delta_{\sigma(1)},\dots,
\delta_{\sigma(\mu_{\tilde{\gamma}})}), \;\sigma\in
Sym_{\mu_{\tilde{\gamma}}}.
$$

\item Gabrielov thansformations
$$
G_i(j)(\delta_1,\dots,\delta_{\mu_{\tilde{\gamma}}})=
(\delta_1,\dots,\delta_{j-1}, h_i(\delta_j),\delta_{j+1},\dots,
\delta_{\mu_{\tilde{\gamma}}}).
$$
\end{enumerate}

One can also discuss the Dynkin diagrams corresponding to a
distinguished basis. The above group actions also act on the
Dynkin diagrams. The following proposition was proved by Gabrielov
(cf. \cite{E}).

\begin{prop}\label{prop-picar-base} Any element in $D_*$ can be obtained from a fixed element by iterations of
transformations (1) and (3).
\end{prop}

\subsubsection*{Quantum Picard-Lefschetz theory}

By a cobordism argument, the classical Picard-Lefschetz theory can
be seen at the level of the moduli space.

By Theorem \ref{thm-cobo-nonsing}, we have the following corollary.

\begin{crl}\label{crl-cobo-nonsing} The virtual cycle
$[\MMr_{g,k}(\bgamma,\bvkappa)]^{vir}$ in $H_*(\MMr_{g,k}(\bgamma))$
is independent of the  various choices we made to construct the
Kuranishi structure, which include the cut-off functions
$\beta,\varpi$, the obstruction bundle $E_\sigma$ and the partition
of unity. Assume that $\MMr_{g,k}(\bgamma;\bvkappa(\lambda))$ is a
family of moduli spaces depending on a perturbation path. If for
each $\lambda$, the moduli space
$\MMr_{g,k}(\bgamma;\bvkappa(\lambda))$ is strongly regular, then
each $[\MMr_{g,k}(\bgamma;\bvkappa(\lambda))]^{vir}$ defines the
same cohomology class in $H_*(\MMr_{g,k}(\bgamma))$.
\end{crl}

Now assume $\MMr_{g,k}(\bgamma',\tilde{\gamma};\bvkappa',\kappa^i)$
to be a fixed strongly regular moduli space. Here $\kappa^i$ is one
of the critical points of $W_{\tilde{\gamma}}+W_{0,\tilde{\gamma}}$.
Then we obtain $\mu_{\tilde{\gamma}}$ virtual cycles $
[\MMr_{g,k}(\bgamma',\tilde{\gamma};\bvkappa',\kappa^i)]^{vir},
i=1,\dots,\mu_{\tilde{\gamma}}$ in $H_*(\MMr_{g,k}(\bgamma))$. Let
$V_{\tilde{\gamma}}\subset H_*(\MMr_{g,k}(\bgamma))$ be the vector
space generated by these virtual cycles.

Let $\alpha^i$ be the corresponding critical value of $\kappa^i$.
We can further require that the order of the critical values
satisfies $\mbox{Im}(\alpha^i)<\mbox{Im}(\alpha^j)$ if $i<j$. This
order determines a basis of $V_{\tilde{\gamma}}$.

The perturbation parameter $(b_1,\dots, b_{N_{\tilde{\gamma}}})$
lies inside one chamber of $\C^{N_{\tilde{\gamma}}}$. By Corollary
\ref{crl-cobo-nonsing}, the virtual cycles do not change if we
only move the point $(b_1,\dots, b_{N_{\tilde{\gamma}}})$ inside
the same chamber.

However, when the point $(b_1,\dots, b_{N_{\tilde{\gamma}}})$
crosses the wall between two chambers, the basis of
$V_{\tilde{\gamma}}$ is transformed to another basis. The change
is given by the following wall crossing formula.

\begin{thm}[{\bf Wall crossing formula}]\label{crl-cobo-sing} Let
$(b_1(\lambda),\dots, b_{N_{\tilde{\gamma}}}(\lambda)),
\lambda\in [-1,1]$ be a generic crossing path in
$\C^{N_{\tilde{\gamma}}}$. Let
$\{\kappa^1(\pm),\dots,\kappa^i(\pm),\kappa^{i+1}(\pm),\dots,\kappa^{\mu_{N_{\tilde{\gamma}}}}(\pm)\}$
be the set of ordered critical points at $\lambda=\pm 1$. We can
assume that $\kappa^j(\pm)=\kappa^j$ is fixed for $j\neq i$,
$\kappa^i(\pm)=\kappa^i(\lambda=\pm 1)$ and
$\mbox{Im}(\alpha^i(\lambda=0))=\mbox{Im}(\alpha^{i+1})$.\

If the perturbation satisfies $\mbox{Re}
\alpha^i(\lambda)<\mbox{Re} \alpha^{i+1}$, we have the
left-transformation:
\begin{align}
&[\MMr_{g,k}(\bgamma',\tilde{\gamma};\bvkappa',\kappa^j(+))]^{vir}=
[\MMr_{g,k}(\bgamma',\tilde{\gamma};\bvkappa',\kappa^j(-))]^{vir},\;\forall j\neq i,i+1\label{wall-proof-1}\\
&[\MMr_{g,k}(\bgamma',\tilde{\gamma};\bvkappa',\kappa^i(+))]^{vir}=
[\MMr_{g,k}(\bgamma',\tilde{\gamma};\bvkappa',\kappa^{i+1}(-))]^{vir}+\nonumber\\
&R_{i,i+1}\cdot[\MMr_{g,k}(\bgamma',\tilde{\gamma};\bvkappa',\kappa^i(-))]^{vir}\label{wall-proof-2}\\
&[\MMr_{g,k}(\bgamma',\tilde{\gamma};\bvkappa',\kappa^{i+1}(+))]^{vir}=
[\MMr_{g,k}(\bgamma',\tilde{\gamma};\bvkappa',\kappa^i(-))]^{vir},\label{wall-proof-3}
\end{align}
where $R_{i,i+1 }$ is the intersection number defined as above.\

If the perturbation satisfies $\mbox{Re}
\alpha^i(\lambda)>\mbox{Re} \alpha^{i+1}$, we have the
right-transformation:
\begin{align}
&[\MMr_{g,k}(\bgamma',\tilde{\gamma};\bvkappa',\kappa^j(+))]^{vir}=
[\MMr_{g,k}(\bgamma',\tilde{\gamma};\bvkappa',\kappa^j(-))]^{vir},\;\forall j\neq i,i+1\\
&[\MMr_{g,k}(\bgamma',\tilde{\gamma};\bvkappa',\kappa^{i}(+))]^{vir}=
[\MMr_{g,k}(\bgamma',\tilde{\gamma};\bvkappa',\kappa^{i+1}(-))]^{vir},\\
&[\MMr_{g,k}(\bgamma',\tilde{\gamma};\bvkappa',\kappa^{i+1}(+))]^{vir}=
[\MMr_{g,k}(\bgamma',\tilde{\gamma};\bvkappa',\kappa^{i}(-))]^{vir}+\nonumber\\
&R_{i,i+1}\cdot[\MMr_{g,k}(\bgamma',\tilde{\gamma};\bvkappa',\kappa^{i+1}(-))]^{vir}.
\end{align}
\end{thm}

\begin{proof} It suffices to prove
the left transformation formula.\

At first, we shall prove (\ref{wall-proof-1}). Define the moduli
space
$$
\WW_{g,k}^\lambda(\kappa^j)=\cup_{\lambda\in
[-1,1]}\{\lambda\}\times \MMrs_{g,k}(\bgamma',\tilde{\gamma};
\bvkappa',\kappa^j;\lambda).
$$
By Theorem \ref{thm-cobo-nonsing}, this moduli space is a compact
Hausdorff space carrying an orientable Kuranishi structure with
boundaries. Its boundaries consist of two parts:
$\MMrs_{g,k}(\bgamma',\tilde{\gamma}; \bvkappa',\kappa^j(+))$ and
$\MMrs_{g,k}(\bgamma',\tilde{\gamma}; \bvkappa',\kappa^j(-))$, which
are strongly regular moduli spaces. Let $\{(U_\sigma(\pm),
E_{\sigma}(\pm), s_\sigma(\pm))\}$ be a good coordinate system of
Kuranishi neighborhoods of $\MMrs_{g,k}(\bgamma',\tilde{\gamma};
\bvkappa',\kappa^j(\pm))$ respectively. Then by Theorem 6.4 of
\cite{FO}, the section $\{s_\sigma(\pm)\}$ can be approximated by a
sequence of multisections $\{\tilde{s}^\pm_{\sigma,n}\}$. By Lemma
17.4 of \cite{FO}, we can extend the multisections
$\{\tilde{s}^\pm_{\sigma,n}\}$ to the multisections
$\{\tilde{s}^\lambda_{\sigma,n}\}$ over
$\WW_{g,k}^\lambda(\kappa^j)$ such that the restriction
$\tilde{s}^\lambda_{\sigma,n}|_{\lambda=\pm
1}=\tilde{s}^\pm_{\sigma,n}$. By the proof of Theorem 4.9 of
\cite{FO}, the zero set $(\tilde{s}^\lambda_{\sigma,n})^{-1}(0)$ is
a singular chain satisfying
$$
\pat
(\tilde{s}^\lambda_{\sigma,n})^{-1}(0)=(\tilde{s}^+_{\sigma,n})^{-1}(0)-(\tilde{s}^-_{\sigma,n})^{-1}(0).
$$

Define the map $\Pi: \WW_{g,k}^\lambda(\kappa^j)\rTo
\MMr_{g,k}(\bgamma)$ by
$$
[(\frkc, \bu)]\mapsto [(\frkc_0)],
$$
where $\frkc_0$ is the $W$ curve obtained from $\frkc$ by
shrinking the soliton components.

Since the map $\Pi: \WW_{g,k}^\lambda(\kappa^j)\rTo
\MMr_{g,k}(\bgamma)$ is strongly continuous, we have
$$
\pat
\Pi_*((\tilde{s}^\lambda_{\sigma,n})^{-1}(0))=\Pi_*((\tilde{s}^+_{\sigma,n})^{-1}(0))
-\Pi_*((\tilde{s}^-_{\sigma,n})^{-1}(0)).
$$
Therefore, we obtain
$$
[\MMr_{g,k}(\bgamma',\tilde{\gamma};
\bvkappa',\kappa^j(+))]^{vir}=[\MMr_{g,k}(\bgamma',\tilde{\gamma};
\bvkappa',\kappa^j(-))]^{vir}.
$$
In particular, the cobordism argument can be applied to the moduli
space
$$
\cup_{\lambda\in [0,1]}\{\lambda\}\times
\MMrs_{g,k}(\bgamma',\tilde{\gamma}; \bvkappa',\kappa^j;\lambda)
$$
to obtain the following relation:
\begin{equation}
[\MMr_{g,k}(\bgamma',\tilde{\gamma};
\bvkappa',\kappa^j(+))]^{vir}=[\MMrs_{g,k}(\bgamma',\tilde{\gamma};
\bvkappa',\kappa^j; \lambda=0)]^{vir}.
\end{equation}
So we have finished the proof of (\ref{wall-proof-1}).

The proof of (\ref{wall-proof-3}) is almost the same as the proof
of (\ref{wall-proof-1}). The sole difference is that the decorated
index $\kappa^i$ is moving when $\lambda$ changes from $+1$ to
$-1$. In particular, one can also prove
\begin{equation}\label{wall-proof-4}
[\MMrs_{g,k}(\bgamma',\tilde{\gamma};\bvkappa',\kappa^{i};
\lambda=0)]^{vir}=[\MMr_{g,k}(\bgamma',\tilde{\gamma};\bvkappa',\kappa^{i}(-))]^{vir}.
\end{equation}

To prove equality (\ref{wall-proof-2}), we consider the following
moduli space:
$$
\WW_{g,k}^\lambda(\kappa^{i+1})=\cup_{\lambda\in
[-1,1]}\{\lambda\}\times \MMr_{g,k}(\bgamma',\tilde{\gamma};
\bvkappa',\kappa^{i+1};\lambda).
$$
By Theorem \ref{thm-cobo-sing}, this moduli space can carry an
oriented Kuranishi structure with boundaries. The boundaries
consists of three parts:$\MMr_{g,k}(\bgamma',\tilde{\gamma};
\bvkappa',\kappa^{i+1};\lambda=+1)$,
$\MMr_{g,k}(\bgamma',\tilde{\gamma};
\bvkappa',\kappa^{i+1};\lambda=-1)$ and
$\MMrs_{g,k}(\bgamma',\tilde{\gamma};
\bvkappa',\kappa^{i};\lambda=0)\#S^{\tilde{\gamma}}(\kappa^i(0),\kappa^{i+1})$.

If the perturbation path of the parameter $b(\lambda)$ is generic,
then the solitons of
$S^{\tilde{\gamma}}(\kappa^i(0),\kappa^{i+1})$ are Morse-Smale
flows and there are exactly $R_{i,i+1}$ elements, where
$R_{i,i+1}$ is the intersection number of the vanishing cycles
representing the critical points $\kappa^i(0)$ and $\kappa^{i+1}$.

If $\bu_{+-}\in S^{\tilde{\gamma}}(\kappa^i(0),\kappa^{i+1})$ and
$\bu_1\in \MMrs_{g,k}(\bgamma',\tilde{\gamma};
\bvkappa',\kappa^{i};\lambda=0)$, then their neighborhoods in
$\WW_{g,k}^\lambda(\kappa^{i+1})$ are exactly the same as those
described by Lemmas \ref{lm-BPS-kura1} and \ref{lm-BPS-kura4}. Now
using a similar cobordism argument, one can show that
$[\MMr_{g,k}(\bgamma',\tilde{\gamma};
\bvkappa',\kappa^{i+1};\lambda=+1)]^{vir}$ is cobordant to
$[\MMr_{g,k}(\bgamma',\tilde{\gamma};
\bvkappa',\kappa^{i+1};\lambda=-1)]^{vir}$ and $R_{i,i+1}$ pieces of
$[\MMrs_{g,k}(\bgamma',\tilde{\gamma};
\bvkappa',\kappa^{i};\lambda=0)]^{vir}$. Notice that
\begin{align*}
&[\MMr_{g,k}(\bgamma',\tilde{\gamma};
\bvkappa',\kappa^{i+1};\lambda=+1)]^{vir}=[\MMr_{g,k}(\bgamma',\tilde{\gamma};
\bvkappa',\kappa^{i}(+)]^{vir},\\
&[\MMr_{g,k}(\bgamma',\tilde{\gamma};
\bvkappa',\kappa^{i+1};\lambda=-1)]^{vir}=[\MMr_{g,k}(\bgamma',\tilde{\gamma};
\bvkappa',\kappa^{i+1}(-))]^{vir},
\end{align*}
and
$$
[\MMrs_{g,k}(\bgamma',\tilde{\gamma};
\bvkappa',\kappa^{i};\lambda=0)]^{vir}=[\MMr_{g,k}(\bgamma',\tilde{\gamma};\bvkappa',\kappa^{i}(-))]^{vir}
$$
by relation (\ref{wall-proof-4}), and thus we obtain
(\ref{wall-proof-2}).
\end{proof}

\subsubsection*{Correspondence} With the given order of the
critical points $\{\kappa^1,\dots,
\kappa^{\mu_{\tilde{\gamma}}}\}$ or the critical values
$\{\alpha_1,\dots, \alpha_{\mu_{\tilde{\gamma}}}\}$, we have the
corresponding order of the system of paths
$\{l_1,\dots,l_{\mu_{\tilde{\gamma}}}\}$. This system of paths
induces a distinguished basis of vanishing cycles
$\{\Delta_1,\dots, \Delta_{\mu_{\tilde{\gamma}}}\}$ in
$H_{N_{\tilde{\gamma}}-1}$ or thimbles $\{S_1,\dots,
S_{\mu_{\tilde{\gamma}}}\}$ in $H_{N_{\tilde{\gamma}}}(Y, Y_*)$.
For simplicity, in the following we only discuss the
correspondence between virtual cycles in $V_{\tilde{\gamma}}$ and
thimbles in $H_{N_{\tilde{\gamma}}}(Y, Y_*)$. The result is the
same for vanishing cycles.

We define a linear map
$$
\MMr_{g,k}(\bgamma',\tilde{\gamma};\bvkappa',\cdot):
H_{N_{\tilde{\gamma}}}(Y, Y_*)\rTo V_{\tilde{\gamma}}\subset
H_*(\MMr_{g,k}(\bgamma))
$$
by setting the map between bases:
$$
\MMr_{g,k}(\bgamma',\tilde{\gamma};\bvkappa',S_i):=[\MMr_{g,k}(\bgamma',\tilde{\gamma};\bvkappa',\kappa^i)]^{vir}.
$$
In particular, we define
$$
\MMr_{g,k}(\bgamma',\tilde{\gamma};\bvkappa',-S_i):=-[\MMr_{g,k}(\bgamma',\tilde{\gamma};\bvkappa',\kappa^i)]^{vir}.
$$
Theorem \ref{crl-cobo-sing} implies that the generic crossing path
of perturbation provides the action of the braid group on
$V_{\tilde{\gamma}}$:
\begin{align}
&br^L_j\cdot
([\MMr_{g,k}(\bgamma',\tilde{\gamma};\bvkappa',\kappa^1)]^{vir},\dots,
[\MMr_{g,k}(\bgamma',\tilde{\gamma};\bvkappa',\kappa^{\mu_{\tilde{\gamma}}})]^{vir})\\
&=([\MMr_{g,k}(\bgamma',\tilde{\gamma};\bvkappa',\kappa^1)]^{vir},\dots,
[\MMr_{g,k}(\bgamma',\tilde{\gamma};\bvkappa',\kappa^{j-1})]^{vir},\nonumber\\
&h_j([\MMr_{g,k}(\bgamma',\tilde{\gamma};\bvkappa',\kappa^{j+1})]^{vir}),
 [\MMr_{g,k}(\bgamma',\tilde{\gamma};\bvkappa',\kappa^{j})]^{vir},\dots,
 [\MMr_{g,k}(\bgamma',\tilde{\gamma};\bvkappa',\kappa^{\mu_{\tilde{\gamma}}})]^{vir}).
\end{align}
Here $br^L_j$ means the action of the braid group which is given
by left-transformation.\

Hence, we have the following proposition.
\begin{prop}\label{prop-linear} The map
$\MMr_{g,k}(\bgamma',\tilde{\gamma};\bvkappa',\cdot):H_{N_{\tilde{\gamma}}}(Y,
Y_*)\rTo V_{\tilde{\gamma}}\subset H_*(\MMr_{g,k}(\bgamma))
 $ is a
homomorphism of $Br_(\mu_{\tilde{\gamma}})$-modules. Furthermore,
if $\{S'_i\}$ is another distinguished basis of
$H_{N_{\tilde{\gamma}}}(Y, Y_*)$ and $S=\sum_i x'_i S'_i$, then
$$
\MMr_{g,k}(\bgamma',\tilde{\gamma};\bvkappa',S)=\sum_i
x'_i\MMr_{g,k}(\bgamma',\tilde{\gamma};\bvkappa',S'_i).
$$
\end{prop}

\begin{proof} The second conclusion is obtained by
Proposition \ref{prop-picar-base}.
\end{proof}

\subsection{Axioms for the virtual cycle on the space of rigidified $W$-curves}

\begin{thm}\label{thm:main-rigid}
Let $\Gamma$ be a decorated stable $W$-graph (not necessarily
connected) with each tail $\tau\in T(\Gamma)$ decorated by an
element $\gamma_\tau \in \grp$.  Denote by $k:=|T(\Gamma)|$  the
number of tails of $\Gamma$.   There exists a \emph{virtual cycle}
$$\left[\MMr(\Gamma)\right]^{vir} \in
H_*(\MMr(\Gamma),\Q)\otimes \prod_{\tau \in T(\Gamma)}
H_{N_{\gamma_{\tau}}}(\C^N_{\gamma_{\tau}},W^{\infty}_{\gamma_\tau},
\Q)
$$
satisfying the axioms below. When $\Gamma$ has a single vertex of
genus g, k tails, and no edges (i.e, $\Gamma$ is a corolla), we
denote the virtual cycle by
$\left[\MMr_{g,k}(\bgamma)\right]^{vir}$, where $\bgamma
:=(\gamma_1, \dots, \gamma_k)$.

The following axioms are satisfied:
\begin{enumerate}
\item \textbf{Dimension:}\label{ax:rig-dimension} If $D_\Gamma$ is not
an integer,  
then $\left[\MMr(\Gamma)\right]^{vir}=0$. Otherwise, the cycle
$\left[\MMr(\Gamma)\right]^{vir}$ has degree
\begin{equation}\label{eq:dimension}
6g-6+2k-2\#E(\Gamma)-2D_{\Gamma}=2\left((\chat-3)(1-g) + k -\#E(\Gamma) - \sum_{\tau\in
T(\Gamma)} \iota_{\tau}\right).
\end{equation}
So the cycle lies in $H_r(\MMr(\Gamma),\Q)\otimes \prod_{\tau \in
T(\Gamma)}
H_{N_{\gamma_{\tau}}}(\C^N_{\gamma_{\tau}},W^{\infty}_{\gamma_\tau},
\Q),$ where
$$
r:=6g-6+2k -\#E(\Gamma)-2D-\sum_{\tau\in
T(\Gamma)}N_{\gamma_\tau} =
2\left((\hat{c}-3)(1-g)+k-\#E(\Gamma)-\sum_{\tau\in
T(\Gamma)}\iota(\gamma_{\tau})- \sum_{\tau\in
T(\Gamma)}\frac{N_{\gamma_\tau}}{2}\right).
$$

\item \label{ax:rig-symm}\textbf{Symmetric group invariance} There
is a natural $S_k$-action on $\MMr_{g,k}$ obtained by permuting the
tails. This action also induces the permutation of relative homology
and cohomology groups. So for any $\sigma \in S_k$ the induced map
$$
\sigma_*: H_*(\MMr_{g,k},\Q)\otimes \prod_i
H_{N_{\gamma_{i}}}(\C^N_{\gamma_i}, W^{\infty}_{\gamma_i}, \Q) \to
H_*(\MMr_{g,k},\Q)\otimes \prod_i
H_{N_{\gamma_{\sigma(i)}}}(\C^N_{\gamma_{\sigma(i)}},
W^{\infty}_{\gamma_{\sigma(i)}}, \Q)
$$
satisfies
$$
<\sigma_*\left[\MMr(\Gamma)\right]^{vir},
\sigma_*(\alpha)>=<\left[\MMr(\Gamma)\right]^{vir}, \alpha>,
$$
for any $\alpha\in \prod_i H^{N_{\gamma_{i}}}(\C^N_{\gamma_i},
W^{\infty}_{\gamma_i}, \Q)$.

\item \textbf{Disconnected graphs:} Let $\Gamma =\coprod_{i}
\Gamma_i$ be a stable, decorated $W$-graph which is the disjoint
union of connected $W$-graphs $\Gamma_i$. The classes
$\left[\MMr(\Gamma)\right]^{vir}$ and
$\left[\MMr(\Gamma_i)\right]^{vir}$ are related by
\begin{equation}
\left[\MMr(\Gamma)\right]^{vir}= \left[\MMr(\Gamma_1)\right]^{vir}
\times \dots \times \left[\MMr(\Gamma_d)\right]^{vir}.
\end{equation}

\item \textbf{Degenerating connected graphs:}

Let $\Gamma$ be a connected, genus-$g$, stable, decorated $W$-graph
with a single edge $e$ and let $\tilde{i} : \MMr(\Gamma) \to
\MMr_{g,k}(\bgamma)$ denote the canonical inclusion map. Then there
holds
$$
[\MMr_{g,k}(\bgamma)]^{vir}\cap \MMr(\Gamma)=[\MMr(\Gamma)]^{vir},
$$
or in cohomology form
$$[\MMr(\Gamma)]^{vir}=\tilde{i}^*[\MMr_{g,k}(\bgamma)]^{vir}.$$
\smallskip

\item \textbf{Topological Euler class for the Narrow sector:}
Suppose that all the decorations on tails of $\Gamma$ are
  \emph{narrow}, meaning that $\C^N_{\gamma_i}=0$, and so we can
  omit $H_{N_{\gamma_{i}}}(\C^N_{\gamma_i}, W^{\infty}_{\gamma_i}, \Q)=\Q$ from our notation.

Consider the universal rigidified $W$-structure
$(\LL_1,\dots,\LL_N)$ on the universal curve $\pi:\cC \to
\MMr(\Gamma)$ and the two-term complex of sheaves
$$\pi_*(\LL_i){\to} R^1\pi_*(\LL_i).$$
There is a family of maps
$$W_i=\frac{\partial W}{\partial x_i}: \pi_*(\bigoplus_j \LL_j)
\to \pi_*(K\otimes \LL_i^*)\cong R^1\pi_*(\LL_i)^*.$$ The above
two-term complex is quasi-isomorphic to a complex of vector bundles
\cite{PV}
$$E^0_i\xrightarrow{d_i} E^1_i$$
such that
$$\ker( d_i)\to \coker(d_i)$$
is isomorphic to the original two-term complex.
    $W_i$ is naturally extended (denoted by the same notation) to
    $$\bigoplus_i E^0_i\to (E^1_i)^*.$$
    Choosing an Hermitian metric on $E^1_i$ defines an
    isomorphism $\bar{E}^{1*}_i\cong E^1_i.$ Define the \emph{Witten map} to be the following:
    $$\wit=\bigoplus (d_i+\bar{W}_i):
    \bigoplus_i E^0_i\to\bigoplus_i \bar{E}^{1*}_i\cong \bigoplus_i E^1_i.$$
 Let $\pi_j: \bigoplus_i E^j_i\to \MM$ be
    the projection map. The Witten map defines a proper section
    (denoted by the same notation)
    of the bundle $\wit: \bigoplus_i E^0_i\to \pi^*_0 \bigoplus_i
    E^1_i.$ The above data  defines a topological Euler
    class $\eul(\wit: \bigoplus_i E^0_i\to \bigoplus_i
    E^1_i)$. Then,
    $$[\MMr(\Gamma)]^{vir}=\eul(\wit: \pi^*_0 \bigoplus_i E^1_i\to \bigoplus_i
    E^0_i)\cap [\MMr(\Gamma)].$$

    The above axiom implies two subcases.
\begin{enumerate}
\item{\bf Concavity}:\label{ax:convex}\footnote{This axiom was
  called \emph{convexity} in \cite{JKV1} because the original form of
  the construction outlined by Witten in the $A_{r-1}$ case involved
  the Serre dual of $\LL$, which is convex precisely when our $\LL$ is
  concave.}

Suppose that all tails of $\Gamma$ are narrow.  If
$\pi_*\left(\bigoplus_{i=1}^t\LL_i\right)=0$, then the virtual cycle
is given by capping the top Chern class of  $\left(R^1 \pi_*
\left(\bigoplus_{i=1}^t\LL_i\right)\right)$ with the usual
fundamental cycle of the moduli space:
\begin{equation}\begin{split}
\left[\MMr(\Gamma)\right]^{vir}& =
c_{top}\left(\left(R^1\pi_*\bigoplus_{i=1}^t\LL_i \right)\right)
\cap \left[\MMr(\Gamma)\right]\\
&=c_{D}\left(R^1\pi_*\bigoplus_{i=1}^t\LL_i \right) \cap
\left[\MMr(\Gamma)\right].
\end{split}
\end{equation}
\item   {\bf  Index zero:} \label{ax:wittenmap} Suppose that  $\dim (\MMr(\Gamma))=0$
and all the decorations on tails  are narrow.

If the pushforwards $\pi_* \left(\bigoplus\LL_i\right)$ and
$R^1\pi_* \left(\bigoplus \LL_i\right)$ are both vector bundles of
the same rank, then the virtual cycle is just the degree
$\deg(\wit)$ of the Witten map times the fundamental cycle:
$$\left[\MMr(\Gamma)\right]^{vir} = \deg(\wit)\left[\MMr(\Gamma)\right].$$
\end{enumerate}

\item \textbf{Forgetting tails:}\label{ax:rig-tails}
\begin{enumerate}
    \item Let $\Gamma$
have its $i$th tail decorated with $J^{-1}$, where $J$ is the
exponential grading element of $G$. Further, let $\Gamma'$ be the
decorated $W$-graph obtained from $\Gamma$ by forgetting the $i$th
tail and its decorations.  Assume that $\Gamma'$ is stable, and
denote the forgetting tails morphism by $$\vtr: \MMr(\Gamma) \to
\MMr(\Gamma').$$  We have
\begin{equation}
[\MMr_{g,k}(\Gamma)]^{vir}=(\vtr)^*([\MMr_{g,k-1}(\Gamma')]^{vir}).\end{equation}

\item (1)
$[\MMr_{0,3}(\gamma_1,\gamma_2,J^{-1})]^{vir}=\emptyset$ if
$\gamma_1\neq \gamma_2^{-1}$.\

(2) If $S_i$ and $S^-_i$ are arbitrary  bases in
$H_{N_\gamma}(\C^N_\gamma, (W_\gamma)^{\infty},\Q)$ and
$H_{N_\gamma}(\C^N_\gamma, (W_\gamma)^{-\infty},\Q)$, respectively,
and if $(\eta^{ij})$ is the inverse matrix of $(<S_i, S^-_j>)$,
then
$$
[\MMr_{0,3}(\gamma,\gamma^{-1},J^{-1})]^{vir}=\sum_{i,j}\frac{|G|}{|<\gamma>|}\times
\frac{|G|}{|<J^{-1}>|} \eta^{ij} S_i\otimes S^-_j.
$$
Here $\sum_{i,j}\eta^{ij} S_i\otimes S^-_j$ is called the Casimir
element of the intersection pairing. Furthermore, for any
$\alpha\in H^{N_\gamma}(\C^N_\gamma, (W_\gamma)^{-\infty},\Q) $
and $\beta\in H^{N_\gamma}(\C^N_\gamma, (W_\gamma)^{\infty},\Q)$,
there holds
$$
[\MMr_{0,3}(\gamma,\gamma^{-1},J^{-1})]^{vir}(\alpha,\beta,e_{J^{-1}})=\frac{|G|}{|<\gamma>|}\times
\frac{|G|}{|<J^{-1}>|}<\alpha,\beta>,
$$
where $<\cdot,\cdot>$ is the pairing induced by the intersection
pairing.
\end{enumerate}

\item\textbf{Composition law:}\label{ax:rig-cutting} Given any
genus-$g$ decorated stable $W$-graph $\Gamma$ with $k$ tails, and
given any edge $e$ of $\Gamma$, let $\hGamma$ denote the graph
obtained by ``cutting'' the edge $e$ and replacing it with two
unjoined tails $\tau_+$ and $\tau_-$.  Further, denote by
$$\widetilde{\rho}:\WW(\hGamma) \to \WW(\Gamma)$$ the gluing
loops or gluing trees morphism corresponding to gluing $\tau_+$ to
$\tau_-$, as described in Section~\ref{sec:rigmoduli} Equations
(\ref{eq:glueTreeW}) and (\ref{eq:glueLoopW}).

The virtual cycle $\left[\MMr({\hGamma})\right]^{vir}$ lies in
$$
H_*(\MMr(\hGamma),\Q)\otimes \prod_{\tau \in T(\Gamma)}
H_{N_{\gamma_{\tau}}}(\C^N_{\gamma_{\tau}},W^{\infty}_{\gamma_\tau},
\Q)\otimes
H_{N_{\gamma_{+}}}(\C^N_{\gamma_{+}},W^{\infty}_{\gamma_+}, \Q)
\otimes H_{N_{\gamma_{-}}}(\C^N_{\gamma_{-}},W^{\infty}_{\gamma_-},
\Q),
$$
and the homology
$H_{N_{\gamma_{+}}}(\C^N_{\gamma_{+}},W^{\infty}_{\gamma_+}, \Q)$ is
naturally isomorphic to
$H_{N_{\gamma_{-}}}(\C^N_{\gamma_{-}},W^{\infty}_{\gamma_-}, \Q)$.
So the intersection pairing
$$\langle\, ,
\rangle:H_{N_{\gamma_{+}}}(\C^N_{\gamma_{+}},W^{\infty}_{\gamma_+},
\Q) \otimes
H_{N_{\gamma_{-}}}(\C^N_{\gamma_{-}},W^{\infty}_{\gamma_-}, \Q),\rTo
\C
$$
can be applied to contract these and give a map
\begin{align*}
\langle\, \rangle_{\pm}:&H_*(\MMr(\hGamma),\Q)\otimes \prod_{\tau
\in T(\Gamma)}
H_{N_{\gamma_{\tau}}}(\C^N_{\gamma_{\tau}},W^{\infty}_{\gamma_\tau},
\Q)\otimes
H_{N_{\gamma_{+}}}(\C^N_{\gamma_{+}},W^{\infty}_{\gamma_+}, \Q)
\otimes H_{N_{\gamma_{-}}}(\C^N_{\gamma_{-}},W^{\infty}_{\gamma_-},
\Q)\\
&\rTo H_*(\MMr(\hGamma),\Q)\otimes \prod_{\tau \in T(\Gamma)}
H_{N_{\gamma_{\tau}}}(\C^N_{\gamma_{\tau}},W^{\infty}_{\gamma_\tau},
\Q).
\end{align*}
We have
\begin{equation}\label{eq:composition}
\frac{|G|}{|<\gamma_+>|}\left[\MMr({\Gamma})\right]^{vir}
=\widetilde{\rho}_*
\left\langle\left[\MMr({\hGamma})\right]^{vir}\right\rangle_{\pm}.
\end{equation}

\item \textbf{Sums of Singularities:}

If $W_1 \in \C[z_1,\dots,z_{N_1}]$ and $W_2 \in
C[z_{N_1+1},\dots,z_{N_1+N_2}]$ are two quasi-homogeneous
polynomials with diagonal automorphism groups $G_1$ and $G_2$,
then the diagonal automorphism group of $W=W_1+W_2$ is $G = G_1
\times G_2$. For any $\gamma_1\in G_1$ and $\gamma_2\in G_2$, it
is known \cite{AGV, E} that the following isomorphism holds:
$$
H_{mid}(\C^{N_1+N_2}_{(\gamma_1,\gamma_2)},
W^{\infty}_{(\gamma_1,\gamma_2)}, \Q) \cong
H_{mid}(\C^{N_1}_{\gamma_1}, W^{\infty}_{\gamma_1}, \Q) \otimes
H_{mid}(\C^{N_2}_{\gamma_2}, W^{\infty}_{\gamma_2}, \Q).
$$
Further, since any $G$-decorated stable graph $\Gamma$ is
equivalent to the choice of a $G_1$-decorated graph $\Gamma_1$ and
a $G_2$-decorated graph $\Gamma_2$ which are based on the
underlying graph $\barGamma$, we have the fiber bundle
\begin{equation}
\MMr(\Gamma) = (\WW_1^{rig})(\Gamma_1) \times_{\MM(\barGamma)}
(\WW_2^{rig})(\Gamma_2).
\end{equation}
The natural inclusion
$$\MMr_{g,k}(\Gamma) = (\WW_1^{rig})_{g,k}(\Gamma_1) \times_{\MM_{g,k}(\barGamma)}
(\WW_2^{rig})_{g,k}(\Gamma_2)\rInto^{\Delta}
(\WW_1^{rig})_{g,k}(\Gamma_1) \times (\WW_2^{rig})_{g,k}(\Gamma_2),$$
together with the isomorphism of middle homology, induces a
homomorphism
\begin{multline*}
\Delta^*:\left(H_*((\WW_1^{rig})_{g,k}(\Gamma_1),\Q)\otimes
\prod_{i=1}^k
H_{N_{\gamma_{i,1}}}(\C^N_{\gamma_{i,1}},(W_1)^{\infty}_{\gamma_{i,1}},
\Q)\right) \otimes
\left(H_*((\WW_2^{rig})_{g,k}(\Gamma_2),\Q)\otimes \prod_{i=1}^k
H_{N_{\gamma_{i,2}}}(\C^N_{\gamma_{i,2}},(W_2)^{\infty}_{\gamma_{i,2}},
\Q)\right) \\
\rTo H_*((\overline{\W_1+\W_2})^{rig}_{g,k}(\Gamma),\Q)\otimes
\prod_{i=1}^k
H_{N_{(\gamma_{i,1},\gamma_{i,2})}}(\C^{N_1+N_2}_{(\gamma_{i,1},\gamma_{i,2})},
(W_1+W_2)^{\infty}_{(\gamma_{i,1},\gamma_{i,2})}, \Q),
\end{multline*}
such that the virtual cycle satisfies
\begin{equation}
\Delta^*\left(\left[(\WW_1^{rig})_{g,k}(\Gamma_1)\right]^{vir}\otimes
\left[(\WW_2^{rig})_{g,k}(\Gamma_2)\right]^{vir}\right) =
\left[(\overline{\W_1+\W_2})^{rig}_{g,k}(\Gamma)\right]^{vir}.
\end{equation}

\item \textbf{Deformation Invariance:} Let $W_\tau\in
\C[z_1,\dots,z_N]$ be a family of non-degenerate quasi-homogeneous
polynomials depending smoothly on the parameter $\tau\in S$, where
$S$ is a path connected domain in $\C^m$. Suppose that $G$ is the
common automorphism group of $W_\tau$, then the virtual cycle
$[\MMr_\tau(\Gamma)]^{vir}$ associated to $(W_\tau,G)$ is
independent of $\tau$.

\item \textbf{$G_{max}$-Invariance:} The virtual cycle $[\MMr(\Gamma)]^{vir}$ associated to $(W,G)$ is
$G_{max}$-invariant (refer to the proof of this axiom for the
explanation of the $G_{max}$ action).
\end{enumerate}
\end{thm}

\begin{rem} By a Liouville type theorem for $A_1$-singularity, we see that the $A_1$-theory is
uniquely determined by these axioms and it agrees with the theory
of $2$-spin curves as described in \cite[\S4.5]{JKV1}.
\end{rem}

\subsection{Proof of Theorem \ref{thm:main-rigid} and Theorem \ref{thm:main}}\

In this section, we will first define the virtual cycle, prove the
axioms given by Theorem \ref{thm:main-rigid} and finally prove
Theorem \ref{thm:main}.

\subsubsection*{Definition of virtual cycle $[\MMr_{g,k}(\Gamma)]^{vir}$}

Fix $\bgamma=\{\gamma_1,\dots,\gamma_k\}$ and choose the moduli
space $\MMr_{g,k}(\bgamma,\bvkappa)$ to be strongly regular. For
each $\gamma\in G$, choose the basis $\{S^-_j(\gamma), j=1,\dots,
\mu_\gamma\}$ in $H_{N_\gamma}(\C^N_\gamma,
(W\gamma+W_{0,\gamma})^{-\infty}, \Q)$ corresponding to the critical
points of $W_\gamma+W_{0,\gamma}$ and the dual basis $\{S_j(\gamma),
j=1,\dots, \mu_\gamma\}$ in $H_{N_\gamma}(\C^N_\gamma,
(W_\gamma+W_{0,\gamma})^{\infty}, \Q)$. Then each combination
$\left(S^-_{j_1}(\gamma_1),\dots,S^-_{j_k}(\gamma_k)\right)$
corresponds to the combination of $k$ critical points,
$\bvkappa_{j_1\dots
j_k}:=\left(\kappa^-_{j_1}(\gamma_1),\dots,\kappa^-_{j_k}(\gamma_k)\right)$.
We obtain the virtual cycle
$[\MMr_{g,k}(\Gamma;\bgamma,\bvkappa_{j_1\dots j_k})]^{vir}
=:[\MMr_{g,k}(\Gamma;\bgamma,S^-_{j_1}(\gamma_1),\dots,S^-_{j_k}(\gamma_k))]^{vir}$.
Now we fix a strongly regular parameter $(b_i^0)$; the Gauss-Manin
connection provides the isomorphisms
$$
GM_{(b_i^0)}: H_{N_\gamma}(\C^N_\gamma,
(W\gamma+W_{0,\gamma})^{\pm\infty}, \Q)\rTo
H_{N_\gamma}(\C^N_\gamma, (W_\gamma)^{\pm\infty}, \Q).
$$
Using the isomorphisms we can identify $H_{N_\gamma}(\C^N_\gamma,
(W\gamma+W_{0,\gamma})^{\pm\infty}, \Q)$ with
$H_{N_\gamma}(\C^N_\gamma, (W_\gamma)^{\pm\infty}, \Q)$.

Define
\begin{align}
&\left[\MMr(\Gamma)\right]^{vir}:=\sum_{j_1,\dots,
j_k}\left(\left[\MMr_{g,k}(\Gamma; \bgamma,\bvkappa_{j_1\dots
j_k} )\right]^{vir}\otimes \prod_{i=1}^k S_{j_i}(\gamma_i)\right)\\
&\in H_*(\MMr_{g,k}(\Gamma))\otimes \prod_{\tau \in T(\Gamma)}
H_{N_{\gamma_{\tau}}}(\C^N_{\gamma_{\tau}},W^{\infty}_{\gamma_\tau},
\Q).
\end{align}

By Proposition \ref{prop-linear} We have

\begin{prop}\label{prop:cycle-indep-para} The virtual cycle $\left[\MMr(\Gamma)\right]^{vir}$
is independent of the choice of the basis $\{S_{j_i}(\gamma_i)\}$ of
$H_{N_\gamma}(\C^N_\gamma, (W_\gamma)^{\pm\infty}, \Q)$ at each
marked point $p_i$.
\end{prop}

Since the parallel transport induced by the Gauss-Manin connection
preserves the inner product of the homology bundle, the above
proposition justifies the definition of the virtual cycle
$\left[\MMr(\Gamma)\right]^{vir}$.

\subsubsection*{Proof of Dimension axiom} The dimension of the virtual cycle
$[\MMr_{g,k}(\Gamma;\bgamma,\bvkappa_{j_1\dots j_k})]^{vir}$ was
already calculated  in Theorem \ref{ind-witt-oper}. The real
dimension of the tensor product of the Lefschetz thimbles
$\prod_{i=1}^k S_{j_i}(\gamma_i)$ is $\sum_{i=1}^k N_{\gamma_i}$. So
we obtain the dimension of the virtual cycle
$\left[\MMr(\Gamma)\right]^{vir}$. Notice that its real dimension
can be an odd number. This is a new result compared to the previous
research in $A_r$-spin curves.

\subsubsection*{Proof of symmetric group invariance}

Since the construction of the virtual cycle
$\left[\MMr(\Gamma;\bgamma,\bvkappa)\right]^{vir}$ is independent of
the order of the tails, hence if we denote by $\sigma\Gamma$  the
graph obtained by applying $\sigma$ to the tails of $\Gamma$, then we
have
$$
\sigma_*\left[\MMr(\Gamma)\right]^{vir}=\left[\MMr(\sigma\Gamma)\right]^{vir}=\left[\MMr(\Gamma)\right]^{vir}
$$
as homology classes.

\subsubsection*{Proof of Disconnected graphs} This is obvious.

\subsubsection*{Proof of Degenerating connected graphs} It suffices to prove the following conclusion:
$$
[\MMr_{g,k}(\bgamma,\kappa)]^{vir}\cap
\MMr(\Gamma)=[\MMr(\Gamma;\kappa)]^{vir}.
$$
Let $\{(\Vd\times \Vr\times \Vm, E_\sigma, \aut(\sigma),
s_\sigma)\}$ be an oriented Kuranishi structure of
$\MMr_{g,k}(\bgamma,\kappa)$. The restriction $\{(\Vd\times
\{0\}\times \Vm, E_\sigma, \aut(\sigma), s_\sigma|_\Gamma)\}$
provides a Kuranishi structure of $\MMr(\Gamma;\kappa)$. We can take
a smooth sequence of multisections $s_\sigma^n$ such that
$s^n_\sigma$  approximates $s_\sigma$ and the restriction
$s^n_\sigma|_\Gamma$ approximates $s_\sigma|_\Gamma$. Hence
\begin{align*}
&[\MMr(\Gamma;\kappa)]^{vir}=\Pi\left(\cup_\sigma \left(
(s^{n}_\sigma|_\Gamma)^{-1}(0)/\aut(\sigma)\right)\right)\\
&=\Pi\left(\cup_\sigma \left(((s^{n}_\sigma)^{-1}(0)\cap
(\Vd\times \{0\}\times\Vm))/\aut(\sigma)\right)\right)\\
&=[\MMr_{g,k}(\bgamma,\kappa)]^{vir}\cap \MMr(\Gamma;\kappa),
\end{align*}
where $\Pi: \MMr(\Gamma;\kappa)\rTo \MMr(\Gamma)$ is the
forgetful map introduced before.

\subsubsection*{Proof of the axiom of Topological Euler class for the
narrow sector:}

Since all the marked points are narrow, the perturbation
terms of the perturbed Witten equation only occur near the broad
nodal points. Let $\bb$ be the perturbation parameter, then the
virtual cycle
$$
[\MMr(\Gamma;0)]^{vir}\in H_*(\MMr(\Gamma;0))
$$
is independent of $\bb$.

Let $\lambda\in [0,1]$, and consider the $\lambda$-parameterized
perturbed Witten equation:

\begin{equation}\label{witt-equa-para1}
\bpat_{\cC} u_{i}+\tilde{I}_1\left(\overline{\frac{\pat( W+\lambda
W_{0,\beta})}{\pat u_{i}}}\right)=0, \forall i=1,\dots,N.
\end{equation}

Then we have a family of moduli spaces
$\MMrs_{g,k}(\Gamma;0;\lambda)$. Define
$$
\WW^\lambda_{g,k}=\cup_{\lambda\in [0,1]}\{\lambda\}\times
\MMrs_{g,k}(\Gamma;0;\lambda).
$$

By Corollary \ref{crl-wittenlemma1}, as $\lambda\to 0$, the
solutions $\bu^\lambda$ of the equation (\ref{witt-equa-para1}) will
tend to $0$ solution of the non-perturbed Witten equation. Therefore
the moduli space $\WW^\lambda_{g,k}$ is a compact Hausdorff space.
One can prove in the same way as in the proof of Wall-crossing
formula that it carries an oriented Kuranishi structure
$\{(U_\sigma, E_\sigma, s_\sigma)\}$. Notice that the hypothesis
that all the marked points are narrow is important because it guarantees that the linearized equations of the perturbed and
non-perturbed Witten equations have the same Fredholm index in the
same Sobolev spaces. Since all the marked points are narrow
points, there is no wall-crossing phenomenon and the virtual cycle
$[\MMr(\Gamma;0;\lambda=1)]^{vir}$ is cobordant to the virtual cycle
$[\MMr(\Gamma;0;\lambda=0)]^{vir}$, where the latter is defined using
the non-perturbed Witten equation. Since the moduli space $
\MMr(\Gamma;0;\lambda=0)$ is compact and can be covered by finitely
many Kuranishi neighborhoods, we can take two complex orbifold
vector bundles $E_i^0$ and $E_i^1$ such that the two-term complex
$$
d_i:E_i^0\to E_i^1
$$
is quasi-isomorphic to
$$
U_\sigma^+/\aut(\sigma)\to E_\sigma/\aut(\sigma),
$$
and then to
$$
\ker(d_i)\to \coker(d_i)
$$
on each chart.

Therefore, we have
$$[\MMr(\Gamma)]^{vir}=[\MMr(\Gamma;0;\lambda=1)]^{vir}=[\MMr(\Gamma;0;\lambda=0)]^{vir}=\eul(\wit: \pi^*_0
\bigoplus_i E^1_i\to \bigoplus_i
    E^0_i)\cap [\MMr_{\Gamma}].
$$
The concavity and dimension properties are the natural conclusions
of the topological Euler class axiom.

\subsubsection*{Proof of forgetting tails axiom} (a) Without loss of generality, we assume that $i=k$ and
\begin{align*}
&\gamma_k=J^{-1}=(e^{2\pi i q_1},\dots, e^{2\pi i q_t}),\\
&\gamma^{-1}_k=J=(e^{2\pi i (1-q_1)},\dots, e^{2\pi i (1-q_t)}).
\end{align*}

Since $\C^N_{\mop}=\{0\}$, it suffices to prove
\begin{equation}
(\vtr)_*\left([\MMr_{g,k}(\Gamma;\bgamma',\gamma_k,\bvkappa',0)]\right)=\frac{|G|}{|<J^{-1}>|}
[\MMr_{g,k-1}(\Gamma';\bgamma',\bvkappa')]^{vir}.
\end{equation}
Consider the semi-stable $W$-curve $\frkc_J:=(\C P^1, z_1,z_2,
J,J^{-1}, \LL_1,\dots, \LL_t, \varphi_1,\dots, \varphi_s)$.

By the degree formula, we have
$$
\deg(|\LL_j|)=-1.
$$
The Witten map on $\frkc_J$ has no perturbation term and the index
$$
\ind(WI_{\frkc_J})=0.
$$
Let $\frkc$ be a rigidified $W$-curve having $k$ marked points
with the $k$-th marked point decorated by $\gamma_k$. We want to
study the Witten equations near this narrow point
$(z_k,\gamma_k)$.

If we take the cylindrical metric near $z_k$ (i.e., let
$|\frac{dz}{z}|=1$), then its neighborhood is viewed as an
infinitely long cylinder and $z_k$ is viewed as the infinite point.
Let $u_j=\tilde{u}_j e_j$; then the Witten equation is
\begin{equation}
\frac{\pat \tilde{u}_i}{\pat \bar{z}}+\overline{\frac{\pat
W(\tilde{u}_1,\dots,\tilde{u}_t)}{\pat
\tilde{u}_i}\frac{1}{z}}=0.
\end{equation}
Let $z=e^w$; then we have
\begin{equation}
\frac{\pat \tilde{u}_i}{\pat \bar{w}}+\overline{\frac{\pat
W(\tilde{u}_1,\dots,\tilde{u}_t)}{\pat \tilde{u}_i}}=0, \;\forall
w\in [-\infty,0]\times S^1.
\end{equation}

If we take the smooth metric (i.e., let $|dz|=1$), then the Witten
equation over the orbicurve has the following form:
\begin{equation}\label{WI-orbi-equa}
\frac{\pat \tilde{u}_i}{\pat \bar{z}}+\sum_j\overline{\frac{\pat
W_j(\tilde{u}_1,\dots,\tilde{u}_t)}{\pat
\tilde{u}_i}\frac{1}{z}}|z|^{2q_i}=0.
\end{equation}
We can resolve this equation near $z_k$, and the corresponding
function $\hat{u}$ satisfies the equation
\begin{equation}\label{WI-reso-equa}
\frac{\pat \hat{u}_i}{\pat \bar{z}}+\overline{\frac{\pat
W(\hat{u}_1,\dots,\hat{u}_t)}{\pat \hat{u}_i}}=0.
\end{equation}

Consider the moduli space $\MMr_{g,k}(\Gamma;\bgamma,\bvkappa)$.
Assume that the $k$-th marked point is decorated by
$(\gamma_k,\bone_{\mop})$. Here we also assume that we take the
cylindrical metric near each marked point. On the other hand, we can
set the metric near $z_k$ to be the smooth metric and then the
Witten equation has the form (\ref{WI-orbi-equa}). After resolution,
the solution $\tilde{u}$ corresponds 1-1 to the solution $\hat{u}$
of (\ref{WI-reso-equa}). In particular, the marked point $z_k$
becomes an ordinary point if we only consider the equation.

Let $\MMr_{g,k}(\Gamma;\bgamma', \gamma_k,\bvkappa', sm)$ be the
moduli space consisting of the isomorphism classes of $W$-sections
$(\frkc, \bu)$, where $\bu$ is the solution of the perturbed Witten
equation. Here the symbol "$sm$" means that near the $k$-th marked
point, the metric is set to be the smooth metric. As before, we can
give the Gromov topology to this moduli space. Moreover, we will
show below that it can carry an oriented Kuranishi structure which
is induced by the Kuranishi structure of
$\MMr_{g,k-1}(\Gamma';\bgamma', \bvkappa')$.

The following diagram is commutative:
$$
\begin{CD}
\MMr_{g,k}(\Gamma;\bgamma', \gamma_k,\bvkappa', sm)@>\vtr
>>\MMr_{g,k-1}(\Gamma';\bgamma',
\bvkappa')\\
@VVV   @VVV\\
\MMr_{g,k}(\Gamma;\bgamma)@>\vtr
>>\MMr_{g,k-1}(\Gamma';\bgamma').
\end{CD}
$$
Here $\vtr$ is the forgetful map. This map exists if and only if the
marked point $z_k$ is decorated by the group element $J^{-1}$.

Let $\{(U_\sigma, \Gamma_\sigma, E_\sigma, s_\sigma)\}$ be a
Kuranishi structure of $\MMr_{g,k-1}(\Gamma';\bgamma', \bvkappa')$.
We can choose the obstruction bundle $E_\sigma$ such that its
generating functions have compact support disjoint from the $k-1$
marked points and the special point $z_k$. For a point
$\sigma=(\frkc_\sigma, \bu_\sigma)$, we can take a neighborhood
$(U_\sigma, \Gamma_\sigma, E_\sigma, s_\sigma)$. The class
$[\frkc_\sigma]\in \MMr_{g,k-1}(\Gamma';\bgamma')$ has a
neighborhood $(\hat{U}_\sigma, \Lambda_\sigma)$. Then the embedding
$U_\sigma\hookrightarrow \hat{U}_\sigma$ is a
$\Gamma_\sigma\hookrightarrow \Lambda_\sigma$ equivariant. There is
a universal family of rigidified $W$-curves
$\tilde{U}_\sigma\rTo \hat{U}_\sigma$ on which
$\Lambda_\sigma$ acts. Let $(\widetilde{\frkc}, \bu)$ be a point in
$\MMr_{g,k}(\Gamma;\bgamma', \gamma_k,\bvkappa', sm)$ such that
$\vtr([\widetilde{\frkc}])=[\frkc_\sigma]$; then there is a
neighborhood $\tilde{U}_\sigma/\Lambda_\sigma$ of
$[\widetilde{\frkc}]$ such that
$\vtr(\tilde{U}_\sigma/\Lambda_\sigma)=\hat{U}_\sigma/\Lambda_\sigma$.
Now
$(U_\sigma\times_{\hat{U}_\sigma}\tilde{U}_\sigma)/\Gamma_\sigma$
becomes a neighborhood of $(\widetilde{\frkc}, \bu)$. The
obstruction bundle $E_\sigma$ and the section $s_\sigma$ induce the
obstruction bundle $\widetilde{E}_\sigma$ and the section
$\ts_\sigma$. It is straightforward to prove that
$\{(U_\sigma\times_{\hat{U}_\sigma}\tilde{U}_\sigma), \Gamma_\sigma,
\widetilde{E}_\sigma, s_\sigma\}$ is an oriented Kuranishi structure
of $\MMr_{g,k}(\Gamma;\bgamma', \gamma_k,\bvkappa', sm)$.
Furthermore if $\{s^n_\sigma\}$ is an approximation sequence of the
continuous section $\{s_\sigma\}$ such that the multi-section
$s^n_\sigma$ is a transversal section, then $[(s^n_\sigma)^{-1}(0)]$
defines the virtual cycle $[\MMr_{g,k-1}(\Gamma';\bgamma',
\bvkappa')]^{vir}$. Now the induced multisection $\{\ts^n_\sigma\}$
satisfies the relation
$(\ts^n_\sigma)^{-1}(0)=(s^n_\sigma)^{-1}(0)\times_{\MMr_{g,k-1}
(\Gamma';\bgamma')}\MMr_{g,k}(\Gamma;\bgamma)$. Therefore we have
$$
[\MMr_{g,k}(\Gamma;\bgamma,
\bvkappa,sm)]^{vir}=(\vtr)^*([\MMr_{g,k-1}(\Gamma';\bgamma',\bvkappa']^{vir}).$$

Now to prove this axiom we only need to show the following
identity:
\begin{equation}\label{forg-axiom-cobo}
[\MMr_{g,k}(\Gamma;\bgamma', \gamma_k,\bvkappa',
sm)]^{vir}=[\MMr_{g,k}(\Gamma;\bgamma,\bvkappa)]^{vir}.
\end{equation}

Actually, as the first step one can show after a simple computation
that the two moduli spaces have the same virtual dimension.

The difference between the two moduli spaces  originates from the
choice of different metric around $z_k$. When using the
cylindrical metric, $z_k$ is the infinite point with a infinitely
long cylindrical neighborhood. When using the smooth metric, $z_k$ is
just an ordinary point with a disc neighborhood. So to obtain the
neighborhood of $z_k$ from the infinitely long cylinder one has to
cap a disc at the infinitely far end. The inverse process is just
the degeneration course of a $W$-curve along a circle centered at
$z_k$.

Let $(\frkc, \bu)\in \MMr_{g,k}(\Gamma;\bgamma,\bvkappa)$, and let
$\frkc_J$ be a semi-stable $W$-curve shown in (2). To construct the
Witten map over $\frkc_J$, we can choose the cylindrical metric near
$z_1$ which is decorated by $J$ and choose the smooth metric around
$z_2$. So $z_2$ is actually an ordinary point and one can use the
Witten lemma to show that the Witten equation over $\frkc_J$ has
only the zero solution. When identifying the point $z_k$ on $\frkc$
with $z_1$ on $\frkc_J$, $\frkc\#_{z_k=z_1}\frkc_J$ becomes a nodal
$W$-curve. Notice that this curve is not in the space
$\MMr_{g,k}(\Gamma;\bgamma)$ because the $\frkc_J$ component is not
stable. However, one can also do the gluing operation. The gluing
parameter $\zeta_k$ is a complex number. However we only take a
small real interval $[0,\varepsilon]$ as our gluing parameter.

When $\zeta_k=0$, we do nothing. When $\zeta_k\neq 0$, we obtain a
$W$-curve $\frkc(\zeta_k)$ from $\frkc$. Therefore from the moduli
space $\MMr_{g,k}(\Gamma;\bgamma)$ of rigidified $W$-curves, we
obtain another moduli space $\MMr_{g,k}(\Gamma;\bgamma,\zeta_k)$ by
the gluing operation. These two spaces are totally equivalent as
orbifolds. Let $\MMr_{g,k}(\Gamma;\bgamma', \gamma_k,\bvkappa'
,\zeta_k, sm)$ be the moduli space of $W$-sections over
$\MMr_{g,k}(\Gamma;\bgamma,\zeta_k)$.

Define
$$
\WW=\cup_{\zeta_k\in [0,\varepsilon]} \MMr_{g,k}(\Gamma;\bgamma',
\gamma_k,\bvkappa',\zeta_k, sm).
$$
Using the decay estimate of the Witten equation, one can prove that
this space is compact in the Gromov topology. One can construct an
oriented Kuranishi structure over $\WW$ from the Kuranishi structure
of $\MMr_{g,k}(\Gamma;\bgamma,\bvkappa)$ and
$\MMr_{g,k}(\Gamma;\bgamma', \gamma_k,\bvkappa', sm)$.

Let $(U_\sigma, E_\sigma, s_\sigma)$ be a chart of $(\frkc, \bu)\in
\MMr_{g,k}(\Gamma;\bgamma,\bvkappa)$. For any gluing parameter
$\zeta_k\in (0,\varepsilon]$, we can obtain the $W$-curve
$\frkc(\zeta_k)$. By an index computation and the fact that the
generators of the bundle $E_\sigma$ have compact support away from
$z_k$, we can view $E_\sigma$ as the obstruction bundle over
$\frkc(\zeta_k)$. Now by the implicit function theorem, there exists
a section $s_{\sigma, \zeta_k}: U_\sigma\times
[0,\varepsilon/3]\rTo E_\sigma$. On the other hand, we let
$\MMr_{g,k}(\Gamma;\bgamma', \gamma_k,\bvkappa',\zeta_k=\varepsilon,
sm)=\MMr_{g,k}(\Gamma;\bgamma', \gamma_k,\bvkappa',sm)$ and extend
the Kuranishi structure to $\cup_{\zeta_k\in
[2\varepsilon/3,\varepsilon]} \MMr_{g,k}(\Gamma;\bgamma',
\gamma_k,\bvkappa',\zeta_k, sm)$. By the extension theorem, we can
construct an oriented Kuranishi structure $\{(U_\sigma\times
[0,\varepsilon], E_\sigma, s_{\sigma,\zeta_k})\}$ over $\MM$. So by
a cobordism argument we have (\ref{forg-axiom-cobo}) and then the
final conclusion.

(b) $\MMr_{0,3}(\gamma_1,\gamma_2,J^{-1})$ is empty if $\gamma_1\neq
\gamma_2^{-1}$ for degree reasons. So we only need to consider the
moduli space $\MMr_{0,3}(\gamma,\gamma^{-1},J^{-1})$ and the virtual
cycle $[\MMr_{0,3}(\gamma,\gamma^{-1},J^{-1})]^{vir}$ for any
$\gamma\in G$. The moduli space
$\MMr_{0,3}(\gamma,\gamma^{-1},J^{-1})$ is $0$-dimensional by the
dimension formula. There is only one point in $\MM_{0,3}$ and
$\WW_{0,3}$ respectively. For the space of rigidified $W$-curves, we
have the decomposition
\begin{equation}
\MMr_{0,3}(\gamma,\gamma^{-1},J^{-1},\kappa_{i}^1,\kappa_{j}^2, 0
)=\coprod_{\psi_1,\psi_2,\psi_3}
\MMr_{0,3}(\gamma,\gamma^{-1},J^{-1}, \kappa_{i}^1,\kappa_{j}^2, 0
,\psi_1,\psi_2,\psi_3),
\end{equation}
where $\kappa_{i}^1,\kappa_{j}^2$ are critical points of the
perturbed polynomial $W_\gamma+W_{0,\gamma}$ and
$W_{\gamma^{-1}}+W_{0,\gamma^{-1}}$. Recall that to define the
perturbed Witten equation, we first fix a standard rigidification
and then choose the perturbation term with respect to this
rigidification. After that the perturbation parameters with
respect to the other rigidification is naturally defined, since
the perturbed Witten equation is a globally defined equation and
the solutions are independent of the choice of the rigidification.
Assume that $\psi^+$ and $\psi^-$ are the standard rigidifications
attached to $\gamma$ and $\gamma^{-1}$. We have

\begin{lm} If the perturbation is strongly regular, then
\begin{equation}
[\MMr_{0,3}(\gamma,\gamma^{-1},J^{-1}, \kappa_{i}^1,\kappa_{j}^2, 0
,\psi^+,\psi^-,\psi_3)]^{vir}=\left\{\begin{array}{ll}
1&\;\text{if }\;I^{-1}(\kappa_{i}^1)=\kappa_{j}^2\\
0&\;\text{otherwise},
\end{array}
\right.
\end{equation} where $I=\diag(\xi^{k_1}, \dots,\xi^{k_N})$ and
$\xi^d=-1$.
\end{lm}

\begin{proof} Take the rigidified $W$-curve $\frkc=(\R\times, S^1, p_1, p_2, p_3, \LL, \psi^+, \psi^-,
\psi_3)$ where we can set $p_1$ $p_2$ to be infinitely far points
at $-\infty$ and $+\infty$ and $p_3=(0,0)$. The rigidification
$\psi^\pm$ actually gives the basis
$e^\pm=(e_1^\pm,\dots,e_N^\pm)$ on open half cylinders $U_i,
i=1,2$, where we require that $U_1\cup U_2=\R\times S^1$ and
$U_1\cap U_2=U_3$. Now by the disingularization operation and a
cobordism argument similar to the proof in (a), we can suppose
$p_3$ is an ordinary smooth point (not an infinitely far point)
with the trivial orbifold structure without changing the virtual
cycle. Consider the perturbed Witten equation over $\frkc$:
\begin{equation}
\bpat u_i-\overline{\frac{\pat(W+W_{0,\beta})}{\pat u_i}}=0,
\end{equation}
where $W_{0,\beta}=\sum_i b_i^{\pm}\beta_i^{\pm} u_i$. Let
$(\zeta^+,\theta^+)$ and $(\zeta^-,\theta^-)$ be the local
coordinates on $U_1$ and $U_2$ respectively. On the overlap $U_3$,
we can do the following transformation:
\begin{itemize}
\item Coordinate transformation:
$(\zeta^+,\theta^+)=-(\zeta^-,\theta^-)$; \item Corresponding
transformation of local basis: $I(e^+)=e^-$.
\end{itemize}
If $u$ is a solution of the equation, and
$u=\tilde{u}^{\pm}e^{\pm}$, then we have the transformation:
$I^{-1}(\tilde{u}^+)=\tilde{u}^-$. In particular, we have the
relation:
\begin{equation}\label{forg-proof-b0}
I^{-1}(\tilde{u}^+(-\infty))=\tilde{u}^-(-\infty),
\end{equation}
and the equation in the uniform coordinates and basis
$(\zeta^-,\theta^-), e^-$ has the form
\begin{equation}
\bpat \tilde{u}_i^--\overline{\frac{\pat(W+W_{0,\beta})}{\pat
\tilde{u}_i^-}}=0.
\end{equation}
Note that the perturbation terms near the two ends have a cut-off
function. Actually, we can consider a family of perturbed
equations continuously changing to the following perturbed Witten
equation without changing the cycle:
\begin{equation}\label{forg-proofb1}
\bpat \tilde{u}^-_i-\overline{\frac{\pat(W+W_{0,\gamma})}{\pat
\tilde{u}^-_i}}=0.
\end{equation}
Therefore, to compute the virtual cycle is equivalent to
considering the existence of the solutions of  Equation
(\ref{forg-proofb1}). Multiplying  Equation (\ref{forg-proofb1})
by $\frac{\pat(W+W_{0,\gamma})}{\pat \tilde{u}^-_i}$ and
integrating it over $\R\times S^1$, we have
\begin{equation}
(W_{\gamma^{-1}}+W_{0,\gamma^{-1}})(\tilde{u}^-)|^{+\infty}_{-\infty}=\sum_i\iint
\left|\frac{\pat(W+W_{0,\gamma})}{\pat
\tilde{u}^-_i}\right|^2dsd\theta.
\end{equation}
Since the perturbation is strongly regular, by the above integral
equality  Equation (\ref{forg-proofb1}) has solutions if and only
if $\tilde{u}^-(+\infty)=\tilde{u}^-(-\infty)$. By
(\ref{forg-proof-b0}), this is just
$$
I^{-1}(\kappa_{i}^1)=\kappa_{j}^2.
$$
\end{proof}

\begin{crl}\label{crl-proof-b2} If the perturbation is strongly regular, then
\begin{equation}
[\MMr_{0,3}(\gamma,\gamma^{-1},J^{-1}, \kappa_{i}^1,\kappa_{j}^2,
0)]^{vir}=\left\{\begin{array}{ll}
\frac{|G|}{|<\gamma>|}\times \frac{|G|}{|<J^{-1}>|}&\;\text{if }\;I^{-1}(\kappa_{i}^1)=\kappa_{j}^2\\
0&\;\text{otherwise}.
\end{array}
\right.
\end{equation}
\end{crl}

\begin{proof} For any pair rigidification $(\psi_1,\psi_2)$, there
exist $g_1,g_2\in G$ such that $\psi_1=g\cdot\psi^+$ and
$\psi_2=g_2\cdot \psi^-$. If $I^{-1}(\kappa_{i}^1)=\kappa_{j}^2$,
we can choose $\hat{I}=g_1g_2^{-1}I$ satisfying the condition
$\hat{I}^{-1}(g_1\kappa_{i}^1)=(g_2\kappa_{j}^2)$ such that the
corresponding equation has a unique solution. Thus the virtual
cycle is just the number of different $\hat{I}$, which is
$\frac{|G|}{|<\gamma>|}\times \frac{|G|}{|<J^{-1}>|}$.
\end{proof}

We begin the proof of (b). Let $S_j^-, j=1,\dots,\mu_\gamma$ be a
basis of $H_{N_\gamma}(\C^N_\gamma, (W_\gamma)^{-\infty},\Q)$
which is identified with $H_{N_\gamma}(\C^N_\gamma,
(W_\gamma+W_{0,\gamma})^{-\infty},\Q)$ and let $\{S_j\}$ be the
dual basis in $H_{N_\gamma}(\C^N_\gamma,
(W_\gamma)^{\infty},\Q)\cong H_{N_\gamma}(\C^N_\gamma,
(W_\gamma+W_{0,\gamma})^{\infty},\Q)\cong
H_{N_\gamma}(\C^N_\gamma,(W_{\gamma^{-1}}+W_{0,\gamma^{-1}})^{-\infty},\Q)$.
We have
$$
[\MMr_{0,3}(\gamma,\gamma^{-1},J^{-1})]^{vir}=\sum_{i,j}
[\MMr_{0,3}(\gamma,\gamma^{-1},J^{-1}, \kappa_{i}^1,\kappa_{j}^2,
0)]^{vir}\otimes S_i\otimes S^-_j.
$$
The map $I$ is a 1-1 map from the set of critical points of
$W_\gamma+W_{0,\gamma}$ to that of $W_\gamma+W_{0,\gamma^{-1}}$
(cf. Remark \ref{rem-wequ-equal}) which induces a map
$$
I_*:H_{N_\gamma}(\C^N_\gamma,
(W_\gamma+W_{0,\gamma})^{-\infty},\Q)\rTo
H_{N_\gamma}(\C^N_\gamma,(W_{\gamma^{-1}}+W_{0,\gamma^{-1}})^{-\infty},\Q)
$$
such that $I_*(S_i^-)=S_i$. Thus by Corollary \ref{crl-proof-b2},
we have
$$
[\MMr_{0,3}(\gamma,\gamma^{-1},J^{-1})]^{vir}=\sum_{i}\frac{|G|}{|<\gamma>|}\times
\frac{|G|}{|<J^{-1}>|} S_i\otimes S^-_i.
$$
In general, if $S_i$ and $S^-_i$ are two arbitrary  bases in
$H_{N_\gamma}(\C^N_\gamma, (W_\gamma)^{-\infty},\Q)$ and
$H_{N_\gamma}(\C^N_\gamma, (W_\gamma)^{-\infty},\Q)$ respectively,
and if we denote by $(\eta^{ij})$ the inverse matrix of $(<S_i,
S^-_j>)$, then
\begin{equation}
[\MMr_{0,3}(\gamma,\gamma^{-1},J^{-1})]^{vir}=\sum_{i,j}\frac{|G|}{|<\gamma>|}\times
\frac{|G|}{|<J^{-1}>|} \eta^{ij} S_i\otimes S^-_j.
\end{equation}
Let $\alpha\in H^{N_\gamma}(\C^N_\gamma, (W_\gamma)^{-\infty},\Q)
$ and $\beta\in H^{N_\gamma}(\C^N_\gamma,
(W_\gamma)^{\infty},\Q)$; then
\begin{align}
&[\MMr_{0,3}(\gamma,\gamma^{-1},J^{-1})]^{vir}(\alpha,\beta,e_{J^{-1}})=\sum_i
\frac{|G|}{|<\gamma>|}\times \frac{|G|}{|<J^{-1}>|}
\alpha(S_i)\beta(S_i^-)\\
&=\frac{|G|}{|<\gamma>|}\times
\frac{|G|}{|<J^{-1}>|}<\alpha,\beta>,
\end{align}
where $<\cdot,\cdot>$ is the pairing induced by the intersection
pairing.

\subsubsection*{Proof of Composition law}\label{ax:cutting} We can
choose two bases $S_i^\pm, i=1\dots, \mu_\gamma$ in
$H_{N_{\gamma_\pm}}(\C^N_{\gamma_\pm},
(W_{\gamma_\pm})^{-\infty},\Q)\cong
H_{N_{\gamma_\pm}}(\C^N_{\gamma_\pm},
(W_{\gamma_\pm}+W_{0,\gamma_\pm})^{-\infty},\Q)$ such that
$S_i^\pm$ corresponds to the critical point $\kappa_i$ of
$W_{\gamma_+}+W_{0,\gamma_+}$ and $I(\kappa_i)$ of
$W_{\gamma_-}+W_{0,\gamma_-}$, respectively. We can write the
virtual cycle as
$$
[\MMr_{g,k}(\hat{\Gamma})]^{vir}=\sum_{i,j}
[\MMr_{g,k}(\hat{\Gamma}, S^+_i, S^-_j)]^{vir}\otimes S^-_i\otimes
S^+_j.
$$
Hence we have
$$
\left<[\MMr_{g,k}(\hat{\Gamma})]^{vir} \right>_{\pm}=\sum_{i}
[\MMr_{g,k}(\hat{\Gamma}, \kappa^+_i, I(\kappa^+_i))]^{vir}
$$
and then
$$
\tilde{\rho}_*\left<[\MMr_{g,k}(\hat{\Gamma})]^{vir}
\right>_{\pm}=\frac{|G|}{|<\gamma_+>|}[\MMr_{g,k}(\Gamma)]^{vir},
$$
since we have $\frac{|G|}{|<\gamma_+>|}$ ways to obtain a rigidified
$W$-curve in $\MMr_{g,k}(\Gamma)$ (without rigidification at the
nodal point).

\subsubsection*{Proof of Sums of singularity axiom} Let $\bgamma_i=(\gamma_{i,1},\dots,
\gamma_{i,k})$ for $i=1,2$ and \\
$\bgamma=((\gamma_{1,1},
\gamma_{2,1}),\dots,(\gamma_{1,k},\bgamma_{2,k}))$. Assume that
$H_{N_{\gamma_{1,i}}}(\C^{N_1}_{\gamma_{1,i}},
W_{1,\gamma_{1,i}}^{-\infty}, \Q)$ is generated by a basis
$S^-_{j_i}, j_i=1,\dots,\mu_{\gamma_{1,i}}$ for each $i$ and
$H_{N_{\gamma_{2,i}}}(\C^{N_2}_{\gamma_{2,i}},
W_{2,\gamma_{2,i}}^{-\infty}, \Q)$ is generated by a basis
$\hat{S}^-_{l_i}, l_i=1,\dots,\mu_{\gamma_{2,i}}$. Then the group
$H_{N_{\gamma_{1,i}}+N_{\gamma_{2,j}}}(\C^{N_1+N_2}_{(\gamma_{1,i},\gamma_{2,1})},
W^{-\infty}, \Q)$ is generated by the basis $\{S^-_{j_i}\otimes
\hat{S}^-_{l_j}, j_i=1,\dots, \mu_{\gamma_{1,i}}, l_j=1,\dots,
\mu_{2,j}\}$.

Now we consider the case that $\Gamma$ is a decorated stable
$W$-graph with each tail decorated by the Lefschetz thimble
$\{S^-_{j_i}\otimes \hat{S}^-_{l_i}, i=1,\dots, k\}$. We have the
moduli spaces
$\WW_1^{rig}(\Gamma_1;\bgamma_1,S^-_{j_i}):=\WW_1^{rig}(\Gamma_1;\bgamma_1,S^-_{j_1},\dots
S^-_{j_k})$,
$\WW_2^{rig}(\Gamma_2;\bgamma_2,\hat{S}^-_{l_i}):=\WW_2^{rig}(\Gamma_2;\bgamma_2,\hat{S}^-_{l_1},\dots
\hat{S}^-_{l_k})$, and $\WW^{rig}(\Gamma;\bgamma,(S^-_{j_i},
\hat{S}^-_{l_i})):= \WW^{rig}(\Gamma;\bgamma,(S^-_{j_1},
\hat{S}^-_{l_1}), \dots, (S^-_{j_k},\hat{S}^-_{l_k}))$. These are
compact Hausdorff spaces when given the Gromov topology and have the
relation
$$
\WW^{rig}(\Gamma, \bgamma,(S^-_{j_i}, \hat{S}^-_{l_i}))
=\WW_1^{rig}(\Gamma_1;\bgamma_1,S^-_{j_i})\times_{\MM_{\bar{\Gamma}}}\WW_2^{rig}(\Gamma_2;\bgamma_2,\hat{S}^-_{l_i})).
$$
By Corollary \ref{crl-cobo-nonsing}, the moduli space
$\WW^{rig}(\Gamma, \bgamma,(S^-_{j_i}, \hat{S}^-_{l_i}))$ carries an
orientable Kuranishi structure. Consider its local charts. For any
point $\sigma\in \WW^{rig}(\Gamma)$, the local chart is given by
$(\Vd\times \Vr\times\Vm, E_\sigma,
\aut(\sigma),s_\sigma,\Psi_\sigma)$. Let $\{\hat{s}_\sigma^n\}$ be a
family of multisections approximating $s_\sigma$; then for $n$ large
enough $[\WW^{rig}(\Gamma)]^{vir}$ is the pushdown of
$[(\hat{s}_\sigma^n)^{-1}(0)]$. Since $W=W_1+W_2$, we have the
splitting $\Vm=V_{map,\sigma_1}\times V_{map,\sigma_2},
\aut(\sigma)=\aut(\sigma_1)\times\aut(\sigma_2)$ and
$E_\sigma=E_{\sigma_1}\otimes E_{\sigma_2}$, where $(\Vd\times
\Vr\times V_{map,\sigma_i} , E_{\sigma_i},
\aut(\sigma_i),s_{\sigma_i},\Psi_{\sigma_i})$ for $i=1,2$ is the
local chart of $\WW_i^{rig}(\Gamma_i)$. Thus
$$
[(\hat{s}_{\sigma}^n)^{-1}(0)]=[(\hat{s}_{\sigma_1}^n)^{-1}(0)]
\times_{\MM_{\bar{\Gamma}}}[(\hat{s}_{\sigma_2}^n)^{-1}(0)],
$$
i.e.,
\begin{equation}\label{sum-sing}
[\WW^{rig}(\Gamma, \bgamma,(S^-_{j_i}, \hat{S}^-_{l_i}))]^{vir}
=[\WW_1^{rig}(\Gamma_1;\bgamma_1,S^-_{j_i})]^{vir}\times_{\MM_{\bar{\Gamma}}}
[\WW_2^{rig}(\Gamma_2;\bgamma_2,\hat{S}^-_{l_i})]^{vir}.
\end{equation}

On the other hand, we have the decomposition of the dual spaces
$$
H_{N_{\gamma_{1,i}}+N_{\gamma_{2,j}}}(\C^{N_1+N_2}_{(\gamma_{1,i},\gamma_{2,1})},
W^{\infty}, \Q)=H_{N_{\gamma_{1,i}}}(\C^{N_1}_{\gamma_{1,i}},
W_{1,\gamma_{1,i}}^{\infty}, \Q)\otimes
H_{N_{\gamma_{2,i}}}(\C^{N_2}_{\gamma_{2,i}},
W_{2,\gamma_{2,i}}^{-\infty}, \Q).
$$
This decomposition combined with the identity (\ref{sum-sing})
induces the conclusion.

\subsubsection*{Proof of Deformation invariance axiom} This comes simply
from the cobordism argument for the oriented Kuranishi
structure for the big moduli space involving the parameter $\tau$.

\subsubsection*{Proof of $G_{max}$-Invariance axiom} Note that
\begin{align}
&\left[\MMr(\Gamma)\right]^{vir}:=\sum_{j_1,\dots,
j_k}\left(\left[\MMr_{g,k}(\Gamma; \bgamma,\bvkappa_{j_1\dots
j_k} )\right]^{vir}\otimes \prod_{i=1}^k S_{j_i}(\gamma_i)\right)\\
&\in H_*(\MMr_{g,k}(\Gamma))\otimes \prod_{\tau \in T(\Gamma)}
H_{N_{\gamma_{\tau}}}(\C^N_{\gamma_{\tau}},W^{\infty}_{\gamma_\tau},
\Q).
\end{align}
For $\hat{\gamma}\in G_{max}$, the action of $\hat{\gamma}$ sends
the virtual cycle $\left[\MMr_{g,k}(\Gamma;
\bgamma,\bvkappa_{j_1\dots j_k} )\right]^{vir}$ based on the data
$(\bgamma,\bvkappa_{j_1\dots j_k} )$ to the virtual cycle
$\left[\MMr_{g,k}(\Gamma; \hat{\gamma}\cdot
\bgamma,\hat{\gamma}\cdot\bvkappa_{j_1\dots j_k} )\right]^{vir}$
which is based on the data
$(\hat{\gamma}\cdot\bgamma,\bvkappa_{j_1\dots
j_k}(\hat{\gamma}\cdot\bgamma))$. The tuple of $k$ critical points
$\bvkappa_{j_1\dots
j_k}(\hat{\gamma}\cdot\bgamma)=\left(\kappa^-_{j_1}(\hat{\gamma}\cdot\gamma_1),
\dots,\kappa^-_{j_k}(\hat{\gamma}\cdot\gamma_k)\right)$. Each
critical point $\kappa^-_{j_i}(\hat{\gamma}\cdot\gamma_i)$
corresponds to a Lefschetz thimble
$S^-_{j_i}(\hat{\gamma}\cdot\gamma_i)$ in the middle dimensional
homology class determined by $\hat{\gamma}\cdot\gamma_i$. By
Wall-crossing formula, the action of $\hat{\gamma}$ to the virtual
cycle is realized by the corresponding braid group action (ref.
Proposition \ref{prop-linear}). On the other hand, the
$\hat{\gamma}$ also gives a dual action to the dual Lefschetz
thimbles $S^+_{j_i}(\hat{\gamma}\cdot\gamma_i)$ and so gives a dual
action to the state space $\ch_{W,G}$. Hence, the group $G_{max}$
acts on the virtual cycle $\left[\MMr(\Gamma)\right]^{vir}$ and the
virtual cycle is $G_{max}$-invariant by Proposition
\ref{prop:cycle-indep-para}.

\begin{proof}[{\bf Proof of Theorem~\ref{thm:main}}]

The primary complication is the numerical
   factors.

   Recall following lemma (Lemma 4.2.3 of \cite{FJR2}).
   \begin{lm}\label{lm:RamifCover}
For a diagram of schemes or DM stacks,
    \begin{equation}\label{eq:PushPullDiag}\begin{diagram}
    W&\rTo^{i} &Y\\
    \dTo^{ p} & &\dTo^{ q}\\
    Z &\rTo^{j} & X
    \end{diagram},\end{equation}
where $i$ and $j$ are regular imbeddings of the same codimension,
$p$ and $q$ are finite morphisms, and there exists a finite
surjective morphism $f:W \rTo Z\times_{X}Y$ with $p = pr_1\circ f$
and $i=pr_2 \circ f$, or such that there exists a finite
surjective morphism $f:Z\times_{X}Y\rTo W$ with $pr_1 = p\circ f$
and $pr_2=i\circ f$, then for any $c\in H_*(M_2,\Q)$ we have
\begin{equation}\frac{p_*i^*(c)}{\deg(p)}=\frac{j^*q_*(c)}{\deg(q)}.\end{equation}
\end{lm}
We will use this lemma several times. All the diagrams in the
proof satisfy the conditions of Lemma~\ref{lm:RamifCover}.

Recall that
$$[\WW(\Gamma)]^{vir}=\frac{1}{\deg
so_{\Gamma}}(so_{\Gamma})_*[\MMr(\Gamma)]^{vir}.$$

   The Dimension and Symmetric group invariance axioms follows
   trivially from the corresponding axioms for the  rigidified space.

   Let's prove the Degenerating connected graph axioms.

  Consider diagram
   $$\begin{array}{ccc}
   \MMr(\Gamma)&\stackrel{\tilde{j}}{\to}& \MMr_{g,k}(\bgamma)\\
   \downarrow \so_{\Gamma}&   &\downarrow \so\\
   \WW(\Gamma) &\stackrel{\tilde{j}}{\to}& \WW_{g,k}(\bgamma)
   \end{array}.$$
   It implies that
   $$\frac{1}{\deg(\so_{\Gamma})}(\so_{\Gamma})_*\tilde{j}^*c=\frac{1}{\deg(\so)}\tilde{j}^*\so_*c$$
   for any homology class $c\in H_*(\MMr_{g,k}(\bgamma), \Q).$

   Let $c=[\MMr_{g,k}(\bgamma)]^{vir}$. We obtain
   $$\tilde{j}^*\so_*[\MMr_{g,k}(\bgamma)]^{vir}=\frac{\deg
   (\so)}{\deg(\so_{\Gamma})}(\so_{\Gamma})_*\tilde{j}^*[\MMr_{g,k}(\bgamma)]^{vir}.$$
   Using the formula $[\MMr(\Gamma)]^{vir}=\tilde{j}^*[\MMr_{g,k}(\bgamma)]^{vir}$
   from the  corresponding axiom for the rigidified space, we obtain
   $$\tilde{j}^*[\WW_{g,k}(\bgamma)]^{vir}=\frac{1}{\deg
   (\so)}\tilde{j}^*\so_*[\MMr_{g,k}(\bgamma)]^{vir}=\frac{1}{\deg(\so_{\Gamma})}(\so_{\Gamma})_*[\MMr(\Gamma)]^{vir}
   =[\WW(\Gamma)]^{vir}.$$

   The disconnected graph axiom and the topological Euler class axiom for the narrow sector follow
   trivially from the corresponding axioms of rigidified space.

    Slightly more attention is required for the Composition axiom.
   It is easy to check that
   $$\so_{\Gamma}^*[\WW(\Gamma)]^{vir}=\frac{1}{\deg(\so(\Gamma))}
   \so_{\Gamma}^*(\so_{\Gamma})_*[\MMr(\Gamma)]^{vir}=[\MMr(\Gamma)]^{vir}.$$
   Hence,
   $$\frac{1}{deg
   (\so_{\Gamma_{cut}})}(\so_{\Gamma_{cut}})_*( \so_{\Gamma}\circ\tilde{\rho})^*[\WW(\Gamma)]^{vir}
   =\frac{1}{\deg(\so_{\Gamma_{cut}})}(\so_{\Gamma_{cut}})_*\tilde{\rho}^*[\MMr(\Gamma)]^{vir}$$
   $$=\frac{1}{\deg(\so_{\Gamma_{cut}})}(\so_{\Gamma_{cut}})_*<[\MMr(\Gamma_{cut})]^{vir}>_{\pm}
   =<[\WW(\Gamma_{cut})]^{vir}>_{\pm}.$$
   Recall that $F=\WW(\Gamma_{cut})\times_{\MM(\Gamma)}
   \WW(\Gamma).$ There is a natural map
   $$\tau: \MMr(\Gamma_{cut})\rTo F$$
   such that
   $$\so_{\Gamma}\circ \tilde{\rho}=pr_2\circ \tau,
   \so_{\Gamma_{cut}}=q\circ \tau.$$
   Hence,
   $$\frac{1}{\deg
   (\so_{\Gamma_{cut}})}(\so_{\Gamma_{cut}})_*( \so_{\Gamma}\circ\tilde{\rho})^*[\WW(\Gamma)]^{vir}
   =\frac{1}{\deg
   (\so_{\Gamma_{cut}})}q_*\tau_*\tau^*pr^*_2
   [\WW(\Gamma)]^{vir}$$
   $$=\frac{\deg(\tau)}{\deg
   (\so_{\Gamma_{cut}})}q_*pr^*_2[\WW(\Gamma)]^{vir}=\frac{1}{\deg(q)}q_*pr^*_2[\WW(\Gamma)]^{vir}.$$

   Let's consider the forgetting tail axiom. We have the fiber diagram
    $$\begin{array}{ccc}
    \MMr(\Gamma)&\stackrel{\vartheta}{\to}&\MMr(\Gamma')\\
    \downarrow \so_{\Gamma} &&       \downarrow \so_{\Gamma'}\\
    \WW(\Gamma)&\stackrel{\vartheta}{\to}&\WW(\Gamma')
    \end{array}.$$

    Hence,
   $$\vartheta^*[\WW(\Gamma')]^{vir}=\frac{1}{\deg(\so_{\Gamma'})}\vartheta^*
   (\so_{\Gamma'})_*[\MMr(\Gamma')]^{vir}=\frac{1}{\deg(\so_{\Gamma})}(\so_{\Gamma})_*
   \vartheta^*[\MMr(\Gamma)]^{vir}$$
   $$=\frac{1}{\deg(\so_{\Gamma})}(\so_{\Gamma})_*[\MMr(\Gamma)]^{vir}=[\WW(\Gamma)]^{vir}.$$
   For the second part of the forgetting tail axiom, we  observe
   $\frac{1}{|G/<\gamma>|^2}\sum_{\psi_1, \psi_2}\sum_{i,j}\alpha_1\eta^{ij}_{\psi_1\circ\psi^{-1}_2}\beta_j$
   corresponds to the Casimir element of pairing of invariant
   homology
   $$<,>:
   H_{N_{\gamma}}(\C^N_{\gamma}, W^{-\infty}_{\gamma},
   \Q)^{G}\otimes H_{N_{\gamma^{-1}}}(\C^N_{\gamma^{-1}}, W^{-\infty}_{\gamma^{-1}},
   \Q)^{G}\to \Q.$$
   Therefore, $[\MMr_{0,3}(\gamma, \gamma^{-1}, J^{-1})]^{vir}$ is
   $\frac{\left|\frac{G}{<J^{-1}>}\right|\left|\frac{G}{<\gamma>}\right|^2}{|G|}$ times the
   Casimir element of the pairing on invariant homology. On the other hand, the degree of $\so$ is
   $\left|\frac{G}{<J^{-1}>}\right|\left|\frac{G}{<\gamma>}\right|^2$. Therefore,
   $[\WW_{0,3}(\gamma, \gamma^{-1}, J^{-1})]^{vir}=\frac{1}{|G|}$ times the Casimir element of invariant
   homology.

   The proof of the Sum of Singularities axiom follows from the fact
   that all the relevant maps are fiber products. The deformation
   invariance axiom and the $G_{max}$-invariance axiom are obvious.

 Let $\tilde{W}=W+Z$ where $Z$ has no common monomials with $W$.
There is a natural inclusion map of components
  $$i:  \overline{\tilde{\W}}_{g,k}\rTo \WW_{g,k}(\bgamma).$$
  Furthermore, there is a natural isomorphism
  $$H_{N_{\gamma_i}}(\C^N_{\gamma_i},
  (\tilde{W})^{\infty}_{\gamma_i}, \Q)\cong H_{N_{\gamma_i}}(\C^N_{\gamma_i},
  (W)^{\infty}_{\gamma_i}, \Q).$$
We can restrict the moduli space to the components
$\overline{\tilde{\W}}_{g,k}(\bgamma)$. For the Witten equation, we
consider a family of quasi-homogeneous polynomials
$\tilde{W}_t=W+tZ$. Then, a simple cobordism argument  yields

\begin{thm}
$[\WW_{g,k,G}(\bgamma)]^{vir}=\overline{\tilde{\W}}_{g,k}(\bgamma)\cap
[\WW_{g,k}(\bgamma)]^{vir}$.
\end{thm}
\end{proof}

\

\bibliographystyle{amsplain}

\providecommand{\bysame}{\leavevmode\hbox
to3em{\hrulefill}\thinspace}

\end{document}